\documentclass[a4paper,11pt]{amsart}
\usepackage[latin1]{inputenc}
\usepackage[greek,english]{babel}
\usepackage[all]{xy}
\usepackage{amscd}
\usepackage{amssymb}
\usepackage{amsthm}
\usepackage{enumitem}
\usepackage{mathrsfs,bbm}
\usepackage[dvipdfmx]{xcolor,graphicx}
\newtheoremstyle{component}{}{}{}{}{\itshape}{.}{.5em}{\thmnote{#3}{#1}}
\theoremstyle{plain}
\newtheorem{thm}{Theorem}[section]
\newtheorem{pro}[thm]{Proposition}
\newtheorem{lem}[thm]{Lemma}
\newtheorem{cor}[thm]{Corollary}
\newtheorem{propdef}[thm]{Proposition-Definition}
\newtheorem{thmdef}[thm]{Theorem-Definition}
\newtheorem{Thm}{Theorem}

\theoremstyle{definition}
\newtheorem{defn}[thm]{Definition}
\newtheorem{conj}[thm]{Conjecture}

\theoremstyle{remark}
\newtheorem{rem}[thm]{Remark}
\newtheorem{ex}[thm]{Example}
\theoremstyle{component}
\newtheorem*{component}{}
\newenvironment{claim}[1]{\par\addvspace{\baselineskip}\noindent\underline{\bfseries
Claim}:\space#1}{\par\addvspace{\baselineskip}}
\def\ker{\mathrm{Ker}}
\DeclareFontEncoding{OT2}{}{}
\DeclareFontSubstitution{OT2}{cmr}{m}{n}
\DeclareFontFamily{OT2}{cmr}{\hyphenchar\font45}
\DeclareFontShape{OT2}{cmr}{m}{n}{%
<5><6><7><8><9>gen*wncyr%
<10><10.95><12><14.4><17.28><20.74><24.88>wncyr10}{}
\DeclareFontShape{OT2}{cmr}{b}{n}{%
<5><6><7><8><9>gen*wncyb%
<10><10.95><12><14.4><17.28><20.74><24.88>wncyb10}{}
\DeclareMathAlphabet{\mathcyr}{OT2}{cmr}{m}{n}
\DeclareMathAlphabet{\mathcyb}{OT2}{cmr}{b}{n}
\SetMathAlphabet{\mathcyr}{bold}{OT2}{cmr}{b}{n}
\makeatletter
\def\theequation{\thesection.\@arabic\c@equation}
\@addtoreset{equation}{section}
\makeatother
\title[Iwasawa Main Conjecture for CM Hilbert cuspforms]{The Cyclotomic Iwasawa Main Conjecture \\
for Hilbert cuspforms with complex multiplication}
\author{Takashi Hara}
\author{Tadashi Ochiai}
\address[T.~Hara]{School of Science and Technology for Future Life \\
Tokyo Denki University \\
5~Senju-Asahi-cho, Adachi-ku, Tokyo 120-8551, Japan}
\address[T.~Ochiai]{Graduate School of Science\\
Osaka University\\
1-1 Machikaneyama-cho, Toyonaka, Osaka 560-0043, Japan}
\email[T.~Hara]{t-hara@mail.dendai.ac.jp}
\email[T.~Ochiai]{ochiai@math.sci.osaka-u.ac.jp}
\thanks{This work is partially supported by KAKENHI (Grant-in-Aid for Exploratory Research: Garant Number 24654004, 
Grant-in-Aid for Scientific Research (B): Grant Number 26287005, and Grant-in-Aid for Young Scientists (B): Grant Number 26800014).}
\usepackage{}	
\usepackage{amsmath}	
\begin{document}
\maketitle
\begin{abstract}
We deduce the cyclotomic Iwasawa main conjecture for Hilbert modular
 cuspforms with complex multiplication from the multivariable main
 conjecture for CM number fields. To this end, we study in detail 
the behaviour of the $p$\nobreakdash-adic $L$-functions and 
the Selmer groups attached to CM number fields 
under specialisation procedures. 
\end{abstract}
\tableofcontents 
\section{Introduction} \label{sc:Introduction}
The {\em cyclotomic Iwasawa main conjecture} for elliptic cuspforms 
describes the mysterious relation between the {\em Selmer groups} 
(algebraic objects) and the {\em $p$-adic $L$-functions}  
(analytic objects) attached to elliptic cuspforms. 
When the cuspform $f$ under consideration does not have {\em complex multiplication}, 
the cyclotomic Iwasawa main conjecture for $f$ is valid under some technical conditions 
thanks to \cite[Theorem 17.4]{Kato} and \cite[Theorem 3.6.4]{SU}.  
When the cuspform $f$ under consideration has complex multiplication, 
one can deduce the validity of the cyclotomic Iwasawa 
main conjecture for $f$ from the {\em two-variable Iwasawa main conjecture} 
for the imaginary quadratic field $F$ via the cyclotomic specialisation 
(such an observation has been already made when $f$ is a cuspform of weight two associated to 
a CM elliptic curve at \cite{Rubin_CM}, for example). 
The main purpose of this article is to
generalise this procedure and deduce the cyclotomic Iwasawa main
conjecture for {\em Hilbert modular cuspforms with complex multiplication}
from the multivariable Iwasawa main conjecture for CM fields.

\medskip
Let $p\geq 5$ be a prime number which is fixed throughout the article. We also fix 
a complex embedding $\iota_\infty :\overline{\mathbb{Q}} \hookrightarrow \mathbb{C}$ 
and a $p$-adic embedding $\iota_p: \overline{\mathbb{Q}} \hookrightarrow \overline{\mathbb{Q}}_p$ 
of an algebraic closure $\overline{\mathbb{Q}} $ of the field of rationals $\mathbb{Q}$.   
Let $F$ be a CM number field of degree $2d$ and $F^+$ its maximal totally real subfield. 
We assume that all prime ideals of $F^+$ lying above $p$ are unramified
over $\mathbb{Q}$ and split completely in the quadratic extension $F/F^+$. 
We denote by $I_{F}$ (resp.\ $I_{F^+}$) the set of all embeddings of $F$
(resp.\ $F^+$) into the fixed algebraic closure $\overline{\mathbb{Q}}$ of 
the rational number field. We choose and fix a $p$-ordinary CM type $\Sigma$ of
$F$, which is a subset of $I_F$ satisfying several conditions 
(see Section~\ref{sssc:CMtypes} for the definition of $p$-ordinary CM types). 

We denote by $\widetilde{F}$ the composition of all
$\mathbb{Z}_p$-extensions of $F$ (in $\overline{\mathbb{Q}}$). 
It is well known that $\mathrm{Gal}(\widetilde{F} /F)$ is isomorphic 
to the free $\mathbb{Z}_p$-module of rank 
$d+1+\delta_{F,p}$, where $\delta_{F,p}$ denotes the Leopoldt defect for 
$F$ and $p$. 
We abbreviate the composite field of $\widetilde{F}$ and $F(\mu_p)$ 
as $\widetilde{F}_\infty$. For each finite abelian extension 
$\mathsf{K}$ of $F$ which contains $F(\mu_p)$ and is linearly disjoint
from $\widetilde{F}$ over $F$, we define
$\widetilde{\mathsf{K}}^{\mathrm{CM}}_\infty$ as
the composite field $\mathsf{K}\widetilde{F}$.    

We consider a gr\"o{\ss}encharacter $\eta$ of type $(A_0)$ on $F$. 
Assume that $\eta$ is {\em ordinary with respect to 
the $($fixed$)$ $p$-ordinary CM type $\Sigma$}, or in other words, assume that $\eta$ is unramified at the 
set $\Sigma_p$ of places of $F$ corresponding to the $p$-adic embeddings $\iota_p \circ \sigma$ for $\sigma$ in $\Sigma$. 
By virtue of global class field theory,
there exists a canonical $p$-adic Galois character 
$\eta^{\mathrm{gal}}\colon G_F
\rightarrow \overline{\mathbb{Q}}^\times_p$ corresponding to $\eta$ (we
shall review the construction of $\eta^\mathrm{gal}$ in
Section~\ref{sssc:grossen}).  
Then, as we shall recall later in Section~\ref{sssc:comparison}, 
there exists a finite abelian extension 
$K/F$ and a character $\psi \colon \mathrm{Gal}(K/F) \rightarrow
\overline{\mathbb{Q}}^\times_p$ of finite order
such that $\eta^{\mathrm{gal}}\psi^{-1}$ factors through the Galois
groups of $\widetilde{F}_\infty /F$. 
Throughout Introduction, we denote by $\mathcal{O}$ the ring of integers of an appropriate finite extension of $\mathbb{Q}_p$ which contains the image of $\eta^\mathrm{gal}$.

Now let us briefly explain  
our main results without precision of notation. 
In Section~\ref{ssc:KHT} we shall introduce the notion of
Katz, Hida and Tilouine's 
$p$-adic $L$-function $\mathcal{L}_{p,\Sigma}^{\mathrm{KHT}} (F)$ for the CM number field $F$, 
which is constructed as an element of 
$\widehat{\mathcal{O}}^{\mathrm{ur}}[[\mathrm{Gal}(F_{\mathfrak{C}p^\infty}/F)]]$ (see Theorem \ref{thm:KHT} for details) 
where $\widehat{\mathcal{O}}^{\mathrm{ur}}$ denotes the composition $\mathcal{O} \widehat{\mathbb{Z}}_p^{\mathrm{ur}}$. 
We denote by $\mathcal{L}_p^\Sigma(\psi)$ (a certain modification of) 
the $\psi$-branch of $\mathcal{L}_{p,
\Sigma}^\mathrm{KHT} (F)$, the image of
$\mathcal{L}_{p,\Sigma}^\mathrm{KHT}(F)$  under  
the $\psi$-twisting map $\widehat{\mathcal{O}}^{\mathrm{ur}}[[\mathrm{Gal}(F_{\mathfrak{C}p^\infty}/F)]]
\rightarrow \widehat{\mathcal{O}}^{\mathrm{ur}}[[\mathrm{Gal}
(\widetilde{F}_\infty /F)]]; g \mapsto \psi(g)g|_{\widetilde{F}_\infty}$. 
For an arbitrary continuous character $\rho \colon \mathrm{Gal}(\widetilde{F}_\infty /F) \rightarrow \mathcal{O}^\times$, 
we denote by $\mathrm{Tw}_\rho$ the $\rho$-twisting map 
$$
\widehat{\mathcal{O}}^{\mathrm{ur}}[[\mathrm{Gal} (\widetilde{F}_\infty /F)]]
\longrightarrow 
\widehat{\mathcal{O}}^{\mathrm{ur}}[[\mathrm{Gal} (\widetilde{F}_\infty /F)]];\ 
g \mapsto \rho (g) g .
$$ 
Under these settings, we state our main results as follows. The first
theorem concerns the analytic part of our main results.

\begin{Thm}[=Corollary~$\ref{cor:compare_p-adic_L}$] \label{Thm:L-func} 
Let $\eta$ be a gr\"o{\ss}encharacter of type $(A_0)$ on $F$ which is 
ordinary with respect to a $p$-ordinary CM type $\Sigma$. 
We choose a finite abelian extension $K$ of $F$ and a character $\psi$ of
 $\mathrm{Gal}(K/F)$ such that $\eta^{\mathrm{gal}}\psi^{-1}$ factors through the Galois
groups of $\widetilde{F}_\infty /F$. 
Then the image of the twisted $($multivariable$)$ $p$-adic $L$-function
$\mathrm{Tw}_{\eta^{\mathrm{gal}}\psi^{-1}}(\mathcal{L}_p^\Sigma(\psi ))$ under
 the cyclotomic specialisation map  
$$
\widehat{\mathcal{O}}^{\mathrm{ur}}[[\mathrm{Gal} (\widetilde{F}_\infty
 /F)]] \rightarrow \widehat{\mathcal{O}}^{\mathrm{ur}}[[\mathrm{Gal}
 (F(\mu_{p^\infty})/F)]];\ g\mapsto g\vert_{F(\mu_{p^\infty})}
$$  
coincides, up to a $($componentwise$)$ nonzero constant multiple, 
with the cyclotomic $p$\nobreakdash-adic $L$\nobreakdash-function 
$\mathcal{L}^{\mathrm{cyc}}_p (\vartheta (\eta )) \in \mathcal{O}^{\mathrm{ur}}[[\mathrm{Gal}
 (F(\mu_{p^\infty})/F)]]  \otimes_{\mathbb{Z}_p} \mathbb{Q}_p$ 
associated to the Hilbert modular form $\vartheta (\eta)$ 
obtained as the theta lift of $\eta$. 
Furthermore if the conjecture on the ratio of complex periods  
 $($see Conjecture~$\ref{conj:periods}$ for details$)$ is true,
 they coincide up to a $($componentwise$)$ $p$-adic unit multiple.
\end{Thm}

The algebraic parts of our main results consist of two ingredients: 
the exact control theorem for the Selmer groups (Theorem~\ref{Thm:control}) 
and the triviality of pseudonull submodules for 
the Pontrjagin duals of the  (strict) Selmer groups (Theorem~\ref{Thm:psnull}). 
We denote by $\mathcal{A}^\mathrm{CM}_\eta$ (resp.\ by $\mathcal{A}^\mathrm{cyc}_\eta$) the cofree module of corank one 
over $\mathcal{O}[[\mathrm{Gal}(\widetilde{F}_\infty /F)]]$ (resp.\ over
$\mathcal{O}[[\mathrm{Gal}(F(\mu_{p^\infty})/F)]]$) on which an
element $g$ of the absolute Galois group $G_F$ acts as the
multiplication by $\eta^\mathrm{gal}(g)g|_{\widetilde{F}_\infty}$ (resp.\ by
$\eta^\mathrm{gal}(g)g|_{F(\mu_{p^\infty})}$). 
Let us define $\mathfrak{A}^\mathrm{cyc}$ 
as the kernel of the cyclotomic specialisation map 
$\mathcal{O}[[\mathrm{Gal}(\widetilde{F}_\infty /F)]] 
\twoheadrightarrow \mathcal{O}[[\mathrm{Gal}(F(\mu_{p^\infty}) /F)]];\ g\mapsto g\vert_{F(\mu_{p^\infty})}$. 
Then we note that $\mathcal{A}^\mathrm{cyc}_\eta$ coincides with
 the maximal
$\mathfrak{A}^\mathrm{cyc}$-torsion submodule $\mathcal{A}^\mathrm{CM}_\eta[\mathfrak{A}^\mathrm{cyc}]$ of $\mathcal{A}^\mathrm{CM}_\eta$. 
By a general recipe due to Greenberg, 
we define the Selmer group $\mathrm{Sel}^{\Sigma}_{\mathcal{A}^\mathrm{CM}_\eta}$ 
(resp.\ $\mathrm{Sel}^{\Sigma}_{\mathcal{A}^\mathrm{cyc}_\eta}$)
as a subgroup of the global Galois cohomology $H^1 (F,\mathcal{A}^\mathrm{CM}_\eta)$ 
(resp.\ $H^1 (F,\mathcal{A}^\mathrm{cyc}_\eta)$) 
satisfying local conditions obtained by ordinary filtrations (refer to Definition
\ref{def:Selmer_group}). The following 
{\em exact control theorem} describes the
behaviour of the Selmer group $\mathrm{Sel}^\Sigma_{\mathcal{A}^\mathrm{CM}_\eta}$
under specialisation procedures:  

\begin{Thm}[=Theorem~$\ref{thm:control}$]  \label{Thm:control}
Let $\eta$, $K$ and $\psi$ be as in Theorem~\ref{Thm:L-func}. 
Assume further that the following condition is fulfilled$:$ 
\begin{description}
\item[(ntr)] 
for each maximal ideal $\mathfrak{M}$ of the semilocal Iwasawa algebra 
$\mathcal{O}[[\mathrm{Gal}(\widetilde{F}_\infty /F)]]$ 
and for each prime ideal  $\mathcal{P}$ of $F$ lying above $p$, 
$\mathcal{A}^\mathrm{CM}_\eta[\mathfrak{M}]$ is nontrivial as a $D_\mathcal{P}$-module.  
\end{description}
Then the natural map 
$$
\mathrm{Sel}^\Sigma_{\mathcal{A}^\mathrm{cyc}_\eta}=\mathrm{Sel}^\Sigma_{\mathcal{A}^\mathrm{CM}_\eta[\mathfrak{A}^\mathrm{cyc}]} \longrightarrow \mathrm{Sel}^\Sigma_{\mathcal{A}^\mathrm{CM}_\eta}[\mathfrak{A}^\mathrm{cyc}] 
$$
induced by the inclusion
 $\mathcal{A}^\mathrm{cyc}_\eta=\mathcal{A}^\mathrm{CM}_\eta[\mathfrak{A}^\mathrm{cyc}]
 \rightarrow \mathcal{A}^\mathrm{CM}_\eta$ is an isomorphism. 
\end{Thm}

Let $M_{\Sigma_p}$ be the maximal abelian pro\nobreakdash-$p$ 
extension of $\widetilde{K}^\mathrm{CM}_\infty$
unramified outside the places of $\widetilde{K}^\mathrm{CM}_\infty$ above the set of places $\Sigma_p$ of $F$ as introduced previously. We denote the 
Galois group of $M_{\Sigma_p}/\widetilde{K}^\mathrm{CM}_\infty$ by
$X_{\Sigma_p}$. 
Then the module $X_{\Sigma_p, \mathcal{O}}=X_{\Sigma_p} \otimes_{\mathbb{Z}_p}\mathcal{O}$ 
is naturally regarded as a compact  
$\mathcal{O}[[\mathrm{Gal}(\widetilde{K}^\mathrm{CM}_\infty/F)]]$-module. 
We denote by $X_{\Sigma_p, (\psi)}$ the $\psi$-isotypic quotient
of $X_{\Sigma_p, \mathcal{O}}$; namely, it is defined as the scalar extension 
$X_{\Sigma_p,\mathcal{O}}  \otimes_{\mathcal{O}[\mathrm{Gal}(K/F)]}
\mathcal{O}[\mathrm{Gal}(F(\mu_p)/F)]$ of $X_{\Sigma_p, \mathcal{O}}$ 
with respect to the $\psi$-twisting map  
$\mathcal{O}[\mathrm{Gal}(K/F)] \rightarrow
\mathcal{O}[\mathrm{Gal}(F(\mu_p)/F)]; g \mapsto
\psi(g)g|_{F(\mu_p)}$. Then one observes that 
the Pontrjagin dual of the Selmer group 
$\mathrm{Sel}^\Sigma_{\mathcal{A}^\mathrm{CM}_\eta}$ is pseudoisomorphic to 
$X_{\Sigma_p, (\psi)} \otimes_\mathcal{O}
\mathcal{O}(\eta^{\mathrm{gal},-1}\psi)$ as a module over 
$\mathcal{O}[[\mathrm{Gal}(\widetilde{F}_\infty/F)]]$ (see
Proposition~\ref{prop:iw_mod} for details). The next theorem implies that 
$X_{\Sigma_p, (\psi)}$ has no nontrivial pseudonull submodules under
certain conditions, analogously to results on algebraic structure of Iwasawa modules in classical Iwasawa theory mainly due to Iwasawa and
Greenberg \cite{gr-iw}:

\begin{Thm}[=Corollary~$\ref{cor:adiv_psi}$] \label{Thm:psnull}
Let $\eta$, $K$ and $\psi$ be as in Theorem~\ref{Thm:L-func}. 
Assume that the 
 cardinality of $\mathrm{Gal}(K/F)$ is relatively prime to $p$ and 
the character $\psi$ satisfies 
the nontriviality condition $(\mathrm{ntr})_\psi$, which is 
described analogously to the condition $(\mathrm{ntr})$ in
 Theorem~\ref{Thm:control} $($see the statement of
 Lemma~$\ref{lem:replace_psi}$ for details$)$. Then 
$X_{\Sigma_p, (\psi)}$ has no nontrivial 
pseudonull $\mathcal{O}[[\mathrm{Gal} (\widetilde{F}_\infty
 /F)]]$-submodules. 
\end{Thm}

Note that $\mathcal{O}[[\mathrm{Gal} (\widetilde{F}_\infty
 /F)]]$ is a semilocal Iwasawa algebra each of 
whose components is isomorphic to $\mathcal{O}[[\mathrm{Gal}(\widetilde{F} /F)]]
$, and the completed group algebra
$\mathcal{O}[[\mathrm{Gal}(\widetilde{F} /F)]]$ is
isomorphic to the formal power series ring in
 $(d+1+\delta_{F,p})$-variables over $\mathcal{O}$. 
Thus, the statement above means that $X_{\Sigma_p, (\psi)}$ has 
no nontrivial pseudonull submodules over each of such components.

Theorem~\ref{Thm:psnull} shall be used to study 
the basechange
compatibility of the characteristic ideal of $X_{\Sigma_p, (\psi)}
\otimes_\mathcal{O} \mathcal{O}(\eta^{\mathrm{gal},-1}\psi)$ (or
 equivalently, of the characteristic ideal of the Pontrjagin dual of $\mathrm{Sel}_{\mathcal{A}^\mathrm{CM}_\eta}^\Sigma$) 
under specialisation procedures. 
We carefully discuss this problem by utilising Theorem~\ref{Thm:psnull} and certain {\em inductive arguments}. For details, 
see Section~\ref{ssc:specialisation}.

Combining these results, we finally obtain the following theorem, which is the main result of this article.

\begin{Thm}[=Theorem~$\ref{thm:IMC}$] \label{Thm:IMC}
Let $\eta$ be a gr\"o{\ss}encharacter of type $(A_0)$ on $F$ which is 
ordinary with respect to a $p$-ordinary CM type $\Sigma$. 
 Assume that  all of the following three conditions are fulfilled{\upshape}
 \begin{itemize}
  \item[-] the gr\"o{\ss}encharacter $\eta$ satisfies the nontriviality condition $(\mathrm{ntr})$ introduced in Theorem~$\ref{Thm:control}${\upshape ;}
  \item[-] the Iwasawa main conjecture for the CM number field
	  $F$ is true for the branch character $\psi$ chosen as above$;$ that is, the equality
	  $$
(\mathcal{L}_p^\Sigma(\psi)) = \mathrm{Char}_{\mathcal{O}[[\mathrm{Gal} (\widetilde{F}_\infty /F)]]} 
(X_{\Sigma_p, (\psi)})
	  $$
	  holds as an equation of ideals of $\widehat{\mathcal{O}}^\mathrm{ur}[[\mathrm{Gal}(\widetilde{F}_\infty/F)]]$ $($see Definition$~\ref{def:char}$ for the definition of the characteristic ideal $\mathrm{Char}_{\mathcal{O}[[\mathrm{Gal} (\widetilde{F}_\infty /F)]]} 
(X_{\Sigma_p, (\psi)})$ over the {\em semilocal} Iwasawa algebra $\mathcal{O}[[\mathrm{Gal}(\widetilde{F}_\infty/F)]])${\upshape ;} 
  \item[-] in each component of the semilocal Iwasawa algebra $\widehat{\mathcal{O}}^\mathrm{ur}[[\mathrm{Gal}(\widetilde{F}_\infty/F)]]$, the cyclotomic $p$-adic $L$-function $\mathcal{L}_p^\mathrm{cyc}(\vartheta(\eta))$ of $\vartheta(\eta)$ does not vanish.
 \end{itemize}
Then the cyclotomic Iwasawa main conjecture for the $p$-stabilised Hilbert eigencuspform $f_\eta=\vartheta (\eta)^{p\text{-st}}$ with complex
 multiplication is true up to $\mu$-invariants$;$ in other
 words, we have the equality of ideals of 
 $\widehat{\mathcal{O}}^\mathrm{ur}[[\mathrm{Gal}(F(\mu_{p^\infty})/F)]]\otimes_{\mathbb{Z}_p}
 \mathbb{Q}_p$  
$$
(\mathcal{L}^{\mathrm{cyc}}_p (\vartheta(\eta))) 
= \mathrm{Char}_{\mathcal{O}[[\mathrm{Gal} (F(\mu_{p^\infty }) /F)]]} (\mathrm{Sel}_{\mathcal{A}_{\vartheta(\eta)}^\mathrm{cyc}}^{\Sigma,\vee}),
$$
where the superscript $\vee$ denotes the Pontrjagin dual. If the
 conjecture on the ratio of complex periods
 $($Conjecture~$\ref{conj:periods})$ is true, the equality above holds
 in $\mathcal{O}[[\mathrm{Gal}(F(\mu_{p^\infty})/F)]]$.
\end{Thm}

Recently Fabio Mainardi and Ming-Lun Hsieh have thoroughly studied 
the Iwasawa main conjecture for CM number fields, and Hsieh's results
combined with Leopoldt conjecture imply its validity for $\widetilde{F}_\infty/F$ under certain technical assumptions (see Remark~\ref{rem:IMC} for details). Hence
Theorem~\ref{Thm:IMC} combined with Hsieh's results and Leopoldt conjecture
guarantees the existence of Hilbert modular cuspforms with complex multiplication
for which the cyclotomic Iwasawa main conjecture is true.

\medskip

The detailed content of this article is as follows. 
After a brief review on basic facts of 
gr\"o{\ss}encharacters of type $(A_0)$ and 
(ad\`elic) Hilbert modular cuspforms, we shall introduce in
Section~\ref{sc:analytic_side} two $p$-adic $L$-functions of different
type: the cyclotomic $p$-adic $L$-functions associated to 
(nearly $p$-ordinary) Hilbert modular cuspforms (in Section~\ref{ssc:p-adic_L_cusp})
and Katz, Hida and Tilouine's $p$-adic $L$-functions for CM number
fields (in Section~\ref{ssc:KHT}).
We compare their interpolation formulae in detail
in Section~\ref{ssc:comparison} when the Hilbert modular
cuspform under consideration has complex multiplication, and then verify
Theorem~\ref{Thm:L-func} based on this comparison.
Section~\ref{sc:algebraic_side} is devoted 
to the algebraic parts of our results. There we introduce various Selmer
groups and compare them at a certain extent of precision
in Section~\ref{ssc:defSel}.
We then prove the Exact Control Theorem
(Theorem~\ref{Thm:control}) in Section~\ref{ssc:control},
and discuss the basechange compatibility of 
characteristic ideals of the Pontrjagin duals of Selmer groups 
by recursively applying Greenberg's criterion 
for almost divisibility in Section~\ref{ssc:specialisation} (we
shall briefly review results of Greenberg on
the almost divisibility of the Selmer groups in Section~\ref{ssc:greenberg}). 
In this process Theorem~\ref{Thm:psnull} is proved as the first step 
of our induction argument (see Section~\ref{sssc:first_step}
for details). Combining all the results obtained in
Sections~\ref{sc:analytic_side} and \ref{sc:algebraic_side}, 
we deduce in Section~\ref{sc:specialisation_IMC}
the cyclotomic Iwasawa main conjecture for Hilbert modular
cuspforms with complex multiplication (at least up to a nonzero constant
multiple) from the (multivariable) Iwasawa main conjecture for CM number
fields via the cyclotomic specialisation (Theorem~\ref{Thm:IMC}). 
In Appendix~\ref{app:cm}, 
we verify several basic properties of Galois representations associated to
Hilbert modular cuspforms with complex multiplication after Ribet's
arguments for elliptic cuspforms with complex multiplication 
employed in \cite{Ribet}. Appendix~\ref{app:duality} is devoted to the verification of Proposition~\ref{prop:commutes} in Section~\ref{sc:algebraic_side}, which is one of the technical keys to our algebraic main results.
\subsubsection*{{\bf Notation}}
%
We mainly use the fraktur $\mathfrak{r}$ for the ring of integers 
of {\em an algebraic number field} (which is often regarded 
as the base field of a certain motive); 
the calligraphic $\mathcal{O}$ 
is kept to denote the ring of integers for {\em a $p$-adic field} 
(which is often regarded as the coefficient field of the $p$-adic 
realisation of a certain motive).
We denote the absolute norm of a fractional ideal $\mathfrak{a}$ 
of an algebraic number field by $\mathcal{N}\mathfrak{a}$.

Throughout this article $p$ denotes a prime number which is larger than
or equal to $5$. 
We fix an algebraic closure $\overline{\mathbb{Q}}$ of the rational number field $\mathbb{Q}$ and 
regard all algebraic number fields (that is, all finite extensions of $\mathbb{Q}$) as subfields of $\overline{\mathbb{Q}}$. 
We also fix an embedding 
$\iota_\infty \colon \overline{\mathbb{Q}} \hookrightarrow \mathbb{C}$ of 
$\overline{\mathbb{Q}}$ into the complex number field $\mathbb{C}$, and 
an embedding $\iota_p \colon \overline{\mathbb{Q}} \hookrightarrow \overline{\mathbb{Q}}_p$ of $\overline{\mathbb{Q}}$ into a fixed algebraic closure $\overline{\mathbb{Q}}_p$ of 
the $p$-adic number field $\mathbb{Q}_p$ respectively.

For an algebraic number field $\mathsf{K}$, we denote by $\mathbb{A}_\mathsf{K}$ 
(resp.\ $\mathbb{A}_\mathsf{K}^\times$) the ring of ad\`eles 
(resp.\ the group of id\`eles) of $\mathsf{K}$. The finite part 
(resp.\ the archimedean part) of the ring of ad\`eles 
$\mathbb{A}_\mathsf{K}$ is denoted by $\mathbb{A}_\mathsf{K}^f$ 
(resp.\ $\mathbb{A}_\mathsf{K}^\infty$). 
We associate a modulus $\prod_v v^{\mathrm{ord}_{v}(x)}$ to every
id\`ele $x$ in $\mathbb{A}_\mathsf{K}^\times$, where 
$v$ runs over all places of $\mathsf{K}$. In this article 
we only consider moduli associated to {\em finite id\`eles}, 
and hence we always identify a modulus $\prod_v v^{\mathrm{ord}_v (x)}$
with the corresponding fractional ideal $\prod_v \mathfrak{P}_v^{\mathrm{ord}_v(x)}$ of $\mathsf{K}$, where $\mathfrak{P}_v$ denotes the prime ideal associated to $v$.
 
We shall fix the notion of {\em the standard additive character} throughout this article. 
For each archimedean place $v$ of an algebraic number field $\mathsf{K}$ and for each $x_v \in \mathsf{K}_v$, 
we define the local additive character $\mathbf{e}_{\mathsf{K}_v} \colon \mathsf{K}_v \rightarrow \mathbb{C}^\times$ by 
\begin{align*}
\mathbf{e}_{\mathsf{K}_v}(x_v)= 
\begin{cases}
\exp(2\pi\sqrt{-1}x_v) & \text{if $v$ is real},  \\
\exp(2\pi\sqrt{-1}x_v\bar{x}_v) & \text{if $v$ is complex}
\end{cases}
\end{align*}
where $\bar{x}_v$ denotes the complex conjugate of $x_v$. For each nonarchimedean place $v$ of $\mathsf{K}$, 
we define $\mathbf{e}_{\mathsf{K}_v}$ as 
\begin{align*}
\mathbf{e}_{\mathsf{K}_v}(x_v)=\exp(-2\pi \sqrt{-1} \mathrm{Tr}_{\mathsf{K}/\mathbb{Q}}(\tilde{x}_v)).
\end{align*}
Here $\tilde{x}_v$ denotes an arbitrary element of $\bigcup_{n=1}^{\infty} \mathfrak{P}_v^{-n}$ 
(regarded as a $\mathfrak{r}_\mathsf{K}$-submodule of $\mathsf{K}$)
such that $\tilde{x}_v-x_v$ is contained in the ring of integers of $\mathsf{K}_{\mathfrak{P}_v}$. The ad\`elic standard 
additive character $\mathbf{e}_{\mathbb{A}_\mathsf{K}} \colon \mathbb{A}_\mathsf{K}/\mathsf{K} \rightarrow \mathbb{C}^\times$ is 
defined as the product of local standard characters for all places of $\mathsf{K}$. 
We also define $\mathbf{e}_{\mathbb{A}_\mathsf{K}^\infty}$ as the product of 
all archimedean local additive characters.
 
Let $\mathbb{C}_p$ be the $p$-adic completion of the fixed algebraic closure $\overline{\mathbb{Q}}_p$ of $\mathbb{Q}_p$ and $\mathcal{O}_{\mathbb{C}_p}$ its 
ring of integers. For the ring of integers $\mathcal{O}$ of a finite extension of $\mathbb{Q}_p$, 
we denote a discrete valuation ring $\mathcal{O} \widehat{\mathbb{Z}}_p^{\mathrm{ur}}$ by ${\widehat{\mathcal{O}}}^\mathrm{ur}$ 
where $\widehat{\mathbb{Z}}_p^{\mathrm{ur}}$ is the $p$-adic completion of the maximal unramified extension of $\mathbb{Z}_p$.
 
In this article we adopt the {\em geometric} normalisation 
of global class field theory; more precisely, let $\mathsf{L/K}$ be 
a finite abelian extension of algebraic number fields. Then 
the reciprocity map $(-,\mathsf{L/K}) \colon \mathbb{A}^\times_\mathsf{K} \rightarrow \mathrm{Gal}(\mathsf{L/K})$ 
is normalised to map a uniformiser $\varpi_{\mathfrak{p}}$ of a prime ideal $\mathfrak{p}$ 
relatively prime to the conductor of the extension $\mathsf{L/K}$ 
to the {\em geometric} Frobenius element 
$\mathrm{Frob}_{\mathfrak{p}}$ in $\mathrm{Gal}(\mathsf{L/K})$; that is, 
$a^{(\varpi_{\mathfrak{p}}, \mathsf{L/K})^{-1}}\equiv a^{q_{\mathfrak{p}}} \, (\text{mod}\, \mathfrak{p})$ 
holds for each $a$ in $\mathfrak{r}_\mathsf{K}$ where $q_{\mathfrak{p}}$ denotes 
the cardinality of the residue field $\mathfrak{r}_\mathsf{K}/\mathfrak{p}$.
%
 
If $\mathsf{K}$ is an algebraic number field, the absolute Galois group
$\mathrm{Gal}(\overline{\mathbb{Q}}/\mathsf{K})$ is denoted by
$G_\mathsf{K}$. For a place $v$ of $\mathsf{K}$, we denote 
by $D_v$ and $I_v$ the decomposition group and the inertia group at $v$
respectively. 
For a (possibly infinite) abelian Galois extension $\mathsf{L/K}$ 
of $\mathsf{K}$ and the ring of integer $\mathcal{O}$ of 
a finite extension of $\mathbb{Q}_p$, we define 
$\mathcal{O}[[\mathrm{Gal}(\mathsf{L/K})]]^\sharp$ as the
free $\mathcal{O}[[\mathrm{Gal}(\mathsf{L/K})]]$-module of rank one 
on which $G_\mathsf{K}$ acts via the universal tautological character
\begin{align*}
G_\mathsf{K} \rightarrow
 \mathcal{O}[[\mathrm{Gal}(\mathsf{L/K})]]^\times ; \ g \mapsto
 g|_\mathsf{L}.
\end{align*}

We finally remark that, as for Hodge-Tate $p$-adic representations,
the Hodge-Tate weights are normalised so that the Hodge-Tate weight of
the $p$-adic cyclotomic character $\chi_{p,\mathrm{cyc}}$ is $-1$.
 %
%
%
\section{The analytic side}\label{sc:analytic_side}
%
%
We shall develop  in this section the analytic parts of our main results. 
We first present a brief overview of classical theory on Hilbert modular
cuspforms in Section~\ref{ssc:classical}, and introduce the notion
of the {\em $p$-adic $L$-functions associated to Hilbert modular cuspforms} in
Section~\ref{ssc:p-adic_L_cusp}. We then introduce another type of
$p$-adic $L$-functions in Section~\ref{ssc:KHT}: {\em Katz, Hida and
Tilouine's $p$-adic $L$-functions for CM number fields}.  
In Section~\ref{ssc:comparison}, we consider the cyclotomic 
specialisation of (appropriately twisted) 
Katz, Hida and Tilouine's $p$-adic $L$-function, 
and compare it with the $p$-adic $L$-function of a Hilbert modular
cuspform with complex multiplication.
%
%
\subsection{Classical theory on Hilbert modular cuspforms} \label{ssc:classical}
%
%
%
This subsection is devoted to an overview 
of classical (complex) theory on (ad\`elic) Hilbert modular cuspforms. 
After a brief review of gr\"o{\ss}encharacters of type $(A_0)$ in
Section~\ref{sssc:grossen},  
we define Hilbert modular cuspforms of double digit 
weight after Hida, and summarise basic facts on  
their Fourier expansions and associated (complex) $L$-functions 
in Section~\ref{sssc:cuspform}. 
Section~\ref{sssc:hecke} is a survey of the theory on Hecke
operators for Hilbert modular cuspforms.
We then introduce in Section~\ref{sssc:HilbGalois} 
the notion of Galois representations associated to Hilbert modular cuspforms.
We finally present the notion of Hilbert modular cuspforms 
with complex multiplication in Section~\ref{sssc:CMforms}, which 
play central roles in the present article. 
%
%
\subsubsection{Generalities on gr\"o{\ss}encharacters of type $(A_0)$} \label{sssc:grossen}
%
%
In this paragraph $\mathsf{K}$ denotes a number field (later we 
only consider the cases where $\mathsf{K}$ is either a totally real
number field $F^+$ or a CM number field $F$). 
We denote by $I_\mathsf{K}$ the set of all embeddings of $\mathsf{K}$ into $\overline{\mathbb{Q}}$. 
Let $S_\mathsf{K}(\mathbb{R})$ (resp.\ $S_\mathsf{K}(\mathbb{C})$) denote 
the set of real places of $\mathsf{K}$ (resp.\ the set of complex places
of $\mathsf{K}$), and let $S_{\mathsf{K},\infty}$ denote the set of
archimedian places of $\mathsf{K}$; that is, $S_{\mathsf{K}, \infty}=S_\mathsf{K}(\mathbb{R})\cup S_\mathsf{K}(\mathbb{C})$. 
Recall that each real place corresponds to a unique element of $I_\mathsf{K}$ and 
each complex place corresponds to a unique pair of elements in $I_\mathsf{K}$. 
For each real place, we denote by $\tau_v$ the corresponding element in $I_\mathsf{K}$.
For each complex place $v$, we specify one of the corresponding pair in $I_\mathsf{K}$ 
as $\tau_v$ and identify $\mathsf{K}_v$ with the complex field $\mathbb{C}$ 
via the embedding $\iota_\infty \circ \tau_v \colon \mathsf{K} \hookrightarrow \mathbb{C}$. Then 
the composite of the other one, which we denote by $\bar{\tau}_v$,  with $\iota_\infty$ is 
the complex conjugation of $\iota_\infty \circ \tau_v$. 
  
An id\`ele class character 
$\eta \colon \mathbb{A}_\mathsf{K}^\times/\mathsf{K}^\times \rightarrow \mathbb{C}^\times$ 
is called {\em a gr\"o{\ss}encharacter of type $(A_0)$} (or also called {\em an algebraic Hecke character}) 
of $\mathsf{K}$ if its archimedean part is {\em algebraic}; namely, 
there exists an element $\mu=\sum_{\tau \in I_\mathsf{K}} \mu_\tau \tau$ of $\mathbb{Z}[I_\mathsf{K}]$ 
such that 
\begin{align*}
\eta(x_\infty)=x_\infty^{-\mu} :=\prod_{v \in S_\mathsf{K}(\mathbb{R})} x_v^{-\mu_{\tau_v}} \prod_{v \in S_\mathsf{K}(\mathbb{C})} x_v^{-\mu_{\tau_v}} \bar{x}_v^{-\mu_{\bar{\tau}_v}}
\end{align*}
holds for each element $x_\infty=(x_v)_{v\in S_{\mathsf{K}, \infty}}$ in the identity component of $\mathbb{A}_{\mathsf{K}}^{\infty, \times}=(\mathsf{K}\otimes_\mathbb{Q} \mathbb{R})^\times$ 
(that is, for each element $x_\infty=(x_v)_{v\in S_{\mathsf{K},
\infty}}$ such that $(x_v)_{v\in S_\mathsf{K}(\mathbb{R})}$ is totally
positive).  
Here $\bar{x}_v$ denotes the usual complex conjugate of $x_v$ in $\mathbb{C}$
(we identify $\mathsf{K}_v$ with $\mathbb{C}$ via the specified identification for each complex place $v$). 
The element $\mu$ as above is called {\em the infinity type of $\eta$}. 
It is widely known that $\eta(x)x_\infty^\mu$ is an algebraic number
for each $x$ in $\mathbb{A}_\mathsf{K}$ (here $x_\infty$ denotes the
archimedean part of $x$). 

For each prime ideal $\mathfrak{l}$ of $\mathsf{K}$, we define $e(\mathfrak{l})$ as the minimum among 
nonnegative integers $e$ such that the local component $\eta_{\mathfrak{l}} \colon \mathsf{K}_{\mathfrak{l}}^\times \rightarrow \mathbb{C}^\times$ 
of $\eta$ at $\mathfrak{l}$ factors through 
$\mathsf{K}_{\mathfrak{l}}^\times/(1+\mathfrak{l}^e) \rightarrow \mathbb{C}^\times$. 
When $\eta$ is unramified at $\mathfrak{l}$, we define $e(\mathfrak{l})$ as $0$. 
The integral ideal $\mathfrak{C}(\eta)=\prod_{\mathfrak{l}}\mathfrak{l}^{e(\mathfrak{l})}$ is called 
{\em the conductor} of $\eta$.
We denote by $\eta^*$ the ideal character associated to $\eta$; namely, $\eta^*$ is the character defined by 
\begin{align}\label{equation:idealclasschar}
\eta^*(\mathfrak{a})=\prod_{\mathfrak{l} \nmid \mathfrak{C}(\eta)} \eta_\mathfrak{l}(\varpi_{\mathfrak{l}}^{\mathrm{ord}_{\mathfrak{l}}(\mathfrak{a})})
\end{align}
for each fractional ideal $\mathfrak{a}$ of $\mathsf{K}$ relatively prime to the conductor $\mathfrak{C}(\eta)$, 
where $\varpi_{\mathfrak{l}}$ denotes a uniformiser of $\mathsf{K}_{\mathfrak{l}}$. Note that the associated ideal character $\eta^*$ 
does not depend on the choice of uniformisers since $\eta$ is unramified at each $\mathfrak{l} \nmid \mathfrak{C}(\eta)$. 
\begin{ex}[norm character] \label{ex:norm}
The ad\`elic norm $\lvert \cdot \rvert_{\mathbb{A}_\mathsf{K}}$ is regarded as 
a gr\"o{\ss}encharacter of type $(A_0)$ by virtue of Artin's 
product formula. It has the infinity type $-\sum_{\tau \in I_\mathsf{K}} \tau$ and 
the conductor $\mathfrak{r}_\mathsf{K}$. The ideal character 
associated to $\lvert \cdot \rvert_{\mathbb{A}_\mathsf{K}}$ is $\mathcal{N}_\mathsf{K}^{-1}$, 
the inverse of the absolute norm defined on $\mathsf{K}$. When $\mathsf{K}$ is totally real, 
an arbitrary gr\"o{\ss}encharacter of type $(A_0)$ defined on
 $\mathsf{K}$ is described as 
 $\phi \lvert \cdot \rvert^n_{\mathbb{A}_\mathsf{K}}$ for a certain
 gr\"o{\ss}encharacter $\phi$ of finite order on $\mathsf{K}$ 
 and a certain integer $n$.
\end{ex}
Now we associate to $\eta$ a $p$-adic id\`ele class character 
$\hat{\eta}\colon \mathbb{A}_\mathsf{K}^{\times}/\mathsf{K}^\times \rightarrow \overline{\mathbb{Q}}_p^\times$ as follows; 
for each prime ideal $\mathfrak{p}$ of $\mathsf{K}$ lying above $p$, we denote by $I_{\mathsf{K}, \mathfrak{p}}$ 
the subset of $I_{\mathsf{K}}$ consisting of embeddings $\tau$ such that $\iota_p \circ \tau$ induces the place 
associated to $\mathfrak{p}$. Then we put 
\begin{align*}
x_p^{-\mu}=\prod_{\mathfrak{p} \mid p\mathfrak{r}_\mathsf{K}} x_{\mathfrak{p}}^{-\sum_{\tau \in I_{\mathsf{K},\mathfrak{p}}}\mu_\tau}, \quad \text{for } x_p=(x_\mathfrak{p})_{\mathfrak{p}\mid p\mathfrak{r}_\mathsf{K}} \in \mathbb{A}_{\mathsf{K}}^{p,\times}=(\mathsf{K}\otimes_\mathbb{Q} \mathbb{Q}_p)^\times.
\end{align*}
For each id\`ele $x$ of $\mathsf{K}$, we define $\hat{\eta}(x)$ as 
\begin{align} \label{eq:padicavatar}
\hat{\eta}(x)=\iota_p(\eta(x)x_{\infty}^\mu) x_p^{-\mu} 
\end{align}
where $x_p$ and $x_\infty$ respectively denote the $p$-component and the $\infty$-component of $x$.
Obviously $\hat{\eta}$ is trivial on $\mathsf{K}^\times$ and factors through the finite id\`ele class group $\mathbb{A}_\mathsf{K}^{f,\times}/\mathsf{K}^{\times}$ by construction.
The $p$-adic id\`ele class character $\hat{\eta}$ constructed as
above is called the {\em $p$-adic avatar} of $\eta$, while the (complex)
character $\eta$ is called the {\em complex avatar} of $\hat{\eta}$. 
Note that, via global class field theory, the $p$-adic avatar $\hat{\eta}$ 
corresponds to a unique $p$-adic Galois character $\eta^{\mathrm{gal}}$ 
defined on $\mathrm{Gal}(\mathsf{K}_{\mathfrak{C}(\eta)}/\mathsf{K})$ satisfying $\eta^\mathrm{gal}((x,\mathsf{K}_{\mathfrak{C}(\eta)}/\mathsf{K}))=\hat{\eta}(x)$ for an arbitrary element $x$ in $\mathbb{A}_\mathsf{K}^\times$ where $\mathfrak{C}(\eta)$ denotes the conductor of $\eta$ and $\mathsf{K}_{\mathfrak{C}(\eta)}$ denotes the ray class field modulo $\mathfrak{C}(\eta)$ over $\mathsf{K}$. Refer to {\em Notation} in Section~\ref{sc:Introduction} for our normalisation of the reciprocity map $(-,\mathsf{K}_{\mathfrak{C}(\eta)/\mathsf{K}})$. In particular, we have $\eta^\mathrm{gal}(\mathrm{Frob}_\mathfrak{l})=\eta(\varpi_\mathfrak{l})=\hat{\eta}(\varpi_\mathfrak{l})=\eta^*(\mathfrak{l})$ for each prime ideal $\mathfrak{l}$ of $\mathsf{K}$ relatively prime to $\mathfrak{C}(\eta)$, where $\varpi_\mathfrak{l}$ denotes a uniformiser of the local field $\mathsf{K}_\mathfrak{l}$.  
\subsubsection{Hilbert modular cuspforms of double-digit weight}
\label{sssc:cuspform}
%
Let us recall the definition of (ad\`elic) Hilbert modular cuspforms. 
 We basically follow Hida's description of ad\`elic Hilbert modular forms in \cite{hecke, HIT}, 
 although there might be several different manners to introduce them. In particular, 
 we adopt his {\em double-digit weight convention} (refer to 
 \cite[Section~2.3.2]{HIT}). 

 Let $F^+$ be a totally real number field and $\mathfrak{r}_{F^+}$ 
 the ring of integers of $F^+$.
 We define an algebraic group $G$ over $\mathbb{Z}$ as the Weil restriction of scalars of the general linear group $GL(2)_{/\mathfrak{r}_{F^+}}$ 
 over $\mathfrak{r}_{F^+}$ from 
 $\mathfrak{r}_{F^+}$ to $\mathbb{Z}$. Let $T_0$ be 
 the diagonal torus of $GL(2)_{/\mathfrak{r}_{F^+}}$ and 
 $T_G$ its Weil restriction of scalars from $\mathfrak{r}_{F^+}$ 
 to $\mathbb{Z}$. The character group $X(T_G)$ of $T_G$ is identified with $\mathbb{Z}[I_{F^+}]\times \mathbb{Z}[I_{F^+}]$; specifically, an element 
 $\kappa=(\kappa_1, \kappa_2)$ of $\mathbb{Z}[I_{F^+}]\times \mathbb{Z}[I_{F^+}]$ 
 corresponds to a unique algebraic character ${T_G}_{/\overline{\mathbb{Q}}} \rightarrow {\mathbb{G}_m}_{/\overline{\mathbb{Q}}}$ which induces
 \begin{align*}
 \begin{pmatrix} x_1 & 0 \\
0 & x_2\end{pmatrix} \mapsto x_1^{\kappa_1}x_2^{\kappa_2} , \quad 
 x_i^{\kappa_i}=\prod_{\tau \in I_{F^+}} \tau(x_i)^{\kappa_{i,\tau}}
  \quad (\in \overline{\mathbb{Q}}^\times) \qquad \text{for }i=1,2
 \end{align*}
 on $T_G(\mathbb{Q})=F^{+,\times} \times F^{+,\times}$. Here $\kappa_i$ denotes an element of $\mathbb{Z}[I_{F^+}]$ defined by $\sum_{\tau
 \in I_{F^+}}\kappa_{i,\tau}\tau$ for each $i=1,2$
 which satisfies the following condition:
 \begin{equation} \label{eq:weight}
  \kappa_{1,\tau}+\kappa_{2,\tau} \,\, \text{is a constant independent of $\tau$ in $I_{F^+}$.}
 \end{equation}
 We denote by $[\kappa]$ the constant value  $\kappa_{1,\tau}+\kappa_{2,\tau}$ when the condition \eqref{eq:weight} is satisfied. 
 Note that the diagonal torus $T_G$ contains the center $Z$ of $G$, namely the subgroup consisting of all scalar matrices. Let us define another algebraic torus 
 $T$ as the Weil restriction of scalars of the multiplicative group 
 ${\mathbb{G}_m}_{/\mathfrak{r}_{F^+}}$ from $\mathfrak{r}_{F^+}$ to
 $\mathbb{Z}$. 
 For an integral ideal $\mathfrak{N}$ of $F^+$, we consider the following congruence subgroups of $G(\hat{\mathbb{Z}})$:
 \begin{align*}
 \hat{\Gamma}_0(\mathfrak{N}) &:= \left\{  \left. \begin{pmatrix} a & b
						  \\ c & d \end{pmatrix}
  \in G(\hat{\mathbb{Z}}) \right| c \in
  \mathfrak{N}\hat{\mathfrak{r}}_{F^+}  \right\}, \\
 \hat{\Gamma}_1(\mathfrak{N}) &:=\left\{ \left. \begin{pmatrix} a & b \\ c & d \end{pmatrix} \in G(\hat{\mathbb{Z}}) \right| 
 a-1 \in \mathfrak{N}\hat{\mathfrak{r}}_{F^+},\  c\in \mathfrak{N}\hat{\mathfrak{r}}_{F^+} \right\}.
 \end{align*}  

 Let $\varepsilon_+ \colon \mathbb{A}_{F^+}^\times/F^{+,\times} \rightarrow
 \mathbb{C}^\times$ denote a gr\"o{\ss}encharacter of type $(A_0)$ on $F^+$
 with infinity type $\kappa_1+\kappa_2-\mathsf{t}$, where
the symbol $\mathsf{t}$ denotes the ``trace'' element 
$\sum_{\tau \in I_{F^+}} \tau$ in $\mathbb{Z}[I_{F^+}]$. 
We often identify the ad\`elic points $Z(\mathbb{A}_\mathbb{Q})$ 
of the center $Z$ of $G$ with the id\`ele group
 $\mathbb{A}_{F^+}^\times$ of $F^+$, and regard $\varepsilon_+$ 
 as a character on $Z(\mathbb{A}_\mathbb{Q})$. 
We denote by $\varepsilon \colon T(\hat{\mathbb{Z}})\rightarrow
 \mathbb{C}^\times$ the restriction of the finite part of
 $\varepsilon_+$ to $T(\hat{\mathbb{Z}})=\hat{\mathfrak{r}}_{F^+}^\times$, 
where $\hat{\mathfrak{r}}_{F^+}^\times$ denotes 
the profinite completion of $\mathfrak{r}_{F^+}^\times$. 
 Then one easily observes that if the conductor
 $\mathfrak{C}(\varepsilon)$ of $\varepsilon$ contains $\mathfrak{N}$, the map 
 \begin{align*}
 \begin{pmatrix} a & b \\ c & d \end{pmatrix} \mapsto \varepsilon(a_{\mathfrak{N}})
 \end{align*}
 defines a continuous character $\hat{\Gamma}_0(\mathfrak{N})\rightarrow
 \mathbb{C}^\times$, for which we use the same symbol  
 $\varepsilon$ by abuse of notation. Here $a_\mathfrak{N}$ denotes the projection of $a$ to $F^+_{\mathfrak{N}}=\prod_{\mathfrak{l}\mid \mathfrak{N}}
 F^+_{\mathfrak{l}}$. We denote the pair of the
 characters $(\varepsilon, \varepsilon_+)$ by $\underline{\varepsilon}$,
 which shall play a role of a nebentypus character.

 Let $\mathfrak{h} \subset \mathbb{C}$ be the Poincar\'e upper half plane which consists of all complex numbers whose imaginary parts are positive. 
 Then the identity component $G(\mathbb{R})^+$ of the
 $\mathbb{R}$-valued points $G(\mathbb{R})=GL_2(\mathbb{R})^{I_{F^+}}$
 of $G$ acts on $\mathfrak{h}^{I_{F^+}}$ via the coordinatewise
 M\"obius transformation. We now introduce the {\em automorphy factor of weight $\kappa=(\kappa_1, \kappa_2)$} by 
 \begin{align*}
 J_\kappa(g,z)=\det(g)^{\kappa_1-\mathsf{t}} j(g,z)^{\kappa_2-\kappa_1+\mathsf{t}} \quad \text{for $g=\begin{pmatrix} a & b \\ c & d \end{pmatrix} \in G(\mathbb{R})^+$ and $z=(z_{\tau})_{\tau \in I_{F^+}} \in \mathfrak{h}^{I_{F^+}}$}
 \end{align*}
 where $j(g,z)$ denotes the vector defined by $(c_{\tau}z_{\tau}+d_{\tau})_{\tau
 \in I_{F^+}}$. Here we use the following abbreviations on multi-indices:
 \begin{align*}
 \det(g)^{\kappa_1-\mathsf{t}}=\prod_{\tau \in I_{F^+}} \det(g_\tau)^{\kappa_{1,\tau}-1}, \quad j(g,z)^{\kappa_2-\kappa_1+\mathsf{t}}=\prod_{\tau \in I_{F^+}} (c_{\tau}z_{\tau}+d_{\tau})^{\kappa_{2,\tau}-\kappa_{1,\tau}+1}.
 \end{align*}

 \begin{defn}[Hilbert modular cuspforms] \label{def:cuspforms}
   Let $\kappa$ be an element of $\mathbb{Z}[I_{F^+}]\times \mathbb{Z}[I_{F^+}]$ for which the condition \eqref{eq:weight} is satisfied, and let $\underline{\varepsilon}=(\varepsilon,\varepsilon_+)$ be as above.
A complex-valued function $f \colon
 G(\mathbb{A}_{\mathbb{Q}})\rightarrow \mathbb{C}$ on the ad\`elic
 points $G(\mathbb{A}_{\mathbb{Q}})$ of $G$ is called a {\em Hilbert
 modular cuspform of weight $\kappa$, level $\mathfrak{N}$ 
 and nebentypus $\underline{\varepsilon}$} if it satisfies 
 the following three conditions:
 \begin{enumerate}[label=(HC\arabic*)]
 \item (automorphy) 
  
\noindent let $C_{\mathbf{i}}$ denote the stabiliser subgroup of the vector $\mathbf{i}=(\sqrt{-1}, \dotsc, \sqrt{-1}) \in \mathfrak{h}^{I_{F^+}}$, which is by definition a subgroup of $G(\mathbb{R})^+$. Then the equality
 \begin{align*}
 f(\alpha xuw)=\varepsilon_+(w)\varepsilon(u_f)f(x)J_\kappa(u_{\infty}, \mathbf{i})^{-1}
 \end{align*}
 holds for each $\alpha$ in $G(\mathbb{Q})$, $w$ in $Z(\mathbb{A}_{\mathbb{Q}})$ and $u=u_fu_\infty$ in $\hat{\Gamma}_0(\mathfrak{N})C_{\mathbf{i}}$; 

 \item (holomorphy)
  
\noindent for each $z$ in $\mathfrak{h}^{I_{F^+}}$ let us choose an element $u_\infty$ of $G(\mathbb{R})^+$ satisfying $u_\infty(\mathbf{i})=z$. Then the function
 \begin{align*}
 f_g \colon \mathfrak{h}^{I_{F^+}}\rightarrow \mathbb{C} ; z\mapsto 
 f(gu_\infty)J_\kappa(u_{\infty}, \mathbf{i})
 \end{align*}
 is holomorphic with respect to $z$ for every $g$ in $G(\mathbb{A}_{\mathbb{Q}}^f)$. 
 Note that $f_g$ is well defined independently of the choice of each $u_\infty$ by virtue of the automorphy (HC1); 

 \item (cuspidality)
  
\noindent the integral
 \begin{align*}
 \int_{\mathbb{A}_{F^+}/F^+} f\left( \begin{pmatrix} 1 & u \\ 0 & 1 \end{pmatrix}x \right) du
 \end{align*}
 vanishes for every $x$ in $GL_2(\mathbb{A}_{F^+}^f)$ where $du$ is an additive Haar measure on $\mathbb{A}_{F^+}/F^+$.
 \end{enumerate}
We denote by $S_\kappa(\mathfrak{N}, \underline{\varepsilon};
 \mathbb{C})$ the complex vector space spanned by all Hilbert
 modular cuspforms of weight $\kappa$, level $\mathfrak{N}$ and
 nebentypus $\underline{\varepsilon}$.
\end{defn}

It is well known that the space of Hilbert modular cuspforms $S_\kappa (\mathfrak{N},
 \underline{\varepsilon}; \mathbb{C})$ is of finite dimension.
 We often impose the following constraints on weights of Hilbert modular cuspforms.
 
 \begin{defn} \label{def:weights}
 Let $\kappa=(\kappa_1, \kappa_2)$ be an element of $\mathbb{Z}[I_{F^+}]\times \mathbb{Z}[I_{F^+}]$ for which the condition \eqref{eq:weight} is satisfied.
 \begin{enumerate}[label=(\arabic*)]
  \item The element $\kappa$ is called a {\em cohomological weight}
	if the inequality $\kappa_1< \kappa_2$ holds.
  \item The element $\kappa$ is called a {\em critical weight} if 
       it is cohomological and the inequality 
$\kappa_1 <0 \leq \kappa_2$ holds. 
 \end{enumerate}
 Here inequalities among elements of $\mathbb{Z}[I_{F^+}]$ abbreviate
  corresponding coefficientwise inequalities. For instance, we use the inequality notation $\kappa_1<\kappa_2$ to express that the inequality $\kappa_{1,\tau}<\kappa_{2,\tau}$ holds for every $\tau$ in $I_{F^+}$. 
 \end{defn}

 From now on we always assume that {\em all double-digit weights considered in this article are cohomological}.

 \medskip
 
 Now let $f$ denote a Hecke eigencuspform of cohomological weight $\kappa=(\kappa_1 ,\kappa_2 )$, 
 level $\mathfrak{N}$ and nebentypus $\underline{\varepsilon}$ 
 (we shall give a brief review on Hecke theory in the next
 subsection). To each eigencuspform $f$, Blasius and Rogawski \cite{BR}
 has attached a motive $M(f)_{/F^+}$ defined over $F^+$
 with coefficients in the Hecke field $\mathbb{Q}_f$ associated to $f$,
 which we will introduce later (see the paragraph after Definition~\ref{def:eigenform}). 
The motive $M(f)_{/F^+}$ is pure of weight $[\kappa]$. For each $\tau$ in $I_{F^+}$, the Hodge type of $M(f)_{/F^+}$ with respect to the complex embedding $\iota_\infty \circ \tau \colon F^+ \hookrightarrow \mathbb{C}$ of the field of definition $F^+$ is given by 
 $\{ (\kappa_{1,\tau}, \kappa_{2,\tau}), (\kappa_{2,\tau}, \kappa_{1,\tau}) \}$. 
 In other words, Hida's double-digit weight convention is adapted to 
 the Hodge type of $M(f)_{/F^+}$, 
 and the weight of the cuspform $f$ is critical if and only if the associated motive $M(f)_{/F^+}$ is critical in the sense of Deligne
 \cite{Deligne}.
 
\begin{rem}[on nebentypus characters] \label{rem:neben}
In \cite[Section~2.3.2]{HIT} Hida introduces more general notion on
 nebentypus characters; namely, he considers as a nebentypus character 
a triple $(\varepsilon_1, \varepsilon_2; \varepsilon_+)$
consisting of finite characters $\varepsilon_1$, $\varepsilon_2$ on
 $T(\hat{\mathbb{Z}})$ and a gr\"o{\ss}encharacter $\varepsilon_+$ of type
 $(A_0)$ on $Z(\mathbb{A}_\mathbb{Q})$ with certain constraints. 
The nebentypus introduced here is a special one of Hida's general
 notion. Indeed, our notion of the nebentypus
 $\underline{\varepsilon}=(\varepsilon, \varepsilon_+)$ corresponds 
to a triple $(\varepsilon, 1_{T(\hat{\mathbb{Z}})}; \varepsilon_+)$, 
which satisfies all the required conditions.
The space of Hilbert modular cuspforms $S_\kappa(\mathfrak{N},
 \underline{\varepsilon}; \mathbb{C})$ with nebentypus of the form 
$\underline{\varepsilon}=(\varepsilon, \varepsilon_+)$ is indeed
 contained in the space of Hilbert modular cuspforms
 $S_\kappa(\hat{\Gamma}_1(\mathfrak{N}); \mathbb{C})$ of 
weight $\kappa$ and level $\hat{\Gamma}_1(\mathfrak{N})$, and 
hence we can apply to them  general theory on Hilbert modular forms and 
Hecke algebras of $\hat{\Gamma}_1$\nobreakdash-level structure developed in
 \cite{shimura} and \cite{hecke}. In particular 
the (ad\`elic) Fourier coefficients depend only on fractional ideals of
 $F^+$ under such constraints on nebentypus (as we shall see later in
 Proposition~\ref{prop:fourier}), and hence our convention on 
nebentypus characters seems to be well suited to
arithmetic problems concerning the $L$-functions associated to Hilbert modular forms. 
\end{rem}

We next recall the notion of the Fourier expansions 
 of Hilbert modular forms, and then finish this subsection by introducing the (complex) $L$-functions associated to them.

 \begin{pro} \label{prop:fourier}
 Let $\mathfrak{d}=\mathfrak{d}_{F^+}$ denote the absolute different of $F^+$ and let $F^{+,\times}_+$ denote the set of all totally positive elements of $F^+$. 
 Then each Hilbert modular cuspform $f$ belonging to $S_\kappa
  (\mathfrak{N}, \underline{\varepsilon}; \mathbb{C})$ has the
  $($ad\`elic$)$ Fourier expansion of the following form for each
  $x \in \mathbb{A}_{F^+}$ and $y\in \mathbb{A}_{F^+}^\times${\upshape :}
\begin{align*}
f\left( \begin{pmatrix} y & x \\ 0 & 1 \end{pmatrix} \right) =
 |y|_{\mathbb{A}_{F^+}} \sum_{\xi \in F^{+,\times}_+} C(\xi y\mathfrak{d}; f) (\xi y_\infty)^{-\kappa_1} 
 \mathbf{e}_{\mathbb{A}_{F^+}^\infty}(\sqrt{-1}\xi y_\infty) \mathbf{e}_{\mathbb{A}_{F^+}} (\xi x).
 \end{align*}
 The correspondence $\mathfrak{a} \mapsto C(\mathfrak{a}; f)$ defines a
  complex-valued function on the group of fractional ideals of $F^+$, which vanishes unless $\mathfrak{a}$ is integral. 
 \end{pro}
 
 This is \cite[Proposition~4.1]{hecke}, the proof of which essentially depends on Shimura's classical computation \cite[(2.18)]{shimura}. 
 Refer also to \cite[Proposition~2.26]{HIT}. We call $C(-; f)$ 
the {\em Fourier coefficient of $f$}. 
 A Hilbert modular cuspform $f$ is said to be {\em normalised} if its Fourier coefficient $C(\mathfrak{r}_{F^+}; f)$ at $\mathfrak{r}_{F^+}$ 
 equals $1$. For a $\mathbb{Q}$-subalgebra $A$ of $\mathbb{C}$, we denote by
 $S_\kappa(\mathfrak{N}, \underline{\varepsilon}; A)$ the $A$-subspace of $S_\kappa (\mathfrak{N}, \underline{\varepsilon}; \mathbb{C})$ 
spanned by cuspforms with all Fourier coefficients contained in $A$.

Now let us assume that the Hilbert modular cuspform $f$ under
consideration is a {\em normalised eigenform}.
The (complex) {\em $L$-function associated to $f$} 
 is defined as (the meromorphic continuation of) the Dirichlet series 
 \begin{align*}
 L(f,s)=\sum_{(0)\neq \mathfrak{a}\subseteq \mathfrak{r}_{F^+}} \frac{C(\mathfrak{a};f)}{\mathcal{N}\mathfrak{a}^s}.
 \end{align*}
 For a gr\"o{\ss}encharacter $\theta \colon \mathbb{A}_{F^+}^\times/F^{+,\times} \rightarrow \mathbb{C}^\times$ of type $(A_0)$, we also define 
 {\em the $L$-function associated to $f$ twisted by $\theta$} as (the
 meromorphic continuation of)
 \begin{align*}
 L(f, \theta, s)=\sum_{(0)\neq\mathfrak{a}\subseteq \mathfrak{r}_{F^+}} \frac{C(\mathfrak{a};f)\theta^*(\mathfrak{a})}{\mathcal{N}\mathfrak{a}^s},
 \end{align*}
 where $\theta^\ast$ is the ideal class character which was associated with $\phi$ in $\eqref{equation:idealclasschar}$ 
and $\theta^*(\mathfrak{a})$ is defined to be zero if $\mathfrak{a}$ is not relatively prime to the conductor of $\theta$. 

 \subsubsection{Review on Hecke theory} \label{sssc:hecke}
 
We shall briefly recall Hecke theory on ad\`elic Hilbert modular
cuspforms after \cite[Section~2]{hecke}. As in the previous paragraph, 
we consider the space
$S_\kappa(\mathfrak{N}, \underline{\varepsilon};\mathbb{C})$ of 
Hilbert modular cuspforms of weight $\kappa$, level $\mathfrak{N}$ and 
nebentypus $\underline{\varepsilon}$. Recall that it is contained in 
the space $S_\kappa(\hat{\Gamma}_1(\mathfrak{N}); \mathbb{C})$ of
cuspforms of weight $\mathfrak{N}$ and level $\hat{\Gamma}_1(\mathfrak{N})$.
Now let us define the following monoids:
 \begin{align*}
 \Delta_0(\mathfrak{N})&=\left\{ \left. \begin{pmatrix} a & b \\ c & d
				       \end{pmatrix} \in
  M_2(\hat{\mathfrak{r}}_{F^+})\cap G(\mathbb{A}_{\mathbb{Q}}^f) \right|
  a_\mathfrak{N} \in \mathfrak{r}_{F^+, \mathfrak{N}}^\times,\ c\in
  \mathfrak{N}\hat{\mathfrak{r}}_{F^+} \right\}, \\
 \Delta_1(\mathfrak{N})&=\left\{ \left. \begin{pmatrix} a & b \\ c & d
				       \end{pmatrix} \in
  M_2(\hat{\mathfrak{r}}_{F^+})\cap G(\mathbb{A}_{\mathbb{Q}}^f) \right|
  a-1 \in \mathfrak{N}\hat{\mathfrak{r}}_{F^+},\ c\in
  \mathfrak{N}\hat{\mathfrak{r}}_{F^+} \right\}.
 \end{align*}
Then $\Delta_0(\mathfrak{N})$ contains $\hat{\Gamma}_0(\mathfrak{N})$
and $\Delta_1(\mathfrak{N})$ contains $\hat{\Gamma}_1(\mathfrak{N})$
respectively. We thus consider the action of the double coset algebra
$R(\hat{\Gamma}_1(\mathfrak{N}), \Delta_1(\mathfrak{N}))$ on
$S_\kappa(\hat{\Gamma}_1(\mathfrak{N}); \mathbb{C})$. We refer the readers to \cite[Section~3]{shimura_intro} for details on the theory of double coset algebras. The action of $R(\hat{\Gamma}_1(\mathfrak{N}), \Delta_1(\mathfrak{N}))$ on $S_\kappa(\hat{\Gamma}_1(\mathfrak{N}); \mathbb{C})$ is defined as follows.
For a cuspform $f$ in $S_\kappa(\hat{\Gamma}_1(\mathfrak{N}),
\mathbb{C})$ and an element 
$[\hat{\Gamma}_1(\mathfrak{N})y\hat{\Gamma}_1(\mathfrak{N})]$ of
$R(\hat{\Gamma}_1(\mathfrak{N}), \Delta_1(\mathfrak{N}))$, we set 
\begin{align*}
f\vert_{[\hat{\Gamma}_1(\mathfrak{N})y\hat{\Gamma}_1(\mathfrak{N})]}(g)=\sum_{i=1}^h
 f(gy_i^{-1})
\end{align*}
where $\{y_i \}_{i=1, \dotsc, h}$ is a representative of the left coset
decomposition of
$\hat{\Gamma}_1(\mathfrak{N})y\hat{\Gamma}_1(\mathfrak{N})$:
\begin{align*}
\hat{\Gamma}_1(\mathfrak{N})y\hat{\Gamma}_1(\mathfrak{N})=\bigcup_{i=1}^h
 \hat{\Gamma}_1(\mathfrak{N})y_i.
\end{align*}
Then one easily observes that the action of
 $R(\hat{\Gamma}_1(\mathfrak{N}), \Delta_1(\mathfrak{N}))$ preserves 
the space $S_\kappa(\mathfrak{N}, \underline{\varepsilon}; \mathbb{C})$
 of Hilbert cuspforms with nebentypus $\underline{\varepsilon}$. 
To describe the structure of the double coset algebra
 $R(\hat{\Gamma}_1(\mathfrak{N}), \Delta_1(\mathfrak{N}))$, we here introduce 
typical double coset operators which are often quoted 
as {\em Hecke operators}. 
 Choose a uniformiser $\varpi_{\mathfrak{l}}$ of the local field $F^+_{\mathfrak{l}}$ for 
 each prime ideal $\mathfrak{l}$ of $F^+$, and regard 
 $\begin{pmatrix} 1 & 0 \\ 0 & \varpi_{\mathfrak{l}} \end{pmatrix}$ 
 as an element of $\Delta_1(\mathfrak{N})$ whose local component is 
 the identity matrix at every place except for $\mathfrak{l}$. 
 Then we define 
 \begin{align*}
 T(\mathfrak{l}) &= \hat{\Gamma}_1(\mathfrak{N}) \begin{pmatrix}
 1 & 0 \\ 0 & \varpi_\mathfrak{l} \end{pmatrix}
  \hat{\Gamma}_1(\mathfrak{N}) \quad \text{if $\mathfrak{l}$ does not divide $\mathfrak{N}$}, \\
 U(\mathfrak{l}) &= \hat{\Gamma}_1(\mathfrak{N}) \begin{pmatrix}
 1 & 0 \\ 0 & \varpi_\mathfrak{l} \end{pmatrix}
  \hat{\Gamma}_1(\mathfrak{N}) \quad \text{if $\mathfrak{l}$ divides $\mathfrak{N}$}
 \end{align*}
 for each prime ideal of $F^+$. They are determined independently 
 of the choice of uniformisers. Next let $\mathfrak{b}$ be an integral ideal of $F^+$ relatively prime to $\mathfrak{N}$. 
 For such an ideal $\mathfrak{b}$, choose a finite id\`ele
 $\mathsf{b}\in \hat{\mathfrak{r}}_{F^+}\cap \mathbb{A}_{F^+}^\times$ so
 that it is congruent to $1$ modulo
 $\mathfrak{N}\hat{\mathfrak{r}}_{F^+}$ and its associated modulus 
 coincides with $\mathfrak{b}$. Then we set 
 \begin{align*}
T(\mathfrak{b}, \mathfrak{b})=\hat{\Gamma}_1(\mathfrak{N}) \begin{pmatrix} \mathsf{b} & 0 \\ 0 & \mathsf{b} \end{pmatrix} \hat{\Gamma}_1(\mathfrak{N}),
 \end{align*} 
 which does not depend on the choice of the auxiliary id\`ele
 $\mathsf{b}$. We also use the notation $S(\mathfrak{b})$ 
for $T(\mathfrak{b}, \mathfrak{b})$ as in \cite{shimura}. 
 By virtue of the general theory, 
the double coset algebra $R(\hat{\Gamma}_1(\mathfrak{N}), \Delta_1(\mathfrak{N}))$ is commutative and is freely 
 generated as a $\mathbb{Z}$-algebra by $T(\mathfrak{l})$ for prime ideals relatively prime to $\mathfrak{N}$, $U(\mathfrak{l})$ for prime ideals dividing $\mathfrak{N}$ and $S(\mathfrak{l})$ for prime ideals relatively prime to $\mathfrak{N}$. Moreover 
we obtain the following formula which one can adopt as the definition 
 of the operator $T(\mathfrak{l}^e)$ for prime ideals relatively prime to $\mathfrak{N}$:
 \begin{align*}
 T(\mathfrak{l})T(\mathfrak{l}^e)=T(\mathfrak{l}^{e+1})
 +\mathcal{N}(\mathfrak{l})S(\mathfrak{l})T(\mathfrak{l}^{e-1}) \quad \text{for each $e\geq 1$}.
 \end{align*}
 It is also known that $U(\mathfrak{l}^e)=U(\mathfrak{l})^e$ holds for every
 prime ideal dividing $\mathfrak{N}$ and an arbitrary natural number $e$.

\begin{rem}
In \cite{HIT} and many other articles of Hida, 
the action of $R(\hat{\Gamma}_0(\mathfrak{N}),
 \Delta_0(\mathfrak{N}))$ is defined as in the following manner. 
We first extend the nebentypus character
 $\underline{\varepsilon}=(\varepsilon, \varepsilon_+)$
 to $\Delta_0(\mathfrak{N})$ by just setting 
 \begin{align*}
 \begin{pmatrix} a & b \\ c & d \end{pmatrix} \mapsto \varepsilon(a_{\mathfrak{N}}).
 \end{align*}

Now let $\iota$ denote
 the standard positive involution on $(2\times 2)$-matrices defined by 
\begin{align*}
\begin{pmatrix} a & b \\ c & d \end{pmatrix} \mapsto \begin{pmatrix} d &
						      -b \\ -c & a
						     \end{pmatrix}.
\end{align*}

 For each $x$ in $\Delta_0(\mathfrak{N})$, we define an action of the double coset $[\hat{\Gamma}_0(\mathfrak{N})x \hat{\Gamma}_0(\mathfrak{N})]$ 
on a Hilbert modular cuspform $f$ belonging to 
$S_\kappa(\mathfrak{N}, \underline{\varepsilon}; \mathbb{C})$ by 
 \begin{align*}
  f\vert_{[\hat{\Gamma}_0(\mathfrak{N})x \hat{\Gamma}_0
 (\mathfrak{N})]}(g)=\sum_{i=1}^h \underline{\varepsilon}(x_i^{-\iota})f(gx_i^\iota)
 \end{align*}
 where $\{ x_i\}_{i=1, \dotsc, h}$ is a representative set of the 
 left coset decomposition of
 $\hat{\Gamma}_0(\mathfrak{N})x\hat{\Gamma}_0(\mathfrak{N})$. 
For an element $y$ of $\Delta_1(\mathfrak{N})$, the action of the $\hat{\Gamma}_1(\mathfrak{N})$-double coset $[\hat{\Gamma}_1(\mathfrak{N})y\hat{\Gamma}_1(\mathfrak{N})]$ is compatible with that of the $\hat{\Gamma}_0(\mathfrak{N})$-double coset $[\hat{\Gamma}_0(\mathfrak{N})y\hat{\Gamma}_0(\mathfrak{N})]$. Indeed we obtain
\begin{align*}
\sum_{i=1}^h f(gy_i^{-1})=\sum_{i=1}^hf(g(\det y_i)^{-1}y_i^\iota)
 =\sum_{i=1}^h\varepsilon_+(\det y_i)^{-1}f(gy_i^\iota)
\end{align*}
for
 $\hat{\Gamma}_1(\mathfrak{N})y\hat{\Gamma}_1(\mathfrak{N})=\bigcup_{i=1}^h
 \hat{\Gamma}_1(\mathfrak{N})y_i$, but then $\varepsilon_+(y_i^{-1})$ equals
 $\varepsilon(y_i^{-\iota})$ for every $i$ because $y_i$ is an element of $\Delta_1(\mathfrak{N})$.
\end{rem}

 \begin{defn}[Eigenforms] \label{def:eigenform}
 A Hilbert modular form $f$ belonging to $S_\kappa (\mathfrak{N}, \underline{\varepsilon}; \mathbb{C})$ 
 is called {\em an eigenform} if it is a common eigenvector with respect
  to all double coset operators (or Hecke operators) belonging to $R(\hat{\Gamma}_1(\mathfrak{N}), \Delta_1(\mathfrak{N}))$. 
 \end{defn}

 By the well-known formula
 \begin{align} \label{eq:hecke}
 C(\mathfrak{a}; \left. f\right|_{T(\mathfrak{n})}) =\sum_{\mathfrak{b}\mid \mathfrak{a},\ \mathfrak{b}\mid \mathfrak{n},\ (\mathfrak{b}, \mathfrak{N})=1} \mathcal{N}(\mathfrak{b})C(\mathfrak{b}^{-2}\mathfrak{na}; \left. f\right|_{S(\mathfrak{b})}) 
 \end{align}
 (see \cite[Corollary~4.2]{hecke} for instance), one observes that the Fourier coefficient $C(\mathfrak{l}; f)$ of an eigencuspform $f$ at a prime ideal $\mathfrak{l}$ 
 is obtained as the multiple of $C(\mathfrak{r}_{F^+}; f)$ and the
 eigenvalue with respect to the Hecke operator $T(\mathfrak{l})$ (or $U(\mathfrak{l})$ 
 if $\mathfrak{l}$ divides $\mathfrak{N}$). Recall that the eigenvalues of the Hecke operators acting on $S_\kappa (\mathfrak{N}, \underline{\varepsilon}; \mathbb{C})$  are algebraic numbers due to \cite[Proposition~2.2]{shimura}. Therefore an eigencuspform $f$ belongs to $S_\kappa(\mathfrak{N}, \underline{\varepsilon}; \overline{\mathbb{Q}})$ if it is normalised. We denote by $\mathbb{Q}_f$ the field which one obtains by adjoining to $\mathbb{Q}$ the eigenvalues of the Hecke operators acting on $f$, which we call the {\em Hecke field} associated to $f$.

 To introduce the notion of near ordinarity, we recall the notion of
 {\em normalised Hecke operators} after \cite[Section~3]{hecke}.
 Let $\mathbb{Q}(\kappa_1)$ denote a field which one obtains by
 adjoining to $\mathbb{Q}$ all the elements of the form $x^{\kappa_1}$ for 
 $x$ in $F^{+,\times}$.  Then $\mathbb{Q}(\kappa_1)$ is 
a finite extension of $\mathbb{Q}$. Note that  $\kappa_1$ naturally induces 
 homomorphisms $F^{+,\times}\rightarrow \mathbb{Q}(\kappa_1)^\times$ and  
 $\mathbb{A}_{F^+}^\times \rightarrow \mathbb{A}_{\mathbb{Q}(\kappa_1)}^\times$. It is known that 
there exists an $\mathfrak{r}_{\mathbb{Q}(\kappa_1)}$-subalgebra $A$ of $\mathbb{C}$ 
 satisfying the following condition:
 \begin{quote}
 for each element $x$ in $\mathbb{A}_{F^+}^{f, \times}$, the modulus associated to $x^{\kappa_1}$ is principal as a fractional ideal of $A$.
 \end{quote}
Refer to \cite[p.\ 310]{hecke} for details of the existence of such an
 algebra $A$. Let us choose a uniformiser $\varpi_{\mathfrak{l}}$ 
 of $F^+_{\mathfrak{l}}$ for each prime ideal $\mathfrak{l}$ of $F^+$. 
We take a generator of the modulus associated 
to $\varpi_{\mathfrak{l}}^{\kappa_1}$ (as a fractional ideal of $A$) 
and denote it 
by $\{ \mathfrak{l}^{\kappa_1} \}$. We also 
define $\{ \mathfrak{a}^{\kappa_1} \}$ 
 for an arbitrary fractional ideal $\mathfrak{a}=\prod_{\mathfrak{l}\text{: prime}} \mathfrak{l}^{e(\mathfrak{l})}$ 
 by $\prod_{\mathfrak{l}\text{: prime}} \{\mathfrak{l}^{\kappa_1}
 \}^{e(\mathfrak{l})}$. 

\begin{defn}[Normalised Hecke operators]
 We define elements $T_0(\mathfrak{l})$, $U_0(\mathfrak{l})$ and
 $S_0(\mathfrak{b})$ of $R(\hat{\Gamma}_1(\mathfrak{N}),
 \Delta_1(\mathfrak{N}))\otimes_{\mathbb{Z}} A$ in the following manner;
 \begin{align*}
T_0(\mathfrak{l})&=\{\mathfrak{l}^{\kappa_1}\}^{-1}T(\mathfrak{l}) \quad
  \text{for a prime ideal $\mathfrak{l}$ which does not divide $\mathfrak{N}$}, \\
U_0(\mathfrak{l})&=\{\mathfrak{l}^{\kappa_1}\}^{-1}U(\mathfrak{l}) \quad
  \text{for a prime ideal $\mathfrak{l}$ which divides $\mathfrak{N}$}, \\
S_0(\mathfrak{b})&=\{\mathfrak{b}^{\kappa_1}\}^{-2}S(\mathfrak{b}) \quad
  \text{for an integral ideal $\mathfrak{b}$ which is relatively prime to $\mathfrak{N}$}.
 \end{align*}
The operators $T_0(\mathfrak{l})$, $U_0(\mathfrak{l})$ and 
 $S_0(\mathfrak{b})$ are called {\em normalised Hecke operators}. 
\end{defn}

 \begin{defn}[Near ordinarity]
 Let $\mathfrak{p}$ be a prime ideal of $F^+$ lying above $p$. A normalised eigencuspform $f$ 
 belonging to $S_\kappa(\mathfrak{N}, \underline{\varepsilon}; \overline{\mathbb{Q}})$ is said to be {\em nearly ordinary at $\mathfrak{p}$} 
 (or {\em nearly $\mathfrak{p}$\nobreakdash-ordinary}) if the eigenvalue of $f$ with respect to the normalised Hecke operator $T_0(\mathfrak{p})$ (or $U_0(\mathfrak{p})$ if $\mathfrak{p}$ divides the level $\mathfrak{N}$) 
 is a $p$-unit under the specified embedding $\iota_p \colon
  \overline{\mathbb{Q}}\rightarrow \overline{\mathbb{Q}}_p$. 

A normalised eigencuspform $f$ which is nearly ordinary at 
 all prime ideals $\mathfrak{p}$ lying above $p$ is said to be 
 {\em nearly ordinary at $p$} (or {\em nearly $p$-ordinary}).
 \end{defn}

 Note that the notion of the normalised Hecke operators {\em does depend} 
 on the choice of a generator of $\{ \mathfrak{p}^{\kappa_1} \}$, 
 but the notion of the near ordinarity {\em does not depend} on it since the $p$-adic valuation of 
 $\{\mathfrak{p}^{\kappa_1} \}$, which is regarded as an element of $\overline{\mathbb{Q}}_p$ via the fixed 
 embedding $\iota_p$, is well defined independently of the choice of its generator. 
 The normalisation of Hecke operators discussed above is crucial in Hida's theory
 on $p$-adic Hecke algebras. One of the reasons why it is important is that normalised Hecke operators 
 preserve the space of cuspforms with rational or integral coefficients. 
 We omit the details; see \cite[Section~4]{hecke} and \cite[Section~2.3.3]{HIT}.

We finally introduce the notion of {\em $p$-stabilisation}. 
 Define an operator 
 \begin{align*}
 V(\mathfrak{b}) \colon S_{\kappa}(\mathfrak{N}, \underline{\varepsilon}; \overline{\mathbb{Q}}) \rightarrow S_\kappa(\mathfrak{bN}, \underline{\varepsilon}; \overline{\mathbb{Q}})
 \end{align*}
 for every integral ideal $\mathfrak{b}$ of $F^+$ by 
 \begin{align*}
 \left. f\right|_{V(\mathfrak{b})}(g)=|\mathsf{b}|_{\mathbb{A}_{F^+}} f\left( g 
 \begin{pmatrix} \mathsf{b}^{-1} & 0 \\ 0 & 1 \end{pmatrix} \right)
 \end{align*}
 where $\mathsf{b}$ is a finite id\`ele of $F^+$ whose
 associated modulus coincides with $\mathfrak{b}$. We readily see that 
 the Fourier coefficient of $\left. f\right|_{V(\mathfrak{b})}$ 
 at $\mathfrak{a}$ is given by $C(\mathfrak{b}^{-1}\mathfrak{a}; f)$. 
 Now let $f$ denote a normalised eigencuspform $f$ belonging to $S_{\kappa}(\mathfrak{N}, \underline{\varepsilon}; \mathbb{\overline{Q}})$ 
 and assume that $f$ is nearly ordinary at $p$. Let $\mathfrak{p}$ 
 be a prime ideal of $F^+$ lying above $p$ and suppose that 
 the level $\mathfrak{N}$ is not divisible by $\mathfrak{p}$. 
 Note that the eigenvalue of $f$ with respect to 
 the Hecke operator $T_0(\mathfrak{p})$ is calculated as
 $\{\mathfrak{p}^{\kappa_1} \}^{-1}C(\mathfrak{p};f)$ which we denote 
 by $C_0(\mathfrak{p};f)$.
 The eigenvalue of $f$ with respect to $S_0(\mathfrak{p})$ is also
 calculated as $\{\mathfrak{p}^{\kappa_1}\}^{-2}\varepsilon^*_+(\mathfrak{p})$. 
 Consider the Hecke polynomial of $f$ with respect to the normalised
 Hecke operator $T_0(\mathfrak{p})$; in other words,
 consider the quadratic polynomial defined by
 \begin{align*}
 1-C_0(\mathfrak{p};f)X-\mathcal{N}\mathfrak{p} \{\mathfrak{p}^{\kappa_1}\}^{-2} \varepsilon_+^*(\mathfrak{p})X^2=(1-\alpha_{0, \mathfrak{p}}X)(1-\beta_{0, \mathfrak{p}}X).
 \end{align*}
 We denote two roots of this polynomial (regarded as an elements of $\overline{\mathbb{Q}}$) by $\alpha_{0,\mathfrak{p}}$ and 
 $\beta_{0,\mathfrak{p}}$. Since $f$ is nearly ordinarity at $\mathfrak{p}$, the Hecke eigenvalue 
 $C_0(\mathfrak{p};f)=\alpha_{0,\mathfrak{p}}+\beta_{0,\mathfrak{p}}$ of
 $f$ with respect to $T_0(\mathfrak{p})$ is a $p$-adic unit. 
 This implies that one of the roots has to be a $p$-adic unit 
 (under the fixed embedding $\iota_p$), which we specify as $\alpha_{0, \mathfrak{p}}$. 
 Let us consider the cuspform $f-\left. \{\mathfrak{p}^{\kappa_1}\}
 \beta_{0, \mathfrak{p}} f\right|_{V(\mathfrak{p})} \in S_\kappa(\mathfrak{pN}, \underline{\varepsilon}; \overline{\mathbb{Q}})$. Since $\{ \mathfrak{p}^{\kappa_1} \} \beta_{0, \mathfrak{p}}$ is a root of the quadratic polynomial $1-C(\mathfrak{p}; f)X+\mathcal{N}\mathfrak{p} \varepsilon_+^*(\mathfrak{p})X^2$, 
 the form $f-\{ \mathfrak{p}^{\kappa_1} \}\beta_{0, \mathfrak{p}} \left. f\right|_{V(\mathfrak{p})}$ does not depend on the 
 choice of a generator of $\{ \mathfrak{p}^{\kappa_1} \}$. It has 
 the same eigenvalues as those of $f$ everywhere except at $\mathfrak{p}$ 
 and has the eigenvalue $\alpha_{0, \mathfrak{p}}$ with respect to $U_0(\mathfrak{p})$; 
 hence it is nearly ordinary at $\mathfrak{p}$. Repeating the same procedure for 
 all prime ideals lying above $p$ which do not divide 
 the level $\mathfrak{N}$ of $f$, we obtain the {\em $p$-stabilisation
 $f^{p\text{-st}}$ of $f$}, which is a normalised eigencuspform 
 of level $\mathfrak{N}\prod_{\mathfrak{p}\mid p\mathfrak{r}_{F^+},
 \mathfrak{p}\nmid\mathfrak{N}}\mathfrak{p}$. 
A normalised nearly $p$-ordinary eigencuspform is called a {\em
 $p$-stabilised newform} if it is obtained as the $p$-stabilisation of 
a certain primitive form (in the sense of Miyake
 \cite[p.~185]{miyake}). In particular, the level of a (nearly $p$-ordinary)
 $p$-stabilised newform is divisible by every prime ideal $\mathfrak{p}$
 of $F^+$ lying above $p$.

\subsubsection{Galois representations associated to Hilbert modular cuspforms} \label{sssc:HilbGalois}

We here introduce the notion of {\em Galois representations
associated to Hilbert modular cuspforms}. 

\begin{thmdef} \label{thmdef:Galois_rep}
Let $f$ be a normalised eigencuspform of cohomological weight $\kappa$, 
level $\mathfrak{N}$ and nebentypus $\underline{\varepsilon}$.
Let $\mathcal{K}$ be 
a finite extension of $\mathbb{Q}_p$ containing the Hecke field
$\mathbb{Q}_f$ of $f$ $($under the fixed embedding $\iota_p \colon
 \overline{\mathbb{Q}} \hookrightarrow \overline{\mathbb{Q}}_p)$. 
Then there exists a $2$-dimensional Galois representation 
$V_f$ of $G_{F^+}$ with coefficients in $\mathcal{K}$ such that the equality
\begin{align*}
\det(1-\mathrm{Frob}_\mathfrak{q} X ;
 V_f)=1-C(\mathfrak{q};f)X+\mathcal{N}\mathfrak{q} \varepsilon_+^*(\mathfrak{q})X^2
\end{align*}
holds for every prime ideal $\mathfrak{q}$ which does not divide
$p\mathfrak{N}$. Moreover $V_f$ is an irreducible representation 
of $V_f$.
The Galois representation $V_f$ of $G_{F^+}$ 
is called the {\em Galois representation associated to $f$}.
\end{thmdef}

The existence of such a Galois representation is 
established due to results of many people 
including Ohta \cite{ohta}, Carayol \cite{carayol}, 
 Wiles \cite{wiles}, Taylor \cite{taylor1}, and
 Blasius and Rogawski \cite{BR}. The irreducibility of $V_f$ is 
verified due to Taylor in \cite[Theorem~3.1]{taylor2}
by the same argument as Ribet's one in \cite{Ribet}. 
Note that the Galois representation $V_f$ is uniquely 
determined up to scalar multiples by virtue of 
its irreducibility combined with \v{C}ebotarev's density theorem.

When the Hilbert cuspform $f$ is nearly $p$-ordinary, we can also obtain
precise information on the local behaviour of the associated representation
$V_f$ at places above $p$.

\begin{pro} \label{prop:localGalois}
Let $f$ be a normalised eigencuspform of cohomological weight $\kappa$, 
 level $\mathfrak{N}$ and nebentypus $\underline{\varepsilon}$ which is
 nearly ordinary at $p$ and let the other notation be as in Theorem-Definition~$\ref{thmdef:Galois_rep}$. Then, for each place $\mathfrak{p}$ of $F^+$ lying above $p$, the Galois representation $V_f$ associated to $f$ contains a unique $1$-dimensional $D_\mathfrak{p}$-stable $\mathcal{K}$-subspace $\mathrm{Fil}_\mathfrak{p}^+V_f$ on which the decomposition group $D_\mathfrak{p}$ of $G_{F^+}$ at $\mathfrak{p}$ acts via the $p$-adic character character $\delta_\mathfrak{p}\colon D_\mathfrak{p}\rightarrow \mathcal{K}^\times$ satisfying
\begin{align} \label{eq:unramified_char}
  \delta_\mathfrak{p}(\mathrm{Frob}_{\varpi_\mathfrak{p}})=\alpha_f(\varpi_\mathfrak{p})
\end{align}
 for every uniformiser $\varpi_{\mathfrak{p}}$ of $F^+_\mathfrak{p}$. Here $\mathrm{Frob}_{\varpi_\mathfrak{p}}=(\varpi_\mathfrak{p}, F^{+,\mathrm{ab}}_\mathfrak{p}/F^+_\mathfrak{p})$ denotes the Frobenius element corresponding to the uniformiser $\varpi_\mathfrak{p}$ via the local reciprocity map. 
The value $\alpha_{f}(\varpi_\mathfrak{p})$ appearing in $(\ref{eq:unramified_char})$ is a unique $p$-adic unit root of the quadratic polynomial
  \begin{align} \label{eq:unitroot}
 1-\varpi_\mathfrak{p}^{-\kappa_{1,\mathfrak{p}}}C(\mathfrak{p};f)X-\varpi_{\mathfrak{p}}^{-2\kappa_{1,\mathfrak{p}}}\mathcal{N}\mathfrak{p} \varepsilon_+^*(\mathfrak{p})X^2=(1-\alpha_f(\varpi_\mathfrak{p})X)(1-\beta_f(\varpi_\mathfrak{p})X)
  \end{align}
 where $C(\mathfrak{p};f)$ denotes the eigenvalue of the Hecke operator $U(\mathfrak{p})$ with respect to $f$, and $\kappa_{1,\mathfrak{p}}=\sum_{\tau \colon (\iota_p\circ \tau)^{-1}(\mathfrak{M}_{\overline{\mathbb{Z}}_p})=\mathfrak{p}}\kappa_{1,\tau}$ denotes the summation of $\kappa_{1,\tau}$ over all $\tau \colon F^+\hookrightarrow \overline{\mathbb{Q}}$ such that $\iota_p\circ \tau$ induces the place $\mathfrak{p}$.
\end{pro}

Note that the near $p$-ordinarity of $f$ guarantees that the quadratic equation \eqref{eq:unitroot} indeed has a unique $p$-adic unit root. We also remark that the character $\delta_\mathfrak{p}$ is not unramified in general, and hence the equation \eqref{eq:unramified_char} {\em does depend} on the choice of a uniformiser $\varpi_\mathfrak{p}$ of $F^+_\mathfrak{p}$, contrary to the cases of $p$-ordinary modular form (of parallel weight).

Proposition~\ref{prop:localGalois} is first observed by Mazur and Wiles for $p$-ordinary elliptic modular forms of weight 2 \cite[Chapter~3., Section~2]{MazurWiles} and then verified by Wiles for $p$-ordinary Hilbert modular forms of parallel weight in \cite[Theorem~2.2]{Wiles_prep} and \cite[Theorem~2]{wiles}. The general cases have been verified by Hida in \cite[Theorem I]{nearlyhecke}. The $D_\mathfrak{p}$-stable filtration $\mathrm{Fil}_\mathfrak{p}^+ V_f \subset V_f$ is used to define the local condition at $\mathfrak{p}$ of Greenberg's Selmer group (see also Section~\ref{sssc:general_Selmer}).

\subsubsection{Hilbert modular cuspforms with complex multiplication} \label{sssc:CMforms}

In this subsection we introduce the notion of Hilbert modular cuspforms
with {\em complex multiplication}. The following definition is due to Ribet 
for elliptic modular forms \cite[Section~3]{Ribet}:
 
 \begin{defn}[Cuspform with complex multiplication] \label{def:cm}
  Let $\nu \colon \mathbb{A}_{F^+}^\times/F^{+,\times} \rightarrow
  \mathbb{C}^\times$ be a {\em nontrivial} gr\"o{\ss}encharacter of finite order on $F^{+,\times}$. A Hilbert modular eigencuspform $f$
  of weight $\kappa$, level
  $\mathfrak{N}$ and nebentypus $\underline{\varepsilon}$ is
  said to {\em have complex multiplication by $\nu$}
  if the equality
 \begin{align} \label{eq:cm}
 C(\mathfrak{l}; f)=\nu^*(\mathfrak{l})C(\mathfrak{l}; f)
 \end{align}
  holds for all prime ideals $\mathfrak{l}$ in a set of prime ideals
  of density $1$ in $\mathfrak{r}_{F^+}$. 
  
A  Hilbert modular eigencuspform $f$ of weight $\kappa$, level
 $\mathfrak{N}$ and nebentypus $\underline{\varepsilon}$ is said to 
  {\em have complex multiplication} if $f$ has complex multiplication
  by a certain nontrivial
 gr\"o{\ss}encharacter $\nu\colon \mathbb{A}_{F^+}^\times/F^{+,\times}
 \rightarrow \mathbb{C}^\times$ of finite order on $F^+$.
\end{defn}

 The right hand side of (\ref{eq:cm}) is naturally regarded as the Fourier coefficient at $\mathfrak{l}$ of $f\otimes \nu$: the cuspform 
$f$ twisted by the gr\"o{\ss}encharacter $\nu$ of finite order (see \cite[Proposition 4.5]{shimura} for the definition of $f\otimes \nu$).  
The nebentypus of $f \otimes \nu$ is easily calculated as $\nu^2 \varepsilon$, and by comparing eigenvalues of the Hecke operator $S(\mathfrak{l})$ at $f$ and $f\otimes \nu$, we obtain the equality
 \begin{align*}
 \varepsilon_+^*(\mathfrak{l})=\nu^{*}(\mathfrak{l})^2\varepsilon^*_+(\mathfrak{l})
 \end{align*}
 for every prime ideal $\mathfrak{l}$ prime to levels of $f$ and $f \otimes \nu$. Then \v{C}ebotarev's density theorem 
 forces $\nu^2$ to be trivial, and consequently $\nu$ must be a quadratic character if $f$ has complex multiplication by $\nu$.

 \medskip
 
 An example of Hilbert eigencuspforms which has complex
 multiplication is obtained as the {\em theta lift} of a
 gr\"o{\ss}encharacter of type $(A_0)$ of a CM number fields.
 In order to introduce the notion of theta lifts, let us consider 
 a totally imaginary quadratic extension $F$ of $F^+$, which is by construction a CM number field. We denote by $c$ a unique nontrivial element
 of the Galois group $\mathrm{Gal}(F/F^+)$, which is none other than the conplex conjugation. Now let $\Sigma$ denote a CM type of $F$, that is, a subset
 of $I_F$ such that $I_F$ is decomposed into the disjoint union of $\Sigma$ and $\Sigma^c=\{\, \sigma \circ c \mid \sigma \in \Sigma \,\}$. 
 Then we have a canonical bijection $\Sigma \longrightarrow I_{F^+}$  
 via the restriction $\sigma \mapsto \sigma|_{F^+}$.
 Since each complex place corresponds to a unique element $\sigma$ of the fixed $p$-ordinary CM-type 
$\Sigma$, we identify $\Sigma$ with the set of archimedean places of $F$ and, for each $\sigma \in \Sigma$, 
we specify the identification of $F_\sigma$ with $\mathbb{C}$ 
via the embedding $\iota_\infty \circ \sigma \colon F\hookrightarrow \mathbb{C}$; in other words, 
we identify $\mathbb{A}_F^{\infty}=F\otimes_{\mathbb{Q}} \mathbb{R}$ with $\mathbb{C}^\Sigma$ via the 
isomorphism induced by $x\otimes 1 \mapsto (\iota_\infty \circ \sigma(x))_{\sigma \in \Sigma}$.

Under these identifications, the infinity type $\mu$ 
of a gr\"o{\ss}encharacter $\eta \colon \mathbb{A}_F^{\times}/F^{\times} \rightarrow \mathbb{C}^{\times}$ 
of type $(A_0)$ is described as
$\mu=\sum_{\sigma \in \Sigma} (\mu_{\sigma}\sigma+\mu_{\bar{\sigma}}\bar{\sigma})$, where 
$\bar{\sigma}=\sigma \circ c$ denotes a unique element in $\Sigma^c$ corresponding to $\sigma \in \Sigma$. 
 A gr\"o{\ss}encharacter $\eta$ of type $(A_0)$ on $F$ is 
 said to be {\em $\Sigma$-admissible} (or {\em admissible with respect
to  $\Sigma$}) if its infinity type $\mu$ satisfies $\mu_{\sigma} < \mu_{\bar{\sigma}}$ for every $\sigma$ in $\Sigma$. Given an $\Sigma$-admissible infinity type $\mu$ on $F$, we define a cohomological double digit weight $\kappa_\mu =(\kappa_{\mu,1}, \kappa_{\mu, 2}) \in \mathbb{Z}[I_{F^+}]\times \mathbb{Z}[I_{F^+}]$ on $F^+$ by
\begin{align} \label{eq:weight_cm}
\kappa_{\mu} = \left( \sum_{\sigma \in \Sigma} \mu_{\sigma}
  \left. \sigma \right|_{F^+},\ \sum_{\sigma \in \Sigma} \mu_{\bar{\sigma}}
  \left. \bar{\sigma} \right|_{F^+} \right). 
\end{align}
Given a gr\"o{\ss}encharacter $\eta \colon \mathbb{A}_F^\times/F^\times \rightarrow \mathbb{C}$ of type $(A_0)$ on $F$, we define $\underline{\varepsilon}_{\eta}=(\varepsilon_{\eta},
 \varepsilon_{\eta,+})$ by 
$$ 
\underline{\varepsilon}_\eta =( \left. \nu_{F/F^+}
  \breve{\eta}\right|_{T(\hat{\mathbb{Z}})},\ \nu_{F/F^+} \breve{\eta}) 
$$
where $\breve{\eta}$ is defined to be  $\left. \eta\right|_{\mathbb{A}_{F^+}^\times} | \, \cdot \, |_{\mathbb{A}_{F^+}^\times}$ and  
 $\nu_{F/F^+}$ denotes the quadratic character on $\mathbb{A}_{F^+}^\times/F^{+,\times}$ associated to the quadratic extension $F/F^+$ 
 via global class field theory. Finally let $\mathfrak{D}_{F/F^+}$ denote the relative discriminant of the quadratic extension $F/F^+$.

 \begin{propdef}[Theta lifts] \label{prop:theta} 
Let $F$ be a totally quadratic extension of $F^+$, $\Sigma$ a CM type of $F$, and
 $\mathfrak{C}$ an integral ideal of $F$.  
Let $\eta \colon \mathbb{A}_F^\times /F^\times \rightarrow
  \mathbb{C}^\times$ be a gr\"o{\ss}encharacter of type $(A_0)$ with modulus
  $\mathfrak{C}$ and suppose that the infinity type $\mu$ of $\eta$ 
is $\Sigma$\nobreakdash-admissible.  Then there exists a unique normalised cuspform $\vartheta(\eta)$ of weight $\kappa_\mu$, level $\mathfrak{D}_{F/F^+} \mathfrak{C}\mathfrak{C}^c$ 
 and nebentypus $\underline{\varepsilon}_\eta$ 
 such that its Fourier coefficient $C(\mathfrak{a}; \vartheta(\eta))$ at an integral ideal $\mathfrak{a}$ of $F^+$ is given by $\sum_{\mathfrak{A}\subseteq \mathfrak{r}_F, (\mathfrak{A}, \mathfrak{C})=1; \mathfrak{A}\mathfrak{A}^c=\mathfrak{a}} \eta^*(\mathfrak{A})$. The cuspform $\vartheta(\eta)$ is a common eigenvector 
 of $T(\mathfrak{l})$ for every prime ideal $\mathfrak{l}$ relatively prime to $\mathfrak{D}_{F/F^+}\mathfrak{C}\mathfrak{C}^c$. 
 Furthermore if the modulus $\mathfrak{C}$ of $\eta$ coincides with the
  conductor $\mathfrak{C}(\eta)$ of $\eta$, the resulting cuspform $\vartheta(\eta)$ is 
 primitive in the sense of Miyake \cite[p.~185]{miyake}.  
 \end{propdef}

 The normalised Hilbert modular cuspform $\vartheta(\eta)$ defined as above is
  called the {\em theta lift} of the gr\"o{\ss}encharacter $\eta$ of
  $(A_0)$ on $F$.
 See \cite[Section~7B]{gelbert} for the proof of the proposition, which is 
 based upon Hecke theory on GL(2) through representation theoretic arguments. 

 By its explicit description, 
 the Fourier coefficient $C(\mathfrak{l}; \vartheta(\eta))$ of $\vartheta(\eta)$ at a prime ideal $\mathfrak{l}$ equals $0$ if 
 $\mathfrak{l}$ is inert in $F$; in other words, $C(\mathfrak{l};
 \vartheta(\eta))$ vanishes when $\nu^*_{F/F^+}(\mathfrak{l})$ equals
 $0$ or $-1$. We thus readily observe that 
 $\vartheta(\eta)$ has complex multiplication by
 $\nu_{F/F^+}$. Conversely, if a Hilbert cuspform $f$ belonging to 
$S_\kappa (\mathfrak{N}, \underline{\varepsilon}; \overline{\mathbb{Q}})$ 
 has complex multiplication by a quadratic character $\nu$, 
there exists a totally imaginary quadratic extension $F$ of $F^+$ such that 
 $\nu$ is the character associated to $F/F^+$ and $f$ is described as a 
linear combination of theta lifts of appropriate gr\"o{\ss}encharacters 
of type $(A_0)$ on $F$. We strongly believe that this fact is 
fairly well known, but we shall give a proof of this fact with the language of Galois representation at Proposition~\ref{prop:cm} in Appendix~\ref{app:cm}.

\medskip
In order to let the theta lift $\vartheta(\eta)$ be {\em nearly $p$-ordinary},
we must impose the following {\em ordinarity condition} on the totally imaginary quadratic extension $F/F^+$:
\begin{itemize}
\item[-] {\bfseries (ord$_{F/F^+}$)} all places of $F^+$ lying above $p$
	 split completely in $F$.
\end{itemize}
Then due to the ordinarity condition (ord$_{F/F^+}$), 
there exists a {\em $p$-ordinary CM type} $\Sigma$ of $F$; that is, 
$\Sigma$ is a CM type such that two embeddings $\iota_p \circ\sigma$ and
$\iota_p \circ \sigma \circ c$ of $F$ into $\overline{\mathbb{Q}}_p$ 
define different places of $F$ (lying above $p$) for each $\sigma$ in $\Sigma$.
See Section~\ref{sssc:CMtypes} for details on $p$-ordinary CM types.

Now, under the ordinarity condition (ord$_{F/F^+}$), let $\Sigma$ be a $p$-ordinary CM type of $F$ and $\eta \colon \mathbb{A}_F^\times/F^\times \rightarrow \mathbb{C}^\times$ a $\Sigma$-admissible gr\"o{\ss}encharacter of type $(A_0)$ on $F$. We say $\eta$ to be {\em ordinary with respect to $\Sigma$} (or {\em $\Sigma$-ordinary}) if $\eta$ is unramified at every place $\mathfrak{P}$ contained in $\Sigma_p$. Then the theta lift $\vartheta(\eta)$ of $\eta$ is nearly $p$-ordinary when $\eta$ is ordinary with respect to $\Sigma$.
Conversely if $f$ is a nearly $p$-ordinary $p$-stabilised newform
with complex multiplication, there exist a totally imaginary
quadratic extension $F$ of $F^+$ satisfying (ord$_{F/F^+}$),
a $p$\nobreakdash-ordinary CM type $\Sigma$ and a $\Sigma$-admissible gr\"o{\ss}encharacter
$\eta$ of type ($A_0$) on $F$ ordinary with
respect to $\Sigma$ such that $f$ is described
as $f_\eta:=\vartheta(\eta)^{p\text{-st}}$.
We think that one can verify these facts by looking at the local component at $\mathfrak{p}$ of the automorphic representation $\pi_f$ associated to $f$, but later we shall give a brief proof at Proposition~\ref{prop:cm_ord} in Appendix~\ref{app:cm} based upon the local study of the Galois representation $V_f$ associated to $f$.

%
%
\subsection{The cyclotomic $p$-adic $L$-function for Hilbert
  modular cuspforms} \label{ssc:p-adic_L_cusp}
%
%

We introduce the notion of the cyclotomic $p$-adic $L$-function
associated to Hilbert modular cuspforms in this section. 
We first define the (complex) $p$-optimal periods in
Section~\ref{sssc:optimal_periods}, and then discuss the $p$-adic 
$L$-functions after the result of the second-named author \cite{ochiai}
in Section~\ref{sssc:p-adic_L_cusp}.

Throughout this section $F^+$ denotes a totally real number field
satisfying the following {\em unramifiedness condition}:
\begin{itemize}
\item[-] {\bfseries (unr$_{F^+}$)} the prime $p$ does not ramify in $F^+$.
\end{itemize}
In particular $F^+$ does not contain primitive $p$-th roots of unity. 

\subsubsection{The $p$-optimal complex periods}
\label{sssc:optimal_periods}

We recall the definition of the {\em $p$-optimal complex period} associated
to a normalised eigencuspform $f$ of weight $\kappa$, level
$\mathfrak{N}$ and nebentypus $\underline{\varepsilon}$ in this
paragraph after \cite[Definition~3.5]{ochiai}. 
Let $\kappa_1^{\mathrm{max}}$ denote 
the maximum of the integers $\kappa_{1,\tau}$ over $\tau$ in $I_{F^+}$, 
and set $\tilde{f}(x)=f(x)|\det(x)|_{\mathbb{A}_{F^+}}^{\kappa_1^{\mathrm{max}}+1}$ for every
element $x$ of $G(\mathbb{A}_\mathbb{Q})$. Then we readily see that 
$\tilde{f}$ is a Hilbert modular eigencuspform of weight
$\tilde{\kappa}$, level $\mathfrak{N}$ and nebentypus
$\underline{\tilde{\varepsilon}}$, where $\tilde{\kappa}$ and
$\underline{\tilde{\varepsilon}}$ are defined as 
\begin{align*}
\tilde{\kappa}=(\tilde{\kappa}_1,\tilde{\kappa}_2)=(\kappa_1-(\kappa_1^{\mathrm{max}}+1)\mathsf{t},
 \kappa_2-(\kappa_1^{\mathrm{max}}+1)\mathsf{t}), \,
 \underline{\tilde{\varepsilon}}=(\varepsilon
 \lvert \cdot \rvert_{\mathbb{A}^f_{F^+}}^{2(\kappa_1^{\mathrm{max}}+1)},
 \varepsilon_+\lvert \cdot \rvert_{\mathbb{A}_{F^+}}^{2(\kappa_1^{\mathrm{max}}+1)}).
\end{align*}
In particular the weight $\tilde{\kappa}$ of $\tilde{f}$ is {\em
critical} in the sense of Definition~\ref{def:weights}. The cuspform
$\tilde{f}$ is sometimes called a {\em critical twist of $f$}. 
For a subalgebra $A$ of $\mathbb{C}$ containing all Hecke eigenvalues
associated to $f$, let $\mathfrak{h}_{\tilde{\kappa}}(\mathfrak{N},
\tilde{\underline{\varepsilon}};A)$ denote the image of
$R(\hat{\Gamma}_1(\mathfrak{N}),
\Delta_1(\mathfrak{N}))\otimes_\mathbb{Z} A$ in 
$\mathrm{End}_A(S_{\tilde{\kappa}}(\mathfrak{N},\tilde{\underline{\varepsilon}};A))$.
Then, as an element of $\mathfrak{h}_{\tilde{\kappa}}(\mathfrak{N},\tilde{\underline{\varepsilon}};A)$, the Hecke operator $S(\mathfrak{b})$ is naturally identified with $\tilde{\varepsilon}^*_+(\mathfrak{b})=\varepsilon^*_+(\mathfrak{b})\mathcal{N}\mathfrak{b}^{-2(\kappa_1^\mathrm{max}+1)}$.
Consider a linear map $\lambda_{\tilde{f}} \colon
\mathfrak{h}_{\tilde{\kappa}}(\mathfrak{N},
\tilde{\underline{\varepsilon}};A)\rightarrow A$ 
which sends the Hecke operator $T(\mathfrak{l})$ (resp.\
$U(\mathfrak{l})$) to an eigenvalue 
$\lambda_{\tilde{f}}(\mathfrak{l})$ of $T(\mathfrak{l})$ (resp.\
$U(\mathfrak{l})$) with respect to the eigencuspform $\tilde{f}$ 
for every prime ideal $\mathfrak{l}$ which does not divide
$\mathfrak{N}$ (resp.\ which divides $\mathfrak{l}$). 
Note that $\lambda_{\tilde{f}}(\mathfrak{l})$ equals  $\mathcal{N}
\mathfrak{l}^{-(\kappa_1^{\mathrm{max}}+1)}\lambda_f(\mathfrak{l})$ by construction, where
$\lambda_f(\mathfrak{l})$ is an eigenvalue of $T(\mathfrak{l})$ (or
$U(\mathfrak{l})$ if $\mathfrak{l}$ divides $\mathfrak{N}$) 
with respect to $f$.

\medskip

Let $Y_1(\mathfrak{N})_{/\mathbb{Q}}$ denote 
the Hilbert-Blumenthal modular variety of level $\mathfrak{N}$. 
It is an algebraic variety defined over $\mathbb{Q}$ and obtained as 
the canonical model of the complex analytic variety
$Y_1(\mathfrak{N})=G(\mathbb{Q})\backslash
G(\mathbb{A}_\mathbb{Q})/\hat{\Gamma}_1(\mathfrak{N})C_{\mathbf{i}}$.
Note that if we choose a set of representatives $\{ \mathfrak{c}_i \}_{i=1}^h$ of the strict ray class group $\mathrm{Cl}_{F^+}^+$ of $F^+$, we can decompose the Hilbert-Blumenthal modular variety $Y_1(\mathfrak{N})(\mathbb{C})$ as
\begin{align} \label{eq:var_decomp}
 Y_1(\mathfrak{N})(\mathbb{C}) = \bigsqcup_{i=1}^h \Gamma_1^i(\mathfrak{N}) \backslash \mathfrak{h}^{I_{F^+}},
\end{align}
where $\mathfrak{h}$ denotes the Poincar\'e upper half plane and, for each $i$ with $1\leq i\leq h$, $\Gamma_1^i(\mathfrak{N})$ denotes the arithmetic subgroup of $G(\mathbb{Q})$ defined as
\begin{align*}
\Gamma_1^i(\mathfrak{N})= G(\mathbb{Q}) \cap \left(
\begin{pmatrix} \mathfrak{c}_i & 0 \\ 0 & 1\end{pmatrix}^{-1}\hat{\Gamma_1(\mathfrak{N})C_{\mathbf{i}}}
\begin{pmatrix}
 \mathfrak{c}_i & 0 \\ 0 & 1
\end{pmatrix}\right). 
\end{align*}
Corresponding to the decomposition of the Hilbert-Blumenthal modular variety \eqref{eq:var_decomp}, the space of cuspforms $S_\kappa(\hat{\Gamma}_1(\mathfrak{N}); \mathbb{C})$ of level $\hat{\Gamma}_1(\mathfrak{N})$ is also decomposed as
\begin{align} \label{eq:form_decomp}
 S_{\tilde{\kappa}}(\hat{\Gamma}_1(\mathfrak{N}); \mathbb{C}) \cong  \bigoplus_{i=1}^h S_{\tilde{\kappa}} (\Gamma_1^i(\mathfrak{N});\mathbb{C}) \ ; \ h \mapsto (h_i(z))_{1\leq i\leq h},
\end{align}
where $S_{\tilde{\kappa}}(\Gamma_1^i(\mathfrak{N}); \mathbb{C})$ denotes the space of cuspforms of weight $\tilde{\kappa}$ and level $\Gamma_1^i(\mathfrak{N})$ on $\mathfrak{h}^{I_{F^+}}$. See, for example, \cite[Definition~2.5. and Lemma~2.6.]{ochiai} for details on the decomposition \eqref{eq:form_decomp}. 
In the rest of the article, we choose and fix a set of representatives $\{ \mathfrak{c}_i \}_{i=1}^h$ of $\mathrm{Cl}_{F^+}^+$ so that the $\mathfrak{p}$-component of each $\mathfrak{c}_i$ equals $1$ for all the places $\mathfrak{p}$ of $F^+$ lying above $p$. Next, we define a {\em standard local system} $\mathscr{L}(\tilde{\kappa};A)$ on 
$Y_1(\mathfrak{N})(\mathbb{C})$ for a subring $A$ of $\mathbb{C}$
satisfying the following condition $(*)$:
\begin{quote}
$(*)$ the subring $A$ contains the normal closure of $\mathfrak{r}_{F^+}[\mathfrak{d}_{F^+}^{-1}\mathfrak{c}_{i}^{-1}]$ for each $i$ with $1\leq i\leq h$, where $\mathfrak{d}_{F^+}$ denotes the absolute different of $F^+$.
\end{quote}
For each element $\tau$ of $I_{F^+}$, let
$L(\tilde{\kappa}_\tau;A)=\bigoplus_{m_\tau =0}^{\tilde{\kappa}_{2,\tau}-\tilde{\kappa}_{1,\tau}-1}AX_\tau^{m_\tau } Y_\tau^{\tilde{\kappa}_{2,\tau}-\tilde{\kappa}_{1,\tau}-1-m_\tau }$ 
denote the free $A$-module spanned by all two-variable homogeneous polynomials of degree
$\tilde{\kappa}_{2,\tau}-\tilde{\kappa}_{1,\tau}-1$ with coefficients in $A$. 
Let $g=\begin{pmatrix} a & b \\ c & d\end{pmatrix}$ be an element of $GL_2(F^+)$ such that all conjugate of the matrix component $a,b,c,d$ of $g$ are contained in 
$A$. Such an element $g$ acts from the left on $L(\tilde{\kappa}_\tau ;A)$ by
\begin{multline*}
 X_\tau^{m_\tau}Y_\tau^{\tilde{\kappa}_{2,\tau}-\tilde{\kappa}_{1,\tau}-1-m_\tau} \\
 \mapsto \tau (\det(g)^{\tilde{\kappa}_{1,\tau}})\cdot 
 (\tau (a)X_\tau+\tau (c)Y_\tau)^{m_\tau} (\tau (b)X_\tau+\tau (d)Y_\tau)^{\tilde{\kappa}_{2,\tau}-\tilde{\kappa}_{1,\tau}-1-m_\tau}
\end{multline*}
for each $m_\tau$ with $0\leq m_\tau\leq
\tilde{\kappa}_{2,\tau}-\tilde{\kappa}_{1,\tau}-1$. Set
$L(\tilde{\kappa}; A)=\bigotimes_{\tau \in I_{F^+}}
L(\tilde{\kappa}_\tau; A)$. 
Then we define the {\em standard local system} $\mathscr{L}(\tilde{\kappa};A)$ 
on $Y_1(\mathfrak{N})(\mathbb{C})$ 
as the sheaf of continuous sections of the following covering map:
\begin{align*}
G(\mathbb{Q})\backslash
 G(\mathbb{A}_\mathbb{Q})\times L(\tilde{\kappa};A) /\hat{\Gamma}_1(\mathfrak{N})C_\mathbf{i}
 \rightarrow G(\mathbb{Q})\backslash
 G(\mathbb{A}_\mathbb{Q})/\hat{\Gamma}_1(\mathfrak{N})C_\mathbf{i}=Y_1(\mathfrak{N})(\mathbb{C}).
\end{align*}
Here we consider that $L(\tilde{\kappa};A)$ admits the trivial right action of $\hat{\Gamma}_1(\mathfrak{N})C_\mathbf{i}$, and we let $G(\mathbb{Q})$ and
$\hat{\Gamma}_1(\mathfrak{N})C_\mathbf{i}$ 
act on $G(\mathbb{A}_\mathbb{Q})\times L(\tilde{\kappa};A)$ diagonally.  
The Hecke algebra $\mathfrak{h}_{\tilde{\kappa}}(\mathfrak{N},
\tilde{\underline{\varepsilon}};A)$ then acts on the 
Betti cohomology group $H^d(Y_1(\mathfrak{N})(\mathbb{C}),
\mathscr{L}(\tilde{\kappa}; A))$ via the Hecke correspondences.
Note also that, via the decomposition \eqref{eq:var_decomp}, we can regard the local system $\mathscr{L}(\tilde{\kappa}; A)$ as the sheaf of continuous sections of
\begin{align*}
 \Gamma_1^i(\mathfrak{N})\backslash (\mathfrak{h}^{I_{F^+}} \times L(\tilde{\kappa}; A)) \rightarrow \Gamma_1^i(\mathfrak{N}) \backslash \mathfrak{h}^{I_{F^+}}
\end{align*}
on each connected component $\Gamma_1^i(\mathfrak{N})\backslash \mathfrak{h}^{I_{F^+}}$ of $Y_1(\mathfrak{N})(\mathbb{C})$ of $Y_1(\mathfrak{N})(\mathbb{C})$.

Next let $\mathbb{Q}'_f$ be the composite field $\mathbb{Q}_f F^{+,\mathrm{gal}}$
and $\mathfrak{r}'_f$ its ring of integers. Let us denote by $\mathfrak{r}'_{f,(p)}$ the localisation of $\mathfrak{r}'_f$ at the $p$-adic place
induced by the specified $p$-adic embedding $\mathbb{Q}'_f \subset \overline{\mathbb{Q}}
\xrightarrow{\iota_p} \overline{\mathbb{Q}}_p$. 
Note that an arbitrary $\mathfrak{r}'_{f,(p)}$-algebra 
satisfies the condition $(*)$ since 
$\mathfrak{d}_{F^+}$ is a $p$-adic unit due to the assumption
(unr$_{F^+}$), and it is possible to choose representatives
$\{\mathfrak{c}_i\}_{i=1}^h$ so that they are relatively prime to $p$.
Therefore for an arbitrary $\mathfrak{r}'_{f,(p)}$-algebra $A$,
we can take the maximal $A$-submodule 
$H^d(Y_1(\mathfrak{N})(\mathbb{C}),
\mathscr{L}(\tilde{\kappa}; A))[\lambda_{\tilde{f}}]$ of the Betti cohomology $H^d(Y_1(\mathfrak{N})(\mathbb{C}),
\mathscr{L}(\tilde{\kappa}; A))$ 
(resp. the maximal $A$-submodule $H^d_{c}(Y_1(\mathfrak{N})(\mathbb{C}),
\mathscr{L}(\tilde{\kappa}; A))[\lambda_{\tilde{f}}]$ of the compactly supported Betti cohomology $H^d_{c}(Y_1(\mathfrak{N})(\mathbb{C}),
\mathscr{L}(\tilde{\kappa}; A))$) 
on which the Hecke algebra $\mathfrak{h}_{\tilde{\kappa}}(\mathfrak{N},
\tilde{\underline{\varepsilon}};A)$ acts as the multiplication of
the eigenvalues at $\tilde{f}$.  
Now let $\epsilon$ be an element of $\{ \pm 1\}^{I_{F^+}}$ which we regard
as a character defined on the group of connected components
of the infinite part $GL_2 (\mathbb{R})^{I_{F^+}}$ in $GL_2 (\mathbb{A}_{F^+})$. We consider the composite map
{\small 
\begin{multline}\label{equation:composition_map_period}
H^d_{c}(Y_1(\mathfrak{N})(\mathbb{C}),
\mathscr{L}(\tilde{\kappa}; \mathfrak{r}'_{f,(p)})) \longrightarrow 
H^d_{c}(Y_1(\mathfrak{N})(\mathbb{C}),
\mathscr{L}(\tilde{\kappa}; \mathbb{C})) \longrightarrow 
H^d_{c}(Y_1(\mathfrak{N})(\mathbb{C}),
\mathscr{L}(\tilde{\kappa}; \mathbb{C}))[\lambda_{\tilde{f}}]^{\epsilon }
\end{multline}}%
where the first map is a natural one and the second map is the projection.
As is well known, the group of connected components act on 
$H^d_{c}(Y_1(\mathfrak{N})(\mathbb{C}),
\mathscr{L}(\tilde{\kappa}; \mathbb{C}))$ in a way compatible with the action of 
$\mathfrak{h}_{\tilde{\kappa}}(\mathfrak{N},
\tilde{\underline{\varepsilon}};\mathbb{C})$, and the $\epsilon$-eigenspace
$H^d_{c}(Y_1(\mathfrak{N})(\mathbb{C}),
\mathscr{L}(\tilde{\kappa}; \mathbb{C}))[\lambda_{\tilde{f}}]^{\epsilon}
$ of $H^d_{c}(Y_1(\mathfrak{N})(\mathbb{C}),
\mathscr{L}(\tilde{\kappa}; \mathbb{C}))[\lambda_{\tilde{f}}]$ with respect to this action
is of dimension one over $\mathbb{C}$ for each $\epsilon$. Thus the image of 
$H^d_{c}(Y_1(\mathfrak{N})(\mathbb{C}),
\mathscr{L}(\tilde{\kappa}; \mathfrak{r}'_{f,(p)}))$ under the map \eqref{equation:composition_map_period} is free of rank one over 
the discrete valuation ring $\mathfrak{r}'_{f,(p)}$ for each $\epsilon$; or in
other words, each $H^d_c(Y_1(\mathfrak{N})(\mathbb{C}),\mathscr{L}(\tilde{\kappa};\mathbb{C}))[\lambda_{\tilde{f}}]^\epsilon$ is equipped with an $\mathfrak{r}_{f,(p)}'$-integral structure.

We are now ready to associate the Hilbert cuspform $\tilde{f}$ to a cohomology class $[\tilde{f}]$. Let $(\tilde{f}_i(z))_{1\leq i\leq h}$ denote the element of $\bigoplus_{i=1}^h S_{\tilde{\kappa}}(\Gamma_1^i(\mathfrak{N});\mathbb{C})$ corresponding to $\tilde{f}$ via the decomposition~\eqref{eq:form_decomp}, and let us consider the vector-valued differential form $\omega_{\tilde{f}}=(\omega_{\tilde{f}_i})_{1\leq i\leq h}$ on $Y_1(\mathfrak{N})(\mathbb{C})$ defined as
\begin{align*}
 \omega_{\tilde{f}_i}=\tilde{f}_i(z_\tau) \prod_{\tau \in I_{F^+}}(X_\tau +z_\tau Y_\tau)^{\tilde{\kappa}_{2,\tau}-\tilde{\kappa}_{1,\tau}-1} \bigwedge_{\tau \in I_{F^+}} dz_\tau
\end{align*}
on each connected component $\Gamma_1^i(\mathfrak{N})\backslash \mathfrak{h}^{I_{F^+}}$ of $Y_1(\mathfrak{N})(\mathbb{C})$, where $(z_\tau)_{\tau \in I_{F^+}}$ denotes the standard coordinate of $\mathfrak{h}^{I_{F^+}}$. The integration of $\omega_{\tilde{f}}$ on a $d$-cycle of $Y_1(\mathfrak{N})(\mathbb{C})$ then defines a cohomology class $[\tilde{f}]$ of $H^d(Y_1(\mathfrak{N})(\mathbb{C}), \mathscr{L}(\tilde{\kappa};\mathbb{C}))$, which we call the {\em Eichler-Shimura class} associated to the critical twist $\tilde{f}$ of $f$. We use the same symbol $[\tilde{f}]$ for its image under the composition
\begin{align*}
H^d(Y_1(\mathfrak{N})(\mathbb{C}), \mathscr{L}(\tilde{\kappa}; \mathbb{C})) \stackrel{\text{proj}}{\twoheadrightarrow} H^d(Y_1(\mathfrak{N})(\mathbb{C}),\mathscr{L}(\tilde{\kappa};\mathbb{C}))[\lambda_{\tilde{f}}] 
 \cong H^d_c(Y_1(\mathfrak{N})(\mathbb{C}), \mathscr{L}(\tilde{\kappa};\mathbb{C}))[\lambda_{\tilde{f}}]
\end{align*}
by abuse of notation.

\begin{defn} \label{def:complex_period}
For each $\epsilon$, let us choose an $\mathfrak{r}'_{f,(p)}$-basis $b^\epsilon$ of the image of the rank-one free  
$\mathfrak{r}'_{f,(p)}$-module $H^d_{c}(Y_1(\mathfrak{N})(\mathbb{C}),
\mathscr{L}(\tilde{\kappa}; \mathfrak{r}'_{f,(p)}))$ under the map 
$(\ref{equation:composition_map_period})$. 
Then we define a {\em $p$-optimal complex period $C_{f, \infty}^{\epsilon}\in \mathbb{C}^\times$ of signature $\epsilon$}  of $f$ to be the 
constant given by 
\begin{equation}\label{equation:definition_of_modularsymbolcomplexperiod}
[\tilde{f}]^\epsilon = C_{f, \infty}^{\epsilon} \cdot b^\epsilon .
\end{equation}
 where $[\tilde{f}]^\epsilon$ denotes the image of $[\tilde{f}]$ under the projection onto $H^d(Y_1(\mathfrak{N})(\mathbb{C}), \mathscr{L}(\tilde{\kappa};\mathbb{C}))[\lambda_{\tilde{f}}]^\epsilon$.
The complex period $C_{f, \infty}^{\epsilon}$ depends on the choice of $b^\epsilon $ 
which has an ambiguity of a multiple of an element in $(\mathfrak{r}'_{f,(p)} )^\times$. 
Hence we often regard $C_{f, \infty}^{\epsilon}$ as an element of
$\mathbb{C}^\times /(\mathfrak{r}'_{f,(p)} )^\times$. 
\end{defn}
 
\subsubsection{The cyclotomic $p$-adic $L$-functions}  \label{sssc:p-adic_L_cusp}

We here introduce the notion of
the cyclotomic $p$-adic $L$-functions associated to
 Hilbert modular cuspforms and their
 interpolation formulae  (Theorem~\ref{thm:Hilb_L-func}).

 Let $\mathrm{Cl}_{F^+}^+(p^{\infty}\mathfrak{r}_{F^+})$ 
 be the {\em strict} ray class group of $F^+$ modulo $p^\infty\mathfrak{r}_{F^+}$, which is defined as the projective limit $\mathrm{Cl}_{F^+}^+(p^\infty \mathfrak{r}_{F^+})=\varprojlim_{n\rightarrow \infty} \mathbb{A}_{F^+}^\times/(F^+)^\times I_{p^n\mathfrak{r}_{F^+}}$. Here the subgroup $I_{p^n \mathfrak{r}_{F^+}}$ of $\mathbb{A}_{F^+}^\times$ is defined as 
\begin{align*}
I_{p^n\mathfrak{r}_{F^+}}=\prod_{\mathfrak{l}\nmid p\mathfrak{r}_{F^+}} \mathfrak{r}_{F^+,\mathfrak{l}}^\times \times \prod_{\mathfrak{p}\mid p\mathfrak{r}_{F^+}} (1+\mathfrak{p}^{nm_\mathfrak{p}}\mathfrak{r}_{F^+,\mathfrak{p}}) \times \prod_{w\colon \text{archimedean place of $F^+$}} (\mathfrak{r}_{F^+,w}^\times)_+
\end{align*}
when $p\mathfrak{r}_{F^+}$ is decomposed as $p\mathfrak{r}_{F^+}=\prod_{\mathfrak{p}\mid p\mathfrak{r}_{F^+}} \mathfrak{p}^{m_\mathfrak{p}}$. For an archimedean
 place $w$, we denote by $(\mathfrak{r}_{F^+,w}^\times)_+$ the connected component of $\mathfrak{r}_{F^+,w}^\times$ containing $1$, which is isomorphic to $\mathbb{R}_{>0}$. We denote by $F^+_{p^{\infty}\mathfrak{r}_{F^+}}$ the {\em strict} ray class field modulo $p^{\infty}\mathfrak{r}_{F^+}$ over $F^+$, the field corresponding to $\mathrm{Cl}_{F^+}^+(p^\infty\mathfrak{r}_{F^+})$ via global class field theory. For a ray class character $\phi \colon \mathrm{Cl}^+_{F^+}(p^\infty\mathfrak{r}_{F^+}) \rightarrow \mathbb{C}^\times$ of finite order, we associate its {\em signature} $\mathrm{sgn}(\phi) \in \{\pm 1\}^{I_{F^+}}$ in the following manner. 
 Observe that there is a canonical homomorphism from the component group 
 $\{\pm 1\}^{I_{F^+}}$ of the archimedean part of the idele class group to 
$\mathrm{Cl}^+_{F^+}(p^\infty \mathfrak{r}_{F^+})$.   
Via this homomorphism, we associate a character 
$(\phi_\tau)_{\tau \in I_{F^+}}$ on $\{\pm 1\}^{I_{F^+}}$ to 
$\phi$. We define as $\mathrm{sgn}(\phi)=(\phi_\tau(-1))_{\tau \in I_{F^+}}$.

 We next define the Gaussian sum $G(\phi)$ of the ray class character $\phi\colon \mathrm{Cl}_{F^+}^+(p^\infty \mathfrak{r}_{F^+})\rightarrow
 \mathbb{C}^\times$ of finite order as
 \begin{align} \label{eq:gauss}
 G(\phi)=\sum_{x \in (\mathfrak{C}(\phi)^{-1}/\mathfrak{r}_{F^+})^\times} 
 \phi(x)\exp(2\pi \sqrt{-1} \mathrm{Tr}_{F^+/\mathbb{Q}}(x)),
 \end{align}
 where we denote by $(\mathfrak{C}(\phi)^{-1}/\mathfrak{r}_{F^+})^\times$ the subset of $\mathfrak{C}(\phi)^{-1}/\mathfrak{r}_{F^+}$ consisting of elements whose annihilators exactly coincide with $\mathfrak{C}(\phi)$.
 In the defining equation of $G(\phi)$ we evaluate $\phi$ at an element $x$ of $(\mathfrak{C}(\phi)^{-1}/\mathfrak{r}_{F^+})^\times$ via 
 the following composition:
 \begin{align*}
 (\mathfrak{C}(\phi)^{-1}/\mathfrak{r}_{F^+})^\times \xrightarrow{\sim} (\mathfrak{r}_{F^+}/\mathfrak{C}(\phi))^\times 
 \rightarrow (\mathfrak{r}_{F^+}/\mathfrak{C}(\phi))^\times/
 \mathfrak{r}^\times_{F^+} \hookrightarrow \mathrm{Cl}^+_{F^+}(\mathfrak{C}(\phi)).
 \end{align*}
 For a Galois character $\phi\colon \mathrm{Gal}(F^+_{p^\infty\mathfrak{r}_{F^+}}/F^+)\rightarrow \overline{\mathbb{Q}}^\times$ of finite order, there exists a unique character on $\mathrm{Cl}^+_{F^+}(p^\infty\mathfrak{r}_{F^+})$ whose value at a prime ideal $\mathfrak{l}$ relatively prime to $p\mathfrak{r}_{F^+}$ equals $\phi(\mathrm{Frob}_\mathfrak{l}^{-1})$, the evaluation of $\phi$ at the {\em arithmetic} Frobenius element of $\mathfrak{l}$. By abuse of notation, we use the same symbol $\phi$ for this unique ray class character.

\medskip
 We now state the existence of the cyclotomic $p$-adic $L$-function associated to a Hilbert modular cuspform, which is originally due to Manin \cite[Sections 5 and 6]{Manin}.
For an element $k$ of $\mathbb{Z}[I_{F^+}]$, we define the integer $k^{\mathrm{max}}$ (resp.\ $k_\mathrm{min}$)  as 
 the maximum (resp.\ the minimum) among its coefficients. 
 
 \begin{thm} \label{thm:Hilb_L-func}

 Let $f$ be a normalised nearly $p$-ordinary eigencuspform in $S_\kappa(\mathfrak{N}, \underline{\varepsilon}; \overline{\mathbb{Q}})$ which is stabilised at $p$. 
 We fix a discrete valuation ring 
 $\mathcal{O}$ finite flat over $\mathbb{Z}_p$ which contains 
 $\mathfrak{r}'_{f, (p)}$.
 Then there exists an element $\mathcal{L}_p^\mathrm{cyc}(f)$ of
  $\mathcal{O}[[\mathrm{Gal}(F^+ (\mu_{p^\infty})/F^+)]] \otimes_{\mathbb{Z}_p}\mathbb{Q}_p$ characterised by the interpolation property $($recall that we have put $d=[F^+ :\mathbb{Q}])$

 \begin{align} \label{eq:interp_cusp}
 \chi_{p,\mathrm{cyc}}^j \phi(\mathcal{L}_p^\mathrm{cyc}(f))=
 \frac{\Gamma(j\mathsf{t}-\kappa_1)}{\Gamma((\kappa_1^{\mathrm{max}}+1)\mathsf{t}-\kappa_1)}G(\phi)\prod_{\mathfrak{p} \mid p\mathfrak{r}_{F^+}} A_{\mathfrak{p}}(f; \phi, j) \frac{L(f, \phi^{-1}, j)}{(-2\pi \sqrt{-1})^{d(j-\kappa_1^{\mathrm{max}}-1)} C_{f, \infty}^{\epsilon_{\phi,j}}} 
 \end{align}
 for an arbitrary integer $j$ satisfying $\kappa_1^{\mathrm{max}}+1\leq j \leq \kappa_{2,\mathrm{min}}$ and an arbitrary character $\phi$ of 
 $\mathrm{Gal}(F^+ (\mu_{p^\infty}) /F^+)$ of finite order. Here $C^{\epsilon_{\phi,j}}_{f, \infty}$ denotes a $p$-optimal complex period of signature $\epsilon_{\phi,j}$ $($see Definition~$~\ref{def:complex_period})$. The signature
  $\epsilon_{\phi,j}$ is defined as $(-1)^{j-\kappa_1^{\mathrm{max}}-1}\mathrm{sgn}(\phi)$.

 The $\mathfrak{p}$-adic multiplier $A_{\mathfrak{p}}(f; \phi, j)$ is defined as 
 \begin{align*}
 A_{\mathfrak{p}}(f;\phi,j)=\begin{cases}
 \displaystyle  1-\frac{\phi(\mathfrak{p})\mathcal{N}\mathfrak{p}^{j-1}}{\alpha_{\mathfrak{p}}(f)} & \text{if $\mathfrak{p}$ does not divide $\mathfrak{C}(\phi)$}, \\
\displaystyle \left( \frac{\mathcal{N}\mathfrak{p}^{j-1}}{\alpha_{\mathfrak{p}}(f)} \right)^{\mathrm{ord}_{\mathfrak{p}}(\mathfrak{C}(\phi))} & \text{if $\mathfrak{p}$ divides $\mathfrak{C}(\phi)$},  
 \end{cases}
\end{align*}
 where $\alpha_{\mathfrak{p}}(f)$ denotes the eigenvalue of $f$ with respect to the Hecke operator $U(\mathfrak{p})$.  
  The gamma factor is abbreviated by using multi-indices as 
 \begin{align*}
 \Gamma(m\mathsf{t}-\kappa_1)=\prod_{\tau \in I_{F^+}}\Gamma(m-\kappa_{1,\tau}) \quad \text{for each $m$ satisfying } m\geq \kappa_1^{\mathrm{max}}+1.
 \end{align*}
 \end{thm}

 \begin{rem}
 The construction of \cite{Manin} is based upon {\em modular symbols} over Hilbert modular varieties, 
 which is a generalisation of the work by Manin himself, Mazur and Swinnerton-Dyer, Vi\v{s}ik, Amice and V\'elu for 
elliptic modular case (see \cite{MTT} for historical reviews). 
 The construction of Theorem~\ref{thm:Hilb_L-func} is revisited by \cite{dimitrov, ochiai} in the context of generalising it into nearly ordinary 
 Hida deformations. 
\end{rem}

\begin{rem}
 There is another fashion of construction of  $p$-adic $L$-functions associated
 to (Hilbert) modular forms which is based upon 
 the theory of {\em Rankin-Selberg convolutions}. 
 Methods of the construction of $p$-adic $L$-functions 
 in this direction are discussed by Panchishkin \cite{panchishkin}, 
 Dabrowski \cite{dabrowski}, Mok \cite{mok} and so on. 
 The complex period appearing in their theory is of Shimura type; namely, 
 it is defined by using the self Petersson inner product of 
 the given cuspform.   
 \end{rem} 

%
\subsection{Katz, Hida and Tilouine's $p$-adic $L$-functions for CM
  number fields} \label{ssc:KHT}
%

In this subsection we introduce the $p$-adic $L$-function (or the $p$-adic measure) for a CM number field
which was first constructed by Katz \cite{katz} for gr\"o{\ss}encharacters of type $(A_0)$ with conductors dividing $p^\infty$, 
and then by Hida and Tilouine \cite[Theorem~II]{HT-katz} for general
gr\"o{\ss}encharacters of type $(A_0)$.

%
\subsubsection{$p$-ordinary CM types} \label{sssc:CMtypes}
%

As in Introduction, let $F$ be a CM number field of degree $2d$ with 
the maximal totally real subfield $F^+$. We denote by $c$ the complex conjugation of $F$, that is, the unique generator of the Galois group of $F/F^+$.
We impose the two assumptions (unr$_F^+$) and (ord$_{F/F^+}$) on $F/F^+$ and $p$, which are introduced at the beginning of Section~\ref{ssc:p-adic_L_cusp} and a little bit before Proposition-Definition \ref{prop:theta} respectively. 

By virtue of the assumption (ord$_{F/F^+}$), one can consider {\em a $p$-ordinary CM type} $\Sigma$ of $F$ (also called {\em a $p$-adic CM type}). 
Namely $\Sigma$ is a subset of $I_F$ satisfying the following two conditions:
\begin{itemize}
\item[-] we have $I_F =\Sigma \cup \Sigma^c$ (disjoint union) where 
$\Sigma^c=\{ \sigma\circ c \in I_F \mid \sigma \in \Sigma \}$; 
\item[-] we have $\{ \text{places of $F$ lying
above $p$} \} = \Sigma_p \cup \Sigma^c_p$ (disjoint union) where $\Sigma_p$ the set of places of $F$ induced by embeddings $\iota_p \circ \sigma$ for 
all $\sigma$ in $\Sigma$ and $\Sigma_p^c$ the set of their complex conjugates.
\end{itemize}
It is not difficult to see that there exists a $p$-ordinary CM type $\Sigma$ if and only if the condition (ord$_{F/F^+}$) is satisfied and that the number of $p$-ordinary CM types is equal to $2^{\sharp \{ \text{places of $F$ lying
above $p$} \}}$. Let us take a $p$\nobreakdash-ordinary CM type $\Sigma$ 
and fix it once and for all. 
Then, by the definition of gr\"o{\ss}encharacter of type $(A_0)$, 
the sum $\mu_\sigma +\mu_{\bar{\sigma}}$ has the same value for every $\sigma$ in $\Sigma$. We denote this constant by $-w$. The infinity type $\mu$ of $\eta$ thus has two expressions as follows:
\begin{align*}
\mu=\sum_{\sigma \in \Sigma} (\mu_\sigma \sigma+\mu_{\bar{\sigma}}\bar{\sigma}) = -w\mathsf{t}-\sum_{\sigma \in \Sigma} r_\sigma (\sigma -\bar{\sigma}). 
\end{align*}

\subsubsection{The CM periods of Katz, Hida and Tilouine} \label{sssc:Kperiod}

We next introduce the complex and $p$-adic periods which 
appear in the interpolation formula of the $p$-adic $L$-functions 
for CM number fields.

Let $\mathfrak{C}$ be an integral ideal of $F$ which is prime to $p$ and 
let us choose and fix an element $\delta$ of $F$ which satisfies the following two conditions: 
\begin{enumerate}[label=(\arabic*${}_\delta$)]
\item the imaginary part $\mathrm{Im}(\sigma(\delta))$ of $\delta$ is positive for all $\sigma$ in $\Sigma$;
\item the alternating form defined by $\langle u,v\rangle_\delta=(uv^c-u^cv)/2\delta$ induces 
an isomorphism between $\mathfrak{r}_F\wedge_{\mathfrak{r}_{F^+}} \mathfrak{r}_F$ and 
$\mathfrak{d}_{F^+}^{-1}\mathfrak{c}^{-1}$ 
for a certain fractional ideal $\mathfrak{c}$ of $F^+$ prime to $p\mathfrak{C}\mathfrak{C}^c$.
\end{enumerate}
Here $\mathfrak{d}_{F^+}$ denotes the absolute different of $F^+$. 
We embed $\mathfrak{r}_F$ into $\mathbb{C}^{\Sigma}$ diagonally via the fixed ($p$-ordinary) CM-type $\Sigma$ 
and denote its image by $\Sigma(\mathfrak{r}_F)$, which becomes a $\mathbb{Z}$-lattice in $\mathbb{C}^\Sigma$.
We denote the complex torus $\mathbb{C}^\Sigma/\Sigma(\mathfrak{r}_F)$ by $X(\mathfrak{r}_F)$. 
Then the pairing $\langle \  , \  \rangle_\delta$ defines {\em a $\mathfrak{c}$\nobreakdash-polarisation} 
on $X(\mathfrak{r}_F)$ and thus the pair $(X(\mathfrak{r}_F),\langle \  , \  \rangle_\delta)$ gives rise to 
an abelian variety $(X(\mathfrak{r}_F), \lambda_\delta)$ equipped with a $\mathfrak{c}$-polarisation 
$\lambda_\delta \colon X(\mathfrak{r}_F)^t \xrightarrow{\sim} X(\mathfrak{r}_F)\otimes_{\mathfrak{r}_{F^+}} \mathfrak{c}$. By construction 
the $\mathfrak{c}$-polarised abelian variety $(X(\mathfrak{r}_F), \lambda_\delta)$ is 
equipped with {\em complex multiplication by $\mathfrak{r}_F$}.
Note that the element $2\delta$ is 
a generator of the fractional ideal $\mathfrak{cd}_F$ of $F$ by the polarisation condition $(2_\delta)$ (see also \cite[Lemme~(5.7.35)]{katz}). 
In particular, if we embed $2\delta$ into the id\`ele group $\mathbb{A}_{F}^\times/F^\times$ diagonally and 
denote $(2\delta)_w$ the component at the prime $w$, we can take $(2\delta)_{\mathfrak{Q}}$ as a generator of 
$\mathfrak{d}_F \otimes_{\mathfrak{r}_F} \mathfrak{r}_{F,\mathfrak{Q}}$ for every prime ideal $\mathfrak{Q}$ of $F$ 
relatively prime to $\mathfrak{c}$. 

Next we endow $(X(\mathfrak{r}_F), \lambda_\delta)$ with a {\em  $\Gamma_{00}(p^\infty)$-level structure}. 
In order to do so, we first decompose $\mathfrak{C}$ into a product $\mathfrak{C}=\mathfrak{F}\mathfrak{F}_c\mathfrak{I}$ 
so that $\mathfrak{FF}_c$ is a product of prime ideals completely split over $F^+$, $\mathfrak{I}$ 
is that of prime ideals inert or ramified over $F^+$, $\mathfrak{F}$ and $\mathfrak{F}_c$ are 
relatively prime and $\mathfrak{F}_c^c$ (the complex conjugate of
$\mathfrak{F}_c$) contains $\mathfrak{F}$. We fix such a decomposition
of $\mathfrak{C}$ once and for all.
We put $\mathfrak{f}$ as $\mathfrak{FI}\cap F^+$ and $\mathfrak{f}_c$ as $\mathfrak{F}_c\mathfrak{I}\cap F^+$; 
then $\mathfrak{f}$ and $\mathfrak{f}_c$ are integral ideals of $F^+$ and $\mathfrak{f}_c$ contains $\mathfrak{f}$.  
We choose a differential id\`ele $\mathsf{d}_{F^+}=(\mathsf{d}_v)_v$ of $F^+$, a finite  id\`ele of $F^+$ whose associated modulus 
coincides with $\mathfrak{d}_{F^+}$, so that the following conditions are fulfilled:
\begin{itemize}
\item[-] the local component $\mathsf{d}_v$ equals $1$ unless $v$ divides $p\mathfrak{fd}_{F^+}$;
\item[-] if $\mathfrak{Q}$ is a prime divisor of $\mathfrak{F}$ and $\mathfrak{q}$ is the unique prime ideal of $F^+$ lying below $\mathfrak{Q}$, the local component $\mathsf{d}_\mathfrak{q}$ of $\mathsf{d}$ is given by $(2\delta)_\mathfrak{Q}$, where we identify 
$F^+_\mathfrak{q}$ with $F_\mathfrak{Q}$ via the isomorphism induced by the canonical inclusion $F^+\hookrightarrow F$.
\end{itemize}
Then the composition
\begin{align*}
(\mathfrak{d}_{F^+}^{-1} &\otimes_\mathbb{Z} \mathbb{G}_m)[p^\infty](\mathbb{C}) 
\cong 
\mathfrak{d}_{F^+}^{-1} \prod_{\mathfrak{p} \mid p\mathfrak{r}_{F^+} }( \mathfrak{p}^{-\infty}  / \mathfrak{r}_{F^+} ) 
\\
&\xrightarrow{\times \mathsf{d}_{F^+}} 
\prod_{\mathfrak{p} \mid p\mathfrak{r}_{F^+} } (\mathfrak{p}^{-\infty}  / \mathfrak{r}_{F^+} ) 
 \overset{\sim}{\longrightarrow} 
 \prod_{\mathfrak{P}\in \Sigma_p} (\mathfrak{P}^{-\infty} / \mathfrak{r}_{F} ) \hookrightarrow \mathbb{C}^\Sigma/\Sigma (\mathfrak{r}_F )=X(\mathfrak{r}_F).
\end{align*}
 induces a $\Gamma_{00}(p^\infty)$-level structure $i(\mathfrak{r}_F) \colon (\mathfrak{d}_{F^+}^{-1}\otimes_\mathbb{Z} \mathbb{G}_m)[p^\infty]\hookrightarrow X(\mathfrak{r}_F)$ over $\mathbb{C}$. 
 The theory of complex multiplication enables us to find a model of the triple 
 $(X(\mathfrak{r}_F), \lambda_\delta, i(\mathfrak{r}_F))$ over a valuation ring 
 which is obtained as the inverse image  under $\iota_p$ of a certain finite integral extension $\widehat{\mathcal{O}}'$ of 
 $\widehat{\mathcal{O}}^{\mathrm{ur}}$. 
 Moreover, since we admit the unramifiedness condition (unr$_{F^+}$), we can take $\widehat{\mathcal{O}}'$ as
 $\widehat{\mathcal{O}}^{\mathrm{ur}}$ itself due to the fundamental theorem of the theory of complex multiplication 
 combined with Serre and Tate's criterion for good reduction. Let us denote the inverse image of $\widehat{\mathcal{O}}^{\mathrm{ur}}$ under $\iota_p$ by $\mathcal{W}$ and take a $\mathfrak{r}_F\otimes_\mathbb{Z} \mathcal{W}$-basis $\omega(\mathfrak{r}_F)$ of the module of invariant differentials $\underline{\omega}_{X(\mathfrak{r}_F)/\mathcal{W}}$. Recall that $\underline{\omega}_{X(\mathfrak{r}_F)/\mathcal{W}}$ is 
 an invertible module over $\mathfrak{r}_F\otimes_\mathbb{Z} \mathcal{W}$ since $X(\mathfrak{r}_F)$ has complex multiplication 
 by $\mathfrak{r}_F$). 

 By construction, $X(\mathfrak{r}_F)$ admits a canonical complex uniformisation 
 $\Pi \colon \mathbb{C}^\Sigma \twoheadrightarrow X(\mathfrak{r}_F)$ defined as the quotient 
 with respect to the lattice $\Sigma(\mathfrak{r}_F)$. It induces an isomorphism between the modules of 
 invariant differentials
 \begin{align*}
 \Pi^* \colon \underline{\omega}_{X(\mathfrak{r}_F)}
  \xrightarrow{\sim} \bigoplus_{\sigma \in \Sigma} \mathbb{C} du_\sigma
 \end{align*}
 and we can take $\omega_{\mathrm{trans}}(\mathfrak{r}_F)=\sum_{\sigma \in \Sigma} du_\sigma$ as a $\mathfrak{r}_F\otimes_\mathbb{Z} \mathbb{C}$-basis of the right hand side. Then we define the {\em complex CM period} $\Omega_{\mathrm{CM},\infty} = 
 (\Omega_{\mathrm{CM}, \infty, \sigma})
 \in (\mathfrak{r}_F\otimes_\mathbb{Z} \mathbb{C})^\times=\mathbb{C}^{\times, \Sigma}$ 
 by the following equality:
 \begin{align*}
 \Pi^* \omega(\mathfrak{r}_F)=\Omega_{\mathrm{CM},\infty}\, \omega_{\mathrm{trans}}(\mathfrak{r}_F).
 \end{align*}
 On the other hand, (the $p$-part of) the $\Gamma_{00}(p^\infty)$-level structure induces an isomorphism 
 $i_p \colon (\mathfrak{d}_{F^+}^{-1}\otimes_\mathbb{Z} \mathbb{G}_m)^{\wedge} \xrightarrow{\sim} X(\mathfrak{r}_F)^{\wedge}$ 
 between the formal completions along the identity sections over $\widehat{\mathcal{O}}^{\mathrm{ur}}$, 
 and hence we obtain the isomorphism
 \begin{align*}
 i_p^* \colon \underline{\omega}_{X(\mathfrak{r}_F)/\widehat{\mathcal{O}}^{\mathrm{ur}}} \xrightarrow{\sim} 
  \bigoplus_{\sigma \in \Sigma} \widehat{\mathcal{O}}^{\mathrm{ur}}\frac{dT_\sigma}{T_\sigma}.
 \end{align*}
 We can take $\omega_{\mathrm{can}}(\mathfrak{r}_F)=\sum_{\sigma \in \Sigma} dT_\sigma/T_\sigma$ 
 as a $\mathfrak{r}_F\otimes_\mathbb{Z} \widehat{\mathcal{O}}^{\mathrm{ur}}$-basis of the right hand side, 
 and we define the {\em  $p$-adic CM period} $\Omega_{\mathrm{CM},p}
 =(\Omega_{\mathrm{CM},p,\sigma}) \in (\mathfrak{r}_F\otimes_\mathbb{Z} \widehat{\mathcal{O}}^\mathrm{ur})^\times=(\widehat{\mathcal{O}}^{\mathrm{ur,\times}})^\Sigma$ 
 by the following equality:
 \begin{align*}
 i_p^* \omega(\mathfrak{r}_F)=\Omega_{\mathrm{CM},p}\, \omega_{\mathrm{can}}(\mathfrak{r}_F).
 \end{align*}

 \begin{rem}
 One readily observes by the construction above that, 
 when one replaces $\omega(\mathfrak{r}_F)$ 
 with another $\mathfrak{r}_F\otimes_\mathbb{Z} \mathcal{W}$-basis of $\underline{\omega}_{X(\mathfrak{r}_F)/\mathcal{W}}$, 
 both $\Omega_{\mathrm{CM},\infty}$ and $\Omega_{\mathrm{CM},p}$ are multiplied by {\em the same} value 
 contained in $(\mathfrak{r}_F\otimes_\mathbb{Z} \mathcal{W})^\times$; 
 therefore 
 {\em the ratio} of the pair $(\Omega_{\mathrm{CM},\infty}, \Omega_{\mathrm{CM},p})$ is 
 well defined independently of the choice of a basis $\omega(\mathfrak{r}_F)$ of $\underline{\omega}_{X(\mathfrak{r}_F)/\mathcal{W}}$.
 \end{rem}

\subsubsection{The $p$-adic $L$-functions for CM fields}

In order to state the interpolation formula of the $p$-adic $L$-function for the CM number field $F$, 
we here introduce the notion of 
{\em dual gr\"o{\ss}encharacters}; for a gr\"o{\ss}encharacter $\eta$ of type $(A_0)$ on $F$, 
the dual gr\"o{\ss}encharacter $\check{\eta}$ of $\eta$ is defined by $\eta(x)\check{\eta}(x^c)=|x|_{\mathbb{A}_F}$ for every $x$ in $\mathbb{A}_F^\times$. In the language of ideal characters, it is characterised as $\eta^*(\mathfrak{A})\check{\eta}^*(\mathfrak{A}^c)=\mathcal{N}\mathfrak{A}^{-1}$ for every fractional ideal $\mathfrak{A}$ relatively prime to $\mathfrak{C}(\eta)$.

Now we are ready to introduce the results of Katz, Hida and Tilouine.

\begin{thm}[{\cite[Theorem~(5.3.0)]{katz}, \cite[Theorem~II]{HT-katz}}] \label{thm:KHT}
Let $p$ be an odd prime number and $F$ a CM field of degree $2d$ with the maximal totally real subfield $F^+$. Assume that 
$F$, $F^+$ and $p$ satisfy both the conditions {\upshape (unr$_{F^+}$)} and {\upshape
 (ord$_{F/F^+}$)}. Let us choose and fix a $p$-ordinary CM type $\Sigma$ of $F$. 
Let $\mathfrak{C}=\mathfrak{F}\mathfrak{F}_c\mathfrak{l}$ be an integral ideal of $F$ 
relatively prime to $p$ and $\delta$ a purely imaginary element of $F$ 
satisfying 
both the conditions $(1_\delta)$ and $(2_\delta)$ stated at the beginning of Section~$~\ref{sssc:Kperiod}$. Then there exists a unique element 
$\mathcal{L}_{p,\Sigma}^{\mathrm{KHT}}(F)$ in the Iwasawa algebra 
$\widehat{\mathcal{O}}^{\mathrm{ur}}[[\mathrm{Gal}(F_{\mathfrak{C}p^\infty}/F)]]$, where $F_{\mathfrak{C}p^\infty}$ denotes the ray class field modulo $\mathfrak{C}p^{\infty}$ over $F$, satisfying 
\begin{align*}
\frac{\eta^{\mathrm{gal}}(\mathcal{L}_{p, \Sigma}^{\mathrm{KHT}}(F))}
{\Omega_{\mathrm{CM},p}^{w\mathsf{t}+2\mathsf{r}}} &=
(\mathfrak{r}_F^\times : \mathfrak{r}_{F^+}^\times)W_p(\eta)\frac{(-1)^{wd}(2\pi)^{|\mathsf{r}|}\Gamma_\Sigma(w\mathsf{t}+\mathsf{r})}{\sqrt{|D_{F^+}|}\mathrm{Im}(2\delta)^{\mathsf{r}}} \\
& \qquad \qquad \times \prod_{\mathfrak{L} \mid \mathfrak{C}}(1-\eta^*(\mathfrak{L}))\prod_{\mathfrak{P}\in \Sigma_p} 
\{ (1-\eta^*(\mathfrak{P}^c))(1-\check{\eta}^*(\mathfrak{P}^c))\} \frac{L(\eta^*,0)}{\Omega_{\mathrm{CM},\infty}^{w\mathsf{t}+2\mathsf{r}}}
\end{align*}
for each gr\"o{\ss}encharacter $\eta$ of type $(A_0)$ 
with conductor dividing $\mathfrak{C}p^{\infty}$ such that 
\begin{enumerate}[label={\upshape (\roman*)}]
\item the conductor of $\eta$ is divisible by all prime factors 
of $\mathfrak{F}${\upshape ;}
\item the infinity type $\mu=-w\mathsf{t}-\sum_{\sigma \in \Sigma}r_\sigma (\sigma-\bar{\sigma})$ of $\eta$ satisfies
either of the followings{\upshape ;}
\begin{enumerate}[label={\upshape (\alph*)}]
\item $w\geq 1$ and $r_\sigma \geq 0$ for all $\sigma$ in $\Sigma${\upshape ;}
\item $w\leq 1$ and $w+r_\sigma-1\geq 0$ for all $\sigma$ in $\Sigma$.
\end{enumerate}
\end{enumerate}
 The {\em local $\varepsilon$-factor $W_p(\eta)$ at $p$} is defined as 
\begin{align*}
W_p(\eta)=\prod_{\mathfrak{P}\in \Sigma_p} \mathcal{N}\mathfrak{P}^{-e(\mathfrak{P})} \eta_{\mathfrak{P}}(\varpi_{\mathfrak{P}}^{-e(\mathfrak{P})}) \sum_{x\in (\mathfrak{r}_{F,\mathfrak{P}}/\mathfrak{P}^{e(\mathfrak{P})})^\times}
\eta_{\mathfrak{P}}(x)e_{\mathfrak{P}}(\varpi_{\mathfrak{P}}^{-e(\mathfrak{P})}(2\delta)_{\mathfrak{P}}^{-1}x)
\end{align*}
where $\varpi_{\mathfrak{P}}$ denotes a uniformiser of the local field $F_{\mathfrak{P}}$ and $e(\mathfrak{P})$ denotes the exponent of $\mathfrak{P}$ in the conductor of $\eta$. In the equation above we diagonally embed $2\delta$ into the 
id\`ele group $\mathbb{A}_F^\times$ and denote its $\mathfrak{P}$-component by $(2\delta)_{\mathfrak{P}}$.
\end{thm}

In Theorem~\ref{thm:KHT} we use the following convention on multi-indices;
\begin{align*}
\Omega_{\mathrm{CM}, ?}^{w\mathsf{t}+2\mathsf{r}}=\prod_{\sigma \in
 \Sigma}\Omega_{\mathrm{CM},? ,\sigma}^{w+2r_\sigma} \quad \text{for } ?=\text{CM or $p$}, \qquad 
|\mathsf{r}|=\sum_{\sigma \in \Sigma} r_\sigma, \\
 \Gamma_\Sigma(w\mathsf{t}+\mathsf{r})=\prod_{\sigma \in \Sigma}\Gamma(w+r_{\sigma}), \quad \mathrm{Im}(2\delta)^{\mathsf{r}}=\prod_{\sigma\in\Sigma} \mathrm{Im}(\sigma(2\delta))^{r_{\sigma}}.
\end{align*}

\begin{rem} 
 The $p$-adic $L$-function $\mathcal{L}_{p, \Sigma}^{\mathrm{KHT}}(F)$ {\em does depend} on the choice of $\delta$ satisfying 
 the conditions $(1_\delta)$ and $(2_\delta)$, but we can explicitly describe effects on the interpolation formula  when we replace $\delta$ 
 by another purely imaginary element $\delta'$ satisfying the polarisation conditions $(1_\delta)$ and $(2_\delta)$. 
 In particular, we readily observe that the $p$-adic valuation of $\mathcal{L}_{p, \Sigma}^{\mathrm{KHT}}(F)$ does not change 
 after such a replacement of $\delta$. Refer, for example, to \cite[Section~5.8]{katz}. 
 We also remark that the assumption (unr$_{F^+}$) is not required 
 in Katz, Hida and Tilouin's original construction of $\mathcal{L}_{p,\Sigma}^{\mathrm{KHT}}(F)$. However, 
 without (unr$_{F^+}$), the constructed $p$-adic $L$-function $\mathcal{L}_{p,\Sigma}^{\mathrm{KHT}}(F)$ might not be an element of $\widehat{\mathcal{O}}^{\mathrm{ur}}[[\mathrm{Gal}(F_{\mathfrak{C}p^\infty}/F)]]$ 
 but an element of $\mathcal{O}_{\mathbb{C}_p}[[\mathrm{Gal}(F_{\mathfrak{C}p^\infty}/F)]]$.
\end{rem}

%
\subsection{Comparison of the $p$-adic $L$-functions} \label{ssc:comparison}
%

In this subsection we specialise Katz, Hida and Tilouine's $p$-adic $L$-function of 
  $\mathcal{L}^\mathrm{KHT}_{p, \Sigma}(F)$ for a CM number field 
 to a fixed arithmetic weight parametrisation 
and obtain an element of the cyclotomic Iwasawa algebra
which interpolates critical values of the $L$-function 
associated to a Hilbert modular form with complex multiplication. 
At the end of this subsection, we compare the specialised element  
 with the $p$-adic $L$-function $\mathcal{L}_p^\mathrm{cyc}(f)$ 
 constructed by the second-named author, and formulate a certain conjecture on the relation between Katz's complex CM periods and 
 modular symbolic complex periods.
 
%
\subsubsection{Comparison of the interpolation formulae}
%

 Let the notation be as in the previous subsection. Recall that, in particular, $\eta$ denotes a $\Sigma$-admissible gr\"o{\ss}encharacter of type $(A_0)$ on $F$) for an appropriate $p$-ordinary CM type $\Sigma$. Now assume that 
 $\vartheta(\eta)$ is a primitive form; in particular, 
 $\vartheta(\eta)$ is an eigenvector with respect to the Hecke operator $T(\mathfrak{p})$ (or $U(\mathfrak{p})$) for each $\mathfrak{p}$ lying above $p$.
 Then, as we have already seen at the end of Section~\ref{sssc:CMforms},
 the cuspform $\vartheta(\eta)$ is nearly $p$-ordinary if and only 
 if  $\eta$ is ordinary with respect to $\Sigma$,
that is, $\eta$ is unramified at every $\mathfrak{P}$ contained in $\Sigma_p$.
 Assume that $\eta$ is ordinary with respect to $\Sigma$ and let $\vartheta(\eta)^{p\text{-st}}$ denote the $p$-stabilisation of $\vartheta(\eta)$. 
 In particular, the eigenvalue with respect to the normalised Hecke operator
 $U_0(\mathfrak{p})$ is given by $\{\mathfrak{p}^{\kappa_1}\}^{-1}\eta^*(\mathfrak{P})$ for each prime ideal $\mathfrak{p}$ lying above $p$; 
 here $\mathfrak{P}$ is an element of $\Sigma_p$ satisfying $\mathfrak{PP}^c=\mathfrak{p}$. One readily observes that the $p$-adic valuation 
 of $\{ \mathfrak{p}^{\kappa_1}\}^{-1} \eta^*(\mathfrak{P})$ coincides 
 with that of the evaluation of the $p$-adic avatar $\hat{\eta}$ 
 introduced in Section~\ref{sssc:grossen} at a uniformiser $\varpi_{\mathfrak{P}}$ of $F_\mathfrak{P}$. Since the $p$-adic avatar $\hat{\eta}$ takes values into $p$-adic units, we see that  $\vartheta(\eta)^{p\text{-st}}$ is indeed nearly ordinary at $p$.

 \begin{pro} \label{prop:sp}
Assume that we have chosen and fixed an element $\delta$ of $F$ satisfying 
  the conditions $(1_\delta)$ and $(2_\delta)$ in Section~$\ref{sssc:Kperiod}$.
 Then under the settings and the notation as above, there exists 
an element $\mathcal{L}^\mathrm{cyc}_{p,\mathrm{CM}}(\eta)$ of 
 $\mathcal{O}_{\mathbb{C}_p}[[\mathrm{Gal}(F^+(\mu_{p^{\infty}})/F^+)]]$ characterised by 
 the interpolation property
 \begin{align} \label{eq:interp_CM_special}
 \chi_{p,\mathrm{cyc}}^j\phi(\mathcal{L}^{\mathrm{cyc}}_{p,\mathrm{CM}}(\eta))=
 \Gamma(j\mathsf{t}-\kappa_{\mu,1})G(\phi) 
\frac{\displaystyle \prod_{\mathfrak{p} \mid
  p\mathfrak{r}_{F^+}}A_{\mathfrak{p}}(\vartheta(\eta)^{p\text{-st}}, \phi, j) L(\vartheta(\eta)^{p\text{-st}}, \phi^{-1}, j)}{(-2\pi \sqrt{-1})^{jd} 
(-\widetilde{\Omega}_{\mathrm{CM},\infty})^{\kappa_{\mu,2}} 
(-\Omega_{\mathrm{CM},\infty})^{-\kappa_{\mu,1}}} 
 \end{align}
 for an arbitrary natural number $j$ satisfying $\kappa^\mathrm{max}_{\mu ,1 }+1\leq j \leq
  (\kappa_{\mu, 2})_\mathrm{min}$ and an arbitrary character $\phi$ of finite order
  of $\mathrm{Gal}(F^+(\mu_{p^\infty})/F^+)$.
 Here $\Omega_{\mathrm{CM},\infty}$ denotes Katz's complex CM period
  introduced in Section~$\ref{sssc:Kperiod}$ and  
$\widetilde{\Omega}_{\mathrm{CM},\infty}$ denotes the modified complex 
CM period defined by 
 $((2\delta)^{-1}\otimes 2\pi\sqrt{-1})^{-1}
  \Omega_{\mathrm{CM},\infty}$, which we regard as 
 an element of $(F^+\otimes \mathbb{C})^\times)$. 
The Gaussian sum $G(\phi^{-1})$ 
 is defined as $(\ref{eq:gauss})$, and 
 the $\mathfrak{p}$-adic multiplier $A_{\mathfrak{p}}(\vartheta(\eta)^{p\text{-st}},\phi,j)$ is defined by 
 \begin{align*}
 A_{\mathfrak{p}}(\vartheta(\eta)^{p\text{-st}}, \phi, j)=
 \begin{cases}
 \displaystyle  1-\frac{\phi(\mathfrak{p})\mathcal{N}\mathfrak{p}^{j-1}}{\eta^*(\mathfrak{P})} & 
 \text{if $\mathfrak{p}$ does not divide $\mathfrak{C}(\phi)$}, \\
\displaystyle \left( \frac{\mathcal{N}\mathfrak{p}^{j-1}}{\eta^*(\mathfrak{P})} \right)^{\mathrm{ord}_{\mathfrak{p}}(\mathfrak{C}(\phi))} & 
 \text{if $\mathfrak{p}$ divides $\mathfrak{C}(\phi)$}.
 \end{cases}
 \end{align*}
 \end{pro}
 
\begin{rem}
 Prasanna and the second-named author have already constructed an object corresponding to $\mathcal{L}^{\mathrm{cyc}}_{p,\mathrm{CM}}(\eta)$ in elliptic modular cases \cite[Theorem~2.4]{OP}, 
 and our modified complex CM period 
$\widetilde{\Omega}_{\mathrm{CM}, \infty}$ is 
 a counterpart of the CM period $\Omega_\infty^{\mathrm{CM}}$ defined
 there.
\end{rem}

 We firstly prepare two  elementary lemmata required for the construction of
 the $p$-adic measure $\mathcal{L}^{\mathrm{cyc}}_{p,\mathrm{CM}}(\eta)$.

\begin{lem} \label{lem:unit_index}
The index $(\mathfrak{r}_F^\times : \mathfrak{r}_{F^+}^\times)$ is
 relatively prime to $p$.
\end{lem}

\begin{proof}
Let $W_F$ (resp.\ $W_{F^+}$) denote the group of roots of unity 
contained in $F$ (resp.\ in $F^+$).  Then
 we have $(\mathfrak{r}_F^\times : \mathfrak{r}_{F^+}^\times) = Q_F(W_F : W_{F^+}) $ where $Q_F$ is the Hasse's unit index
defined as $(\mathfrak{r}_F^\times : W(F) \mathfrak{r}_{F^+}^\times)$.  
 It is known that $Q_F$ can take only two possible values $1$ or $2$ (see \cite[Satz~14]{Hasse} or
 \cite[Theorem~4.12]{Was}). In
 particular $Q_F$ is not divisible by $p$ because we assume that $p$ is
 odd.  Moreover the unit group $\mathfrak{r}_F^\times$
 contains none of the  $p$-power roots of unity since
 $F/\mathbb{Q}$ is unramified at every place lying above $p$ due to the
 assumptions (unr$_{F^+}$) and (ord$_{F/F^+}$). Hence the index $(W_F:W_{F^+})$ is not
 divisible by $p$ either. This completes the proof. 
\end{proof}

 \begin{lem} \label{lem:disc}
 Let $\delta$ be an element of $F$ satisfying the conditions $(1_\delta)$ and $(2_\delta)$ at the beginning of  Section~$\ref{sssc:Kperiod}$. 
 Then the product $\prod_{\sigma \in \Sigma} \sigma(2\delta)$ is an element of $\mathbb{Z}_p^\times$.
 \end{lem}

 \begin{proof}
Recall that $2\delta$ is a generator of
the $p$-part of the absolute discriminant of $F$, or in other words, 
$2\delta(\mathfrak{r}_F \otimes_\mathbb{Z}
  \mathbb{Z}_p)=\mathfrak{d}_F\otimes_\mathbb{Z}
  \mathbb{Z}_p$ holds (refer to \cite[Lemma~(5.7.35)]{katz}). 
The assumptions (unr$_{F^+}$) and (ord$_{F/F^+}$) imply that $\mathfrak{d}_F
  \otimes_\mathbb{Z} \mathbb{Z}_p$ is trivial, and hence 
$2\delta$ is a $p$-adic unit; that is, the image of $2\delta$ in
  $F\otimes_\mathbb{Z} \mathbb{Z}_p$ is contained in $(\mathfrak{r}_F
  \otimes_\mathbb{Z} \mathbb{Z}_p)^\times$. Then the image of $2\delta$ 
under the composition 
\begin{align*}
(\mathfrak{r}_F \otimes_\mathbb{Z} \mathbb{Z}_p)^\times
 \xrightarrow{\text{projection}} \underset{\mathfrak{P}\in
 \Sigma_p}{\prod} \mathfrak{r}_{F, \mathfrak{P}}^\times
 \stackrel{\sim}{\longrightarrow} (\mathfrak{r}_{F^+} \otimes_\mathbb{Z}
 \mathbb{Z}_p)^\times \xrightarrow{\mathrm{Nr}_{F^+/\mathbb{Q}}\otimes \mathrm{id}}
 \mathbb{Z}_p^\times
\end{align*}
coincides with $\prod_{\sigma \in \Sigma} \sigma(2\delta)$ by definition, where the middle isomorphism in
  the diagram above is the identification induced by the fixed
  $p$-ordinary CM type $\Sigma$:
\begin{align*}
\mathfrak{r}_{F^+} \otimes_\mathbb{Z} \mathbb{Z}_p \xrightarrow{\sim}
 \prod_{\mathfrak{P} \in \Sigma_p} \mathfrak{r}_{F, \mathfrak{P}} \, ;
 \quad 
 x\otimes 1 \mapsto (\iota_p \circ \sigma(x))_{\iota_p \circ\sigma \in \Sigma_p}.
\end{align*}
  Therefore $\prod_{\sigma \in \Sigma} \sigma(2\delta)$ is an element of $\mathbb{Z}_p^\times$.
 \end{proof}
 
 Taking the bijection $\Sigma \xrightarrow{\sim} I_{F^+}; \sigma \mapsto \sigma\vert_{F^+}$ into account, we abbreviated the product $\prod_{\sigma \in \Sigma}\sigma(2\delta)$ as $(2\delta)^\mathsf{t}$ in the following arguments.

We now explain how to construct the $p$-adic measure
$\mathcal{L}_{p,\mathrm{CM}}^\mathrm{cyc}(\eta)$. Let $\mathfrak{C}$ denote 
the prime-to-$p$ part of the conductor of $\eta$ and consider 
Katz, Hida and Tilouine's measure
$\mathcal{L}_{p,\Sigma}^\mathrm{KHT}(F)$ introduced in
Theorem~\ref{thm:KHT}, which is by construction 
an element of
$\widehat{\mathcal{O}}^\mathrm{ur}[[\mathrm{Gal}(F_{\mathfrak{C}p^\infty}/F)]]$. 
Define a $p$-adic integer $\tilde{d}=\tilde{d}(\delta)$
to be  $(-1)^d(2\delta)^{\mathsf{t}}
=\prod_{\sigma \in \Sigma}\sigma(-2\delta)$, which is indeed a $p$-adic unit
in $\mathbb{Z}_p^\times$ by Lemma~\ref{lem:disc}.
Hence there exists a unique element $z_{\tilde{d}}$ of
$\mathrm{Gal}(F(\mu_{p^\infty})/F)$ corresponding to $\tilde{d}$ via $\chi_{p,\mathrm{cyc}}$. 

Let $\sigma_0$ denote an element of $\Sigma$ such that $\mu_\sigma 
(=\kappa_{\mu,1,\sigma|_{F^+}})$ takes the maximum $\kappa_{\mu,1}^{\mathrm{max}}$, 
and let $\mathfrak{P}_0$ denotes the corresponding element of
$\Sigma_p$: the prime ideal of $F$ induced by the embedding 
$\iota_p\circ \sigma_0 \colon F \hookrightarrow \overline{\mathbb{Q}}_p$. 
We denote the lay class field modulo $\mathfrak{C}
\prod_{\mathfrak{P}^*\in (\Sigma_p \cup \Sigma_p^c)\setminus \{ \mathfrak{P}_0\}} (\mathfrak{P}^*)^\infty$ over $F$ 
by $F^{(0)}_{\mathfrak{C}p^\infty}$.
Since $\eta$ is ordinary with respect to $\Sigma$, the $p$-adic avatar
of $\eta \lvert \cdot \rvert_{\mathbb{A}_F}^{\kappa_{\mu,1}^{\mathrm{max}}}$ is unramified
at $\mathfrak{P}_0$ by its construction, and therefore the corresponding
Galois character
$\eta^\mathrm{gal}\chi_{p,\mathrm{cyc}}^{\kappa_{\mu,1}^{\mathrm{max}}}$ 
factors through the Galois group 
$\mathrm{Gal}(F_{\mathfrak{C}p^\infty}^{(0)}/F)$. 
The ray class field $F_{\mathfrak{C}p^\infty}^{(0)}$ is contained in
$F_{\mathfrak{C}p^\infty}$ by definition. On the other hand,
 the $p$-adic cyclotomic extension $F(\mu_{p^\infty})/F$ is totally ramified
 at $\mathfrak{P}$ since $F$ does not contain any $p$-power root of unity due to the conditions (unr$_{F^+}$) and (ord$_{F/F^+}$). By compairing
 ramification, we see that $F^{(0)}_{\mathfrak{C}p^{\infty}}$ is
linearly disjoint from $F(\mu_{p^\infty})$ over $F$. The Galois
group of $F(\mu_{p^\infty})F_{\mathfrak{C}p^\infty}^{(0)}/F$ is thus 
decomposed as the direct product of $\mathrm{Gal}(F(\mu_{p^\infty})/F)$ and
$\mathrm{Gal}(F_{\mathfrak{C}p^\infty}^{(0)}/F)$. Via this decomposition 
we regard $(z_{\tilde{d}}, 1)$ as an element of
$\mathrm{Gal}(F(\mu_{p^\infty})F_{\mathfrak{C}p^\infty}^{(0)}/F)$ and 
let $\tilde{z}_{\tilde{d}}$ denote its arbitrary lift to
$\mathrm{Gal}(F_{\mathfrak{C}p^\infty}/F)$. Then we define  
$\mathcal{L}_{p,\mathrm{CM}}^\mathrm{cyc}(\eta)\in \widehat{\mathcal{O}}^\mathrm{ur}[[\mathrm{Gal}(F^+(\mu_{p^\infty})/F^+)]]$
as the product of a power of Katz's $p$\nobreakdash-adic CM period
$\Omega_{\mathrm{CM},p}^{-(\kappa_{\mu,2}-\kappa_{\mu ,1})}$ and 
the image of the element $\mathrm{Ex}(F,\delta)
\mathcal{L}_{p,\Sigma}^\mathrm{KHT}(F)$ under the map 
\begin{align*}
\widehat{\mathcal{O}}^\mathrm{ur}[[\mathrm{Gal}(F_{\mathfrak{C}p^\infty}/F)]]
 \rightarrow
 \widehat{\mathcal{O}}^\mathrm{ur}[[\mathrm{Gal}(F(\mu_{p^\infty})/F)]]
 \xrightarrow{\sim}
 \widehat{\mathcal{O}}^\mathrm{ur}[[\mathrm{Gal}(F^+(\mu_{p^\infty})/F^+)]]
\end{align*}
defined as $g \mapsto \eta^\mathrm{gal}(g)g|_{F^+(\mu_{p^\infty})}$ 
for each element $g$ of $\mathrm{Gal}(F_{\mathfrak{C}p^{\infty}}/F)$. 
Here $\mathrm{Ex}(F,\delta)$ denotes an element defined by the product 
$(\mathfrak{r}_F^\times : \mathfrak{r}_{F^+}^\times)^{-1}  \sqrt{|D_{F^+}|}
(-2\delta)^{-\kappa_{\mu,1}^{\mathrm{max}}\mathsf{t}}\tilde{z}_{\tilde{d}(\delta)}^{-1}$, which we call the {\em  extra factor}.
Note that $\mathcal{L}_{p,\mathrm{CM}}^\mathrm{cyc}(\eta)$ is indeed $p$-adically integral since $Ex(F,\delta)$ is a $p$-adic unit by
virtue of Lemma~\ref{lem:unit_index}, Lemma~\ref{lem:disc} and the assumption (unr$_{F^+}$).

 \begin{proof}[{Proof of Proposition~$\ref{prop:sp}$}]
We shall deduce the desired interpolation formula of
  $\mathcal{L}_{p,\mathrm{CM}}^\mathrm{cyc}(\eta)$ from the interpolation formula of
  $\mathcal{L}_{p,\Sigma}^{\mathrm{KHT}}(F)$ (see
  Theorem~\ref{thm:KHT}). Note that the evaluation of
  $\mathcal{L}_{p,\mathrm{CM}}^\mathrm{cyc}(\eta)$ at $\chi_{p,\mathrm{cyc}}^j\phi$ is
  exactly the same as the  evaluation of  
  $\mathrm{Ex}(F,\delta)\mathcal{L}_{p,\Sigma}^{\mathrm{KHT}}(F)$ at the
  character $\eta^{\mathrm{gal}}\chi_{p,\mathrm{cyc}}^j\phi$ by definition.

\begin{component}[$\blacktriangleright$ Extra factor] 
--- The evaluation of the extra factor $\mathrm{Ex}(F,\delta)$
 at the character $\eta^\mathrm{gal}\chi_{p,\mathrm{cyc}}^j\phi$ is
 calculated as
\begin{equation} \label{eq:extra}
\begin{aligned}
\eta^\mathrm{gal}\chi_{p,\mathrm{cyc}}^j\phi(\mathrm{Ex}(F,\delta)) &=
(\mathfrak{r}_F^\times : \mathfrak{r}_{F^+}^\times)^{-1}\sqrt{|D_{F^+}|}
(-2\delta)^{-\kappa_{\mu,1}^{\mathrm{max}}\mathsf{t}}
\eta^\mathrm{gal} \chi_{p,\mathrm{cyc}}^j \phi(\tilde{z}_{\tilde{d}}^{-1}) \\
&= (\mathfrak{r}_F^\times :
 \mathfrak{r}_{F^+}^\times)^{-1}\sqrt{|D_{F^+}|} (-2\delta)^{-\kappa_{\mu,1}^{\mathrm{max}}\mathsf{t}} 
 \chi_{p,\mathrm{cyc}}^{j-\kappa_{\mu,1}^{\mathrm{max}}}\phi(z_{\tilde{d}}^{-1}) \eta^\mathrm{gal}\chi_{p,\mathrm{cyc}}^{\kappa^\mathrm{max}_{\mu,1}}(1) \\
&= (\mathfrak{r}_F^\times : \mathfrak{r}_{F^+}^\times)^{-1}
 \sqrt{|D_{F^+}|} \phi(z_{\tilde{d}})^{-1} (-2\delta)^{-j\mathsf{t}}.
\end{aligned}
\end{equation}
At the second equality, we just replace $\tilde{z}_{\tilde{d}}$ by the corresponding element $(z_{\tilde{d}},1)$ in the product  $\mathrm{Gal}(F(\mu_{p^\infty})/F)\times \mathrm{Gal}(F^{(0)}_{\mathfrak{C}p^\infty}/F)$.
\end{component}

\begin{component}[$\blacktriangleright$ Interpolation region]
--- The infinity type $\mu_j$ of the gr\"o{\ss}encharacter 
$\eta^{\mathrm{gal}}\chi_{p,\mathrm{cyc}}^j \phi$ is given 
by $\sum_{\sigma \in \Sigma}
 ((\mu_\sigma-j)\sigma+(\mu_{\bar{\sigma}}-j) \bar{\sigma})$. 
We then define integers $w_j$ and $r_{j,\sigma}$ by the equation
\begin{align*}
\mu_j=-w_j\mathsf{t}-\sum_{\sigma \in \Sigma}r_{j,\sigma}
 (\sigma-\bar{\sigma}).
\end{align*}
More concretely $w_j$ and $r_{j,\sigma}$ are defined as follows: 
\begin{align*}
w_j=2j-[\kappa_\mu] \text{\quad and \quad }
 r_{j,\sigma}=\mu_{\bar{\sigma}}-j =\kappa_{\mu,2,\sigma|_{F^+}}-j \quad
 \text{for each $\sigma$ in $\Sigma$}. 
\end{align*}
The interpolation region of Katz, Hida and Tilouine's measure for
 gr\"o{\ss}encharacters of the form
 $\eta^\mathrm{gal}\chi_{p,\mathrm{cyc}}^j \phi$ is given by 
\begin{quote}
($w_j\geq 1$ and $r_{j,\sigma}\geq 0 \quad {}^\forall \sigma\in 
\Sigma$)\  or \  ($w_j\leq 1$ and $r_{j,\sigma}+w_j-1\geq 0 \quad {}^\forall
 \sigma \in \Sigma$).
\end{quote}
The solution of the simultaneous inequalities above with respect to $j$ is 
then calculated as $\max \{\mu_\sigma \mid \sigma \in \Sigma\} +1  \leq j \leq
 \min \{ \mu_{\bar{\sigma}} \mid \sigma \in \Sigma\}$, which coincides with the
 desired interpolation region $\kappa_{\mu,1}^\mathrm{max}+1\leq j\leq (\kappa_{\mu,2})_\mathrm{min}$ by the definition of $\kappa_\mu=(\kappa_{\mu,1},\kappa_{\mu,2})$.
\end{component}

\begin{component}[$\blacktriangleright$ Periods]
--- The equality $\Omega_{\mathrm{CM}, ?}^{w_j\mathsf{t}+2\mathsf{r}_j}=\Omega_{\mathrm{CM},
 ?}^{\kappa_{\mu,2}-\kappa_{\mu,1}}$ holds for $\mathsf{r}_j=\sum_{\sigma
 \in \Sigma} r_{j,\sigma} \sigma$ and $?=\infty,\ p$. In particular 
the contribution of the $p$-adic CM period appearing 
in the interpolation formula
 of $\mathcal{L}_{p,\Sigma}^{\mathrm{KHT}}(F)$ is canceled by the
 construction of $\mathcal{L}_{p,\mathrm{CM}}^\mathrm{cyc}(\eta)$. Moreover we obtain 
\begin{align} \label{eq:complex_period}
\frac{1}{\Omega_{\mathrm{CM},\infty}^{w_j\mathsf{t}+\mathsf{r}_j}} 
=\frac{1}{\Omega_{\mathrm{CM},\infty}^{\kappa_{\mu,
 2}-\kappa_{\mu,1}}} 
=\frac{(-1)^{[\kappa_\mu]d}(-2\delta)^{\kappa_{\mu,2}}}{(-2\pi \sqrt{-1})^{|\kappa_{\mu,2}|}} \cdot 
\frac{1}{(-\widetilde{\Omega}_{\mathrm{CM},\infty})^{\kappa_{\mu, 2}}(-\Omega_{\mathrm{CM},
 \infty})^{-\kappa_{\mu,1}}}.
\end{align}
\end{component}

\begin{component}[$\blacktriangleright$ $L$-value] 
--- First note that 
 the Galois character $\eta^\mathrm{gal}\chi_{p,\mathrm{CM}}^j\phi$
 corresponds to the ideal character
 $\eta^*\mathcal{N}^{-j}\phi^{-1}$ (see the paragraphs before Theorem~\ref{thm:Hilb_L-func}). Here we define $\phi^{-1}(\mathfrak{A})$ 
for a fractional ideal $\mathfrak{A}$ of $F$ by
 $\phi^{-1}(\mathcal{N}_{F/F^+}\mathfrak{A})$. The gr\"o{\ss}encharacter $\eta$
 is ramified at every prime ideal $\mathfrak{L}$ dividing $\mathfrak{C}$
 since $\mathfrak{C}$ is by definition the prime-to-$p$ part of the conductor of $\eta$, and hence 
the local term $1-\eta^*\mathcal{N}^{-j}\phi^{-1}(\mathfrak{L})$ at such $\mathfrak{L}$ equals $1$. By comparing the Dirichlet series expressions, we can readily check that  the $L$-value $L(\eta^*\mathcal{N}^{-j}\phi^{-1},0)$ exactly equals $L(\vartheta(\eta), \phi^{-1},j)$ by the construction of the theta lift.
 For each place $\mathfrak{p}$ of $F^+$ lying above $p$, the
 local factor of $L(\vartheta(\eta), \phi^{-1}, j)$ at $\mathfrak{p}$ is
 given by 
\begin{align*}
L_\mathfrak{p}(\vartheta(\eta),
 \phi,j)&=\{1-C(\vartheta(\eta),\mathfrak{p})\phi^{-1}(\mathfrak{p})\mathcal{N}
 \mathfrak{p}^{-j}+\eta^*(\mathfrak{p}\mathfrak{r}_F)\nu_{F/F^+}^*\phi^{-2}(\mathfrak{p}) \mathcal{N}
 \mathfrak{p}^{-2j}\}^{-1} \\
&= (1-\eta^*\phi^{-1}(\mathfrak{P})\mathcal{N}
 \mathfrak{P}^{-j})^{-1}(1-\eta^*\phi^{-1}(\mathfrak{P}^c)(\mathcal{N}
 \mathfrak{P}^c)^{-j})^{-1}
\end{align*}
for $\mathfrak{p}=\mathfrak{PP}^c$ with $\mathfrak{P}$ 
in $\Sigma_p$. Therefore  the equality
\begin{align*}
(1-\eta^*\phi^{-1}(\mathfrak{P}^c)(\mathcal{N}\mathfrak{P}^c)^{-j})L_\mathfrak{p}(\vartheta(\eta),
 \phi, j)=(1- \eta^*\phi^{-1}(\mathfrak{P}) \mathcal{N}\mathfrak{P}^{-j})^{-1}
\end{align*}
holds and the right hand side of this equation is no other than the local
 component at $\mathfrak{p}$ of the $L$-value
 $L(\vartheta(\eta)^{p\text{-st}},\phi,j)$ by the definition of the
 $p$-stabilisation of $\vartheta(\eta)$. Finally we obtain by the
 definition of the dual gr\"o{\ss}encharacter an equation
\begin{align*}
1-(\eta^* \mathcal{N}^{-j}\phi^{-1})\check{\mathstrut}\,(\mathfrak{P}^c) =
\begin{cases}
\displaystyle
 1-\frac{\phi(\mathfrak{p})\mathcal{N}\mathfrak{P}^{j-1}}{\eta^*(\mathfrak{P})} &
 \text{when $\phi$ is unramified at $\mathfrak{P}$}, \\
1 & \text{otherwise}
\end{cases}
\end{align*}
for each $\mathfrak{P}$ in $\Sigma_p$ (recall that
 $\eta^*\mathcal{N}^{-j}$ is unramified at $\mathfrak{P}$ due to the
 assumption (unr$_{F^+}$)), and it coincides with the $\mathfrak{p}$-adic
 multiplier $A_\mathfrak{p}(\vartheta(\eta)^{p\text{-st}}; \phi,j)$  
when $\mathfrak{p}=\mathcal{N}_{F/F^+}\mathfrak{P}$ does not divide
the conductor $\mathfrak{C}(\phi)$ of $\phi$.  Consequently we have 
\begin{equation} \label{eq:L-value}
\begin{aligned}
\prod_{\mathfrak{L}\mid
 \mathfrak{C}}(1-\eta^*\mathcal{N}^{-j}\phi^{-1}(\mathfrak{L})) 
 \prod_{\mathfrak{P}\mid
 \Sigma_p}\{(1-&\eta^*\mathcal{N}^{-j}\phi^{-1}(\mathfrak{P}^c))(1-(\eta^*\mathcal{N}^{-j}\phi^{-1})\check{\mathstrut}\,
 (\mathfrak{P}^c))\}L(\eta^*\mathcal{N}^{-j}\phi^{-1},
 0) \\
&=\prod_{\mathfrak{p}\mid p\mathfrak{r}_{F^+}, \mathfrak{p} \nmid \mathfrak{C}(\phi)}
 A_\mathfrak{p}(\vartheta(\eta)^{p\text{-st}}; \phi, j)
 L(\vartheta(\eta)^{p\text{-st}}, \phi^{-1}, j).
\end{aligned}
\end{equation}
\end{component}

\begin{component}[$\blacktriangleright$ Local $\varepsilon$-factor at $p$]
--- Recall that for each $\mathfrak{P}$ in $\Sigma_p$ and
 $\mathfrak{p}=\mathfrak{PP}^c$, we obtain a specified identification 
\begin{align} \label{eq:identification}
\mathfrak{r}_{F^+, \mathfrak{p}} \xrightarrow{\sim} \mathfrak{r}_{F, \mathfrak{P}}
\end{align}
induced by $\tau_\mathfrak{p}(x) \mapsto \sigma_\mathfrak{P}(x)$ for 
 each element $x$ of $\mathfrak{r}_{F^+}$, where  
$\sigma_\mathfrak{P}\colon F\hookrightarrow \overline{\mathbb{Q}}$ 
denotes a unique  embedding contained in $\Sigma$ 
such that $\iota_p \circ \sigma_\mathfrak{P}$ induces $\mathfrak{P}$, and
 $\tau_\mathfrak{p}$ denotes the restriction of $\sigma_\mathfrak{P}$ to
 $F^+$. Then $(2\delta)_\mathfrak{P}=\sigma_\mathfrak{P}(2\delta)$
 corresponds to a unique element
 $(2\delta)_\mathfrak{p}$ of $\mathfrak{r}_{F^+, \mathfrak{p}}^\times$
 under the identification (\ref{eq:identification}). We define
 $\varpi_\mathfrak{p}$ as an element of $\mathfrak{r}_{F^+,
 \mathfrak{p}}$ corresponding to the fixed uniformiser 
 $\varpi_\mathfrak{P}$ of $F_\mathfrak{P}$ via (\ref{eq:identification}), which is a
 uniformiser of $F^+_\mathfrak{p}$.

 \medskip
 
 Note that the id\`elic character corresponding to $\eta^\mathrm{gal}\chi_{p,\mathrm{cyc}}^j \phi$ is $\eta \lvert \cdot \rvert_{\mathbb{A}_F}^j \phi^{-1}$,  where $\lvert \cdot \rvert_{\mathbb{A}_F}$ denotes the id\`elic norm character on $F$ (see Example~\ref{ex:norm} for details). Here we use the same symbol $\phi^{-1}$ for the id\`elic character corresponding to the ideal character $\phi^{-1}$ by abuse of notation. Since $\eta \lvert \cdot \rvert_{\mathbb{A}_F}^j$ is unramified at each
 place $\mathfrak{P}$ in $\Sigma_p$ by virtue of the assumption (unr$_{F^+}$), 
the exponent $e(\mathfrak{P})$ of $\mathfrak{P}\in \Sigma_p$ 
 in the conductor of $\eta \lvert \cdot \rvert_{\mathbb{A}_F}^j \phi^{-1}$ exactly equals
 the exponent of $\mathfrak{p}=\mathfrak{PP}^c$ in the conductor of $\phi$, which we denote by 
$e_\phi(\mathfrak{p})$. By using these facts, we calculate the local $\varepsilon$-factor $W_p(\eta \lvert \cdot \rvert_{\mathbb{A}_F}^j\phi^{-1})$ in the following. For each $\mathfrak{P}$ in $\Sigma_p$, set 
 \begin{align*}
  W_\mathfrak{P}&(\eta \lvert\cdot \rvert_{\mathbb{A}_F}^j \phi^{-1}) \\
  =&\mathcal{N}\mathfrak{P}^{-e_\phi(\mathfrak{p})} (\eta\lvert \cdot \rvert_{\mathbb{A}_F}^j\phi^{-1})_\mathfrak{P}(\varpi_\mathfrak{P}^{-e_\phi(\mathfrak{p})}) \sum_{x\in (\mathfrak{r}_{F^+,\mathfrak{p}}/\mathfrak{p}^{e_\phi(\mathfrak{p})})^\times} (\eta \lvert \cdot \rvert_{\mathbb{A}_F}^j \phi^{-1})_{\mathfrak{P}}(x) e_{\mathfrak{P}}(\varpi_\mathfrak{P}^{-e_\phi(\mathfrak{p})}(2\delta)_\mathfrak{P}^{-1}x).
 \end{align*}
Then $W_p(\eta \lvert \cdot \rvert_{\mathbb{A}_F}^j \phi^{-1})$ obviously equals the product $\prod_{\mathfrak{P}\in \Sigma_p} W_\mathfrak{P}(\eta \lvert \cdot \rvert_{\mathbb{A}_F}^j \phi^{-1})$. Note also that when $e_\phi(\mathfrak{p})$ equals $0$, or in other words, when $\phi$ does not ramify at $\mathfrak{p}$, the local term $W_\mathfrak{P}(\eta \lvert \cdot \rvert_{\mathbb{A}_F}^j\phi^{-1})$ at $\mathfrak{P}$ is trivial by its definition.

Now assume that $\phi$ is ramified at $\mathfrak{p}=\mathfrak{PP}^c$. Then by direct computation, we have
\begin{multline*}
 \mathcal{N}
 \mathfrak{P}^{-e_\phi(\mathfrak{p})} (\eta\lvert \cdot \rvert_{\mathbb{A}_F}^j
 \phi^{-1})_\mathfrak{P}(\varpi_\mathfrak{P}^{-e_\phi(\mathfrak{p})}) 
= \displaystyle
 \left(
 \frac{\mathcal{N}\mathfrak{P}^{j-1}}{\eta^*(\mathfrak{P})}
 \right)^{e_\phi(\mathfrak{p})}
 \phi_\mathfrak{p}^{-1}(\mathcal{N}_{F/F^+}\varpi_\mathfrak{P}^{-e_\phi(\mathfrak{p})}) \\
= 
 A_\mathfrak{p}(\vartheta(\eta)^{p\text{-st}},\phi,j)
 \phi_\mathfrak{p}(\varpi_\mathfrak{p}^{e_\phi(\mathfrak{p})}).
\end{multline*}
Next let us take an element $\alpha$ of $F^+$
satisfying $\alpha_\mathfrak{p}=\varpi_\mathfrak{p}^{e_\phi(\mathfrak{p})}$ for each $\mathfrak{p}$ of $F^+$ lying above $p$.  
Note that $(\eta\lvert \cdot \rvert_{\mathbb{A}_F}^j)_\mathfrak{P}(x)$ is trivial for each $\mathfrak{P}$ in $\Sigma_p$ and an arbitrary element $x$ in
 $\mathfrak{r}_\mathfrak{P}^\times$ by virtue of the assumption (unr$_{F^+}$). 
 By using this, we can calculate as
 \begin{align} \label{eq:P-part_epsilon}
\begin{aligned}
 \phi_\mathfrak{p}(\varpi_\mathfrak{p}^{e_\phi(\mathfrak{p})}) &\sum_{x \in
 (\mathfrak{r}_{F,\mathfrak{P}}/\mathfrak{P}^{e(\mathfrak{P})})^\times}
 (\eta\lvert\cdot\rvert_{\mathbb{A}_F}^j
 \phi^{-1})_\mathfrak{P}(x)e_\mathfrak{P}(\varpi_\mathfrak{P}^{-e(\mathfrak{P})}(2\delta)_\mathfrak{P}^{-1}x) \\
&=
 \phi_\mathfrak{p}(\varpi_\mathfrak{p}^{e_\phi(\mathfrak{p})}) \sum_{x \in
 (\mathfrak{r}_{F^+,\mathfrak{p}}/\mathfrak{p}^{e_\phi(\mathfrak{p})})^\times}
 \phi_\mathfrak{p}^{-1}(x)e_\mathfrak{p}(\varpi_\mathfrak{p}^{-e_\phi(\mathfrak{p})}(2\delta)_\mathfrak{p}^{-1}x) \\ 
&= 
 \phi_\mathfrak{p}^{-1}((-2\delta)_\mathfrak{p}) \sum_{x \in
 (\mathfrak{r}_{F^+,
 \mathfrak{p}}/\mathfrak{p}^{e_\phi(\mathfrak{p})})^\times}
 \phi_\mathfrak{p}^{-1}(x\alpha^{-1}_\mathfrak{p})
 e_\mathfrak{p}(-x\alpha_\mathfrak{p}^{-1}) 
\end{aligned}
 \end{align}
 In the third equality we change the variable of the summation using the
 fact that the correspondance $x\mapsto x(-2\delta)_\mathfrak{p}$ induces
 an automorphism of $(\mathfrak{r}_{F^+,\mathfrak{p}}/\mathfrak{p}^{e_\phi(\mathfrak{p})})^\times$. Since $2\delta$ is a $p$-adic unit, 
 the evaluation of the $p$-adic avatar of the norm character
 $\lvert \cdot \rvert_{\mathbb{A}_{F^+}}$ at the id\`ele of $F^+$ defined as 
 $(2\delta)_{p\mathfrak{r}_{F^+}}=((2\delta)_\mathfrak{p})_{\mathfrak{p}\mid
 p\mathfrak{r}_{F^+}}$ is calculated as $\prod_{\mathfrak{P}\in
 \Sigma_p} \sigma_\mathfrak{P}(2\delta)^{\sharp \Sigma_\mathfrak{P}}$, 
 which coincides with  $(-1)^d\tilde{d}=(-1)^d\chi_{\mathrm{cyc},p}(z_{\tilde{d}})$. 
This implies that the id\`ele
 $(-2\delta)_{p\mathfrak{r}_{F^+}}$ corresponds to the element
 $z_{\tilde{d}}$ of $\mathrm{Gal}(F^+(\mu_{p^\infty})/F^+)$ via the
 reciprocity map in global class field theory. In particular, the product 
$\prod_{\mathfrak{p}\mid p\mathfrak{r}_{F^+}}
 \phi_\mathfrak{p}^{-1}((-2\delta)_\mathfrak{p})$ coincides with
 $\phi(z_{\tilde{d}})$. Note that,
 by the definition of the ideal characters corresponding to the gr\"o{\ss}encharacters, the equality $\prod_{\mathfrak{p}\mid
 p\mathfrak{r}_{F^+}}\phi_\mathfrak{p}^{-1}((x\alpha^{-1})_\mathfrak{p})=\phi(x\alpha^{-1}\mathfrak{r}_{F^+})$
 holds for an arbitrary element $x$ in $F^+$
 relatively prime to $p$. Taking the product of \eqref{eq:P-part_epsilon} over prime ideals $\mathfrak{p}$ of $\mathfrak{r}_{F^+}$ lying above $p$, we obtain the equality
\begin{align*}
\prod_{\mathfrak{p}\mid p\mathfrak{r}_{F^+}}  \phi_\mathfrak{p}^{-1}((-2\delta)_\mathfrak{p}) &\sum_{x\in
(\mathfrak{r}_{F^+, \mathfrak{p}}/\mathfrak{p}^{e_\phi(\mathfrak{p})})^\times}  
\phi_\mathfrak{p}^{-1}
 (x\alpha_\mathfrak{p}^{-1})e_\mathfrak{p}(-x\alpha_\mathfrak{p}^{-1}) \\
&= \phi(z_{\tilde{d}}) \sum_{x\in (\mathfrak{r}_{F^+}/\mathfrak{C}(\phi))^\times}
 \phi(x\alpha^{-1}\mathfrak{r}_{F^+}) \exp(2\pi \sqrt{-1} \mathrm{Tr}_{F^+/\mathbb{Q}}(x\alpha^{-1})),
\end{align*}
and the last term is no other than the product of $\phi(z_{\tilde{d}})$ and the Gaussian sum $G(\phi)$ defined as (\ref{eq:gauss}) under the isomorphism
\begin{align*}
(\mathfrak{C}(\phi)^{-1}/\mathfrak{r}_{F^+})^\times \xrightarrow{\sim}
 (\mathfrak{r}_{F^+}/\mathfrak{C}(\phi))^\times; \, x\mapsto \alpha x.
\end{align*}
Consequently we obtain 
\begin{align} \label{eq:epsilon_factor}
W_p(\eta\lvert \cdot\rvert_{\mathbb{A}_F}^j\phi^{-1})=\prod_{\mathfrak{P}\in \Sigma_p}W_\mathfrak{P}(\eta \lvert \cdot \rvert_{\mathbb{A}_F}^j \phi^{-1})=\phi(z_{\tilde{d}})G(\phi)\prod_{\mathfrak{p} \mid
 p\mathfrak{r}_{F^+}, \mathfrak{p} \mid \mathfrak{C}(\phi)}
 A_\mathfrak{p}(\vartheta(\eta)^{p\text{-st}}, \phi, j).
\end{align}
\end{component}
  
\begin{component}[$\blacktriangleright$ Other coefficients]
--- One calculates as 
\begin{equation} \label{eq:others}
\begin{aligned}
(\mathfrak{r}_F^\times : \mathfrak{r}_{F^+}^\times)&
 \frac{(-1)^{w_jd}(2\pi)^{|r_j|}
 \Gamma_\Sigma(w_j\mathsf{t}+\mathsf{r}_j)}{\sqrt{|D_{F^+}|}\mathrm{Im}(2\delta)^{\mathsf{r}_j}} \\
& = (\mathfrak{r}_F^\times :
 \mathfrak{r}_{F^+}^\times)\frac{(-1)^{2j-[\kappa_\mu]} (-2\pi
 \sqrt{-1})^{|\kappa_{\mu,2}-j|}\Gamma(j\mathsf{t}-\kappa_{\mu,1})}{\sqrt{|D_{F^+}|}(-2\delta)^{\kappa_{\mu,2}-j\mathsf{t}}} \\
&= \frac{\Gamma(j\mathsf{t}-\kappa_{\mu,1})}{(-2\pi\sqrt{-1})^{jd}}
 \cdot \frac{(-1)^{[\kappa_\mu]}(\mathfrak{r}_F^\times :
 \mathfrak{r}_{F^+}^\times)(-2\pi
 \sqrt{-1})^{|\kappa_{\mu,2}|}(-2\delta)^{j\mathsf{t}}}{\sqrt{|D_{F^+}|}(-2\delta)^{\kappa_{\mu,2}}}.
\end{aligned}
\end{equation}
\end{component}

Combining all the equations (\ref{eq:extra}),
  (\ref{eq:complex_period}),
  (\ref{eq:L-value}),(\ref{eq:epsilon_factor}) and 
(\ref{eq:others}), we obtain the desired interpolation formula
  (\ref{eq:interp_CM_special}) of the $p$-adic measure $\mathcal{L}_{p,\mathrm{CM}}^\mathrm{cyc}(\eta)$.
 \end{proof}

As in Section~\ref{sssc:p-adic_L_cusp}, we identify the archimedean part
of $\mathrm{Cl}^+_{F^+}(p^\infty\mathfrak{r}_{F^+})$ with $\{ \pm
1\}^{I_{F^+}}$, and denote by $w_\tau$ for each $\tau$ in $I_{F^+}$ the
element of the archimedean part of $\mathrm{Cl}^+_{F^+}(p^\infty
\mathfrak{r}_{F^+})$ such that only its $\tau$-component equals to $-1$.
Set $e_\tau^{\pm}=(1\pm w_\tau)/2$ and define an idempotent $e_\epsilon$
of $\mathcal{O}[[\mathrm{Cl}^+_{F^+}(p^\infty\mathfrak{r}_{F^+})]]$ by 
\begin{align*}
e_\epsilon=\prod_{\tau \in I_{F^+}} e_\tau^{\epsilon_\tau}
\end{align*}
for each $\epsilon=(\epsilon_\tau)_{\tau \in I_{F^+}}$ in $\{ \pm
1\}^{ I_{F^+}}$; then for a ray class character $\phi \colon
\mathrm{Cl}_{F^+}^+(p^\infty \mathfrak{r}_{F^+})\rightarrow
\mathbb{C}^\times$ of finite order, we have
\begin{align*}  
\phi(e_\epsilon)=\begin{cases}
1 & \text{when $\epsilon$ coincides with $\mathrm{sgn}(\phi)$}, \\
0 & \text{otherwise}.
\end{cases}
\end{align*}
By abuse of notation, we also use the same notation $e_\epsilon$ for
the image of $e_\epsilon$ under the composite map
\begin{align*}
\mathcal{O}[[\mathrm{Cl}^+_{F^+}(p^\infty \mathfrak{r}_{F^+})]]
 \xrightarrow{\mathrm{rec}}
 \mathcal{O}[[\mathrm{Gal}(F^+_{p^\infty\mathfrak{r}_{F^+}}/F^+)]]
 \twoheadrightarrow
 \mathcal{O}[[\mathrm{Gal}(F^+(\mu_{p^\infty})/F^+))]].
\end{align*}
Then, by comparing the interpolation formulae (\ref{eq:interp_cusp}) and
(\ref{eq:interp_CM_special}), we obtain the following equality holds between the two different 
$p$-adic $L$-functions $\mathcal{L}_{p,\mathrm{CM}}^\mathrm{cyc}(\eta)$
and $\mathcal{L}_p^\mathrm{cyc}(\vartheta(\eta))$ in
$\widehat{\mathcal{O}}^\mathrm{ur}[[\mathrm{Gal}(F^+(\mu_{p^\infty})/F^+)]]$:
\begin{align*}
\mathcal{L}_{p,\mathrm{CM}}^\mathrm{cyc}(\eta)=\sum_{\epsilon \in \{\pm 1\}^{I_{F^+}}} \frac{\Gamma((\kappa_{\mu,1}^{\mathrm{max}}+1)\mathsf{t}-\kappa_{\mu,1})
 C_{\vartheta(\eta), \infty}^{\epsilon}}{(-\widetilde{\Omega}_{\mathrm{CM},\infty})^{\kappa_{\mu,2}}(-\Omega_{\mathrm{CM},
 \infty})^{-\kappa_{\mu,1}}} e_\epsilon \mathcal{L}_p^\mathrm{cyc}(\vartheta(\eta)).
\end{align*}

 In particular we obtain the main result of the analytic part:
 
\begin{cor}[Theorem~\ref{Thm:L-func}] \label{cor:compare_p-adic_L}
The $p$-adic measure $\mathcal{L}_{p,\mathrm{CM}}^\mathrm{cyc}(\eta)$ is 
a nonzero
constant multiple of the cyclotomic $p$-adic zeta function
$\mathcal{L}_p^\mathrm{cyc}(\vartheta(\eta))$ associated to $\vartheta(\eta)$
in each component
of the semilocal Iwasawa algebra
$\widehat{\mathcal{O}}^\mathrm{ur}[[\mathrm{Gal}(F^+(\mu_{p^\infty})/F^+)]]$, 
and each of the two $p$-adic $L$-functions generates the same ideal in
 $\widehat{\mathcal{O}}^\mathrm{ur}[[\mathrm{Gal}(F^+(\mu_{p^\infty})/F^+)]]\otimes_{\mathbb{Z}_p}
 \mathbb{Q}_p$.
\end{cor}

It is widely believed that Corollary~\ref{cor:compare_p-adic_L} holds
even in
$\widehat{\mathcal{O}}^\mathrm{ur}[[\mathrm{Gal}(F^+(\mu_{p^\infty})/F^+)]]$
(or in other words, $\mathcal{L}_{p,\mathrm{CM}}^\mathrm{cyc}(\eta)$ and
$\mathcal{L}_p^\mathrm{cyc}(\vartheta(\eta))$ have ``the same $\mu$-invariants''),
and this speculation leads us to make a conjecture on the ratio of
two complex periods constructed in completely different manners.

\begin{conj} \label{conj:periods}
The ratio of the complex periods
\begin{align*}
\frac{\Gamma((\kappa_{\mu,1}^{\mathrm{max}}+1)\mathsf{t}-\kappa_{\mu,1})C_{\vartheta(\eta),
 \infty}^\epsilon}{(-\widetilde{\Omega}_{\mathrm{CM},
 \infty})^{\kappa_{\mu,2}}(-\Omega_{\mathrm{CM},\infty})^{-\kappa_{\mu,1}}}
\end{align*}
is a $p$-adic unit for an arbitrary element $\epsilon$ in $\{ \pm
 1\}^{I_{F^+}}$ with respect to the fixed embedding $\iota_p \colon
 \overline{\mathbb{Q}}\hookrightarrow \overline{\mathbb{Q}}_p$.
\end{conj}

\begin{rem}
In the case where $F^+$ is the rational number field $\mathbb{Q}$ (elliptic modular cases), the second-named author and Prasanna have obtained a partial result to this conjecture 
(see \cite[Theorem~6.1]{OP}) using the nonvanishing modulo $p$ of special values of the
	 $L$-functions associated to elliptic cuspform (due to 
Stevens \cite[Theorem~2.1]{Stevens} 
	 and Ash-Stevens \cite{AS}) and the modular parametrisation of 
an elliptic curve with complex multiplication. 
However,the nonvanishing modulo $p$ of special values of the $L$-functions 
has not been generalised to general Hilbert modular cuspforms
yet, and there seems to be no generalisation of the theory of modular
parametrisation to Hilbert-Blumenthal modular varieties. 
Hence it seems difficult to generalise the proof of \cite[Theorem~6.1]{OP} 
to general Hilbert modular cases at the present.
\end{rem}

%
%
\section{The algebraic side} \label{sc:algebraic_side}
%
%

We establish algebraic parts of our main results and apply them to the Iwasawa main conjecture for Hilbert modular cuspforms with complex multiplication
in this section. 
We first introduce the Selmer groups associated to 
nearly $p$-ordinary Hilbert cuspforms with complex multiplication, 
and compare them with the Iwasawa module obtained as a certain 
Galois group
$\mathrm{Gal}(M_{\Sigma_p}/\widetilde{K}^\mathrm{CM}_\infty)$ 
(Section~\ref{ssc:defSel}). 
Then we verify the {\em $($exact$)$ control theorem} (Theorem~\ref{thm:control}) 
which describes the behaviour of the (multi-variable) Selmer groups 
under specialisation procedures (Section~\ref{ssc:control}). 
We finally discuss the {\em almost divisibility} of the strict Selmer groups 
and basechange compatibility of the characteristic ideals of 
their Pontrjagin duals (Sections~\ref{ssc:greenberg} and 
\ref{ssc:specialisation}). As an application, we discuss the validity of the cyclotomic Iwasawa main conjecture for Hilbert modular cuspforms with
complex multiplication (Section~\ref{sc:specialisation_IMC}).

\subsection{Selmer groups} \label{ssc:defSel}
 
This subsection is devoted to the definition of various 
Selmer groups and the comparison among them. 
In Section~\ref{sssc:cycSel} we first recall the general definition 
of Selmer groups $\mathrm{Sel}_\mathcal{A}$ 
associated to deformations of Galois representations 
after Greenberg (Definition~\ref{def:Selmer_group}), and then 
define the Selmer group
$\mathrm{Sel}_{\mathcal{A}_f^\mathrm{cyc}}$ 
by applying this general recipe to the cyclotomic deformation $\mathcal{A}^\mathrm{cyc}_f$ of the Galois representation $A_f^\mathrm{cyc}$
associated to a nearly $p$-ordinary Hilbert eigencuspform $f$. 
When the Hilbert cuspform $f$ has complex multiplication, 
or in other words, when $f$ is represented as (the $p$-stabilisation of)
the theta lift $\vartheta(\eta)$ of a gr\"o{\ss}encharacter $\eta$ of type $(A_0)$ defined on 
a certain totally imaginary quadratic extension $F$ of $F^+$, 
we identify $\mathrm{Sel}_{\mathcal{A}_f^\mathrm{cyc}}$ 
with the Selmer group $\mathrm{Sel}_{\mathcal{A}^\mathrm{cyc}_\eta}^\Sigma$ 
associated to the cyclotomic deformation of $\eta$ 
defined with respect to a fixed CM type $\Sigma$ of $F$ (see
Lemma~\ref{lem:Shapiro}). 
In Section~\ref{sssc:comparison}, we introduce the (multi-variable)
Selmer group $\mathrm{Sel}_{\mathcal{A}^\mathrm{CM}_\eta}^\Sigma$ associated to the 
deformation of $\eta$ along the field extension
$\widetilde{F}_\infty/F$, and prove in Proposition~\ref{prop:iw_mod} that the characteristic ideal 
of the Pontrjagin dual of $\mathrm{Sel}_{\mathcal{A}^\mathrm{CM}_\eta}^\Sigma$ 
coincides with (a certain twist of) the characteristic 
ideal of the $\psi$\nobreakdash-isotypic quotient $X_{\Sigma_p, (\psi)}$ of the 
Iwasawa module $X_{\Sigma_p}$ defined in a classical way 
(here $\psi$ denotes a branch character associated to $\eta$; 
see Lemma~\ref{lem:branch} for details on the branch character $\psi$).

\subsubsection{General definition} \label{sssc:general_Selmer}

We first recall the general notion of Selmer groups 
for deformations of Galois representations, which 
is introduced by Ralph Greenberg \cite[Sections~3 and 4]{gr-mot}. 

 Let $\mathcal{R}$ be a complete, noetherian semilocal ring of
 characteristic $0$, and suppose that the residue field
 $\mathcal{R}/\mathfrak{M}$ of $\mathcal{R}$ is 
 a finite field of characteristic $p$ 
 for each maximal ideal $\mathfrak{M}$ of $\mathcal{R}$. 
Let $\mathsf{K}$ be a number field
 and $\mathcal{T}$ a free $\mathcal{R}$-module of finite rank 
on which the absolute Galois group $G_\mathsf{K}$ of $\mathsf{K}$ 
acts continuously and $\mathcal{R}$-linearly. 
 We impose the following constraint on the $\mathcal{R}$-linear
 $G_\mathsf{K}$-representation $\mathcal{T}$:
\begin{quotation}
 the Galois action on $\mathcal{T}$ is unramified outside 
a finite set $S$ of places of $\mathsf{K}$ which contains
all the places lying above $p$ and all the archimedean places.
\end{quotation}
Then the action of $G_\mathsf{K}$ on $\mathcal{T}$ factors through 
the Galois group $\mathrm{Gal}(\mathsf{K}_S/\mathsf{K})$ of 
the maximal Galois extension $\mathsf{K}_S$ over $\mathsf{K}$
which is unramified outside the places of $S$.
For each prime ideal $\mathfrak{p}$ of $\mathsf{K}$ lying above $p$, 
we specify an $\mathcal{R}$-direct summand
$\mathrm{Fil}^+_{\mathfrak{p}} \mathcal{T}$ of $\mathcal{T}$ which 
is stable under the action of the decomposition group 
$D_{\mathfrak{p}}$ of $G_\mathsf{K}$ at $\mathfrak{p}$. 
 In many cases, there exists a canonical (and unique) choice of 
such a direct summand  $\mathrm{Fil}^+_{\mathfrak{p}} \mathcal{T}$ 
for each $\mathfrak{p}$. 
If the Galois representation $\mathcal{T}$ is 
{\em ordinary} (or {\em nearly ordinary}) at each place $\mathfrak{p}$
 above $p$, for example, there exists a canonical 
direct summand $\mathrm{Fil}_\mathfrak{p}^+ \mathcal{T}$ of
 $\mathcal{T}$ induced from what is called 
 the {\em ordinary filtration} at each $\mathfrak{p}$. 
In Section~\ref{sssc:hyp} we shall introduce more
 general notion of local conditions concerning the definition of 
Selmer groups. We denote the Pontrjagin dual
 $\mathrm{Hom}_\mathrm{cts}(\mathcal{R}, \mathbb{Q}_p/\mathbb{Z}_p)$ of
 $\mathcal{R}$ by $\mathcal{R}^{\vee}$. Now consider 
 the discrete $\mathcal{R}$-module $\mathcal{A}$ defined 
as $\mathcal{A}=\mathcal{T}\otimes_\mathcal{R} \mathcal{R}^{\vee}$. 
The absolute Galois group $G_\mathsf{K}$ acts on $\mathcal{A}$ via the first factor 
and we regard $\mathcal{A}$ as a discrete
 $\mathcal{R}$-linear Galois representation of
 $\mathrm{Gal}(\mathsf{K}_S/\mathsf{K})$.  
Note that $\mathcal{A}$ is equipped 
with the specified direct summand $\mathrm{Fil}_\mathfrak{p}^+
 \mathcal{A}$ at each place $\mathfrak{p}$ above $p$ which is induced 
from the specification of direct summands of $\mathcal{T}$; 
namely, $\mathrm{Fil}_\mathfrak{p}^+ \mathcal{A}$
 is defined as $\mathrm{Fil}_\mathfrak{p}^+ \mathcal{T}
 \otimes_\mathcal{R} \mathcal{R}^\vee$. 

 \begin{defn}[Greenberg's Selmer group] \label{def:Selmer_group}
The {\em $($Greenberg's$)$ Selmer group} $\mathrm{Sel}_{\mathcal{A}}$ 
associated to $\mathcal{A}$ is defined as the kernel of 
the global-to-local morphism
 \[
 H^1 (\mathsf{K}_S /\mathsf{K} ,\mathcal{A}) \longrightarrow 
 \underset{\substack{\lambda \in S \\ \lambda \nmid p\infty}}{\prod} 
 H^1 (I_{\lambda} ,\mathcal{A})
 \times \underset{\substack{\mathfrak{p}\in S \\ \mathfrak{p} \vert p\mathfrak{r}_\mathsf{K}}}{\prod} 
 H^1 (I_{\mathfrak{p}} ,\mathcal{A}/\mathrm{Fil}^+_{\mathfrak{p}}
  \mathcal{A}) 
\]
induced by the restriction maps of Galois cohomology groups. 
Here $I_v$ denotes the inertia subgroup of the absolute Galois group
  $G_\mathsf{K}$ at each finite place $v$ of $\mathsf{K}$. 
 \end{defn} 

\subsubsection{Selmer groups for cyclotomic deformations} \label{sssc:cycSel}

Let us recall the following theorem which is due to many people 
including Ohta \cite{ohta}, Carayol \cite{carayol},  
 Wiles \cite{wiles},
 Taylor \cite{taylor1} and Blasius and Rogawski \cite{BR}.
 We quote \cite[Theorem 1,2]{wiles} for nearly ordinary situations 
as follows: 

\begin{thm}\label{theorem:construction_of_gal_rep}
Let $f$ be a nearly $p$-ordinary normalised eigencuspform defined on 
a totally real number field $F^+$ satisfying $($unr$)$, which is of 
cohomological weight $\kappa$, level $\mathfrak{N}$ and nebentypus
$\underline{\varepsilon}$. Let $\mathcal{K}$ be 
a finite extension of $\mathbb{Q}_p$ containing the Hecke field
$\mathbb{Q}_f$ of $f$ and $\mathcal{O}$ the ring of integers of
$\mathcal{K}$.
  
Then there exists a 2-dimensional Galois representation 
$V_f$ of $G_{F^+}$ with coefficients in $\mathcal{K}$ with the following properties$:$ 
\begin{enumerate}[label=$(\arabic*)$]
\item 
For every prime ideal $\mathfrak{q}$ which does not divide
$p\mathfrak{N}$, the following equation holds. 
\begin{align*}
\det(1-\mathrm{Frob}_\mathfrak{q} X ;
 V_f)=1-C(\mathfrak{q};f)X+\mathcal{N}\mathfrak{q} \varepsilon_+(\mathfrak{q})X^2
\end{align*}
\item 
For each place 
$\mathfrak{p}$ of $F^+$ lying above $p$, we have a $D_\mathfrak{p}$-stable $\mathcal{K}$-subspace
$\mathrm{Fil}_\mathfrak{p}^+ V_f \subset V_f$ of dimension one on which the action of
$D_\mathfrak{p}$ is unramified. 
\end{enumerate}
\end{thm}
The Galois representation $V_f$ of $G_{F^+}$ 
is called the {\em Galois representation associated to $f$}.

We define $S^+$ as a finite set of places of $F^+$ 
consisting of all the archimedean
places and all the finite places dividing $p\mathfrak{N}$. 
Then the action of $G_{F^+}$ on $V_f$ factors through the quotient
$\mathrm{Gal}(F^+_{S^+}/F^+)$ of $G_{F^+}$. 
Let $\Lambda^\mathrm{cyc}_{\mathcal{O}}$ denote the Iwasawa algebra 
$\mathcal{O}[[\mathrm{Gal}(F^+(\mu_{p^\infty})/F^+)]]$ over
$\mathcal{O}$. Note that $\Lambda^\mathrm{cyc}_{\mathcal{O}}$ satisfies
all the conditions which we have imposed on the coefficient ring $\mathcal{R}$ 
of a general Galois representation in \ref{sssc:general_Selmer}. 
For a $G_{F^+}$-stable $\mathcal{O}$-lattice $T_f$ of $V_f$ (that is, 
a $G_{F^+}$-stable $\mathcal{O}$-submodule of $V_f$ satisfying $T_f
\otimes_\mathcal{O} \mathcal{K}=V_f$), set 
\begin{align*}
\mathcal{T}_f^\mathrm{cyc} = T_f \otimes_\mathcal{O}
 \Lambda_{\mathcal{O}}^{\mathrm{cyc},\sharp}
\end{align*}
and let $G_{F^+}$ act on $\mathcal{T}_f^\mathrm{cyc}$ diagonally 
(refer to {\em Notation} in
Introduction on the superscript $\sharp$). The $G_{F^+}$-module 
$\mathcal{T}_f^\mathrm{cyc}$ is called the {\em cyclotomic deformation
of $T_f$}. For each place $\mathfrak{p}$ of $F^+$ lying above $p$, 
we define $\mathrm{Fil}_\mathfrak{p}^+ \mathcal{T}_f^\mathrm{cyc}$ as 
\begin{align*}
\mathrm{Fil}_\mathfrak{p}^+ \mathcal{T}_f^\mathrm{cyc} 
= \mathrm{Fil}_\mathfrak{p}^+ T_f \otimes_\mathcal{O}
 \Lambda_{\mathcal{O}}^{\mathrm{cyc}, \sharp}
\end{align*}
equipped with the diagonal action of $G_{F^+}$ (here we define
$\mathrm{Fil}_\mathfrak{p}^+T_f$ as the intersection of $T_f$ and
$\mathrm{Fil}_\mathfrak{p}^+ V_f$). 

\begin{defn}[Selmer group $\mathrm{Sel}_{\mathcal{A}_f^\mathrm{cyc}}$] \label{def:Selmer_f}
The {\em Selmer group $\mathrm{Sel}_{\mathcal{A}_f^\mathrm{cyc}}$
 associated to the cyclotomic deformation of $f$} is the Selmer group 
defined as in
 Definition~\ref{def:Selmer_group} 
for the discrete $\Lambda^\mathrm{cyc}_{\mathcal{O}}$-linear  
$G_{F^+}$\nobreakdash-representation
$\mathcal{A}_f^\mathrm{cyc}=\mathcal{T}_f^\mathrm{cyc}
 \otimes_{\Lambda^\mathrm{cyc}_{\mathcal{O}}}
 \Lambda_{\mathcal{O}}^{\mathrm{cyc}, \vee}$ .
\end{defn}

Note that $\mathrm{Sel}_{\mathcal{A}_f^\mathrm{cyc}}$ {\em does} depend
on the choice of $G_{F^+}$-stable $\mathcal{O}$-lattices $T_f$. 
When $V_f$ is {\em residually irreducible}, the $G_{F^+}$-stable
$\mathcal{O}$-lattice $T_f$ is uniquely determined up to isomorphisms,
and hence the Selmer group $\mathrm{Sel}_{\mathcal{A}_f^\mathrm{cyc}}$
is also uniquely determined up to isomorphisms (independently of the
choice of $T_f$).

\medskip

\subsubsection{Cyclotomic deformations of Hilbert modular cuspforms with complex multiplication} \label{sssc:cyclo_CM}

From now on let us assume that {\em $f$ is a nearly $p$-ordinary 
$p$-stabilised newform with complex multiplication}. 
Then by definition 
there exist a totally imaginary quadratic extension $F$ of $F^+$ 
satisfying the ordinarity condition (ord$_{F/F^+}$) 
and a gr\"o{\ss}encharacter $\eta$ of type $(A_0)$ on $F$ such
that $f$ is represented as the $p$-stabilisation $\vartheta(\eta)^{p\text{-st}}$ of the theta lift 
$\vartheta(\eta)$ of $\eta$ 
(see Proposition~\ref{prop:theta}). 
Since Hecke eigenvalues of $\vartheta(\eta)^{p\text{-st}}$ 
coincide with those of $\vartheta(\eta)$ away 
from prime ideals lying above $p$, the Galois
representation  associated to $\vartheta(\eta)^{p\text{-st}}$ 
is isomorphic to the Galois representation $V_{\vartheta(\eta)}$ 
associated to $\vartheta(\eta)$ by virtue of \v{C}ebotarev's density theorem.
We shall recall in Appendix~\ref{app:cm} that 
the Galois representation $V_{\vartheta(\eta)}$ associated to 
$\vartheta(\eta)$ is isomorphic to the induced representation 
$\mathrm{Ind}^{F^+}_F \mathcal{K}(\eta^\mathrm{gal})$ of
the one-dimensional $G_F$-representation $\mathcal{K}(\eta^\mathrm{gal})$, 
and hence there is a canonical $G_{F^+}$-stable 
$\mathcal{O}$\nobreakdash-lattice of $V_{\vartheta(\eta)}$: namely 
$\mathrm{Ind}^{F^+}_F \mathcal{O}(\eta^\mathrm{gal})$.
We thus adopt $\mathrm{Ind}^{F^+}_F \mathcal{O}(\eta^\mathrm{gal})$ 
as the lattice $T_{\vartheta(\eta)}$ used in the construction of the Selmer
group. 
The ordinary filtration $\mathrm{Fil}^+_{\mathfrak{p}} T_{\vartheta(\eta)}$
at $\mathfrak{p}$ above $p$ is then identified 
with $\mathcal{O}(\eta^\mathrm{gal}|_{D_\mathfrak{P}})$ where
$\mathfrak{P}$ is a unique  element of $\Sigma_p$ which lies above
$\mathfrak{p}$. 
 
We here introduce another Selmer group $\mathrm{Sel}^\Sigma_{\mathcal{A}^\mathrm{cyc}_\eta}$ associated to a gr\"o{\ss}encharacter $\eta$ of the CM number field $F$. 
The restriction of the action of $\mathrm{Gal}(F(\mu_{p^\infty})/F)$ on $F(\mu_{p^\infty})$ to
$F^+(\mu_{p^\infty})$ induces an isomorphism between 
$\mathrm{Gal}(F(\mu_{p^\infty})/F)$ and
$\mathrm{Gal}(F^+(\mu_{p^\infty})/F^+)$, which enables us to identify the Iwasawa algebra ${\mathcal{O}}[[\mathrm{Gal}(F(\mu_{p^\infty})/F)]]$ of $\mathrm{Gal}(F(\mu_{p\infty})/F)$ with $\Lambda_\mathcal{O}^\mathrm{cyc}:={\mathcal{O}}[[\mathrm{Gal}(F^+(\mu_{p^\infty})/F^+)]]$ in a canonical manner. 
Now consider the {\em cyclotomic deformation of $\eta^\mathrm{gal}$} 
\begin{align*}
\mathcal{T}^{\mathrm{cyc}}_\eta=\mathcal{O}(\eta^\mathrm{gal}) \otimes_{\mathcal{O}}
\Lambda_\mathcal{O}^{\mathrm{cyc}, \sharp}
\end{align*}
equipped with the $D_\mathcal{P}$-stable filtration
$\mathrm{Fil}_\mathcal{P}^+ \mathcal{T}_\eta^\mathrm{cyc}$ for each
prime ideal $\mathcal{P}$ of $F$ lying above $p$ defined by 
\begin{align}\label{definition:f_pT_eta}
\mathrm{Fil}_\mathcal{P}^+ \mathcal{T}_\eta^\mathrm{cyc}= 
\begin{cases}
\mathcal{T}_\eta^\mathrm{cyc} & \text{if $\mathcal{P}$ is contained in
 $\Sigma_p$}, \\
0 & \text{otherwise}.
\end{cases}
\end{align}
As usual we let $G_F$ act diagonally on
$\mathcal{T}_\eta^\mathrm{cyc}$. Let $S$ denote the set of places of
$F$ lying above those of $F^+$ in $S^+$. Then one readily observes that 
the diagonal action of $G_F$ on $\mathcal{T}_\eta^{\mathrm{cyc}}$
factors through the quotient $\mathrm{Gal}(F_S/F)$ of $G_F$.

\begin{defn}[Selmer group
 $\mathrm{Sel}^\Sigma_{\mathcal{A}^\mathrm{cyc}_\eta}$] \label{def:Selmer_eta}
The {\em Selmer group $\mathrm{Sel}^\Sigma_{\mathcal{A}^\mathrm{cyc}_\eta}$
 associated to the cyclotomic deformation $\mathcal{A}^\mathrm{cyc}_\eta$ of $\eta^\mathrm{gal}$} (with
 respect to the $p$-ordinary CM type $\Sigma$) is the
 Selmer group in the sense of Definition~\ref{def:Selmer_group} 
 constructed for the discrete $G_F$-representation 
 $\mathcal{A}_\eta^\mathrm{cyc}$ defined as $\mathcal{T}_\eta^\mathrm{cyc}\otimes_{\Lambda_\mathcal{O}^\mathrm{cyc}}\Lambda_\mathcal{O}^{\mathrm{cyc},\vee}$. 
\end{defn}

Then one easily sees that the Selmer groups defined 
in Definitions~\ref{def:Selmer_f} and \ref{def:Selmer_eta} coincide; 
namely,

\begin{lem} \label{lem:Shapiro}
The Selmer group
 $\mathrm{Sel}_{\mathcal{A}_{\vartheta(\eta)}^\mathrm{cyc}}$ associated
 to the cyclotomic deformation  $\mathcal{A}_\eta^\mathrm{cyc}$ of $\vartheta(\eta)$ is
 isomorphic to the Selmer group
 $\mathrm{Sel}^\Sigma_{\mathcal{A}^\mathrm{cyc}_\eta}$ as a
 $\Lambda^\mathrm{cyc}_{\mathcal{O}}$-module.
\end{lem}

\begin{proof}
This is a direct consequence of Shapiro's Lemma. Indeed 
we may identify $\mathcal{A}_{\vartheta(\eta)}^\mathrm{cyc}$ with 
the induced representation $\mathrm{Ind}^{F^+}_F 
\mathcal{A}^\mathrm{cyc}_\eta$ of $\mathcal{A}^\mathrm{cyc}_\eta$ by
 construction (under the canonical identification 
${\mathcal{O}}[[\mathrm{Gal}(F(\mu_{p^\infty})/F)]]
\xrightarrow{\sim}
\Lambda_\mathcal{O}^\mathrm{cyc}:={\mathcal{O}}[[\mathrm{Gal}(F^+(\mu_{p^\infty})/F^+)]]
$), and we therefore obtain the
following isomorphisms by virtue of (generalised) Shapiro's lemma 
(see \cite[Proposition~B.2]{Rubin} for details):
\begin{align*}
H^1(F^+_{S^+}/F^+, \mathcal{A}^\mathrm{cyc}_{\vartheta(\eta)}) &\cong H^1(F_S/F,
 \mathcal{A}^{\mathrm{cyc}}_\eta), \\ 
H^1(I_\mathfrak{l}, \mathcal{A}^\mathrm{cyc}_{\vartheta(\eta)}) &\cong
 \prod_{\mathfrak{L}\mid \mathfrak{l}} H^1(I_\mathfrak{L}, 
 \mathcal{A}^{\mathrm{cyc}}_\eta)^{[\mathbb{F}_{\mathfrak{L}}\colon
 \mathbb{F}_{\mathfrak{l}}]} \quad \text{for every $\mathfrak{l}$
 in $S$ which does not divide $p\infty$}
\end{align*}
where $\mathbb{F}_\mathfrak{L}$ and $\mathbb{F}_\mathfrak{l}$ denote
the residue fields $\mathfrak{r}_F/\mathfrak{L}$ and
$\mathfrak{r}_{F^+}/\mathfrak{l}$ respectively.
Moreover, at each place $\mathfrak{p}$ of $F^+$ lying above $p$, 
we obtain the equality:  
\begin{align*}
H^1(I_\mathfrak{p}, \mathcal{A}^\mathrm{cyc}_{\vartheta(\eta)}/\mathrm{Fil}^+
 \mathcal{A}^\mathrm{cyc}_{\vartheta(\eta)}) = H^1 (I_{\mathfrak{p}},
 \mathcal{O}(\eta^\mathrm{gal}\vert_{D_{\mathfrak{P}^c}})\otimes_{\mathcal{O}}\Lambda_\mathcal{O}^{\mathrm{cyc},\sharp}
 \otimes_{\Lambda_\mathcal{O}} \Lambda_\mathcal{O}^{\mathrm{cyc},\vee}) 
= H^1 (I_{\mathfrak{P}^c}, \mathcal{A}^{\mathrm{cyc}}_\eta), 
\end{align*}
where $\mathfrak{p}\mathfrak{r}_F=\mathfrak{PP}^c$ with $\mathfrak{P}$
 in $\Sigma_p$ and $\mathfrak{P}^c$ in $\Sigma_p^c$.  Hence
$\mathrm{Sel}_{\mathcal{A}^\mathrm{cyc}_{\vartheta(\eta)}}$ 
is canonically identified with 
the kernel of the global-to-local map
\begin{align*}
H^1(F_S/F, \mathcal{A}^{\mathrm{cyc}}_\eta) \rightarrow \prod_{w
 \in S\setminus \Sigma_p} H^1 (I_w, \mathcal{A}^{\mathrm{cyc}}_\eta),
\end{align*}
which is none other than the Selmer group
 $\mathrm{Sel}^\Sigma_{\mathcal{A}^\mathrm{cyc}_\eta}$ by 
Definition~\ref{def:Selmer_eta}.
\end{proof}

%
\subsubsection{Selmer groups associated to CM fields and Iwasawa modules of classical type} \label{sssc:comparison}
%

Recall $\widetilde{F}_\infty =\widetilde{F} (\mu_p)$ introduced at basic notations given before Theorem A of Introduction. 
We first define the multi-variable Selmer group $\mathrm{Sel}^\Sigma_{\mathcal{A}^\mathrm{CM}_\eta}$ 
over $\Lambda^\mathrm{CM}_\mathcal{O}:=\mathcal{O}[[\mathrm{Gal}(\widetilde{F}_\infty/F)]]$ 
similarly as the Selmer group $\mathrm{Sel}^\Sigma_{\mathcal{A}^\mathrm{cyc}_\eta}$ over $\Lambda_\mathcal{O}^\mathrm{cyc}$. 
We set $\mathcal{T}^\mathrm{CM}_\eta=\mathcal{O}(\eta^\mathrm{gal}) \otimes_\mathcal{O}
\Lambda_\mathcal{O}^{\mathrm{CM}, \sharp}$ and equip it with the
diagonal $G_F$-action. 

\begin{defn}[Selmer group $\mathrm{Sel}^\Sigma_{\mathcal{A}^\mathrm{CM}_\eta}$] \label{def:Selmer_CM}
The {\em Selmer group $\mathrm{Sel}^\Sigma_{\mathcal{A}^\mathrm{CM}_\eta}$
 associated to the deformation $\mathcal{A}^\mathrm{CM}_\eta$ of $\eta^\mathrm{gal}$ along the extension 
 $\widetilde{F}_\infty/F$}  (with
 respect to the $p$-ordinary CM type $\Sigma$) is the
 Selmer group defined as in Definition~\ref{def:Selmer_group} for the
 discrete $G_F$-representation 
 $\mathcal{A}^\mathrm{CM}_\eta=\mathcal{T}^\mathrm{CM}_\eta\otimes_{\Lambda_\mathcal{O}}\Lambda_\mathcal{O}^{\mathrm{CM},\vee}$ .
\end{defn}

In the rest of this subsection we relates the Selmer group
$\mathrm{Sel}^\Sigma_{\mathcal{A}^\mathrm{CM}_\eta}$ 
with a certain Iwasawa module $X_{\Sigma_p, (\psi)}$ 
defined in a classical manner by means of the notion of a {\em branch character associated to $\eta$}.

\begin{lem} \label{lem:branch}
Let $\eta$ be a gr\"o{\ss}encharacter of type $(A_0)$ on $F$.
Then there exists a $p$-adic Galois character $\psi \colon
 G_F \rightarrow \overline{\mathbb{Q}}_p^\times$ 
of finite order such that 
$\eta^\mathrm{gal}\psi^{-1}\colon G_F \rightarrow
 \overline{\mathbb{Q}}_p^\times$ factors through the quotient 
 $\mathrm{Gal}(\widetilde{F}_\infty/F)$ of $G_F$. 
Furthermore we may choose such a character $\psi$ so that 
the composite field of $F(\mu_p)$ and $K_\psi$ is abelian 
over $F$ and linearly disjoint from $\widetilde{F}$ over $F$, 
where $K_\psi$ denotes the field corresponding to the kernel of $\psi$. 
\end{lem}

\begin{proof}
Let us construct a character $\psi$ satisfying the desired
 properties. First note that 
the continuous character $\eta^\mathrm{gal}$ factors 
through the Galois group 
$\mathrm{Gal}(F_{\mathfrak{C}(\eta)p^\infty}/F)$ of the ray class field
 $F_{\mathfrak{C}(\eta)p^\infty}$ modulo $\mathfrak{C}(\eta)p^\infty$
 over $F$ (recall that $\mathfrak{C}(\eta)$ denotes the conductor of $\eta$). 
Let $\Delta'$ denote the maximal torsion subgroup of
 $\mathrm{Gal}(F_{\mathfrak{C}(\eta)p^\infty}/F)$, which is known to
 be finite. 
Then the subfield of $F_{\mathfrak{C}(\eta)p^\infty}$
 corresponding to $\Delta'$ coincides with $\widetilde{F}$ by definition, 
and we obtain the exact sequence of abelian groups
\begin{align} \label{ses:branch}
0\rightarrow \Delta' \rightarrow
 \mathrm{Gal}(F_{\mathfrak{C}(\eta)p^\infty}/F) \rightarrow
 \mathrm{Gal}(\widetilde{F}/F) \rightarrow 0.
\end{align}
This short exact sequence splits since 
the Galois group $\mathrm{Gal}(\widetilde{F}/F)$ is 
a free $\mathbb{Z}_p$-module of rank $d+1+\delta_{F,p}$ by definition. 
Now we take an arbitrary
 section $s\colon \mathrm{Gal}(\widetilde{F}/F)\rightarrow
 \mathrm{Gal}(F_{\mathfrak{C}(\eta)p^\infty}/F)$ and denotes by $K'$
the intermediate field of $F_{\mathfrak{C}(\eta)p^\infty}/F$
 corresponding to $s(\mathrm{Gal}(\widetilde{F}/F))$. By construction 
$\mathrm{Gal}(K'/F)$ is isomorphic to $\Delta'$. We then define
 $\psi$ as the composition
\begin{align*}
\psi \colon G_F \twoheadrightarrow \mathrm{Gal}(K'/F) \cong \Delta'
 \xrightarrow{\eta^\mathrm{gal}\vert_{\Delta'}} \overline{\mathbb{Q}}_p^\times.
\end{align*}
By construction, it is obvious that $\eta^\mathrm{gal}\psi^{-1} \colon G_F \rightarrow
 \overline{\mathbb{Q}}_p^\times$ 
 factors through $\mathrm{Gal}(\widetilde{F}_\infty/F)$ and that 
$K_\psi$ is contained in $K'$.

Since the short exact sequence (\ref{ses:branch}) splits, 
the fields $K'$ and $\widetilde{F}$ are linearly disjoint over $F$. 
 Furthermore since $[F(\mu_p) :F ]$ is prime to $p$, the Galois group $\mathrm{Gal}(F(\mu_p)/F)$ should be contained in the torsion part $\mathrm{Gal}(K'/F)$ of $\mathrm{Gal}(F_{\mathfrak{C}(\eta)p^\infty}/F)$, and therefore
 the field $K'$ contains $F(\mu_p)$. 
The composite field of $F(\mu_p)$ and $K_\psi$, 
which is a subfield of $K'$, is thus abelian over $F$ and 
linearly disjoint from $\widetilde{F}$ over $F$.
\end{proof}

\begin{defn}[branch character] \label{def:branch}
Let $\eta$ be a gr\"o{\ss}encharacter of type $(A_0)$ on $F$. 
We call a character $\psi \colon G_F\rightarrow
 \overline{\mathbb{Q}}_p^\times$ of finite order satisfying the
 assertions proposed in Lemma~\ref{lem:branch} a {\em branch character} associated to $\eta$. 
\end{defn}

We denote by $K$ the composite field of $F(\mu_p)$ and $K_\psi$, and
 by $\widetilde{K}^\mathrm{CM}_\infty$ the composite field of
 $K$ and $\widetilde{F}_\infty$ as in Introduction.  
We set $\Delta=\mathrm{Gal}(K/F)$,
 $\Delta^\mathrm{cyc}=\mathrm{Gal}(F(\mu_p)/F)$ and
 $\widetilde{\Gamma}=\mathrm{Gal}(\widetilde{F}/F)$.  
The cardinality of $\Delta$ is then the product of the order of
 $\psi$ and the extension degree $[K:K_\psi]$ of $K$ over $K_\psi$. Note that 
$[K:K_\psi]$ is relatively prime to $p$ since it divides $p-1$.

Now let us consider the following commutative diagram with exact rows and
columns:
\begin{align} \label{eq:comparison}
\xymatrix{
 &  &  H^1(\Delta, (\mathcal{A}^\mathrm{CM}_\eta)^{G_K}) \ar[r] \ar[d]_{\mathrm{inf}} & \underset{w \in
 S \setminus \Sigma_p}{\prod} H^1(I_{K/F ,w},
 (\mathcal{A}^\mathrm{CM}_\eta)^{I_{K,\tilde{w}_0}}) \ar[d]^{\mathrm{Inf}} \\
0 \ar[r] & \mathrm{Sel}^\Sigma_{\mathcal{A}^\mathrm{CM}_\eta} \ar[r] \ar[d]  & H^1(F_S/F,
 \mathcal{A}^\mathrm{CM}_\eta) \ar[r]  \ar[d]_{\mathrm{res}} &
 \underset{w\in  S\setminus \Sigma_p}{\prod} H^1(I_{F,w},
 \mathcal{A}^\mathrm{CM}_\eta) \ar[d]^{\mathrm{Res}} \\
0 \ar[r] & \ker(\phi_{F_S/K})  \ar[r]  & 
 H^1(F_S/K, \mathcal{A}^\mathrm{CM}_\eta)^{\Delta}  \ar[r]^{\phi_{F_S/K} \qquad} \ar[d]  & \underset{w \in S \setminus
 \Sigma_p}{\prod} \underset{\tilde{w} \mid w}{\prod} H^1(I_{K,\tilde{w}},
 \mathcal{A}^\mathrm{CM}_\eta)  \\
 &  & 
H^2(\Delta, (\mathcal{A}^\mathrm{CM}_\eta)^{G_K}) &
}
\end{align}
where $\phi_{F_S/K}$ denotes the global-to-local map induced from the
restriction morphisms. The rows are exact by definition. 
The middle column of the diagram (\ref{eq:comparison}) is induced from the 
inflation-restriction exact sequences associated to the short exact
sequence 
\begin{align*}
\xymatrix{
1 \ar[r] & \mathrm{Gal}(F_S/K) \ar[r] & \mathrm{Gal}(F_S/F) \ar[r] &
 \Delta \ar[r] & 1}
\end{align*}
of abelian groups.

The map $\mathrm{Res}$ in the right column is induced 
from restriction maps, and the exactness of the right column is 
also deduced from the inflation-restriction exact sequence. 
More precisely, we define the map $\mathrm{Res}=(\mathrm{Res})_{w \in S\setminus
\Sigma_p}$ in the following manner. 
First we choose and fix a decomposition group $D_{F,w}$ 
of $\mathrm{Gal}(F_S/F)$ for each place $w$ in $S\setminus \Sigma_p$. We always consider 
the inertia subgroup $I_{F,w}$ of $\mathrm{Gal}(F_S/F)$ 
to be contained in the fixed decomposition group $D_{F,w}$ 
at such a place $w$.
There exists a unique place $\tilde{w}_0$ of $K$ lying above $w$ 
which is fixed under the action of $D_{F,w}$. We identify the intersection 
of $D_{F,w}$ and $\mathrm{Gal}(F_S/K)$
(resp.\ the intersection of $I_{F,w}$ and $\mathrm{Gal}(F_S/K)$) 
with the decomposition group $D_{K, \tilde{w}_0}$ (resp.\ the inertia subgroup $I_{K, \tilde{w}_0}$) 
of $\mathrm{Gal}(F_S/K)$ at $\tilde{w}_0$. 
We identify the quotient group $I_{F,w}/I_{K, \tilde{w}_0}$ 
with the inertia subgroup $I_{K/F ,w}$ 
of $\Delta=\mathrm{Gal}(K/F)$ at $w$ (note that 
the inertia subgroup of $\Delta$ at $w$ is well defined since $\Delta$ 
is abelian). For each place $\tilde{w}$ of $K$ lying above $w$, 
we also fix the decomposition group $D_{K, \tilde{w}}$
of $\mathrm{Gal}(F_S/K)$ (and denote its inertia subgroup by $I_{K,
\tilde{w}}$), and choose an element $\sigma_{\tilde{w}}$ 
of $\mathrm{Gal}(F_S/F)$ 
so that $\sigma_{\tilde{w}}I_{F,w}\sigma_{\tilde{w}}^{-1}$ contains 
$I_{K,\tilde{w}}$. Then the map $\mathrm{Res}_w$ is  
defined as the composition 
{\small
\begin{align*}
\mathrm{Res}_w \colon H^1(I_{F,w}, \mathcal{A}^\mathrm{CM}_\eta) \rightarrow 
\prod_{\tilde{w} \mid w} H^1(I_{F,w}, \mathcal{A}^\mathrm{CM}_\eta)
 \xrightarrow{\sim}
 \prod_{\tilde{w} \mid w} H^1(\sigma_{\tilde{w}}I_{F,w} 
\sigma^{-1}_{\tilde{w}}, \mathcal{A}^\mathrm{CM}_\eta) \rightarrow \prod_{\tilde{w}
 \mid w} H^1(I_{K,\tilde{w}}, \mathcal{A}^\mathrm{CM}_\eta),
\end{align*}
}where the first map is the diagonal map, the second one is an
isomorphism induced by the conjugation with respect to
$(\sigma_{\tilde{w}})_{\tilde{w}\mid w}$, and the last map is the usual 
(componentwise) restriction map. It is easy to observe that the kernel
of $\mathrm{Res}_w$ coincides with the kernel of the restriction map 
$H^1(I_{F,w}, \mathcal{A}^\mathrm{CM}_\eta) \rightarrow H^1(I_{K,
\tilde{w_0}}, \mathcal{A}^\mathrm{CM}_\eta)$, and hence the right column
of the diagram (\ref{eq:comparison}) is also exact.

Under these settings we shall compare the Selmer group $\mathrm{Sel}_{\mathcal{A}_\eta^\mathrm{cyc}}$ with $\mathrm{Ker}(\phi_{F_S/K})$.

\begin{rem} \label{rem:vanish}
Since all the cohomology groups $H^1 (\Delta,
 (\mathcal{A}^\mathrm{CM}_\eta)^{G_K})$,
$H^2(\Delta, (\mathcal{A}^\mathrm{CM}_\eta)^{G_K})$ 
and $H^1(I_{K/F ,w}, 
 (\mathcal{A}^\mathrm{CM}_\eta)^{I_{K,\tilde{w}_0}})$ are $p$-torsion modules
 annihilated by the cardinality
of $\Delta$, {\em all of them 
 vanish when the order of $\psi$ is relatively
 prime to $p$} (recall that $[K: K_\psi]$ is not
 divisible by $p$). In this case one may immediately conclude that
 $\mathrm{Sel}^\Sigma_{\mathcal{A}^\mathrm{CM}_\eta}$ is isomorphic to $\mathrm{Ker}(\phi_{F_S/K})$ 
by applying the snake lemma to $(\ref{eq:comparison})$. 
\end{rem}

\begin{lem} \label{lem:cpsn}
The module $(\mathcal{A}^\mathrm{CM}_\eta)^{G_K}$ is a 
copseudonull $\Lambda^\mathrm{CM}_\mathcal{O}$-module$;$ that is, 
its Pontrjagin dual is a pseudonull $\Lambda^\mathrm{CM}_\mathcal{O}$-module. 
In particular, the cohomology group $H^i(\Delta,
 (\mathcal{A}^\mathrm{CM}_\eta)^{G_K})$ is 
copseudonull as a $\Lambda^\mathrm{CM}_\mathcal{O}$-module for $i=1,2$.
\end{lem}

\begin{proof} 
Since $G_K$ is a subgroup of $G_F$ of finite index, 
its image $\mathfrak{G}$ 
under the natural surjection $G_F \twoheadrightarrow
 \mathrm{Gal}(\widetilde{F}_\infty/F)$ is also a subgroup of
 $\mathrm{Gal}(\widetilde{F}_\infty/F)$ of finite index. 
Moreover since the free part $\mathrm{Gal}(\widetilde{F}/F)$ 
of $\mathrm{Gal}(\widetilde{F}_\infty/F)$ is 
isomorphic to $\mathbb{Z}_p^{d+1+\delta_{F,p}}$, we may choose 
a basis $\{ \gamma_1, \dotsc , \gamma_{d+1+\delta_{F,p}}\}$ of
 $\mathrm{Gal}(\widetilde{F}/F)$ so that
 $\{\gamma_1^{p^{e_1}}, \dotsc,
 \gamma_{d+1+\delta_{F,p}}^{p^{e_{d+1+\delta_{F,p}}}}\}$ forms a basis 
of the free part of $\mathfrak{G}$ 
for certain nonnegative integers $e_1,
 \dotsc, e_{d+1+\delta_{F,p}}$ (due to elementary divisor theory). 
The Pontrjagin dual of $(\mathcal{A}^\mathrm{CM}_\eta)^{G_K}$ is then
 isomorphic to a certain quotient of
 $\Lambda^\mathrm{CM}_\mathcal{O}/J$,
 where $J$ is the ideal of $\Lambda^\mathrm{CM}_\mathcal{O}$ generated by 
 $\eta^\mathrm{gal}(\gamma_j^{p^{e_j}})\gamma_j^{p^{e_j}}|_{\widetilde{F}_\infty}-1$
 for $1\leq j\leq d+1+\delta_{F,p}$. Let $\varpi$ denote a uniformiser of
 $\mathcal{O}$. Then the
 $\Lambda^\mathrm{CM}_\mathcal{O}$-module
 $\Lambda^\mathrm{CM}_\mathcal{O}/(\varpi\Lambda^\mathrm{CM}_\mathcal{O}+J)$
 is clearly finite, which
 implies that the height of $J$ is greater than or equal to
 $d+1+\delta_{F,p}$ (recall that the Krull dimension of
 $\Lambda^\mathrm{CM}_\mathcal{O}$ is $d+2+\delta_{F,p}$). 
The Pontrjagin dual of $(\mathcal{A}^\mathrm{CM}_\eta)^{G_K}$ is obviously annihilated by
 $J$, and it is thus pseudonull as a $\Lambda^\mathrm{CM}_\mathcal{O}$-module since
 $d$ is a positive integer.   
\end{proof}

\begin{lem} \label{lem:cpsn-loc}
The local cohomology group $H^1(I_{K/F ,w},
 (\mathcal{A}^\mathrm{CM}_\eta)^{I_{K,\tilde{w}_0}})$ is copseudonull
 as a module over $\Lambda^\mathrm{CM}_\mathcal{O}$ for each place $w$
 of $F$ in $S\setminus \Sigma_p$.
\end{lem}

\begin{proof}
First we assume that $w$ is a place of $F$ 
contained in $\Sigma_p^c$. We shall prove that 
$(\mathcal{A}^\mathrm{CM}_\eta)^{I_{K,\tilde{w}_0}}$ is 
copseudonull as a $\Lambda^\mathrm{CM}_\mathcal{O}$-module, which immediately implies the desired 
conclusion on the local cohomology group at $w$. 
Since $K$ is a finite extension of $F$ and 
$w$ is totally ramified in the cyclotomic
extension $F(\mu_{p^\infty})/F(\mu_{p^{m_w}})$ for a certain nonnegative integer $m_w$, 
the image of $I_{K,\tilde{w}_0}$ in $\mathrm{Gal}(F(\mu_{p^\infty})/F)$
 is infinite. We take an element $x$ of $I_{K,\tilde{w}_0}$ 
 whose image in $\mathrm{Gal}(F(\mu_{p^\infty})/F)$ is 
of infinite order. Let $F_\infty^{(w)}/F$ be the composite of all $\mathbb{Z}_p$-extensions over $F$ 
unramified outside $w$, and take an element $y$ of $I_{K, \tilde{w}_0}$ whose image in
 $\mathrm{Gal}(F_\infty^{(w)}/F)$ is of infinite order. Comparing the ramification at $w$, one readily sees that $F(\mu_{p^\infty})$ and $F_\infty^{(w)}$ are linearly disjoint over $F$. 
 The Pontrjagin dual of
 $(\mathcal{A}^\mathrm{CM}_\eta)^{I_{K,\tilde{w}_0}}$ is then isomorphic to 
 a certain quotient of $\Lambda^\mathrm{CM}_\mathcal{O}/J_w$, where 
 $J_w$ is an ideal of $\Lambda^\mathrm{CM}_\mathcal{O}$ generated by
 $\eta^\mathrm{gal}(x)x\vert_{\widetilde{F}_\infty}-1$ and
 $\eta^\mathrm{gal}(y)y\vert_{\widetilde{F}_\infty}-1$. 
 Furthermore the linear disjointness of $F(\mu_{p^\infty})$ and $F_\infty^{(w)}$ over $F$ implies that  $\eta^\mathrm{gal}(x)x\vert_{\widetilde{F}_\infty}-1$ and
 $\eta^\mathrm{gal}(y)y\vert_{\widetilde{F}_\infty}-1$ forms a regular
 sequence in $\Lambda^\mathrm{CM}_\mathcal{O}$. 
The height of $J_w$, which is contained in the annihilator ideal of $(\mathcal{A}_\eta^\mathrm{CM})^{I_{K,\tilde{w}}}$, thus equals two, and hence 
$(\mathcal{A}^\mathrm{CM}_\eta)^{I_{K, \tilde{w}_0}}$ is copseudonull as a $\Lambda^\mathrm{CM}_\mathcal{O}$-module.
 
Next assume that $w$ is a place of $F$ contained in $S$ 
but not lying above $p$, 
and let $\ell$ denote the residue characteristic at $w$. The inertia subgroup $I_{K, \tilde{w}_0}$  
acts on $\mathcal{A}^\mathrm{CM}_\eta$
 trivially because $\widetilde{F}_\infty/F$ is unramified at $w$. 
  
The definition of the branch character $\psi$ implies that $\eta^\mathrm{gal}\psi^{-1}$ is 
ramified only at places lying above $p$. Hence $K$ is unramified outside $p$ over the field $K_\psi$ corresponding to 
the kernel of $\psi$, and the inertia subgroup 
$I_{K/F ,w}$ of $\mathrm{Gal}(K/F)$ at $w$ is naturally regarded as 
that of $\mathrm{Gal}(K_\psi/F)$. 
 Under this identification the inertia subgroup $I_{K/F ,w}$ acts
 on $\mathcal{A}_\eta^\mathrm{CM}$ via the composition $I_{K/F,w}\hookrightarrow \mathrm{Gal}(K_\psi/F)\xrightarrow{\, \psi \,}\mathcal{O}^\times$.

 We shall prove that $H^1(I_{K/F ,w},\mathcal{A}^\mathrm{CM}_\eta)$ is trivial. 
 When $w$ is unramified in the finite abelian extension $K/F$,
 the inertia subgroup $I_{K/F ,w}$ is trivial and there is nothing to 
prove in this case.  
 Hence we assume that $w$ is ramified in $K/F$.  
 The restriction of $\psi$ to $I_{K/F ,w}$ is then not trivial since $\psi\colon \mathrm{Gal}(K_\psi/F)\rightarrow
 \overline{\mathbb{Q}}^\times$ is injective by definition. 
 Let $I_{K/F ,w}^\mathrm{w}$ be the $\ell$-Sylow subgroup (the wild part) of $I_{K/F ,w}$ and
 $I_{K/F ,w}^\mathrm{t}$ the tame quotient $I_{K/F ,w}/I_{K/F ,w}^\mathrm{w}$ of $I_{K/F,w}$. Note that the cohomology group 
$H^1(I_{K/F ,w}^\mathrm{w}, \mathcal{A}^\mathrm{CM}_\eta)$ is trivial
 because it is annihilated by the cardinality of 
 $I_{K/F ,w}^\mathrm{w}$, which is relatively prime to $p$. 
By the inflation-restriction exact sequence 
\begin{align*}
0 \rightarrow H^1(I_{K/F ,w}^\mathrm{t},
 (\mathcal{A}^\mathrm{CM}_\eta)^{I_{K/F ,w}^\mathrm{w}}) \rightarrow
 H^1(I_{K/F ,w}, \mathcal{A}^\mathrm{CM}_\eta) \rightarrow
 H^1(I_{K/F ,w}^\mathrm{w},
 \mathcal{A}^\mathrm{CM}_\eta)^{I_{K/F ,w}^\mathrm{t}}
\end{align*}
combined with 
 the triviality of $H^1(I_{K/F ,w}^\mathrm{w}, \mathcal{A}^\mathrm{CM}_\eta)$, 
 we can identify $H^1(I_{K/F ,w}, \mathcal{A}^\mathrm{CM}_\eta)$ with $H^1(I_{K/F ,w}^\mathrm{t},
 (\mathcal{A}^\mathrm{CM}_\eta)^{I_{K/F ,w}^\mathrm{w}})$.
Therefore it suffices to verify that $H^1(I^\mathrm{t}_{K/F,w},(\mathcal{A}^\mathrm{CM}_\eta)^{I^\mathrm{w}_{K/F},w})$ is trivial in order to prove the vanishing of $H^1(I_{K/F},\mathcal{A}^\mathrm{CM}_\eta)$. 
 
If the action of $I_{K/F ,w}^\mathrm{w}$ on $\mathcal{A}^\mathrm{CM}_\eta$ is
 not trivial, there exists an element $z$ of
 $I_{K/F ,w}^\mathrm{w}$ such that
 $\psi^{\mathrm{gal}}(z)-1$ does not equal zero. 
Since $\psi^\mathrm{gal}(z)$ is a nontrivial $\ell$-power root of unity 
in $\mathcal{O}$, we easily see  that $\psi^\mathrm{gal}(z)-1$ is 
 a nontrivial unit of $\mathcal{O}$. 
 
 By definition $(\mathcal{A}_\eta^\mathrm{CM})^{I_{K/F,w}^\mathrm{w}}$ is annihilated by the unit $\psi^\mathrm{gal}(z)-1$ of $\mathcal{O}$, and hence it is trivial. This obviously implies the triviality of $H^1(I_{K/F,w}^\mathrm{t},(\mathcal{A}_\eta^\mathrm{CM})^{I_{K/F,w}^\mathrm{w}})$.
If $I_{K/F ,w}^\mathrm{w}$ acts trivially on
 $\mathcal{A}^\mathrm{CM}_\eta$, the nontriviality of the action of $I_{K/F,w}$ on $\mathcal{A}_\eta^\mathrm{CM}$ implies that the the tame quotient
 $I_{K/F ,w}^\mathrm{t}$ acts nontrivially on
 $\mathcal{A}^\mathrm{CM}_\eta$. In other words, if we denote 
 a generator of the cyclic group $I_{K/F ,w}^\mathrm{t}$ by $z'$, 
 $\psi^\mathrm{gal}(z')-1 $ is a nonzero element of $\Lambda^\mathrm{CM}_\mathcal{O}$. 
 The first cohomology group $H^1(I_{K/F ,w}^\mathrm{t},
 \mathcal{A}^\mathrm{CM}_\eta)$ of the finite cyclic group
 $I_{K/F ,w}^\mathrm{t}$ is described as
 $\mathcal{A}^\mathrm{CM}_\eta/(\psi^\mathrm{gal}(z')-1)\mathcal{A}^\mathrm{CM}_\eta$, and we thus deduce its triviality from the divisibility of $\mathcal{A}^\mathrm{CM}_\eta$ as a $\Lambda^\mathrm{CM}_\mathcal{O}$-module.
\end{proof}

The diagram~(\ref{eq:comparison}) combined with Lemmata~\ref{lem:cpsn} and \ref{lem:cpsn-loc} implies
that both the kernel and the cokernel of the natural map
$\mathrm{Sel}^\Sigma_{\mathcal{A}^\mathrm{CM}_\eta}\rightarrow \ker(\phi_{F_S/K})$ are 
copseudonull $\Lambda^\mathrm{CM}_\mathcal{O}$\nobreakdash-modules. In particular, 
the characteristic ideals of Pontrjagin duals of these $\Lambda^\mathrm{CM}_\mathcal{O}$-modules coincide with each other. 
Now we investigate the kernel of $\phi_{F_S/K}$ more precisely,
by generalising Greenberg's arguments made around the equation (7) in \cite{gr-mot} to multi-variable cases.
For this purpose we prepare the following technical lemma.

\begin{lem} \label{lem:vanish-free}
Let $\mathcal{G}$ be a profinite group and 
$\rho \colon \mathcal{G} \rightarrow \mathcal{O}^\times$ 
a continuous character, where $\mathcal{O}$ denotes the ring of integers 
of a finite extension of $\mathbb{Q}_p$. 
We denote by $\mathcal{T}_\rho$ a continuous $\mathcal{G}$-representation 
 $\mathcal{O}(\rho) \otimes_\mathcal{O}
 \mathcal{O}[[\Gamma]]^\sharp$ equipped with the diagonal action of
 $\mathcal{G}$, and by $\mathcal{A}_\rho$ its associated discrete
 $\mathcal{G}$-representation defined by $\mathcal{T}_\rho
 \otimes_{\mathcal{O}[[\Gamma]]} \mathcal{O}[[\Gamma]]^\vee$. 
  
Suppose that $\mathcal{G}$ admits a free abelian pro-$p$ 
quotient $\Gamma$ of finite $\mathbb{Z}_p$-rank and that the kernel of the natural surjection
 $\mathcal{G}\twoheadrightarrow \Gamma$ is contained in the kernel of
 $\rho$. 

Then for an arbitrary free pro-$p$ subgroup $\Gamma'$ of $\Gamma$, the first
 cohomology group $H^1(\Gamma', \mathcal{A}_\rho)$ is trivial.
\end{lem}

Note that, in the statement of Lemma \ref{lem:vanish-free}, the free abelian quotient $\Gamma$ naturally acts on
$\mathcal{A}_\rho$ thanks to the assumption on the kernel of $\rho$.  

\begin{proof}
We verify the claim by induction on the $\mathbb{Z}_p$-rank $n$ of
 $\Gamma'$. First consider the case where $n$ equals $1$. 
Then if we choose an element $\tilde{\gamma}$ of $\mathcal{G}$ so that 
its image $\gamma$ in $\Gamma$ topologically generates $\Gamma'$, 
the first cohomology group $H^1(\Gamma', \mathcal{A}_\rho)$ 
is described as the quotient 
$\mathcal{A}_\rho/(\rho(\tilde{\gamma})\gamma-1)\mathcal{A}_\rho$. 
Therefore the claim holds since $\mathcal{A}_\rho$ is a divisible
 $\mathcal{O}[[\Gamma]]$-module and $\rho(\tilde{\gamma})\gamma-1$ is 
a nonzero element of $\mathcal{O}[[\Gamma]]$. 
 
For general $n$, let us choose an arbitrary free pro-$p$ subgroup 
 $\Gamma''$ of $\Gamma'$ of $\mathbb{Z}_p$-rank $n-1$, and
 consider the inflation-restriction exact sequence
\begin{align} \label{eq:infres}
\xymatrix{ 0 \ar[r] & H^1(\Gamma'/\Gamma'', \mathcal{A}_\rho^{\Gamma''}) 
\ar[r] & H^1(\Gamma', \mathcal{A}_\rho) \ar[r] & H^1(\Gamma'',
 \mathcal{A}_\rho)^{\Gamma'/\Gamma''}}.
\end{align}
 The induction hypothesis implies the triviality
 of $H^1(\Gamma'', \mathcal{A}_\rho)$. We readily see
 that the cohomology group $H^1(\Gamma'/\Gamma'', \mathcal{A}_\rho^{\Gamma''})$ is also trivial; 
 indeed, since the $\Gamma''$-coinvariant $(\mathcal{T}_\rho)_{\Gamma''}$ of
 $\mathcal{T}_\rho$ is isomorphic to $\mathcal{O}[[\Gamma/\Gamma'']]$ as an
 $\mathcal{O}[[\Gamma/\Gamma'']]$-module, its Pontrjagin dual 
$\mathcal{A}_\rho^{\Gamma''}$ is a divisible
 $\mathcal{O}[[\Gamma/\Gamma'']]$-module. Therefore
 $H^1(\Gamma'/\Gamma'', \mathcal{A}_\rho^{\Gamma''})$ vanishes by the
 same reason as the case where $n$ equals $1$. Consequently 
 the exact sequence (\ref{eq:infres}) implies that 
the cohomology group $H^1(\Gamma', \mathcal{A}_\rho)$ is also 
 trivial .
\end{proof}

Applying Lemma~\ref{lem:vanish-free} to both the source and the target
of the local-to-global morphism $\phi_{F_S/K}$, we obtain the following
corollaries. 

\begin{cor} \label{cor:global_coh}
The restriction map induces an isomorphism between $H^1(F_S/K,
 \mathcal{A}^\mathrm{CM}_\eta)$ and
 $H^1(F_S/\widetilde{K}^\mathrm{CM}_\infty,
 \mathcal{A}^\mathrm{CM}_\eta)^{\widetilde{\Gamma}}$.
\end{cor}

Here we identify $\widetilde{\Gamma}=\mathrm{Gal}(\widetilde{F}/F)$ with
the Galois group $\mathrm{Gal}(\widetilde{K}^\mathrm{CM}_\infty/K)$ in the canonical manner.

\begin{proof}
Due to the inflation-restriction exact sequence
{\small 
\begin{align*}
\xymatrix{0 \ar[r] & H^1(\widetilde{\Gamma},
 \mathcal{A}^\mathrm{CM}_\eta)
 \ar[r] & H^1(F_S/K, \mathcal{A}^\mathrm{CM}_\eta) \ar[r] &
 H^1(F_S/\widetilde{K}^\mathrm{CM}_\infty,
 \mathcal{A}^\mathrm{CM}_\eta)^{\widetilde{\Gamma}} \ar[r] & H^2(\widetilde{\Gamma},
 \mathcal{A}^\mathrm{CM}_\eta)},
\end{align*}
}it suffices to show that $H^i(\widetilde{\Gamma},
 \mathcal{A}^\mathrm{CM}_\eta)$
 is trivial for $i=1,2$ (note that
 $\mathrm{Gal}(F_S/\widetilde{K}^\mathrm{CM}_\infty)$ acts trivially on 
 $\mathcal{A}^\mathrm{CM}_\eta$). The first cohomology vanishes by a direct 
 consequence of Lemma~\ref{lem:vanish-free}. The second cohomology
 vanishes since $\widetilde{\Gamma}$ is a free abelian pro-$p$ group and 
 thus its cohomological dimension is equal to or less than $1$. This completes
 the proof.
\end{proof}

\begin{cor} \label{cor:local_coh}
Let $w$ be a place of $F$ contained in $S\setminus \Sigma_p$. Then for
 each place $\tilde{w}$ of $K$ lying above $w$, 
the restriction morphism 
\begin{align} \label{eq:restriction_loc}
H^1(I_{K,\tilde{w}},\mathcal{A}^\mathrm{CM}_\eta) \rightarrow
 \prod_{\hat{w} \mid \tilde{w}}
 H^1(I_{\widetilde{K}^\mathrm{CM}_\infty,\hat{w}}, \mathcal{A}^\mathrm{CM}_\eta)
\end{align}
is injective $($here $I_{\widetilde{K}^\mathrm{CM}_\infty, \hat{w}}$
 denotes the inertia subgroup of
 $\mathrm{Gal}(F_S/\widetilde{K}^\mathrm{CM}_\infty)$ at $\hat{w})$.
\end{cor}

\begin{proof}
By the inflation-restriction exact sequence, we see that the kernel of 
the restriction map (\ref{eq:restriction_loc}) is isomorphic to 
$H^1(I(\widetilde{K}^\mathrm{CM}_\infty/K)_w,
 \mathcal{A}^\mathrm{CM}_\eta)$, where we denotes by
 $I(\widetilde{K}^\mathrm{CM}_\infty/K)_w$ the inertia subgroup of
 $\mathrm{Gal}(\widetilde{K}^\mathrm{CM}_\infty/K)$ at $w$. Since
 $I(\widetilde{K}^\mathrm{CM}_\infty/K)_w$ is regarded as a
 $\mathbb{Z}_p$-submodule of the free $\mathbb{Z}_p$-module
 $\widetilde{\Gamma}\cong \mathrm{Gal}(\widetilde{K}^\mathrm{CM}_\infty/K)$ of finite rank,  it is also 
 free as a $\mathbb{Z}_p$-module. Applying
 Lemma~\ref{lem:vanish-free} to $H^1(I(\widetilde{K}^\mathrm{CM}_\infty/K)_{\tilde{w}},
 \mathcal{A}^\mathrm{CM}_\eta)$, we conclude that it is trivial.  
\end{proof}

Corollaries~\ref{cor:global_coh} and \ref{cor:local_coh} imply that 
the kernel of $\phi_{F_S/K}$ is isomorphic to that of 
\begin{align*}
H^1(F_S/\widetilde{K}^\mathrm{CM}_\infty,
 \mathcal{A}^\mathrm{CM}_\eta)^{\Delta \times \widetilde{\Gamma}}
 \rightarrow \prod_{w\in S\setminus \Sigma_p} \prod_{\hat{w}\mid w}
 H^1(I_{\widetilde{K}^\mathrm{CM}_\infty, \hat{w}},
 \mathcal{A}^\mathrm{CM}_\eta), 
\end{align*}
or in other words, the kernel of the restriction map
\begin{align*}
\phi_{F_S/\widetilde{K}^\mathrm{CM}_\infty} \colon \mathrm{Hom}_{\Delta \times \widetilde{\Gamma}}(\mathrm{Gal}(F_S/\widetilde{K}^\mathrm{CM}_\infty)^\mathrm{ab},
 \mathcal{A}^\mathrm{CM}_\eta)
 \rightarrow \prod_{w\in S\setminus
 \Sigma_p} \prod_{\hat{w}\mid w}
 \mathrm{Hom}(I_{\widetilde{K}^\mathrm{CM}_\infty, \hat{w}}^\mathrm{ab},
 \mathcal{A}^\mathrm{CM}_\eta).
\end{align*}
Let $M_{\Sigma_p}$ denote the maximal abelian pro-$p$ extension 
of $\widetilde{K}^\mathrm{CM}_\infty$ unramified outside places
above $\Sigma_p$, and let $X_{\Sigma_p}$ denote the Galois group of
$M_{\Sigma_p}/\widetilde{K}^\mathrm{CM}_\infty$. Then as in classical Iwasawa theory, an 
element $g$ of $\mathrm{Gal}(\widetilde{K}^\mathrm{CM}_\infty/F)$
acts on $X_{\Sigma_p}$ by $x\mapsto 
{}^g x:=\tilde{g}x\tilde{g}^{-1}$, where $\tilde{g}$ denotes 
an arbitrary lift of $g$ to $\mathrm{Gal}(M_{\Sigma_p}/F)$. 
Furthermore we define the {\em maximal $\psi$\nobreakdash-isotypic quotient} 
$X_{\Sigma_p, (\psi)}$ of $X_{\Sigma_p}$ as 
\begin{align*}
X_{\Sigma_p, (\psi)} =(X_{\Sigma_p} \otimes_{\mathbb{Z}_p} \mathcal{O})
 \otimes_{\mathcal{O}[\Delta]} \mathcal{O}[\Delta^\mathrm{cyc}].
\end{align*}
Here the second tensor product is defined with respect to the map 
\begin{align*}
\mathcal{O}[\Delta]\rightarrow \mathcal{O}[\Delta^\mathrm{cyc}]\ ;
 \ \delta \mapsto \psi(\delta)\delta\vert_{F(\mu_p)} \quad \text{for
 $\delta$ in $\Delta=\mathrm{Gal}(K/F)$}
\end{align*}
(recall that $\Delta^\mathrm{cyc}$ is defined as $\mathrm{Gal}(F(\mu_p)/F)$). Since an element $\delta$ of 
$\Delta$ acts on $\mathcal{A}^\mathrm{CM}_\eta$ via the multiplication
by $\psi(\delta)\delta\vert_{F(\mu_p)}$, the kernel of
$\phi_{F_S/\widetilde{K}^\mathrm{CM}_\infty}$ is calculated as 
\begin{align*}
\ker(\phi_{F_S/\widetilde{K}^\mathrm{CM}_\infty}) &=
 \mathrm{Hom}_{\Delta \times \Gamma}(X_{\Sigma_p},
 \mathcal{A}^\mathrm{CM}_\eta) =\mathrm{Hom}_{\mathcal{O}[[\Delta\times
 \Gamma]]}(X_{\Sigma_p} \otimes_{\mathbb{Z}_p}\mathcal{O},
 \mathcal{A}^\mathrm{CM}_\eta) \\
&\cong \mathrm{Hom}_{\Lambda^\mathrm{CM}_\mathcal{O}}(X_{\Sigma_p,
 (\psi)}, \mathcal{A}^\mathrm{CM}_\eta).
\end{align*}

In order to investigate the structure of
$\mathrm{Hom}_{\Lambda^\mathrm{CM}_\mathcal{O}}(X_{\Sigma_p, (\psi)},
\mathcal{A}^\mathrm{CM}_\eta)$, we introduce the notion of 
{\em twisting} of 
finitely generated $\Lambda^\mathrm{CM}_\mathcal{O}$-module. 
Recall that we have defined the {\em $\rho$-twisting map} 
\begin{align*}
\mathrm{Tw}_\rho \colon \Lambda^\mathrm{CM}_\mathcal{O} \xrightarrow{\sim}
 \Lambda^\mathrm{CM}_\mathcal{O} ; g \mapsto \rho(g)g \quad \text{for $g$ in
 $\mathrm{Gal}(\widetilde{F}_\infty/F)$}
\end{align*}
for an arbitrary continuous character $\rho \colon
\mathrm{Gal}(\widetilde{F}_\infty/F)\rightarrow \mathcal{O}^\times$ 
in Section~\ref{sc:Introduction}. For an arbitrary  
$\Lambda^\mathrm{CM}_\mathcal{O}$-module $M$, 
we define the {\em $\rho$-twist} $\mathrm{Tw}_\rho(M)$ of $M$ as the
$\mathcal{O}$-module $M$ on which an element $r$ of $\Lambda^\mathrm{CM}_\mathcal{O}$ acts 
as the scalar multiplication by $\mathrm{Tw}_\rho(r)$. 
The following lemma describes basic properties of twisting of $\Lambda^\mathrm{CM}_\mathcal{O}$-modules.

\begin{lem} \label{lem:twist}
Let $\rho \colon \mathrm{Gal}(\widetilde{F}_\infty/F)\rightarrow
 \mathcal{O}^\times$ be a continuous character and $M$ a 
 $\Lambda^\mathrm{CM}_\mathcal{O}$-module.
\begin{enumerate}[label=$(\arabic*)$]
\item The Pontrjagin dual $(\mathrm{Tw}_\rho(M))^\vee$ of
      $\mathrm{Tw}_\rho(M)$ is isomorphic to
      $\mathrm{Tw}_\rho(M^\vee)$, the $\rho$\nobreakdash-twist of the
      Pontrjagin dual $M^\vee$ of $M$, as a
      $\Lambda^\mathrm{CM}_\mathcal{O}$-module. 
\item For an arbitrary $\Lambda^\mathrm{CM}_\mathcal{O}$-module $N$, 
      we obtain an equality of $\Lambda^\mathrm{CM}_\mathcal{O}$-modules
\begin{align*} 
\mathrm{Hom}_{\Lambda^\mathrm{CM}_\mathcal{O}}(\mathrm{Tw}_\rho(M),N)
 = \mathrm{Hom}_{\Lambda^\mathrm{CM}_\mathcal{O}}(M,
 \mathrm{Tw}_{\rho^{-1}}(N)).
\end{align*}
\item Assume that $M$ is finitely generated and torsion as a
      $\Lambda^\mathrm{CM}_\mathcal{O}$-module. 
      Then $\mathrm{Tw}_\rho(M)$ is pseudonull 
      if and only if $M$ itself is pseudonull. Furthermore 
      we obtain the following equality of ideals of $\Lambda^\mathrm{CM}_\mathcal{O}:$
\begin{align*}
\mathrm{Char}_{\Lambda^\mathrm{CM}_\mathcal{O}}(\mathrm{Tw}_\rho(M))=\mathrm{Tw}_{\rho^{-1}}(\mathrm{Char}_{\Lambda^\mathrm{CM}_\mathcal{O}}(M)).
\end{align*}
\end{enumerate}
\end{lem}

In the statement of Lemma~\ref{lem:twist} (3), the characteristic ideals of finitely generated torsion modules over the {\em semilocal} Iwasawa algebra $\Lambda^\mathrm{CM}_\mathcal{O}$ are defined componentwisely. We shall recall the precise definition of them in Definition~\ref{def:char}.
Note that the Pontrjagin dual
$M^\vee=\mathrm{Hom}_\mathrm{cts}(M, \mathbb{Q}_p/\mathbb{Z}_p)$ of a
$\Lambda^\mathrm{CM}_\mathcal{O}$-module $M$ is equipped with the
$\Lambda^\mathrm{CM}_\mathcal{O}$-module structure defined by
$r\phi(x):=\phi(rx)$, which is incompatible with the usual action 
of $\mathrm{Gal}(\widetilde{F}_\infty/F)$ on
$\mathrm{Hom}_\mathrm{cts}(M, \mathbb{Q}_p/\mathbb{Z}_p)$ defined by $g\phi(x):=\phi(g^{-1}x)$. 

\begin{proof}
The assertions (1) and (2) directly follow from the definition of $p$-twisting.
For the assertion (3), assume that $M$ is a pseudonull module. 
By definition there exist distinct nonzero elements $r_1$ and $r_2$ in the
 annihilator ideal of $M$ which are relatively prime to each other. 
One then readily observes
 that both $\mathrm{Tw}_{\rho^{-1}}(r_1)$ and $\mathrm{Tw}_{\rho^{-1}}(r_2)$ 
annihilate $\mathrm{Tw}_\rho(M)$ and are relatively prime to each
 other, which implies that $\mathrm{Tw}_\rho(M)$ is also pseudonull. 
Similarly one readily verifies the converse implication.
 
 Now let us consider
 the statement on characteristic ideals in the assertion (3).
We first reduce  the claim to the case where $M$ is an elementary
 $\Lambda^\mathrm{CM}_\mathcal{O}$-module of the form
 $\Lambda^\mathrm{CM}_\mathcal{O}/(a)$, by considering componentwisely and
 using the structure theorem of finitely generated torsion
 $\mathcal{O}[[\mathrm{Gal}(\widetilde{F}/F)]]$-modules. 
 Then the reduced claim obviously holds since there exists an isomorphism 
\begin{align*}
\Lambda^\mathrm{CM}_\mathcal{O}/(\mathrm{Tw}_{\rho^{-1}}(a))\xrightarrow{\sim}
 \mathrm{Tw}_\rho(\Lambda^\mathrm{CM}_\mathcal{O}/(a))\ ; \ r \mod
 (\mathrm{Tw}_{\rho^{-1}}(a)) \mapsto \mathrm{Tw}_\rho(r) \mod (a)
\end{align*}
of $\Lambda^\mathrm{CM}_\mathcal{O}$-modules. 
\end{proof}

Since  $\mathcal{A}^\mathrm{CM}_\eta$ is isomorphic to
$\mathrm{Tw}_{\eta^\mathrm{gal}\psi^{-1}}(\Lambda^{\mathrm{CM},\vee}_\mathcal{O})$
as a $\Lambda^\mathrm{CM}_\mathcal{O}$-module, 
we may calculate by using Lemma~\ref{lem:twist} the kernel of
$\phi_{F_S/\widetilde{K}^\mathrm{CM}_\infty}$ as 
\begin{align*}
\ker(\phi_{F_S/\widetilde{K}^\mathrm{CM}_\infty}) &=
 \mathrm{Hom}_{\Lambda^\mathrm{CM}_\mathcal{O}}(X_{\Sigma_p, (\psi)},
 \mathrm{Tw}_{\eta^\mathrm{gal}\psi^{-1}}(\Lambda^\mathrm{CM}_\mathcal{O})^\vee)\\ 
&\cong
 \mathrm{Hom}_{\Lambda^\mathrm{CM}_\mathcal{O}}(\mathrm{Tw}_{\eta^\mathrm{gal}\psi^{-1}}(\Lambda^\mathrm{CM}_\mathcal{O}),
 X_{\Sigma_p, (\psi)}^\vee) \\
&=\mathrm{Hom}_{\Lambda^\mathrm{CM}_\mathcal{O}}(\Lambda^\mathrm{CM}_\mathcal{O},
 \mathrm{Tw}_{\eta^{\mathrm{gal}, -1}\psi} (X_{\Sigma_p, (\psi)})^\vee) 
\cong \mathrm{Tw}_{\eta^{\mathrm{gal}, -1}\psi} (X_{\Sigma_p,
 (\psi)})^\vee.
\end{align*}
As a consequence we obtain the following proposition.

\begin{pro} \label{prop:iw_mod} 
The Pontrjagin dual of the Selmer group
 $\mathrm{Sel}^\Sigma_{\mathcal{A}^\mathrm{CM}_\eta}$ is pseudoisomorphic to  
 $\mathrm{Tw}_{\eta^{\mathrm{gal},-1}\psi}(X_{\Sigma_p, (\psi)})$ as a
 $\Lambda^\mathrm{CM}_\mathcal{O}$-module. In particular, 
we obtain the following equality among characteristic ideals$:$
\begin{align*}
\mathrm{Char}_{\Lambda^\mathrm{CM}_\mathcal{O}}
 \mathrm{Sel}_{\mathcal{A}^\mathrm{CM}_\eta}^{\Sigma,\vee} 
 =\mathrm{Char}_{\Lambda^\mathrm{CM}_\mathcal{O}} (\mathrm{Tw}_{\eta^{\mathrm{gal},-1}\psi}(X_{\Sigma_p, (\psi)}))  
= \mathrm{Tw}_{\eta^\mathrm{gal}\psi^{-1}}(\mathrm{Char}_{\Lambda^\mathrm{CM}_\mathcal{O}}
 X_{\Sigma_p,(\psi)}). 
\end{align*}
\end{pro} 

\begin{rem} \label{rem:sel_psi}
We can define $\Lambda^\mathrm{CM}_\mathcal{O}$-modules $\mathcal{T}^\mathrm{CM}_\psi$ and 
$\mathcal{A}^\mathrm{CM}_\psi$ similarly to $\mathcal{T}^\mathrm{CM}_\eta$ and
$\mathcal{A}^\mathrm{CM}_\eta$: namely, we 
set $\mathcal{T}^\mathrm{CM}_\psi=\mathcal{O}(\psi)
\otimes_\mathcal{O} \Lambda_\mathcal{O}^{\mathrm{CM}, \sharp}$ and 
$\mathcal{A}^\mathrm{CM}_\psi=\mathcal{T}^\mathrm{CM}_\psi \otimes_{\Lambda^\mathrm{CM}_\mathcal{O}}
\Lambda_\mathcal{O}^{\mathrm{CM}, \vee}$. 
Then we readily verify that the Pontrjagin dual
 of the Selmer group $\mathrm{Sel}^\Sigma_{\mathcal{A}^\mathrm{CM}_\psi}$ associated to
 $\mathcal{A}^\mathrm{CM}_\psi$ and the maximal $\psi$-isotypic quotient
$X_{\Sigma_p, (\psi)}$ of $X_{\Sigma_p}\otimes_{\mathbb{Z}_p}
 \mathcal{O}$ are pseudoisomorphic as 
$\Lambda^\mathrm{CM}_\mathcal{O}$-modules by the same argument as above 
(and they are isomorphic when 
the cardinality of $\Delta$ is relatively prime to $p$).  
In other words, 
 $\mathrm{Sel}^\Sigma_{\mathcal{A}^\mathrm{CM}_\eta}$ is
 copseudoisomorphic to
 $\mathrm{Tw}_{\eta^{\mathrm{gal},-1}\psi}(\mathrm{Sel}^\Sigma_{\mathcal{A}_\psi^\mathrm{CM}})$
 as a $\Lambda^\mathrm{CM}_\mathcal{O}$-module.
\end{rem}

We have defined the finitely generated $\Lambda^\mathrm{CM}_\mathcal{O}$-module
$X_{\Sigma_p}\otimes_{\mathbb{Z}_p} \mathcal{O}$ 
in a similar way to the manner where  one defines 
the (classical) Iwasawa module $X_{\{p\}}$: the Galois group
$\mathrm{Gal}(M_{\{p\}}/\mathsf{K}(\mu_{p^\infty}))$ of the maximal 
abelian pro-$p$ extension $M_{\{p\}}$ of $\mathsf{K}(\mu_{p^\infty})$
which is unramified outside places lying above $p$ (here $\mathsf{K}$
denotes an algebraic number field). Proposition~\ref{prop:iw_mod}
describes the relation between Selmer groups (of Greenberg type) and 
Iwasawa modules constructed in a classical way. 
The algebraic structure of the Iwasawa module $X_{\Sigma_p, (\psi)}$ 
has been thoroughly studied by Perrin-Riou \cite{PR_CM} (for
imaginary quadratic fields) and Hida and Tilouine \cite{HT-aIMC} (for
general CM fields).

\subsection{Exact control theorem} \label{ssc:control}

 We shall prove in this subsection 
the {\em exact control theorem} (Theorem~\ref{Thm:control}) 
for Selmer groups $\mathrm{Sel}^\Sigma_{\mathcal{A}^\mathrm{CM}_\eta}$ 
associated  to a gr\"o{\ss}encharacter $\eta$ of type $(A_0)$ defined on 
a CM number field $F$ (see Definition~\ref{def:Selmer_CM}). 
Recall that the Iwasawa algebra
$\Lambda^\mathrm{CM}_\mathcal{O}=\mathcal{O}[[\mathrm{Gal}(\widetilde{F}_\infty/F)]]$
is a semilocal ring 
 each of whose components is isomorphic to 
 the ring of formal power series over $\mathcal{O}$ in 
$d+1+\delta_{F,p}$ variables. Therefore there exists a regular sequence 
$\{ x_1 , \dotsc , x_{d + \delta_{F,p}} \}$ in $\Lambda^\mathrm{CM}_\mathcal{O}$ 
such that each $x_j$ belongs to the kernel $\mathfrak{A}^\mathrm{cyc}$ of 
the augmentation map $\Lambda^\mathrm{CM}_\mathcal{O}
\twoheadrightarrow
\Lambda_\mathcal{O}^\mathrm{cyc}=\mathcal{O}[[\mathrm{Gal}(F(\mu_{p^\infty})
/F)]]$. Such a regular 
sequence $\{ x_1, \dotsc, x_{d+\delta_{F,p}}\}$ is contained in 
the Jacobson radical of $\Lambda^\mathrm{CM}_\mathcal{O}$, and thus 
its arbitrary permutation is also regular.
For each $j$ with $0\leq j \leq d+\delta_{F,p}$, 
let $\mathfrak{A}_{j}$ denote the ideal of $\Lambda^\mathrm{CM}_\mathcal{O}$ generated by $x_1,\dotsc,x_j$:
\begin{equation}\label{equation:x_1tox_j} 
\mathfrak{A}_{j} = (x_1 ,\ldots ,x_j ).
\end{equation}
As the convention we use the symbol $\mathfrak{A}_{0}$ for the zero ideal of $\Lambda^\mathrm{CM}_\mathcal{O}$. 
Let $\mathcal{A}^\mathrm{CM}_\eta[\mathfrak{A}_{j}]$ denote the
$\Lambda^\mathrm{CM}_\mathcal{O}$-submodule of
$\mathcal{A}^\mathrm{CM}_\eta$ consisting of all the
elements annihilated by $\mathfrak{A}_{j}$.

We here introduce the following nontriviality condition (ntr)$_{\mathfrak{P}}$ 
on $\mathcal{A}^\mathrm{CM}_\eta$ 
for each place $\mathfrak{P}$ of $F$ lying above $p$:
\begin{description}
\item[{(ntr)$_{\mathfrak{P}}$}] for each maximal ideal $\mathfrak{M}$ 
	   of the semilocal Iwasawa algebra
	   $\Lambda^\mathrm{CM}_\mathcal{O}$,
	   the maximal $\mathfrak{M}$-torsion submodule 
	   $\mathcal{A}^\mathrm{CM}_\eta[\mathfrak{M}]$ of
	   $\mathcal{A}^\mathrm{CM}_\eta$ is not trivial as a
	   $D_{\mathfrak{P}}$-module.
\end{description}

 \begin{thm}[Exact control theorem, {Theorem~\ref{Thm:control}}] \label{thm:control}
Assume that the nontriviality condition {\upshape
  (ntr)$_{\mathfrak{P}}$} is fulfilled 
for every place $\mathfrak{P}$ contained in $\Sigma_p^c$.
Let $\{ x_1, \dotsc, x_{d+\delta_{F,p}}\}$ be an arbitrary regular
 sequence in $\Lambda^\mathrm{CM}_\mathcal{O}$ all of whose elements
  belong to $\mathfrak{A}^\mathrm{cyc}$, and let us define $\mathfrak{A}_j$ as
  the ideal of $\Lambda^\mathrm{CM}_\mathcal{O}$ generated by $x_1, x_2, \dotsc, x_j$. Then the natural map  
 \begin{align*}
 \mathrm{Sel}^\Sigma_{\mathcal{A}^\mathrm{CM}_\eta[\mathfrak{A}_{j}]} 
\longrightarrow 
\mathrm{Sel}^\Sigma_{\mathcal{A}^\mathrm{CM}_\eta}[\mathfrak{A}_{j}]
 \end{align*} 
induced by the natural inclusion $\mathcal{A}^\mathrm{CM}_\eta[\mathfrak{A}_j] \hookrightarrow \mathcal{A}^\mathrm{CM}_\eta$ is an isomorphism for each $j$ with $0\leq j\leq d+\delta_{F,p}$.  
\end{thm}

For the proof of Theorem~\ref{thm:control}, 
we first  replace our Selmer group 
$\mathrm{Sel}^\Sigma_{\mathcal{A}^\mathrm{CM}_\eta}$ 
by the {\em strict Selmer group} $\mathrm{Sel}_{\mathcal{A}_\eta^\mathrm{CM}}^{\Sigma,\mathrm{str}}$ defined below. We then control the strict Selmer group under 
specialisation with respect to the regular sequence  
$\{ x_1 , \dotsc, x_j \}$.

\begin{defn} \label{def:str_Selmer}
For each $j$ with $0\leq j\leq d+\delta_{F,p}$, 
we define the {\em strict Selmer group}
$\mathrm{Sel}^{\Sigma, \mathrm{str}}_{\mathcal{A}^\mathrm{CM}_\eta[\mathfrak{A}_{j}]}$ 
as the kernel of the global-to-local map
\begin{align*}
 H^1 (F_S /F ,\mathcal{A}^\mathrm{CM}_\eta[\mathfrak{A}_{j}]) \longrightarrow 
 \underset{\substack{\lambda \in S \\ \lambda \nmid p\infty}}{\prod} 
 H^1 (I_{\lambda} ,\mathcal{A}^\mathrm{CM}_\eta[\mathfrak{A}_{j}])
 \times \underset{ \mathfrak{P}\in \Sigma^c_p}{\prod} 
 H^1 (F_{\mathfrak{P}} ,\mathcal{A}^\mathrm{CM}_\eta[\mathfrak{A}_{j}] ).
 \end{align*} 
\end{defn}

Since, for the strict Selmer group $\mathrm{Sel}_{\mathcal{A}^\mathrm{CM}_\eta[\mathfrak{A}_{j}]}^{\Sigma,\mathrm{str}}$,
only the local conditions at places contained in $\Sigma_p^c$  
are modified compared to the definition of the usual Selmer group
$\mathrm{Sel}_{\mathcal{A}^\mathrm{CM}_\eta[\mathfrak{A}_{j}]}^\Sigma$
in Definition~\ref{def:Selmer_CM}, we have the following comparison result:

 \begin{lem} \label{lem:replace}
 Let the notation be as in Theorem~$\ref{thm:control}$ and 
  assume that the condition {\upshape (ntr)$_{\mathfrak{P}}$} is fulfilled
  for each place $\mathfrak{P}$ contained in $\Sigma_p^c$.
Then for every $0\leq j \leq d+\delta_{F,p}$ the natural injection $\mathrm{Sel}^{\Sigma, \mathrm{str}}_{\mathcal{A}^\mathrm{CM}_\eta[\mathfrak{A}_{j}]} \hookrightarrow \mathrm{Sel}^\Sigma_{\mathcal{A}^\mathrm{CM}_\eta[\mathfrak{A}_{j}]}$ induces an isomorphism.
 \end{lem}

 \begin{proof}[Proof of {Lemma~$\ref{lem:replace}$}] 
We consider the following commutative diagram:
 \begin{align} \label{eq:sel_vs_strsel} 
\xymatrix{
 0  \ar[r] &
  \mathrm{Sel}^{\Sigma, \mathrm{str}}_{\mathcal{A}^\mathrm{CM}_\eta[\mathfrak{A}_{j}]} \ar[r]
  \ar[d] &
 H^1(F_{S}/F,\mathcal{A}^\mathrm{CM}_\eta[\mathfrak{A}_{j}]) \ar[r] \ar@{=}[d] & 
R(F, \mathcal{A}^\mathrm{CM}_\eta[\mathfrak{A}_{j}]) \ar[d]^{\alpha_j} \\ 
 0 \ar[r] & \mathrm{Sel}^\Sigma_{\mathcal{A}^\mathrm{CM}_\eta[\mathfrak{A}_{j}]} \ar[r] & H^1
  (F_S /F , \mathcal{A}^\mathrm{CM}_\eta[\mathfrak{A}_{j}])  \ar[r] & R_\mathrm{str}(F, \mathcal{A}^\mathrm{CM}_\eta[\mathfrak{A}_{j}]),
}
  \end{align}
where we use the following
  abbreviation on the direct products of local cohomology groups:
\begin{align*}
R(F, \mathcal{A}^\mathrm{CM}_\eta[\mathfrak{A}_{j}]) &= \underset{\substack{\lambda \in S \\ \lambda \nmid p\infty}}{\prod} 
 {H^1 (I_{\lambda} ,\mathcal{A}^\mathrm{CM}_\eta[\mathfrak{A}_{j}])}
\times
 \underset{\mathfrak{P}\in \Sigma^c_p}{\prod}{H^1 (I_{\mathfrak{P}} ,
 \mathcal{A}^\mathrm{CM}_\eta[\mathfrak{A}_{j}])} , \\
R_\mathrm{str}(F, \mathcal{A}^\mathrm{CM}_\eta[\mathfrak{A}_{j}]) &=\underset{\substack{\lambda \in S \\ \lambda \nmid p\infty}}{\prod} 
 {H^1 (I_{\lambda},\mathcal{A}^\mathrm{CM}_\eta[\mathfrak{A}_{j}])}
\times  
 \underset{\mathfrak{P} \in \Sigma^c_p}{\prod} {H^1 (F_{\mathfrak{P}}
 ,\mathcal{A}^\mathrm{CM}_\eta[\mathfrak{A}_{j}])}.
\end{align*}
  The right vertical map $\alpha_j$ of the diagram \eqref{eq:sel_vs_strsel} is defined as the usual restriction
  maps on the $\Sigma_p^c$-components and the identity maps on the other
  components. 
The snake lemma implies that the natural map $\mathrm{Sel}^{\Sigma, \mathrm{str}}_{\mathcal{A}^\mathrm{CM}_\eta[\mathfrak{A}_{j}]}
  \rightarrow 
  \mathrm{Sel}^\Sigma_{\mathcal{A}^\mathrm{CM}_\eta[\mathfrak{A}_{j}]}$ is an injection 
  whose cokernel is isomorphic to 
  a certain submodule of $\mathrm{Ker} (\alpha_j )$.   
On the other hand, $\mathrm{Ker} (\alpha_j )$ is isomorphic to 
$\underset{\mathfrak{P}\in \Sigma^c_p}{\prod} H^1
  (D_{\mathfrak{P}}/I_{\mathfrak{P}},\mathcal{A}^\mathrm{CM}_\eta[\mathfrak{A}_{j}]^{I_{\mathfrak{P}}})$
  due to the inflation-restriction sequence.

  We now verify that each component $H^1
  (D_{\mathfrak{P}}/I_{\mathfrak{P}},\mathcal{A}^\mathrm{CM}_\eta[\mathfrak{A}_{j}]^{I_{\mathfrak{P}}})$ 
  of the direct product above is trivial. 
Since the the Pontrjagin dual $(\mathcal{A}^\mathrm{CM}_\eta )^\vee$ of $\mathcal{A}^\mathrm{CM}_\eta$ is 
isomorphic to $\Lambda^\mathrm{CM}_\mathcal{O}$ as a $\Lambda^\mathrm{CM}_\mathcal{O}$-module, 
we have $(\mathcal{A}^\mathrm{CM}_\eta[\mathfrak{A}_{j}]^{I_{\mathfrak{P}}})^\vee \cong \Lambda_\mathcal{O}/J_{\mathfrak{P}}$ 
with the annihilator ideal $J_{\mathfrak{P}}$ of $(\mathcal{A}^\mathrm{CM}_\eta[\mathfrak{A}_{j}]^{I_{\mathfrak{P}}})^\vee$. 
 We denote by $f_{\mathfrak{P}}$ the value of the 
 Frobenius element at $\mathfrak{P}$ acting on
  $\mathcal{A}^\mathrm{CM}_\eta[\mathfrak{A}_{j}]^{I_{\mathfrak{P}}}$. Then
  the unramified cohomology group 
$H^1 (D_{\mathfrak{P}}/I_{\mathfrak{P}},\mathcal{A}^\mathrm{CM}_\eta[\mathfrak{A}_{j}]^{
  I_{\mathfrak{P}}})$ is described as the quotient $\mathcal{A}^\mathrm{CM}_\eta[\mathfrak{A}_{j}]^{
  I_{\mathfrak{P}}}/(f_{\mathfrak{P}}-1)
  \mathcal{A}^\mathrm{CM}_\eta[\mathfrak{A}_{j}]^{
  I_{\mathfrak{P}}}$, and hence 
  it is trivial if and only if the multiplication by  
 $f_{\mathfrak{P}}-1$ on
  $\mathcal{A}^\mathrm{CM}_\eta[\mathfrak{A}_{j}]^{I_{\mathfrak{P}}}$
  is surjective; in other words, it is trivial  
  if and only if the multiplication by
  $f_{\mathfrak{P}}-1$ induces 
  an injective endomorphism on $\Lambda^\mathrm{CM}_\mathcal{O}/J_{\mathfrak{P}}$. 
  The latter condition is obviously fulfilled when
  $f_{\mathfrak{P}}-1$ is invertible in
  $\Lambda^\mathrm{CM}_\mathcal{O}/J_{\mathfrak{P}}$, or equivalently, 
  when it is invertible in 
  $\Lambda^\mathrm{CM}_\mathcal{O}/\mathfrak{M}$ for each maximal
  ideal of $\Lambda^\mathrm{CM}_\mathcal{O}$ containing
  $J_{\mathfrak{P}}$. The condition (ntr)$_{\mathfrak{P}}$ thus implies the
  triviality of 
  each cohomology group $H^1(D_{\mathfrak{P}}/I_{\mathfrak{P}},
  \mathcal{A}^\mathrm{CM}_\eta[\mathfrak{A}_{j}]^{I_{\mathfrak{P}}})$
  since the condition (ntr)$_{\mathfrak{P}}$ asserts that the value of $f_\mathfrak{P}$ does not equal $1$.
  Consequently, the strict Selmer
  group $\mathrm{Sel}^{\Sigma, \mathrm{str}}_{\mathcal{A}^\mathrm{CM}_\eta[\mathfrak{A}_{j}]}$ is isomorphic to 
$\mathrm{Sel}^\Sigma_{\mathcal{A}^\mathrm{CM}_\eta[\mathfrak{A}_{j}]}$ for every 
 $j$ with $0 \leq j \leq d+\delta_{F,p}$ when we assume the condition {\upshape (ntr)$_{\mathfrak{P}}$} for each 
place $\mathfrak{P}$ contained in $\Sigma_p^c$. 
 \end{proof}

Let us return to the proof of Theorem \ref{thm:control}. 

 \begin{proof}[Proof of Theorem~$\ref{thm:control}$]
In order to control the strict Selmer group $\mathrm{Sel}^{\Sigma,
  \mathrm{str}}_{\mathcal{A}^\mathrm{CM}_\eta[\mathfrak{A}_{j}]}$,
  we consider the following diagram for each $j$: 
{
 \begin{align}\label{equation:control_strict}
\xymatrix{
 0 \ar[r] & \mathrm{Sel}^{\Sigma,
  \mathrm{str}}_{\mathcal{A}^\mathrm{CM}_\eta[\mathfrak{A}_{j+1}]}
  \ar[r] \ar[d] & H^1 (F_{S}/F , 
\mathcal{A}^\mathrm{CM}_\eta[\mathfrak{A}_{j+1}]) 
\ar[r] \ar[d]_{\alpha_j} & 
R_\mathrm{str}(F, \mathcal{A}^\mathrm{CM}_\eta[\mathfrak{A}_{j+1}])  \ar[d]_{\beta_j} \\ 
 0 \ar[r] & \mathrm{Sel}^{\Sigma,
  \mathrm{str}}_{\mathcal{A}^\mathrm{CM}_\eta[\mathfrak{A}_{j}]}[x_{j+1}]
  \ar[r] &
 H^1 (F_S /F , \mathcal{A}^\mathrm{CM}_\eta[\mathfrak{A}_{j}])[x_{j+1}] 
\ar[r] & R_\mathrm{str}(F, \mathcal{A}^\mathrm{CM}_\eta[\mathfrak{A}_{j}])[x_{j+1}],
}
  \end{align}}
where the symbol $R_\mathrm{str}(F,
  \mathcal{A}^\mathrm{CM}_\eta[\mathfrak{A}_{j}])$ denotes the groups introduced in the diagram (\ref{eq:sel_vs_strsel}).
The middle and right vertical maps $\alpha_j ,\beta_j$ are induced from the long exact
  sequence in Galois cohomology associated to the following short exact
  sequence of $\mathrm{Gal}(F_S/F)$-modules:
\begin{align} \label{ses:control}
0 \rightarrow \mathcal{A}^\mathrm{CM}_\eta[\mathfrak{A}_{j+1}] \rightarrow
 \mathcal{A}^\mathrm{CM}_\eta[\mathfrak{A}_{j}] \xrightarrow{\times x_{j+1}}
 \mathcal{A}^\mathrm{CM}_\eta[\mathfrak{A}_{j}] \rightarrow 0.
\end{align}
In particular $\alpha_j$ is surjective by construction and it suffices 
  to verify that the map $\alpha_j$ (resp.\ $\beta_j$) is injective  
  in order to prove that the map $\mathrm{Sel}^\Sigma_{\mathcal{A}^\mathrm{CM}_\eta[\mathfrak{A}_{j+1}]}\rightarrow
  \mathrm{Sel}^\Sigma_{\mathcal{A}^\mathrm{CM}_\eta[\mathfrak{A}_{j}]}[x_{j+1}]$
  in consideration is injective (resp.\ surjective) by the snake lemma 
applied on (\ref{equation:control_strict}).

As for the kernel of $\alpha_j$, we first observe that it is isomorphic to the quotient module $H^0(F_S/F, \mathcal{A}^\mathrm{CM}_\eta[\mathfrak{A}_{j}])/x_{j+1} 
H^0(F_S/F, \mathcal{A}^\mathrm{CM}_\eta[\mathfrak{A}_{j}])$ by the long exact sequence of cohomology of $\mathrm{Gal}(F_S/F)$ obtained 
  from (\ref{ses:control}). Obviously the global zeroth cohomology group
  $H^0 (F_S/F,\mathcal{A}^\mathrm{CM}_\eta[\mathfrak{A}_{j}])$ is a submodule of the local zeroth cohomology group $H^0(F_{\mathfrak{P}},
  \mathcal{A}^\mathrm{CM}_\eta[\mathfrak{A}_{j}])$ for an arbitrary place $\mathfrak{P}$ in $\Sigma_p^c$. As we shall see in the next paragraph, the latter cohomology group $H^0(F_{\mathfrak{P}},
  \mathcal{A}^\mathrm{CM}_\eta[\mathfrak{A}_{j}])$ is trivial 
  under the condition (ntr)$_{\mathfrak{P}}$. This implies that $\mathrm{Ker} (\alpha_j)$ is trivial.

Next we verify the triviality of the kernel of $\beta_j$. 
For each place $w$ in $S$, let $\beta_{j,w}$ denote the map induced by
  $\beta_j$ on the $w$-component of $R_\mathrm{str}(F,\mathcal{A}_\eta^\mathrm{CM}[\mathfrak{A}_{j+1}])$. Then by the long exact sequence of group cohomology for $D_w$- or $I_w$-modules associated to the short exact sequence (\ref{ses:control}), we have 
\begin{align*}
\ker(\beta_{j,w})\cong \begin{cases}
 H^0 (F_w ,\mathcal{A}^\mathrm{CM}_\eta[\mathfrak{A}_{j}])/x_{j+1} 
 H^0 (F_w
		,\mathcal{A}^\mathrm{CM}_\eta[\mathfrak{A}_{j}])
		& \text{for $w\in \Sigma_p^c$}, \\
H^0 (I_w ,\mathcal{A}^\mathrm{CM}_\eta[\mathfrak{A}_{j}])/x_{j+1} 
 H^0 (I_w
		,\mathcal{A}^\mathrm{CM}_\eta[\mathfrak{A}_{j}])
		& \text{for $w \in S\setminus (\Sigma_p\cup \Sigma^c_p \cup \Sigma_\infty )$}. \\
\end{cases}
\end{align*} 
For a place $w$ in $\Sigma^c_p$, the cohomology group 
 $H^0 (F_w
  ,\mathcal{A}^\mathrm{CM}_\eta[\mathfrak{A}_{j}])$ itself
  is trivial under the condition (ntr)$_{w}$.  
 In fact, it is easy to redescribe $H^0 (F_w
  ,\mathcal{A}^\mathrm{CM}_\eta[\mathfrak{A}_{j}])$ as
  $H^0 (D_w/I_w,\mathcal{A}^\mathrm{CM}_\eta[\mathfrak{A}_{j}]^{
  I_w})$ by definition. We denote by
  $f_w$ the value of the Frobenius element
  at $w$ acting on
  $\mathcal{A}^\mathrm{CM}_\eta[\mathfrak{A}_{j}]^{ I_{w}}$. 
  Then, as in the proof of Lemma~$\ref{lem:replace}$, we readily see that
  the Pontrjagin dual of $H^0 (D_w/I_w,\mathcal{A}^\mathrm{CM}_\eta[\mathfrak{A}_{j}]^{
  I_w})$ is isomorphic to the cokernel of the multiplication of $f_w-1$ on
  $\Lambda^\mathrm{CM}_\mathcal{O}/J_w$, where $J_w$ is the annihilator ideal of $(\mathcal{A}_\eta^\mathrm{CM}[\mathfrak{A}_j]^{I_w})^\vee$. The element $f_w-1$ is, however,
  a unit of $\Lambda^\mathrm{CM}_\mathcal{O}/J_w$ thanks to the condition (ntr)$_w$
  as discussed in the proof of Lemma~\ref{lem:replace}, and in particular the
  cokernel of the multiplication of $f_w-1$ is trivial.
  This completes the proof of the triviality of the kernel of
  $\beta_{j,w}$.

For a place $w$ in $S\setminus (\Sigma_p\cup \Sigma^c_p \cup \Sigma_\infty )$, 
  the inertia subgroup 
$I_w$ acts on $\mathcal{O}(\eta)$ through a finite
  quotient and acts on $\Lambda^{\mathrm{CM},\sharp}_\mathcal{O}$
  trivially. 
Let $E_w$ denote the (finite) image of $I_w$ under the Galois
  character $\eta^\mathrm{gal}$ 
  and let $\varpi^{n_w} \mathcal{O}$ denote the ideal of 
$\mathcal{O}$ generated by every element of the form $\zeta-1$ with 
$\zeta$ belonging to $E_w$ (here $\varpi$ denotes a uniformiser of
  $\mathcal{O}$). Then one readily sees that the cohomology group $H^0(I_w,
  \mathcal{A}^\mathrm{CM}_\eta[\mathfrak{A}_{j}])$ is 
none other than the maximal $\varpi^{n_w}$-torsion submodule 
of $\mathcal{A}^\mathrm{CM}_\eta[\mathfrak{A}_{j}]$.
The Pontrjagin dual of $\ker(\beta_{j,w} )$ is thus isomorphic to 
 the kernel of the multiplication of $x_{j+1}$ on
  $\Lambda^\mathrm{CM}_\mathcal{O}/(\varpi^{n_w}, x_1, \dotsc,
  x_j)$. 
The sequence $ x_1, \dotsc, x_{j+1}, \varpi^{n_w} $ is, however, 
  a regular sequence and contained in the Jacobson radical of
  $\Lambda^\mathrm{CM}_\mathcal{O}$, and thus its permutation 
  $\varpi^{n_w}, x_1, \dotsc, x_{j+1}$ is also a regular sequence in
  $\Lambda_{\mathcal{O}}^{\mathrm{CM}}$. We therefore see that the kernel of
  $\beta_{j,w}$ is trivial.
 \end{proof}

We finally remark that, by applying the same arguments as the proof of
Lemma~\ref{lem:replace} to the discrete $\Lambda^\mathrm{CM}_\mathcal{O}$-module
$\mathcal{A}^\mathrm{CM}_\psi$, we obtain the following result.

\begin{lem} \label{lem:replace_psi}
The natural injection
 $\mathrm{Sel}^{\Sigma,\mathrm{str}}_{\mathcal{A}^\mathrm{CM}_\psi}\hookrightarrow
 \mathrm{Sel}^\Sigma_{\mathcal{A}^\mathrm{CM}_\psi}$ is a copseudoisomorphism of
 discrete $\Lambda^\mathrm{CM}_\mathcal{O}$-modules. Furthermore,  
if the following condition {\upshape(ntr)}$_{\psi,\mathfrak{P}}$
 on $\mathcal{A}^\mathrm{CM}_\psi$ is fulfilled for each place
 $\mathfrak{P} \in \Sigma^c_p$, the injection above is an isomorphism. 
\begin{description}
\item[(ntr)$_{\psi, \mathfrak{P}}$] For each maximal ideal $\mathfrak{M}$ of
	   $\Lambda^\mathrm{CM}_\mathcal{O}$, 
	   the maximal $\mathfrak{M}$-torsion submodule
	   $\mathcal{A}^\mathrm{CM}_\psi[\mathfrak{M}]$ of $\mathcal{A}_\psi^\mathrm{CM}$ is not trivial as
	   a $D_{\mathfrak{P}}$-module. 
\end{description}
\end{lem}

\begin{proof}
We may verify that
 $(\mathcal{A}^\mathrm{CM}_\psi)^{I_{\mathfrak{P}}}$ is a
 copseudonull $\Lambda^\mathrm{CM}_\mathcal{O}$-module for each place
 $\mathfrak{P}$ of $F$ in $\Sigma_p^c$ by arguments 
 similar to the proof of 
 the $\Lambda^\mathrm{CM}_\mathcal{O}$-copseudonullity of
 $(\mathcal{A}^\mathrm{CM}_\eta)^{I_{\mathfrak{P}}}$; see the proof of
 Lemma~\ref{lem:cpsn-loc}. Therefore the first half of the claim is true 
because the cokernel of the natural injection 
 $\mathrm{Sel}^{\Sigma,\mathrm{str}}_{\mathcal{A}^\mathrm{CM}_\psi}\hookrightarrow
 \mathrm{Sel}^\Sigma_{\mathcal{A}^\mathrm{CM}_\psi}$ is a submodule of 
 the direct product of unramified cohomology groups
 $H^1(D_{\mathfrak{P}}/I_{\mathfrak{P}},
 (\mathcal{A}^\mathrm{CM}_\psi)^{I_{\mathfrak{P}}})$ for places 
 $\mathfrak{P}$ in $\Sigma_p^c$.

The second half of the claim is verified by the same argument as the
 proof of Lemma~\ref{lem:replace}.
\end{proof}

%
\subsection{Greenberg's criterion for almost divisibility} \label{ssc:greenberg}
%

 As is well known, the characteristic ideal 
 of a finitely generated torsion module over 
 a complete noetherian regular local domain 
is {\em not necessarily preserved} under basechange 
(or specialisation) procedures. 
Indeed the existence of a {\em nontrivial pseudonull 
submodule} causes peculiar behaviour of the characteristic ideal
under specialisation.  
Therefore, when we discuss specialisation 
 of the multi-variable Iwasawa main conjecture (or,
 in particular, specialisation of multi-variable 
Selmer groups), 
it is crucial to check whether the Pontrjagin dual 
 of the Selmer group contains 
 nontrivial pseudonull submodules or not. 
Greenberg has recently presented in \cite{gr-sel} certain sufficient conditions 
for the pseudonull submodule of the Pontrjagin dual
 of the Selmer group to be trivial, which is applicable to quite 
general situations. In this subsection we introduce various hypotheses which 
 are necessary to state Greenberg's criterion, and then 
we briefly review the main results of \cite{gr-sel}.

 \subsubsection{Algebraic settings} \label{sssc:setting}
%
  
Let $\Lambda_0$ be the ring $\mathcal{O}[[T_1, \dotsc, T_m]]$
of formal power series over the ring of integers $\mathcal{O}$
of a finite extension  of $\mathbb{Q}_p$, and let 
 $\mathcal{R}$ be a $\Lambda_0$-algebra
 which is isomorphic to the direct product of a finite number
 of copies of $\Lambda_0$. For each cofinitely generated
 discrete $R$-module $\mathcal{A}$, we define the {\em $\mathcal{R}$-corank $\mathrm{corank}_\mathcal{R}(\mathcal{A})$ of $\mathcal{A}$} as a {\em finite set} $(\mathrm{corank}_{\mathcal{R}_i}(e_i\mathcal{A}))_{i\in I}$ of nonnegative integers, where each $\mathcal{R}_i$ denotes a local component of 
 the semilocal ring $\mathcal{R}$ cut out by an idempotent $e_i$.
We endow the set of the coranks of cofinitely generated discrete $\mathcal{R}$-modules with the {\em componentwise partial order}; namely, the notation $\mathrm{corank}_\mathcal{R}(\mathcal{A}_1)\leq\mathrm{corank}_\mathcal{R}(\mathcal{A}_2)$ means that  $\mathrm{corank}_{\mathcal{R}_i}(e_i\mathcal{A}_1)\leq \mathrm{corank}_{\mathcal{R}_i}(e_i\mathcal{A}_2)$ holds for every $i$ in $I$. The {\em characteristic ideal} of a finitely generated torsion $\mathcal{R}$-module for the semilocal ring $\mathcal{R}$ is also defined componentwisely as follows.
   
   \begin{defn}[Characteristic ideal of $\mathcal{R}$-modules] \label{def:char}
    Let $M$ be a finitely generated torsion $\mathcal{R}$-module. Then we define the {\em $(\mathcal{R}$-$)$characteristic ideal $\mathrm{Char}_\mathcal{R}(M)$ of $M$} as the ideal of $\mathcal{R}$ corresponding to $\prod_{i\in I}\mathrm{Char}_{\mathcal{R}_i}(e_iM)$ under the indecomposable decomposition $\mathcal{R}=\prod_{i\in I}\mathcal{R}_i$ of $\mathcal{R}$. Here $\mathrm{Char}_{\mathcal{R}_i}(e_iM)$ denotes the characteristic ideal of the finitely generated torsion module $e_iM$ over the regular local ring $\mathcal{R}_i$, which is defined in the usual manner.
   \end{defn}
   
   By definition the characteristic ideal $\mathrm{Char}_\mathcal{R}(M)$ is a principal ideal of $\mathcal{R}$.
   
  \begin{rem}
 In \cite{gr-surj} and \cite{gr-sel}, Greenberg assumes that the coefficient ring $\mathcal{R}$ is a {\em local} ring equipped with several good properties (which Greenberg calls a {\em reflexive ring}). In our setting the coefficient ring $\mathcal{R}$ is a {\em semilocal} ring and is no longer local. However each local component of $\mathcal{R}$ is a regular local ring isomorphic to $\Lambda_0$, which is compatible with Greenberg's setting. Indeed it is not difficult at all to extend Greenberg's results of \cite{gr-surj} and \cite{gr-sel} to our semilocal coefficient case  by using the componentwise decomposition of the coefficient ring $\mathcal{R}$ and $\mathcal{R}$-modules. In Section~\ref{sssc:criterion} we shall introduce statements of Greenberg's results  of \cite{gr-surj} and \cite{gr-sel} extended to the semilocal coefficient case.
  \end{rem}
  
  Now let $\mathsf{K}$, $S$ and $\mathcal{T}$ be as in Section~\ref{sssc:general_Selmer}. Since the semilocal ring $\mathcal{R}=\prod_{i\in I}\mathcal{R}_i$ satisfies all the conditions introduced in Section~\ref{sssc:general_Selmer}, we can consider the Selmer group of $\mathcal{R}$-linear $G_\mathsf{K}$-representations. 
For the discrete $\mathcal{R}$-module 
 $\mathcal{A}=\mathcal{T}\otimes_{\mathcal{R}} \mathcal{R}^{\vee}$ associated to $\mathcal{T}$,
 we specify an $\mathcal{R}$-submodule $L(\mathsf{K}_v, \mathcal{A})$ of 
 the local Galois cohomology group $H^1(\mathsf{K}_v, \mathcal{A})$ for
 each $v$ in $S$, which we call a {\em local condition at $v$}. 
 We denote such a specification of local conditions by $\mathcal{L}$ for brevity. 
Set $Q_{\mathcal{L}}(\mathsf{K}_v, \mathcal{A})$ as 
 the quotient $H^1(\mathsf{K}_v, \mathcal{A})/L(\mathsf{K}_v, \mathcal{A})$ 
 for each $v$ in $S$.  We define $L(\mathsf{K}, \mathcal{A})$ as the product of the $\mathcal{R}$-submodules $L(\mathsf{K}_v, \mathcal{A})$ for all places $v$ in $S$, and similarly we introduce notation on products of local cohomology groups as follows:
 \begin{align*}
 P(\mathsf{K}, \mathcal{A}) = \prod_{v\in S} H^1(\mathsf{K}_v, \mathcal{A}), \quad Q_{\mathcal{L}}(\mathsf{K}, \mathcal{A}) = \prod_{v \in S} Q_{\mathcal{L}}(\mathsf{K}_v, \mathcal{A}). 
 \end{align*}
 The {\em $\mathcal{L}$-Selmer group $\mathrm{Sel}_{\mathcal{L}}(\mathsf{K}, \mathcal{A})$ 
 associated to $\mathcal{A}$} is defined 
 to be the kernel of the natural global-to-local homomorphism 
 \begin{align} \label{map:global-to-local}
 \phi_{\mathcal{L}} \colon H^1(\mathsf{K}_S/\mathsf{K}, \mathcal{A}) \rightarrow Q_{\mathcal{L}}(\mathsf{K},\mathcal{A})
 \end{align}
 induced by restriction morphisms of Galois cohomologies. By definition, $\mathrm{Sel}_{\mathcal{L}}(\mathsf{K},\mathcal{A})$ is an $\mathcal{R}$-submodule of the 
 first cohomology group $H^1(\mathsf{K}_S/\mathsf{K}, \mathcal{A})$ of 
 the Galois group $\mathrm{Gal}(\mathsf{K}_S/\mathsf{K})$ with coefficients 
 in $\mathcal{A}$. When we take the {\em trivial} specification
 $\mathcal{L}_{\mathrm{triv}}$, or in other words, when we impose the 
 minimal local condition $L(\mathsf{K}_v, \mathcal{A})=0$ on each place $v$ in
 $S$, the $\mathcal{L}_{\mathrm{triv}}$-Selmer group 
 is denoted by $\mathcyr{Sh}^1(\mathsf{K}, S,
 \mathcal{A})$ in \cite{gr-coh, gr-sel} and referred as 
 the {\em fine $S$-Selmer group} associated to
 $\mathcal{A}$ after Coates and Sujatha \cite{CS}.

Concerning the algebraic structure of the $\mathcal{L}$-Selmer groups, the following two statements are known to be equivalent (see \cite[Proposition~2.4]{gr-coh} for the proof):
 \begin{itemize}
 \item[-] the Pontrjagin dual $\mathrm{Sel}^{\vee}_{\mathcal{L}}(\mathsf{K},\mathcal{A})$ of the $\mathcal{L}$-Selmer group 
 $\mathrm{Sel}_{\mathcal{L}}(\mathsf{K},\mathcal{A})$ does not contain nontrivial $\Lambda_0$-pseudonull submodules; 
 \item[-] the $\mathcal{L}$-Selmer group $\mathrm{Sel}_{\mathcal{L}}(\mathsf{K},\mathcal{A})$ is {\em almost divisible} as a discrete $\Lambda_0$-module; that is, 
 the equality
          $\mathfrak{P}\mathrm{Sel}_{\mathcal{L}}(\mathsf{K},\mathcal{A})=\mathrm{Sel}_{\mathcal{L}}(\mathsf{K},\mathcal{A})$
          holds for all but finitely many prime ideals $\mathfrak{P}$ of
	  height one of $\Lambda_0$.
 \end{itemize}
 
 %
 \subsubsection{Various hypotheses} \label{sssc:hyp}
 %

Greenberg has thoroughly studied almost $\Lambda_0$-divisibility
of the $\mathcal{L}$-Selmer group and established certain useful
criteria for almost $\Lambda_0$-divisibility in \cite{gr-coh, gr-surj, gr-sel}. Now let us introduce 
various hypotheses which are necessary 
 to state Greenberg's criteria (see also \cite[Section~2.1]{gr-sel}).

 The first two hypotheses concern the Kummer (or Cartier) dual 
 $\mathcal{T}^*=\mathrm{Hom}_{\mathrm{cts}}(\mathcal{A},
 \mu_{p^\infty})$ of $\mathcal{A}$. For a place $v$ contained in $S$, we 
consider the following two statements:
 \begin{description}
 \item[(LOC$^{(1)}_{\mathcal{A},v}$)] the local Galois invariant submodule $H^0_{\mathrm{cts}}(\mathsf{K}_v, \mathcal{T}^*)$ of $\mathcal{T}^*$ is trivial;
 \item[(LOC$^{(2)}_{\mathcal{A},v}$)] the quotient module 
 $\mathcal{T}^*/H^0_{\mathrm{cts}}(\mathsf{K}_v, \mathcal{T}^*)$ is 
reflexive as an $\mathcal{R}$-module.
 \end{description}

The next hypothesis concerns the {\em generalised second
 Tate-\v{S}afarevi\v{c} group} $\mathcyr{Sh}^2(\mathsf{K},S, \mathcal{A})$ defined as the kernel of the global-to-local morphism
 \begin{align*}
  H^2(\mathsf{K}_S/\mathsf{K}, \mathcal{A}) \rightarrow \prod_{v\in S} H^2(\mathsf{K}_v, \mathcal{A})
 \end{align*}
 induced by usual restriction morphisms. Then we consider:
 \begin{description}
 \item[(LEO$_{\mathcal{A}}$)] the generalised second Tate-\v{S}afarevi\v{c} 
 group $\mathcyr{Sh}^2(\mathsf{K},S,\mathcal{A})$ is {\em cotorsion} as an $\mathcal{R}$\nobreakdash-module.
 \end{description}
 One of the significant features concerning the hypothesis
 (LEO$_{\mathcal{A}}$) is that it behaves well under specialisation procedures
 with respect to height-one prime ideals; namely, 
the condition (LEO$_\mathcal{A}$) holds 
 if and only if the condition (LEO$_{\mathcal{A}[\varPi]}$) holds
 (as a condition on the $\Lambda_0 /(\varPi)$-module $\mathcal{A}[\varPi]$) 
 for all but finitely many prime ideals $(\varPi)$ of height one of
 $\Lambda_0$ (generated by a prime element $\varPi$). Here we denote 
 by $\mathcal{A}[\varPi]$ the maximal $\varPi$-torsion submodule of $\mathcal{A}$. 
We refer to \cite[Lemma~4.4.1 and Remark~2.1.3]{gr-coh} for the proof of this property.

 Finally we introduce the hypothesis on the global-to-local morphism $\phi_{\mathcal{L}}$:
 \begin{description}
 \item[(SUR$_{\mathcal{A,L}}$)] the global-to-local morphism $\phi_{\mathcal{L}}$ is surjective.
 \end{description}

 \subsubsection{Greenberg's criterion} \label{sssc:criterion}

The following criterion for the almost divisibility of
$\mathrm{Sel}_\mathcal{L}(\mathsf{K}, \mathcal{A})$ is due to 
Greenberg \cite{gr-sel}. We state it for modules over the {\em semilocal} ring
$\mathcal{R}=\prod_{i\in I}\mathcal{R}_i$, contrary to the settings
in \cite{gr-sel}.

 \begin{thm}[{\cite[Proposition~4.1.1]{gr-sel}}] \label{thm:adiv}
Let $\mathcal{R}$ be a finite $\Lambda_0$-algebra which is isomorphic to the direct product of finitely many copies of $\Lambda_0$, and let $\mathsf{K}$, $S$, $\mathcal{A}$ be as above.
 Assume that the local condition $L(\mathsf{K},\mathcal{A}) (\subset P(\mathsf{K},\mathcal{A}))$ 
 is almost $\Lambda_0$-divisible. Suppose also that all the conditions 
 {\upshape (LOC$^{(1)}_{\mathcal{A},v_0}$)} $($for a certain nonarchimedean place $v_0$ in $S)$, 
 {\upshape (LOC$^{(2)}_{\mathcal{A},v}$)} $($for every place $v$ in $S)$, {\upshape (LEO$_{\mathcal{A}}$)} and 
 {\upshape (SUR$_{\mathcal{A,L}}$)} are fulfilled. 
 
 Then the $\mathcal{L}$-Selmer group $\mathrm{Sel}_{\mathcal{L}}(\mathsf{K},\mathcal{A})$ is almost divisible as an $\Lambda_0$-module.
 \end{thm}

\begin{rem} 
Among various assumptions of Theorem~\ref{thm:adiv}, the almost $\Lambda_0$-divisibility of $L(\mathsf{K},\mathcal{A})$ and the latter two hypotheses (LEO$_\mathcal{A}$), (SUR$_{\mathcal{A,L}}$) are 
 rather nontrivial and not so easy to verify. 
 In fact, the hypothesis 
 (LEO$_{\mathcal{A}}$) is closely related to 
 the {\em weak Leopoldt conjecture} in classical settings and is quite nontrivial 
 (see \cite[Introduction and
 Section~6.D]{gr-coh} for further discussion and for several examples 
 where the hypothesis (LEO$_{\mathcal{A}}$) is not valid). 
 The surjectivity condition (SUR$_{\mathcal{A},\mathcal{L}}$) of
 the global-to-local morphism $\phi_\mathcal{L}$ is closely related 
 to the triviality of the {\em dual Selmer group},
 as discussed later in Section~\ref{sssc:dSel}. 
 Finally, the local condition 
$L(\mathsf{K}, \mathcal{A})$ often tends to be {\em not} almost
 divisible; even in our CM setting, the unramified cohomology group 
 $H^1_\mathrm{ur}(F_{\mathfrak{P}^c}, \mathcal{A}^\mathrm{CM}_\eta)$
 at a place $\mathfrak{P}^c$ in $\Sigma_p^c$  might not 
 be almost divisible in general, and we cannot directly apply Greenberg's 
 criterion to $\mathrm{Sel}^\Sigma_{\mathcal{A}^\mathrm{CM}_\eta}$.
 This is one of the reasons why 
 we replace our Selmer group
 $\mathrm{Sel}^\Sigma_{\mathcal{A}^\mathrm{CM}_\eta}$ with 
 the {\em strict 
 Selmer group} $\mathrm{Sel}^{\Sigma,
 \mathrm{str}}_{\mathcal{A}^\mathrm{CM}_\eta}$ in Section~\ref{ssc:specialisation}.
\end{rem}

\subsubsection{Dual Selmer groups and the surjectivity hypothesis} \label{sssc:dSel}

 By virtue of Poitou and Tate's long exact sequence on Galois cohomology
 groups \cite[(8.6.10)]{NSW}, the cokernel of the global-to-local homomorphism 
 $\phi_{\mathcal{L}}$ is represented in terms of the {\em dual Selmer group} 
 $\mathrm{Sel}_{\mathcal{L}^*}(\mathsf{K}, \mathcal{T}^*)$ 
 of $\mathrm{Sel}_{\mathcal{L}}(\mathsf{K},\mathcal{A})$, which
 enables us to check the hypothesis (SUR$_{\mathcal{A,L}}$) by investigating the
 triviality of $\mathrm{Sel}_{\mathcal{L}^*}(\mathsf{K},
 \mathcal{T}^*)$. We here define the dual Selmer
 group $\mathrm{Sel}_{\mathcal{L}^*}(\mathsf{K}, \mathcal{T}^*)$ 
and introduce a criterion for its triviality, which is also due to
 Greenberg \cite{gr-surj}. In the following paragraphs the subscript ``cts'' denotes 
the Galois cohomology groups of {\em continuous} cocycles.

Let $\mathcal{T}^*=\mathrm{Hom}_{\mathrm{cts}}(\mathcal{A},
\mu_{p^\infty})$ 
denote the Kummer dual of $\mathcal{A}$. Then 
the natural pairing $\mathcal{A}\times \mathcal{T}^*\rightarrow
\mu_{p^\infty}$ combined with the cup product of the Galois cohomology 
induces the local Tate pairing
 \begin{align} \label{eq:tate}
 H^1(\mathsf{K}_v, \mathcal{A})\times H^1_{\mathrm{cts}}(\mathsf{K}_v, \mathcal{T}^*)\rightarrow
 \mathbb{Q}_p/\mathbb{Z}_p
\end{align}
for each $v$ in $S$, which is a perfect pairing as is well known. We 
 specify a subgroup $L^*(\mathsf{K}_v, \mathcal{T}^*)$ of
 $H^1_{\mathrm{cts}}(\mathsf{K}_v, \mathcal{T}^*)$ as the orthogonal complement 
 of $L(\mathsf{K}_v, \mathcal{A})$ under the local Tate pairing (\ref{eq:tate}).
 We denote 
 such specifications of submodules of the local Galois cohomology groups $H^1_\mathrm{cts}(\mathsf{K}_v,\mathcal{T}^*)$ by $\mathcal{L}^*$. The dual Selmer group $\mathrm{Sel}_{\mathcal{L}^*}(\mathsf{K}, \mathcal{T}^*)$
 for $\mathcal{T}^*$ is then defined as the kernel of 
the global-to-local homomorphism
\begin{align*}
\phi_{\mathcal{L}^*} \colon H^1_{\mathrm{cts}}(\mathsf{K}_S/\mathsf{K},
 \mathcal{T}^*) \rightarrow Q_{\mathcal{L}^*}(\mathsf{K}, \mathcal{T}^*),
\end{align*}
where $Q_{\mathcal{L}^*}(\mathsf{K}, \mathcal{T}^*)$ is defined as the direct
product $\prod_{v\in S}
 H^1_{\mathrm{cts}}(\mathsf{K}_v, \mathcal{T}^*)/L^*(\mathsf{K}_v, \mathcal{T}^*)$.
Meanwhile the {\em  fine $S$-Selmer group
 $\mathcyr{Sh}^1(\mathsf{K},S,\mathcal{T}^*)$ associated to
 $\mathcal{T}^*$} is defined as the local-to-global map
 $H^1_\mathrm{cts}(\mathsf{K}_S/\mathsf{K}, \mathcal{T}^*) \rightarrow \prod_{v\in
 S} H^1_\mathrm{cts}(\mathsf{K}_v, \mathcal{T}^*)$. Then one readily verifies 
that the Pontrjagin dual of the cokernel of $\phi_{\mathcal{L}}$ is 
 isomorphic to the quotient $\mathrm{Sel}_{\mathcal{L}^*}(\mathsf{K},
 \mathcal{T}^*)/\mathcyr{Sh}^1(\mathsf{K},S,\mathcal{T}^*)$, and the
 Pontrjagin dual of the cokernel of $\phi_{\mathcal{L}^*}$ is isomorphic to
 $\mathrm{Sel}_\mathcal{L}(\mathsf{K},
 \mathcal{A})/\mathcyr{Sh}^1(\mathsf{K}, S, \mathcal{A})$ \cite[Proposition~3.1.1]{gr-surj}. 
In particular, the triviality of the 
 dual Selmer group $\mathrm{Sel}_{\mathcal{L}^*}(\mathsf{K}, \mathcal{T}^*)$
 implies the validity of the hypothesis (SUR$_{\mathcal{A}, \mathcal{L}}$).

Greenberg himself has given a sufficient condition for the dual Selmer
 group $\mathrm{Sel}_{\mathcal{L}^*}(\mathsf{K}, \mathcal{T}^*)$ to vanish. 
 In order to state it, we here introduce another hypothesis:
 \begin{description}
 \item[(CRK$_{\mathcal{A,L}}$)] the following equality among $\mathcal{R}$-coranks holds (recall the definition and conventions on the $\mathcal{R}$-corank in Section~\ref{sssc:setting}):
 \begin{align*}
\mathrm{corank}_{\mathcal{R}}H^1(\mathsf{K}_S/\mathsf{K}, \mathcal{A})=\mathrm{corank}_{\mathcal{R}}\mathrm{Sel}_{\mathcal{L}}(\mathsf{K},\mathcal{A})+
 \mathrm{corank}_{\mathcal{R}}Q_{\mathcal{L}}(\mathsf{K},\mathcal{A}).
 \end{align*}
 \end{description}
Recall that we have the following equality on $\mathcal{R}$-coranks by the definition of the $\mathcal{L}$-Selmer group 
as the kernel of the global-to-local morphism (\ref{map:global-to-local}): 
$$ 
\mathrm{corank}_{\mathcal{R}} H^1(\mathsf{K}_S/\mathsf{K}, \mathcal{A})
 =\mathrm{corank}_{\mathcal{R}} \mathrm{Sel}_{\mathcal{L}}(\mathsf{K},\mathcal{A})+
 \mathrm{corank}_{\mathcal{R}}
 Q_{\mathcal{L}}(\mathsf{K},\mathcal{A})-\mathrm{corank}_{\mathcal{R}}
 \mathrm{Coker}(\phi_\mathcal{L}). 
$$ 
Hence, it is obvious 
that the following inequality always holds:
\begin{align*}
\mathrm{corank}_{\mathcal{R}} H^1(\mathsf{K}_S/\mathsf{K}, \mathcal{A})
\leq \mathrm{corank}_{\mathcal{R}} \mathrm{Sel}_{\mathcal{L}}(\mathsf{K},\mathcal{A})+
 \mathrm{corank}_{\mathcal{R}}
 Q_{\mathcal{L}}(\mathsf{K},\mathcal{A}).
\end{align*}
It is also obvious that 
the condition (CRK$_{\mathcal{A,L}}$) is valid if and only if 
{\em the cokernel of the global-to-local morphism $\phi_\mathcal{L}$ is 
cotorsion as an $\Lambda_0$-module}.

 \begin{pro}[{\cite[Proposition~3.2.1]{gr-surj}}]  \label{prop:surj}
 Suppose that $\mathcal{A}$ is a divisible $\mathcal{R}$-module and that the
  condition {\upshape (CRK$_{\mathcal{A,L}}$)} holds for $\mathcal{A}$
  and a specification $\mathcal{L}$. 
 Furthermore assume that at least one of the following conditions is fulfilled{\upshape :}
 \begin{enumerate}[label={\upshape (\alph*)}]
 \item for each maximal ideal $\mathfrak{M}$ of $\mathcal{R}$, the maximal $\mathfrak{M}$-torsion submodule  $\mathcal{A}[\mathfrak{M}]$
 has no subquotient isomorphic to $\mu_p$ 
 as a Galois representation of $\mathrm{Gal}(\overline{\mathsf{K}}/\mathsf{K})$ over $\mathbb{F}_p${\upshape ;}
 \item the discrete module $\mathcal{A}$ is cofree as an $\Lambda_0$-module and, for each maximal ideal $\mathfrak{M}$ of $\mathcal{R}$, the maximal $\mathfrak{M}$-torsion submodule  $\mathcal{A}[\mathfrak{M}]$ has 
 no quotient isomorphic to $\mu_p$ as a Galois representation of $\mathrm{Gal}(\overline{\mathsf{K}}/\mathsf{K})$ over $\mathbb{F}_p${\upshape ;}
 \item there exists a place $v_0$ contained in $S$ such that $H^0_\mathrm{cts}(\mathsf{K}_{v_0}, \mathcal{T}^*)$ is trivial and 
 that $Q_{\mathcal{L}}(\mathsf{K}_{v_0}, \mathcal{A})$ is divisible as an $\Lambda_0$-module.
 \end{enumerate}
 Then the dual Selmer group $\mathrm{Sel}_{\mathcal{L}^*}(\mathsf{K},
  \mathcal{T}^*)$ is trivial{\upshape ;} in particular, the surjectivity condition
  {\em (SUR$_{\mathcal{A,L}}$)} holds for $\mathcal{A}$ and $\mathcal{L}$.
 \end{pro}

%
 \subsection{Inductive specialisation of the characteristic ideals} \label{ssc:specialisation}
%

 In the rest of this section we shall state the main result of
 the algebraic side of this article (Theorem~\ref{thm:specialisation}) and prove it. 
As in Section~\ref{sssc:cycSel}, we denote by $F^+$ a totally real number
field of degree $d$ which satisfies the condition (unr$_{F^+}$). 
Let $f$ be a $p$-ordinary
$p$-stabilised newform of cohomological weight $\kappa$, level
$\mathfrak{N}$ and nebentypus $\underline{\varepsilon}$ defined on
$F^+$, and suppose that $f$ has complex multiplication. 
We denote by $\eta$ the gr\"o{\ss}encharacter of type $(A_0)$ defined on 
a totally imaginary quadratic extension $F$ over $F^+$ satisfying the
ordinarity condition (ord$_{F/F^+}$), to which the cuspform $f=\vartheta(\eta)$ 
is associated. As we have already mentioned, it is always possible to 
assume that $\eta$ is ordinary with respect to an 
appropriate $p$-ordinary CM type $\Sigma$ of $F$.
We choose a branch character $\psi$ associated to $\eta$ and fix it
(see Definition~\ref{def:branch}). 

Recall the $\psi$-branch $\mathcal{L}^\Sigma_p(\psi)$ of Katz, Hida and
Tilouine's $p$-adic $L$-function
$\mathcal{L}_{p,\Sigma}^\mathrm{KHT}(F)$ defined in Introduction; namely
$\mathcal{L}_p^\Sigma(\psi)$ is the image of
$\mathcal{L}_{p,\Sigma}^\mathrm{KHT}(F)$ under the $\psi$-twisting map 
\begin{align*}
\widehat{\mathcal{O}}^\mathrm{ur}[[\mathrm{Gal}(F_{\mathfrak{C}p^\infty}/F)]]\rightarrow
 \widehat{\mathcal{O}}^\mathrm{ur}[[\mathrm{Gal}(\widetilde{F}_\infty/F)]]\
 ; \ g\mapsto \psi(g) g\vert_{\widetilde{F}_\infty}.
\end{align*}
\begin{thm} \label{thm:specialisation}
Let the notation be as above. Furthermore assume the following three conditions{\upshape ;}
\begin{itemize}
\item[-] the nontriviality condition {\upshape \bfseries (ntr)$_{\mathfrak{P}}$} for
	 every place $\mathfrak{P}$ of $F$ contained in  $\Sigma_p^c${\upshape ;}
\item[-] {\upshape \bfseries (IMC$_{F,\psi}$)} \ the $((d+1+\delta_{F,p})$-variable$)$ 
Ottawa's main conjecture
\begin{align*}
\mathrm{Char}_{\Lambda^\mathrm{CM}_\mathcal{O}} (X_{\Sigma_p, (\psi)})
 =(\mathcal{L}_p^\Sigma(\psi))
\end{align*}
holds as an equality of ideals in $\Lambda^\mathrm{CM}_\mathcal{O}\hat{\otimes}_{\mathcal{O}} \, \widehat{\mathcal{O}}^\mathrm{ur}$ for the CM number field $F$ and the branch character
	 $\psi${\upshape ;}
\item[-] {\upshape \bfseries
	 (NV$_{\mathcal{L}^\mathrm{cyc}_p(\vartheta(\eta))}$)} \ the
	 cyclotomic $p$-adic $L$-function
	 $\mathcal{L}^\mathrm{cyc}_p(\vartheta(\eta))$
	 {\em does not vanish} in the sense that each component of
	 $\mathcal{L}^\mathrm{cyc}_p(\vartheta(\eta))$ in the
	 indecomposable decomposition of $\Lambda^\mathrm{cyc}_\mathcal{O}\hat{\otimes}_\mathcal{O} \widehat{\mathcal{O}}^\mathrm{ur}$ does not equal zero.
\end{itemize}
Then we have the following equality of ideals of $\Lambda^\mathrm{cyc}_{\mathcal{O}}$
\begin{align} \label{eq:specialisation}
 (\mathrm{Char}_{\Lambda_\mathcal{O}^\mathrm{CM}}(\mathrm{Sel}^\Sigma_{\mathcal{A}^\mathrm{CM}_\eta})^\vee)
  \otimes_{\Lambda_\mathcal{O}^\mathrm{CM}}
  \Lambda_\mathcal{O}^\mathrm{cyc} =\mathrm{Char}_{\Lambda^\mathrm{cyc}_{\mathcal{O}}}
  (\mathrm{Sel}_{\mathcal{A}_{\vartheta(\eta)}^\mathrm{cyc}})^\vee 
\end{align}
 where the tensor product in the left hand side is taken with respect to the canonical quotient map $\Lambda_\mathcal{O}^\mathrm{CM}\twoheadrightarrow \Lambda_\mathcal{O}^\mathrm{cyc}$.
\end{thm}

The rest of this section is devoted to the proof of
Theorem~\ref{thm:specialisation}. 
We shall verify Theorem~\ref{thm:specialisation} by 
 induction on the Krull dimension of the coefficient ring $\Lambda_\mathcal{O}^{\mathrm{cyc}}$, 
applying repeatedly the {\em specialisation lemma} introduced below in 
 Section~\ref{sssc:sp_lemma} (Lemma~\ref{lem:specialisation}), Greenberg's 
 criterion for almost divisibility (Theorem~\ref{thm:adiv}) and the
 exact control theorem (Theorem~\ref{thm:control}).

\begin{rem}
 The conditions (IMC$_{F,\psi}$) and 
 (NV$_{\mathcal{L}^\mathrm{cyc}_p(\vartheta(\eta))}$) imply an important algebraic property of the Selmer group:  
\begin{itemize}
\item[-] {\bfseries
      (COT$_{\mathcal{A}^\mathrm{cyc}_{\vartheta(\eta)}}$}) \ the
      Selmer group $\mathrm{Sel}_{\mathcal{A}^\mathrm{cyc}_{\vartheta(\eta)}}$ is
      a cotorsion $\Lambda^\mathrm{cyc}_\mathcal{O}$-module. 
\end{itemize}

In order to deduce the conclusion of Theorem \ref{thm:specialisation}, we may replace the 
analytic condition (NV$_{\mathcal{L}^\mathrm{cyc}_p(\vartheta(\eta))}$) 
by the algebraic condition (COT$_{\mathcal{A}^\mathrm{cyc}_{\vartheta(\eta)}}$). 
We also remark that we use the analytic condition (NV$_{\mathcal{L}^\mathrm{cyc}_p(\vartheta(\eta))}$) 
only at the final step of our inductive argument (see Section~\ref{sssc:final_step}).
\end{rem}

\subsubsection{The specialisation lemma} \label{sssc:sp_lemma}

 Firstly, we recall the following elementary lemma (which we shall refer as the {\em specialisation lemma} later), describing  the behaviour of characteristic ideals 
 under specialisation procedures. 

  \begin{lem}\label{lem:specialisation}
   Let $\mathcal{R}$ be a finite $\Lambda_0$-algebra which is isomorphic to the direct product of finitely many copies of $\Lambda_0$, and let $M$ be an $\mathcal{R}$-module which is finitely generated and torsion. Assume that $M$ contains no nontrivial pseudonull $\Lambda_0$-submodules. 
 Let $(\varPi)$ denote a prime ideal of height one of $\Lambda_0$
  and assume that $(\varPi)$ does not divide the 
 characteristic ideal $\mathrm{Char}_{\mathcal{R}}(M)$ of $M$. 
 Then the quotient module $M/\varPi M$ is finitely generated and 
  torsion as an $\mathcal{R}/\varPi \mathcal{R}$-module, and the basechange $\mathrm{Char}_{\mathcal{R}}(M)\otimes_{\mathcal{R}} \mathcal{R}/\varPi\mathcal{R}$ of the 
 characteristic ideal $\mathrm{Char}_{\mathcal{R}}(M)$ coincides with 
 the characteristic ideal $\mathrm{Char}_{\mathcal{R}/\varPi \mathcal{R}}(M/\varPi M)$ of 
 the $\mathcal{R}/\varPi \mathcal{R}$-module $M/\varPi M$ as an ideal of the quotient ring $\mathcal{R}/\varPi \mathcal{R}$.
 \end{lem}
 
Here we remark that (commutative) noetherian regular local rings are
unique factorisation domains, and hence every height-one prime ideal of $\Lambda_0$  is a principal ideal.

We readily verify the lemma above in essentially the same way as the proof of
\cite[Lemma~3.1]{och_euler}. We thus omit the proof here, just emphasising
  that the triviality of the pseudonull $\Lambda_0$-submodule of $M$ plays
  a crucial role in the verification of Lemma~\ref{lem:specialisation}.

\subsubsection{Settings on regular sequences} \label{sssc:regular}

We now apply the results of Sections~\ref{ssc:greenberg} and \ref{ssc:specialisation} to the case where $\mathcal{R}$ is the {\em semilocal} Iwasawa algebra $\Lambda_\mathcal{O}^\mathrm{CM}=\mathcal{O}[[\mathrm{Gal}(\widetilde{F}_\infty/F)]]$ and $\Lambda_0$ is the {\em local} Iwasawa algebra $\mathcal{O}[[\mathrm{Gal}(\widetilde{F}/F)]]$, which is isomorphic to the ring of formal power series in $(d+1+\delta_{F,p})$-variables. Here we henceforth choose and fix a splitting of the Group extension
\begin{align*}
 \xymatrix{
1  \ar[r] & \mathrm{Gal}(\widetilde{F}_\infty/\widetilde{F})\cong \mathrm{Gal}(F(\mu_p)/F) \ar[r] & \mathrm{Gal}(\widetilde{F}_\infty/F) \ar[r] & \mathrm{Gal}(\widetilde{F}/F) \ar[r] & 1}
\end{align*}
and regard $\mathcal{O}[[\mathrm{Gal}(\widetilde{F}/F)]]$ as a subring of $\Lambda^\mathrm{CM}_\mathcal{O}$ by using this splitting. This identification endows the semilocal algebra $\Lambda_\mathcal{O}^\mathrm{CM}$ with the $\Lambda_0$-module structure.
In the following arguments, we inductively find  
elements $ \gamma_1, \dotsc, \gamma_{d+\delta_{F,p}} $ 
of $\mathrm{Gal}(\widetilde{F}/F)$ so that 
$\gamma_1-1, \dotsc, \gamma_{d+\delta_{F,p}}-1$ is 
a regular sequence of $\Lambda_0=\mathcal{O}[[\mathrm{Gal}(\widetilde{F}/F)]]$ 
contained in $\mathfrak{A}^{\mathrm{cyc}}$ and they satisfy 
certain ``nice'' properties. We here prepare 
notation on regular sequences. 
Let $j$ be a natural number with $1\leq j\leq d+\delta_{F,p}$ 
and suppose that we have already chosen elements 
$\gamma_1, \dotsc, \gamma_j$ of $\mathrm{Gal}(\widetilde{F}/F)$
such that $\gamma_1-1, \dotsc, \gamma_j-1$ is a regular sequence 
of $\Lambda_0$ contained in
$\mathfrak{A}^{\mathrm{cyc}}$. 
We set $x_k=\gamma_k-1$ for each $k$ with $1\leq k \leq j$ and 
let $\mathfrak{A}_j$ denote the ideal of
$\Lambda_\mathcal{O}^\mathrm{CM}$ generated by $x_1, \dotsc, x_j$ with 
the convention that $\mathfrak{A}_0$ denotes the zero ideal in
$\Lambda^{\mathrm{CM}}_\mathcal{O}$. 
Then the notation introduced here is compatible with that introduced
in Section~\ref{ssc:control}. We also define $\Lambda_0^{(j)}$ to be
the quotient ring $\Lambda_0/(x_1,x_2,\dotsc, x_j)$, which is a regular local ring isomorphic to the ring of formal power series in $(d+1+\delta_{F,p}-j)$-variables.

 Now we introduce the local condition $\mathcal{L}_{\mathrm{str}}$ 
 corresponding to the strict Selmer group; namely we set for each
 nonarchimedean place

$v$ in $S$
 \begin{align*}
 L_{\mathrm{str}}(F_v, \mathcal{A}^{\mathrm{CM}}_\eta[\mathfrak{A}_{j}]) = \begin{cases} 
 H^1_{\mathrm{ur}}(F_v, \mathcal{A}^{\mathrm{CM}}_\eta[\mathfrak{A}_{j}]) & \text{for $v$ in $S\setminus (\Sigma_p\cup \Sigma_p^c)$}, \\
 H^1(F_v, \mathcal{A}^{\mathrm{CM}}_\eta[\mathfrak{A}_{j}]) & \text{for $v \in \Sigma_p$}, \\
 0 & \text{for $v \in \Sigma_p^c$}, 
 \end{cases}
 \end{align*}
 where $H^1_{\mathrm{ur}}(F_\lambda, \mathcal{A}^{\mathrm{CM}}_\eta[\mathfrak{A}_{j}])$
 denotes {\em the unramified cohomology group} which is defined to be  
 $H^1(D_\lambda/I_\lambda ,\mathcal{A}^{\mathrm{CM}}_\eta[\mathfrak{A}_{j}]^{I_\lambda})$. 

We also introduce the local condition $L_\mathrm{str}(F, \mathcal{A}_\psi^{\mathrm{CM}})$ for the discrete $\Lambda^\mathrm{CM}_\mathcal{O}$-module 
$\mathcal{A}_\psi^{\mathrm{CM}}$ in the same manner.

The discrete $\Lambda^{\mathrm{CM}}_\mathcal{O}/\mathfrak{A}_{j}$-module
$\mathcal{A}^{\mathrm{CM}}_\eta[\mathfrak{A}_{j}]$ corresponds to
the continuous Galois representation $\mathcal{T}^{\mathrm{CM}}_\eta/\mathfrak{A}_{j}\mathcal{T}^{\mathrm{CM}}_\eta$;
that is, the Pontrjagin duality induces an isomorphism of $\Lambda^{\mathrm{CM}}_{\mathcal{O}}/\mathfrak{A}_{j}$-linear
$\mathrm{Gal}(F_S/F)$-representations
\begin{align*} 
\mathcal{A}^{\mathrm{CM}}_\eta[\mathfrak{A}_{j}] \cong 
(\mathcal{T}^{\mathrm{CM}}_\eta/\mathfrak{A}_{j}\mathcal{T}^{\mathrm{CM}}_\eta)
\otimes_{\Lambda^{\mathrm{CM}}_\mathcal{O}/\mathfrak{A}_{j}}
(\Lambda^{\mathrm{CM}}_{\mathcal{O}}/\mathfrak{A}_{j})^\vee.
\end{align*} 
We thus denote the Kummer dual of
$\mathcal{A}^{\mathrm{CM}}_\eta[\mathfrak{A}_{j}]$ by
$(\mathcal{T}^{\mathrm{CM}}_\eta/\mathfrak{A}_{j}\mathcal{T}^{\mathrm{CM}}_\eta)^*$, following Greenberg's notation introduced in
Section~\ref{sssc:dSel}.
Then the local condition $\mathcal{L}^*_\mathrm{str}$
for the dual strict Selmer group
$\mathrm{Sel}_{\mathcal{L}^*_\mathrm{str}}(F,
(\mathcal{T}^{\mathrm{CM}}_\eta/\mathfrak{A}_{j}\mathcal{T}^{\mathrm{CM}}_\eta)^*)$
is calculated as
 \begin{multline} \label{eq:dSel_eta}
 L^*_{\mathrm{str}}(F_v,
  (\mathcal{T}^{\mathrm{CM}}_\eta/\mathfrak{A}_{j}\mathcal{T}^{\mathrm{CM}}_\eta)^*)
 = \begin{cases} 
 H^1_\mathrm{ur}(F_v , (\mathcal{T}^{\mathrm{CM}}_\eta/\mathfrak{A}_{j}\mathcal{T}^{\mathrm{CM}}_\eta)^*) 
 & \text{for $v  \in S\setminus (\Sigma_p\cup \Sigma_p^c)$}, \\
0 & \text{for $v\in \Sigma_p$}, \\
  H^1_\mathrm{cts}(F_v, (\mathcal{T}^{\mathrm{CM}}_\eta/\mathfrak{A}_{j}\mathcal{T}^{\mathrm{CM}}_\eta)^*) & \text{for $v\in \Sigma_p^c$},
 \end{cases}
 \end{multline}
where the unramified cohomology $H^1_\mathrm{ur}(F_\lambda,
(\mathcal{T}^{\mathrm{CM}}_\eta/\mathfrak{A}_{j}\mathcal{T}^{\mathrm{CM}}_\eta)^*)$
is defined in the usual manner as $H^1_\mathrm{cts}(D_\lambda/I_\lambda,
(\mathcal{T}^{\mathrm{CM}}_\eta/\mathfrak{A}_{j}\mathcal{T}^{\mathrm{CM}}_\eta)^{*,
I_\lambda})$. Indeed (\ref{eq:dSel_eta}) follows for places of $F$
lying above $p$ directly from the definition of the local condition of the dual
Selmer groups based upon the local Tate duality (\ref{eq:tate}). 
Now we temporary abbreviate
$\mathcal{A}^{\mathrm{CM}}_\eta[\mathfrak{A}_{j}]$
as $\mathcal{A}$ and 
$(\mathcal{T}^{\mathrm{CM}}_\eta/\mathfrak{A}_{j}\mathcal{T}^{\mathrm{CM}}_\eta)^*$
as $\mathcal{T}^*$ for brevity. 
Then for a place of $F$ contained in $S$ but not lying above $p$, it is well
known that the unramified cohomology groups are the orthogonal
complements of each other under the local Tate pairing (\ref{eq:tate}) 
for {\em finite} Galois modules $\mathcal{A}[\mathfrak{J}^n]$ and
$\mathcal{T}^*/\mathfrak{J}^n\mathcal{T}^*$. Here $\mathfrak{J}$ 
is the Jacobian radical of 
$\Lambda^{\mathrm{CM}}_{\mathcal{O}}$ and $n$ is an arbitrary
natural number. Since $\mathcal{T}^*$ is obviously
complete with respect to the $\mathfrak{J}$-adic topology, 
we readily obtain (\ref{eq:dSel_eta})
for such a place by employing standard limit arguments based upon Tate's
theorem (refer to \cite[Corollary~2.3.5]{NSW} and
\cite[Remark~3.5.1]{gr-coh} for example).

\medskip
 \subsubsection{Preliminary step$:$ verification of several hypotheses} \label{sssc:pre_step}

 As the preliminary step for our inductive arguments, we verify the local hypotheses
 (LOC$^{(1)}_{\mathcal{A}^{\mathrm{CM}}_\eta[\mathfrak{A}_{j}],v}$),
 (LOC$^{(2)}_{\mathcal{A}^{\mathrm{CM}}_\eta[\mathfrak{A}_{j}],v}$) and the
 almost $\Lambda_0^{(j)}$-divisibility of 
 $L_{\mathrm{str}}(F, \mathcal{A}^{\mathrm{CM}}_\eta[\mathfrak{A}_{j}])$ for every $j$.   
 Firstly, we  readily observe that the local hypothesis
 (LOC$^{(1)}_{\mathcal{A}^{\mathrm{CM}}_\eta[\mathfrak{A}_{j}], v}$) is valid 
 for {\em every} nonarchimedean place $v$ in $S$ (and hence the local
 hypothesis (LOC$^{(2)}_{\mathcal{A}^{\mathrm{CM}}_\eta[\mathfrak{A}_{j}], v}$) is
 automatically satisfied for every place $v$ in $S$). In fact 
 the Kummer dual of $\mathcal{A}^{\mathrm{CM}}_\eta[\mathfrak{A}_{j}]$
 is a free
 $\Lambda^{\mathrm{CM}}_\mathcal{O}/\mathfrak{A}_{j}$-module of rank one 
 on which every element $g$ of $\mathrm{Gal}(F_S/F)$ acts by the
 multiplication of 
 $\chi_{p,\mathrm{cyc}}^{-1}\eta^\mathrm{gal}(g^{-1})g^{-1}\vert_{\widetilde{F}_\infty}$. 
 Now let $v$ be a nonarchimedean place $v$ in $S$. 
 Since every nonarchimedean place of $F$ does not split completely in
 the cyclotomic $\mathbb{Z}_p$-extension $F^{\mathrm{cyc}}_\infty$ of $F$, 
 the image of the decomposition subgroup $D_v$ 
in $\mathrm{Gal}(F^\mathrm{cyc}_\infty/F)$ is not trivial. 
In particular there exists an element $g_0$ of $D_v$ such that the image
 of $\xi_{g_0}:=\chi_{p,\mathrm{cyc}}^{-1}\eta^\mathrm{gal}(g_0^{-1})g_0^{-1}-1$
 in
 $\Lambda_\mathcal{O}^{\mathrm{CM}}/\mathfrak{A}_{j}$ does not vanish
 (because it is nonzero in the quotient $\Lambda_\mathcal{O}^{\mathrm{cyc}}$ of
 $\Lambda_\mathcal{O}^{\mathrm{CM}}/\mathfrak{A}_{j}$). The definition of $\xi_{g_0}$ implies that, under the indecomposable decomposition of $\Lambda_\mathcal{O}^\mathrm{CM}$, no components of $\xi_0$ equal zero.
 Since every indecomposable component of $\Lambda^{\mathrm{CM}}_\mathcal{O}/\mathfrak{A}_{j}$ is a domain, we readily see that the $\xi_0$-torsion submodule of the free $\Lambda^\mathrm{CM}_\mathcal{O}$-module $\mathcal{A}_\eta^\mathrm{CM}[\mathfrak{A}_j]$ is trivial. This implies
 that the $D_v$-invariant of the Kummer dual of 
 $\mathcal{A}^{\mathrm{CM}}_\eta[\mathfrak{A}_{j}]$ equals zero,
 and thus  (LOC$^{(1)}_{\mathcal{A}^{\mathrm{CM}}_\eta[\mathfrak{A}_{j}],v}$) holds for every $j$ and for every nonarchimedean $v$ in $S$.

 We now verify the almost divisibility of the local condition $L_\mathrm{str}(F_v,\mathcal{A}^{\mathrm{CM}}_\eta[\mathfrak{A}_{j}])$. 
 There is nothing to prove for a place in $\Sigma^c_p$ since the local condition is trivial at such a place. For a place $\mathfrak{P}$ in $\Sigma_p$, 
 the local condition $L_\mathrm{str}(F_\mathfrak{P},\mathcal{A}^{\mathrm{CM}}_\eta[\mathfrak{A}_{j}])$ coincides with the whole cohomology group
$H^1(F_{\mathfrak{P}},\mathcal{A}^{\mathrm{CM}}_\eta[\mathfrak{A}_{j}])$, and its almost $\Lambda_0^{(j)}$-divisibility follows from the
 hypothesis (LOC$^{(2)}_{\mathcal{A}^{\mathrm{CM}}_\eta[\mathfrak{A}_{j}], \mathfrak{P}}$) and \cite[Proposition~5.4]{gr-coh}. 
Finally take a place $\lambda$ from $S\setminus (\Sigma_p\cup \Sigma_p^c)$. By the same argument as the proof of Theorem~\ref{thm:control}, 
 the image of the inertia subgroup $I_\lambda$ under the
 Galois character $\eta^\mathrm{gal}$
 is finite and contained in the set of roots of unity in $\mathcal{O}$. 
 Thus $\mathcal{A}^{\mathrm{CM}}_\eta[\mathfrak{A}_{j}]^{I_\lambda}$ is isomorphic to 
 $\mathcal{A}^{\mathrm{CM}}_\eta[\mathfrak{A}_{j}, \varpi^n\Lambda_\mathcal{O}^{\mathrm{CM}}]$ 
 for some nonnegative integer $n$, which is almost $\Lambda_0^{(j)}$-divisible (observe that $\mathcal{A}^\mathrm{CM}_\eta[\mathfrak{A}_j,\varpi^n \Lambda^\mathrm{CM}_\mathcal{O}]$ is divisible for every height-one prime ideal of $\Lambda_j^{(0)}$ relatively prime to $\varpi\Lambda_0^{(j)}$). 
 Since $H^1_\mathrm{ur}(F_\lambda, \mathcal{A}^{\mathrm{CM}}_\eta[\mathfrak{A}_{j}])$
 is isomorphic to a quotient of $\mathcal{A}^{\mathrm{CM}}_\eta[\mathfrak{A}_{j}]^{I_\lambda}$, 
  $H^1_\mathrm{ur}(F_\lambda, \mathcal{A}^{\mathrm{CM}}_\eta[\mathfrak{A}_{j}])$ must also be 
 almost $\Lambda$-divisible.  
 
\begin{rem}\label{rem:hyp_psi}
It is also possible to  verify the hypotheses
 (LOC$^{(1)}_{\mathcal{A}_\psi^{\mathrm{CM}},v}$) and
 (LOC$^{(2)}_{\mathcal{A}_\psi^{\mathrm{CM}},v}$) for each nonarchimedean
 place contained in $S$ and the almost $\Lambda_0$-divisibility
 of the local condition $L_\mathrm{str}(F, \mathcal{A}_\psi^{\mathrm{CM}})$
 by exactly the same arguments as above.
\end{rem}

\medskip
 \subsubsection{First step$:$ almost divisibility of the initial strict
   Selmer group}  \label{sssc:first_step}

We shall verify the almost
$\Lambda_0$-divisibility  
of the initial strict Selmer group 
$\mathrm{Sel}^{\Sigma,\mathrm{str}}_{\mathcal{A}^{\mathrm{CM}}_\eta}$. 
Identifying  $\mathrm{Sel}^{\Sigma,\mathrm{str}}_{\mathcal{A}^{\mathrm{CM}}_\eta}$ with $\mathrm{Tw}_{\eta^{\mathrm{gal},-1}\psi}(\mathrm{Sel}^{\Sigma,\mathrm{str}}_{\mathcal{A}_\psi^{\mathrm{CM}}})$ by exactly the same arguments as  in Remark~\ref{rem:sel_psi}, we readily see that the almost $\Lambda_0$-divisibility
of $\mathrm{Sel}^{\Sigma,\mathrm{str}}_{\mathcal{A}^{\mathrm{CM}}_\eta}$ is
equivalent to that of
$\mathrm{Sel}^{\Sigma,\mathrm{str}}_{\mathcal{A}_\psi^{\mathrm{CM}}}$.
Concerning the verification of the almost $\Lambda_0$-divisibility of $\mathrm{Sel}^{\Sigma,\mathrm{str}}_{\mathcal{A}_\psi^\mathrm{CM}}$, we first observe that an arbitrary place $\mathfrak{P}$ in
$\Sigma_p$ satisfies the extra condition (c) proposed in Proposition~\ref{prop:surj}. It is obvious because $Q_{\mathcal{L}_\mathrm{str}}(F_\mathfrak{P},
\mathcal{A}^{\mathrm{CM}}_\psi)$ is trivial for each $\mathfrak{P}$ in $\Sigma_p$. 
Thus, in order to deduce the almost
$\Lambda_0$\nobreakdash-divisibility of
$\mathrm{Sel}^{\Sigma,\mathrm{str}}_{\mathcal{A}^{\mathrm{CM}}_\eta}$ 
from Theorem~\ref{thm:adiv} and Proposition~\ref{prop:surj},
it suffices to verify the remaining two hypotheses (CRK$_{\mathcal{A}_\psi^{\mathrm{CM}},
\mathcal{L}_{\mathrm{str}}}$) and (LEO$_{\mathcal{A}_\psi^{\mathrm{CM}}}$); recall that we have already verified in the preceding step (Section~\ref{sssc:pre_step}) all the other hypotheses required in Theorem~\ref{thm:adiv}.

We have observed in Remark~\ref{rem:sel_psi} and Lemma~\ref{lem:replace_psi} that the Iwasawa module $X_{\Sigma_p,
(\psi)}$ and the Pontrjagin dual of the
strict Selmer group
$\mathrm{Sel}^{\Sigma,\mathrm{str}}_{\mathcal{A}_\psi}$ are
pseudoisomorphic to each other as $\Lambda_\mathcal{O}^{\mathrm{CM}}$-modules.

Now we introduce the following hypothesis concerning the algebraic structure of 
the Iwasawa module $X_{\Sigma_p, (\psi)}$:
\begin{description} 
 \item[(TOR$_{X_{\Sigma_p, (\psi)}}$)] the
	    $\Lambda_\mathcal{O}^{\mathrm{CM}}$-module $X_{\Sigma_p, (\psi)}$ is a torsion module.
\end{description}

  Note that the condition (TOR$_{X_{\Sigma_p, (\psi)}}$) is 
  equivalent to the following condition on the algebraic structure of the strict
  Selmer group $\mathrm{Sel}^{\Sigma,\mathrm{str}}_{\mathcal{A}_\psi^\mathrm{CM}}$ by the discussion above: 
 \begin{description}
\item[(COT$_{\mathcal{A}_\psi^{\mathrm{CM}}}$)]  the strict Selmer group
$\mathrm{Sel}^{\Sigma, \mathrm{str}}_{\mathcal{A}_\psi^{\mathrm{CM}}}$ is cotorsion as a $\Lambda_\mathcal{O}^{\mathrm{CM}}$-module.
 \end{description}

\begin{lem} \label{lem:tors}
The following two statements are equivalent for the discrete
 $\Lambda^{\mathrm{CM}}_{\mathcal{O}}$-module $\mathcal{A}^{\mathrm{CM}}_\psi${\upshape ;}
\begin{enumerate}[label=$(\arabic*)$]
\item both the conditions {\em (CRK$_{\mathcal{A}_\psi^{\mathrm{CM}}, \mathcal{L}_{\mathrm{str}}}$)} and 
{\em (LEO$_{\mathcal{A}_\psi^{\mathrm{CM}}}$)} hold{\upshape ;}
 \item the condition {\em (COT$_{\mathcal{A}_\psi^{\mathrm{CM}}}$)} $($or, equivalently, the condition {\em (TOR$_{X_{\Sigma_p,
       (\psi)}}$)}$)$ holds.
\end{enumerate}
\end{lem}

\begin{proof}
 We abbreviate the
 $\Lambda^{\mathrm{CM}}_{\mathcal{O}}$-corank of a cofinitely generated
 $\Lambda^{\mathrm{CM}}_{\mathcal{O}}$-module $M$
 just as ``$\mathrm{corank}~M$'' in the proof of Lemma~\ref{lem:tors} so
 as to simplify the notation.
We shall henceforth 
 verify that the following equation among $\Lambda_\mathcal{O}^{\mathrm{CM}}$-coranks holds:
 \begin{align} \label{eq:key_corank_eq}
 \mathrm{corank}~
  \mathrm{Sel}^{\Sigma, \mathrm{str}}_{\mathcal{A}_\psi^{\mathrm{CM}}}=\mathrm{corank}~ \mathcyr{Sh}^2(F, S, \mathcal{A}_\psi^{\mathrm{CM}})+\mathrm{corank}~ \mathrm{Coker}(\phi_{\mathcal{L}_\mathrm{str}}).
 \end{align}
Once the equation (\ref{eq:key_corank_eq}) is proved, 
the equivalence between the assertions (1) and (2) 
immediately follows; for the assertion (1) is 
fulfilled if and only if the right hand side of the equation
 (\ref{eq:key_corank_eq}) equals zero, whereas the assertion (2) is
 fulfilled if and only if the left hand side of (\ref{eq:key_corank_eq})
 equals zero (recall that the validity of the corank condition
 (CRK$_{\mathcal{A}^{\mathrm{CM}}_\psi,\mathcal{L}_\mathrm{str}}$)
 is equivalent to the $\Lambda^{\mathrm{CM}}_{\mathcal{O}}$-cotorsionness of
 the cokernel of $\phi_{\mathcal{L}_\mathrm{str}}$). 
 
In order to deduce 
 the equation (\ref{eq:key_corank_eq}), we consider the equation
\begin{equation} \label{eq:ingr_corank}
\begin{aligned} 
\mathrm{corank}~ \mathrm{Sel}^{\Sigma,
 \mathrm{str}}_{\mathcal{A}^{\mathrm{CM}}_\psi} 
 =\mathrm{corank}~& H^1(F_S/F, \mathcal{A}^{\mathrm{CM}}_\psi) \\
&-
 \mathrm{corank}~
 Q_{\mathcal{L}_\mathrm{str}}(F,\mathcal{A}^{\mathrm{CM}}_\psi)+
\mathrm{corank}~ \mathrm{Coker}(\phi_{\mathcal{L}_\mathrm{str}}),
\end{aligned}
\end{equation}
which is deduced from the exact sequence
{\small 
\begin{align*}
\xymatrix{
0 \ar[r] & \mathrm{Sel}^{\Sigma,
 \mathrm{str}}_{\mathcal{A}^{\mathrm{CM}}_\psi} \ar[r] &
 H^1(F_S/F, \mathcal{A}^{\mathrm{CM}}_\psi) \ar[r]^{\phi_{\mathcal{L}_\mathrm{str}}} &
 Q_{\mathcal{L}_\mathrm{str}}(F,
 \mathcal{A}^{\mathrm{CM}}_\psi) \ar[r] &
 \mathrm{Coker}(\phi_{\mathcal{L}_\mathrm{str}}) \ar[r] & 0
}
\end{align*}
}defining the strict Selmer group $\mathrm{Sel}^{\Sigma,
 \mathrm{str}}_{\mathcal{A}^{\mathrm{CM}}_\psi}$, 
and calculate the right hand side of (\ref{eq:ingr_corank}) 
by utilising Euler-Poincar\'e characteristic formulae studied in \cite{gr-coh}.

We first calculate the $\Lambda^{\mathrm{CM}}_{\mathcal{O}}$-corank of the
 global cohomology group $H^1(F_S/F,
 \mathcal{A}^{\mathrm{CM}}_\psi)$. 
By the global Euler-Poincar\'e characteristic formula \cite[Proposition~4.1]{gr-coh}, we have  
 the equality:
 \begin{multline}\label{equation:semiglobal_epc}
\mathrm{corank}~  H^1(F_S/F, \mathcal{A}_\psi^{\mathrm{CM}})
\\ 
=  \mathrm{corank}~H^0(F_S/F, \mathcal{A}_\psi^{\mathrm{CM}})+
\mathrm{corank}~H^2(F_S/F,
 \mathcal{A}_\psi^{\mathrm{CM}})+d\, \mathrm{corank}~(\mathcal{A}_\psi^{\mathrm{CM}}).
\end{multline}
Recall that $d$ denotes the extension degree $[F^+:\mathbb{Q}]$ 
of the totally real field $F^+$ over $\mathbb{Q}$. 
 We claim that $H^0(F_S/F, \mathcal{A}^{\mathrm{CM}}_\psi)$ is
 a cotorsion $\Lambda^{\mathrm{CM}}_{\mathcal{O}}$-module.       
Indeed $\mathcal{A}^{\mathrm{CM}}_{\psi}$ is naturally 
identified with the Pontrjagin dual of
 $\Lambda^{\mathrm{CM}}_\mathcal{O}$, on which an element $g$
 of $\mathrm{Gal}(F_S/F)$ acts by the multiplication of $\psi(g)g\vert_{\widetilde{F}_\infty}$. Let us take an arbitrary 
element $g_0$ of $\mathrm{Gal}(F_S/F)$ whose image in
 $\mathrm{Gal}(\widetilde{F}/F)$ under the natural surjection $\mathrm{Gal}(F_S/F)\twoheadrightarrow
 \mathrm{Gal}(\widetilde{F}/F)$ is nontrivial. Then we readily see that 
the $g_0$-invariant of $\mathcal{A}^{\mathrm{CM}}_\psi$ is
 isomorphic to a $\Lambda^{\mathrm{CM}}_\mathcal{O}$-module 
$(\Lambda^{\mathrm{CM}}_\mathcal{O}/(\psi(g_0)g_0\vert_{\widetilde{F}_\infty}-1) \Lambda^{\mathrm{CM}}_\mathcal{O})^\vee$, which is obviously a cotorsion
 $\Lambda^{\mathrm{CM}}_{\mathcal{O}}$-module. The zeroth cohomology $H^0(F_S/F, \mathcal{A}^\mathrm{CM}_\psi)$ is obviously a subgroup of the $g_0$-invariant of $\mathcal{A}_\psi^\mathrm{CM}$, and thus it is also a cotorsion $\Lambda_\mathcal{O}^\mathrm{CM}$:
\begin{equation}\label{equation:cotorsion_h0}
\mathrm{corank}~ H^0(F_S/F, \mathcal{A}^{\mathrm{CM}}_\psi) =0. 
\end{equation}
We next investigate the $\Lambda^{\mathrm{CM}}_{\mathcal{O}}$-corank of the second  cohomology group $H^2(F_S/F, \mathcal{A}_\psi^{\mathrm{CM}})$. 
 The nine-term exact sequence due to Poitou and
 Tate (see \cite[{(8.6.3) i)}]{NSW} and \cite[Section~4.B]{gr-coh}) 
implies that the cokernel of the 
 global-to-local homomorphism
 \begin{align*}
 \phi^{(2)} \colon H^2(F_S/F, \mathcal{A}_\psi^{\mathrm{CM}}) \rightarrow \prod_{v\in S \setminus \Sigma_\infty } H^2(F_v, \mathcal{A}_\psi^{\mathrm{CM}})
 \end{align*}
 is isomorphic to the Pontrjagin dual of $H^0_\mathrm{cts}(F_S/F, (\mathcal{T}_\psi^\mathrm{CM})^*)$, which is trivial because the hypothesis 
 (LOC$^{(1)}_{\mathcal{A}_\psi^{\mathrm{CM}},v}$) holds 
for every $v$ in $S \setminus \Sigma_\infty$. 
The kernel of $\phi^{(2)}$ is 
 $\mathcyr{Sh}^2(F,S,\mathcal{A}_\psi^{\mathrm{CM}})$ by definition. 
 Furthermore, by virtue of the local Tate duality, the Pontrjagin dual of $H^2(F_v, \mathcal{A}_\psi^{\mathrm{CM}})$ 
 is isomorphic to $H^0_\mathrm{cts}(F_v,
 (\mathcal{T}_\psi^{\mathrm{CM}})^*)$ for each $v$ in $S \setminus \Sigma_\infty$,
 and hence it equals zero by the hypothesis
 (LOC$^{(1)}_{\mathcal{A}_\psi^{\mathrm{CM}},v}$) again. 
Combining these calculations, we conclude that the corank of the whole second cohomology group $H^2(F_S/F, \mathcal{A}^\mathrm{CM}_\psi)$ equals that of the  generalised second Tate-\v{S}afarevi\v{c} group $\mathcyr{Sh}^2(F,S,\mathcal{A}_\psi^\mathrm{CM})$:
\begin{equation}\label{equation:h2andsha2}
\mathrm{corank}~ H^2(F_S/F,
 \mathcal{A}_\psi^{\mathrm{CM}})
=\mathrm{corank}~ \mathcyr{Sh}^2(F,S,\mathcal{A}_\psi^{\mathrm{CM}}).
\end{equation}
By 
$(\ref{equation:semiglobal_epc})$, $(\ref{equation:cotorsion_h0})$ and 
$(\ref{equation:h2andsha2})$, we have the following formula:
\begin{align} \label{eq:gl_formula}
\mathrm{corank}~ H^1(F_S/F, \mathcal{A}^{\mathrm{CM}}_\psi) =
 \mathrm{corank}~ \mathcyr{Sh}^2(F,S,\mathcal{A}_\psi^{\mathrm{CM}})+d\, \mathrm{corank}~\mathcal{A}^{\mathrm{CM}}_\psi.
\end{align}

 We now study the $\Lambda^{\mathrm{CM}}_{\mathcal{O}}$-corank of
 $Q_{\mathcal{L}_\mathrm{str}}(F, \mathcal{A}_\psi^{\mathrm{CM}})$. 

First, let us take a place $\mathfrak{P}^c$ in $\Sigma_p^c$.  
Applying the local Euler-Poincar\'e characteristic formula
 \cite[Proposition~4.2]{gr-coh} to the local cohomology group 
$H^1(F_{\mathfrak{P}^c},
 \mathcal{A}_\psi^{\mathrm{CM}})$, we obtain
\begin{align*}
& \mathrm{corank} ~  Q_{\mathcal{L}_{\mathrm{str}}}(F_{\mathfrak{P}^c},
 \mathcal{A}_\psi^{\mathrm{CM}}) =\mathrm{corank}~ H^1(F_{\mathfrak{P}^c}, \mathcal{A}_\psi^{\mathrm{CM}}) \\
& = \mathrm{corank}~ H^0(F_{\mathfrak{P}^c},
 \mathcal{A}_\psi^{\mathrm{CM}})+\mathrm{corank}~ H^2(F_{\mathfrak{P}^c},
 \mathcal{A}_\psi^{\mathrm{CM}})+[F_{\mathfrak{P}^c}\colon
 \mathbb{Q}_p]\, \mathrm{corank}~ \mathcal{A}_\psi^{\mathrm{CM}}.
 \end{align*}
We claim that the $\Lambda^{\mathrm{CM}}_{\mathcal{O}}$-coranks of
 $H^0(F_{\mathfrak{P}^c},
 \mathcal{A}^{\mathrm{CM}}_\psi)$ and $H^2(F_{\mathfrak{P}^c},
 \mathcal{A}^{\mathrm{CM}}_\psi)$ are both equal to zero.
Indeed, since each place of $F$ above $p$ does not split completely in
 $\widetilde{F}/F$,  
the image of $D_{\mathfrak{P}_c}$ in
 $\mathrm{Gal}(\widetilde{F}/F)$  contains a nontrivial
 element. We thus apply the same arguments as those we made for
 $H^0(F_S/F, \mathcal{A}_\psi^{\mathrm{CM}})$ and conclude
 that $H^0(F_{\mathfrak{P}^c},
 \mathcal{A}^{\mathrm{CM}}_\psi)$ is cotorsion as a $\Lambda^{\mathrm{CM}}_{\mathcal{O}}$-module. 
 On the other hand, as we have already checked in the computation of
 the $\Lambda^{\mathrm{CM}}_{\mathcal{O}}$-corank of $H^2(F_S/F,
 \mathcal{A}^{\mathrm{CM}}_\psi)$, the second local cohomology group 
 $H^2(F_{\mathfrak{P}^c}, \mathcal{A}_\psi^{\mathrm{CM}})$ is
 trivial 
 due to the hypothesis (LOC$^{(1)}_{\mathcal{A}_\psi^{\mathrm{CM}},
 \mathfrak{P}^c}$). Therefore we obtain the following formula for each
 place $\mathfrak{P}^c$ in $\Sigma_p^c$:
\begin{align} \label{eq:p_formula}
\mathrm{corank}~Q_{\mathcal{L}_\mathrm{str}}(F_{\mathfrak{P}^c},
 \mathcal{A}^{\mathrm{CM}}_\psi)
 =[F_{\mathfrak{P}^c}:\mathbb{Q}_p]
 \mathrm{corank}~\mathcal{A}^{\mathrm{CM}}_\psi.
\end{align}
Next, let $\lambda$  be a place  in $S \setminus \Sigma_p$. 
We  also calculate the corank of 
$Q_{\mathcal{L}_{\mathrm{str}}}(F_\lambda, \mathcal{A}_\psi^{\mathrm{CM}})$ 
by applying the local Euler-Poincar\'e characteristics
 formula \cite[Proposition~4.2]{gr-coh} to the local cohomology group
 $H^1(F_\lambda, \mathcal{A}_\psi^{\mathrm{CM}})$ as follows:
 \begin{align*} 
&  \mathrm{corank}~ Q_{\mathcal{L}_{\mathrm{str}}}(F_\lambda,
 \mathcal{A}_\psi^{\mathrm{CM}}) \\
&=\mathrm{corank}~ H^1(F_\lambda,
 \mathcal{A}_\psi^{\mathrm{CM}})-\mathrm{corank}~
 H^1_\mathrm{ur}(F_\lambda, \mathcal{A}_\psi^{\mathrm{CM}}) \\
 &= \mathrm{corank}~ H^0(F_\lambda,
 \mathcal{A}_\psi^{\mathrm{CM}})+\mathrm{corank}~ H^2(F_\lambda,
 \mathcal{A}_\psi^{\mathrm{CM}}) -\mathrm{corank}~
 H^1_\mathrm{ur}(F_\lambda, \mathcal{A}_\psi^{\mathrm{CM}}).
\end{align*} 
We claim that the right hand side of this equality is zero, that is, 
\begin{align} \label{eq:tame_formula}
\mathrm{corank}~ Q_{\mathcal{L}_\mathrm{str}}(F_\lambda,
 \mathcal{A}^{\mathrm{CM}}_\psi) =0
\end{align}
holds. Indeed the second cohomology group $H^2(F_\lambda,
 \mathcal{A}^{\mathrm{CM}}_\psi)$ is trivial by the hypothesis
 (LOC$^{(1)}_{\mathcal{A}_\psi^{\mathrm{CM}}, \lambda}$) 
as we have already mentioned, and the $\Lambda^{\mathrm{CM}}_{\mathcal{O}}$-corank of
 the unramified cohomology 
$H^1_\mathrm{ur}(F_\lambda, \mathcal{A}_\psi^{\mathrm{CM}})$
 equals  that of $H^0(F_\lambda, \mathcal{A}_\psi^{\mathrm{CM}})$ 
since the residue characteristic at $\lambda$ does not equal $p$; 
this fact is well known for finite Galois modules and 
we readily generalise it  by using a specialisation trick 
similar to one used in the proof of \cite[Proposition~4.1]{gr-coh}.

Substituting (\ref{eq:gl_formula}), (\ref{eq:p_formula}) and
 (\ref{eq:tame_formula}) for the equation (\ref{eq:ingr_corank}), we obtain
 the desired equation (\ref{eq:key_corank_eq}) as 
{\small 
\begin{align*}
&\mathrm{corank}~\mathrm{Sel}^{\Sigma,
 \mathrm{str}}_{\mathcal{A}^{\mathrm{CM}}_\psi} \\
&=
 \mathrm{corank}~\mathcyr{Sh}^2(F,S,\mathcal{A}_\psi^{\mathrm{CM}})+d\, \mathrm{corank}~\mathcal{A}^{\mathrm{CM}}_\psi
 \\
&\text{\phantom{$\mathrm{corank}~\mathcyr{Sh}^2(F,S,\mathcal{A}_\psi^{\mathrm{CM}})$}} -\sum_{\mathfrak{P}^c
 \in \Sigma_p^c} [F_{\mathfrak{P}^c}\colon \mathbb{Q}_p]
\, \mathrm{corank}~\mathcal{A}^{\mathrm{CM}}_\psi+\mathrm{corank}~
 \mathrm{Coker}(\phi_{\mathcal{L}_\mathrm{str}}) \\
&=\mathrm{corank}~\mathcyr{Sh}^2(F,S,\mathcal{A}_\psi^{\mathrm{CM}})+\mathrm{corank}~
 \mathrm{Coker}(\phi_{\mathcal{L}_\mathrm{str}})+\left(
 d-\sum_{\mathfrak{p}\mid p\mathfrak{r}_{F^+}} [F^+_\mathfrak{p} :
 \mathbb{Q}_p ]\right)
 \mathrm{corank}~\mathcal{A}^{\mathrm{CM}}_\psi \\
&=\mathrm{corank}~\mathcyr{Sh}^2(F,S,\mathcal{A}_\psi^{\mathrm{CM}})+\mathrm{corank}~
 \mathrm{Coker}(\phi_{\mathcal{L}_\mathrm{str}}).
\end{align*}}
\end{proof}

The assumption (TOR$_{X_{\Sigma_p, (\psi)}}$) clearly holds 
if the Galois group
$X_{\Sigma_p}=\mathrm{Gal}(M_{\Sigma_p}/\widetilde{K}^\mathrm{CM}_\infty)$ itself is torsion as an 
$\mathcal{O}[[\mathrm{Gal}(\widetilde{K}^\mathrm{CM}_\infty/F)]]$-module.  
As a part of their research 
on the anti-cyclotomic Iwasawa main conjecture for CM
number fields, Hida and Tilouine have thoroughly studied 
the torsionness property of $X_{\Sigma_p}$ in
\cite[Section~1.2]{HT-aIMC}, and showed that $X_{\Sigma_p}$ is torsion 
over $\mathcal{O}[[\mathrm{Gal}(\widetilde{K}^\mathrm{CM}_\infty/F)]]$
if {\em the $\Sigma$-Leopoldt condition} (the condition
$(\mathscr{L}_{\widetilde{K}^\mathrm{CM}_\infty, \Sigma})$ in the
terminology of \cite{HT-aIMC}) is valid
for the extension $\widetilde{K}^\mathrm{CM}_\infty/F$. See \cite[Theorem~1.2.2~(ii)]{HT-aIMC}
for details of the discussion (note that our extension
$\widetilde{K}^\mathrm{CM}_\infty$ is always regarded as a subfield of
the ray class field $F_{\mathfrak{C}p^\infty}$ modulo
$\mathfrak{C}p^\infty$ over $F$ for an appropriate integral ideal
$\mathfrak{C}$ of $F$). Since the extension
$\widetilde{K}^\mathrm{CM}_\infty/F$  contains the {\em cyclotomic
$\mathbb{Z}_p$\nobreakdash-extension}, we observe that 
the $\Sigma$-Leopoldt condition is valid for the extension
$\widetilde{K}^\mathrm{CM}_\infty/F$ by reducing to the validity of the weak
Leopold conjecture for the cyclotomic $\mathbb{Z}_p$-extension of an
arbitrary number field due to the classical works of Iwasawa and
Greenberg \cite[Proof of Theorem~3]{gr-iw} (or by directly applying
\cite[Theorem~1.2.2~(iii)]{HT-aIMC}). 
Consequently the assumption (TOR$_{X_{\Sigma_p, (\psi)}}$) is
fulfilled and we have verified the following proposition:

\begin{pro} \label{prop:ini_adiv}
The strict Selmer group
 $\mathrm{Sel}^{\Sigma,\mathrm{str}}_{\mathcal{A}_\psi^{\mathrm{CM}}}$
 of $\mathcal{A}^{\mathrm{CM}}_\psi$ is cotorsion over $\Lambda_\mathcal{O}^\mathrm{CM}$ and almost divisible over $\Lambda_0$. The same claim holds for the strict Selmer group 
 $\mathrm{Sel}^{\Sigma,\mathrm{str}}_{\mathcal{A}^{\mathrm{CM}}_\eta}$
 of $\mathcal{A}^{\mathrm{CM}}_\eta$.
\end{pro}

As a corollary of our computation so far, we deduce 
a consequence on the algebraic structure of the Iwasawa module $X_{\Sigma_p,
(\psi)}$ when the Pontrjagin dual of
$\mathrm{Sel}^{\Sigma,\mathrm{str}}_{\mathcal{A}_\psi^{\mathrm{CM}}}$ is exactly
isomorphic to $X_{\Sigma_p, (\psi)}$. Namely, 

\begin{cor}[Theorem~\ref{Thm:psnull}] \label{cor:adiv_psi}
Assume that the nontriviality condition
 {\upshape (ntr)$_{\psi,\mathfrak{P}^c}$} 
is valid for each place $\mathfrak{P}^c$ of $F$ belonging to
 $\Sigma_p^c$ and that the order of the branch character $\psi$ is relatively
 prime to $p$ $($in
 particular $\psi$ is not trivial$)$.
Then the Iwasawa module $X_{\Sigma_p, (\psi)}$ does not
 contain nontrivial pseudonull $\Lambda_0$-submodules.
\end{cor}

\begin{proof}
Under the assumptions, the Pontrjagin dual of
 $\mathrm{Sel}^{\Sigma,\mathrm{str}}_{\mathcal{A}_\psi^{\mathrm{CM}}}$ is exactly
 isomorphic to $X_{\Sigma_p, (\psi)}$ by
 Remark~\ref{rem:vanish} and Lemma~\ref{lem:replace_psi}. The statement
 then follows from Proposition~\ref{prop:ini_adiv}.
\end{proof}  

\begin{rem}
The triviality of pseudonull submodules of the Iwasawa module 
$X_{\Sigma_p, (\psi)}$ has been already studied by Perrin-Riou in
\cite[Th\'eor\`eme~2.4]{PR_CM} only when $F$ is an imaginary quadratic field and
 $\psi$ is a gr\"o{\ss}encharacter associated to an elliptic curve with complex
 multiplication. Her method essentially utilises  
Wintenberger's structure theorem on projective limits of local unit groups \cite{Wintenberger} 
combined with Greenberg's classical result on 
the triviality of the pseudonull submodules contained in the 
``$p$-ramified Iwasawa module'' \cite[Proposition~5]{Gr_Galois}, which 
 is rather different from ours. It was not clear to us if the method
 used in \cite{PR_CM} could be extended to general CM number fields or not.
\end{rem} 

\medskip
 \subsubsection{Intermediate steps$:$ inductive specialisation of the Selmer
   group} \label{sssc:int_step}

Next we shall inductively specialise the strict Selmer group
$\mathrm{Sel}^{\Sigma,\mathrm{str}}_{\mathcal{A}^{\mathrm{CM}}_\eta}$ so that 
the characteristic ideal of its Pontrjagin dual behaves compatibly with
respect to each specialisation procedure. 

We use the same notation as in Section~\ref{sssc:regular}; in particular $\mathfrak{A}_j$ denotes the ideal of $\Lambda_\mathcal{O}^\mathrm{CM}$ generated by a regular sequence $x_1,\dotsc, x_j$ with $x_k=\gamma_k-1$.
We also recall that the ideal $\mathfrak{A}^\mathrm{cyc}$ of $\Lambda_\mathcal{O}^\mathrm{CM}$ is defined as the
kernel of the natural surjection $\Lambda_\mathcal{O}^\mathrm{CM}\twoheadrightarrow \Lambda_\mathcal{O}^\mathrm{cyc}$. Now let us consider the following three conditions $(\Gamma0)_j$, $(\Gamma1)_j$
and $(\Gamma2)_j$ on the fixed elements $\gamma_1,\dotsc, \gamma_j$ of $\mathrm{Gal}(\widetilde{F}/F)$:

\begin{enumerate}[label=$(\Gamma\arabic*)_j$]
\setcounter{enumi}{-1}
 \item the sequence $x_1, \dotsc, x_j$ is a regular sequence of
       $\Lambda^{\mathrm{CM}}_\mathcal{O}$ contained in
       $\mathfrak{A}^{\mathrm{cyc}}${\upshape ;}
 \item the hypothesis {\upshape
       (LEO$_{\mathcal{A}^{\mathrm{CM}}_\eta[\mathfrak{A}_{j}]}$)}
       is fulfilled for $\mathcal{A}_\eta^\mathrm{CM}[\mathfrak{A}_j]${\upshape ;}
\item the dual Selmer group $\mathrm{Sel}_{\mathcal{L}^*_{\mathrm{str}}}(F,
 (\mathcal{T}_\eta^{\mathrm{CM}}/\mathfrak{A}_{j}\mathcal{T}_\eta^{\mathrm{CM}})^*)$
      is trivial.
\end{enumerate}
Here we regard the condition $(\Gamma0)_0$ as the empty condition. The following proposition is the key of our specialisation arguments. 

 \begin{pro} \label{prop:induction}
  Let the notation be as above and assume that the condition 
  {\upshape (ntr)$_{\mathfrak{P}^c}$} is fulfilled
  for every place $\mathfrak{P}^c$ of $F$ contained in $\Sigma_p^c$.
  Suppose that, for a natural number $j$ with $1\leq j\leq d+\delta_{F,p}-1$,
  all the conditions $(\Gamma0)_{j-1}$, $(\Gamma1)_{j-1}$ and
  $(\Gamma2)_{j-1}$  are fulfilled on a set of elements  $\gamma_1, \dotsc, \gamma_{j-1}$ of  $\mathrm{Gal}(\widetilde{F}/F^\mathrm{cyc}_\infty)$.
  Let $\mathrm{Sel}^{\Sigma,\mathrm{str}}_{\mathcal{A}^{\mathrm{CM}}_\eta[\mathfrak{A}_{j-1}]}$ be the strict Selmer group of $\mathcal{A}_\eta^\mathrm{CM}[\mathfrak{A}_{j-1}]$ and $(\mathrm{Sel}^{\Sigma,\mathrm{str}}_{\mathcal{A}^{\mathrm{CM}}_\eta[\mathfrak{A}_{j-1}]})^\vee$ its Pontrjagin dual. 
Then there exists an element $\gamma_j$ of
  $\mathrm{Gal}(\widetilde{F}/F^\mathrm{cyc}_\infty)$ such that
  the element $x_j=\gamma_j-1$ of $\Lambda_\mathcal{O}^{\mathrm{CM}}$ does not
  divide in $\Lambda^{\mathrm{CM}}_{\mathcal{O}}/\mathfrak{A}_{j-1}$ the
  characteristic ideal of the torsion part of
  $(\mathrm{Sel}^{\Sigma,\mathrm{str}}_{\mathcal{A}^{\mathrm{CM}}_\eta[\mathfrak{A}_{j-1}]})^\vee$, and all the conditions $(\Gamma0)_j, (\Gamma1)_j$ and $(\Gamma2)_j$ are fulfilled on the tuples $ \gamma_1, \dotsc, \gamma_{j-1}$ and $\gamma_j$.
 \end{pro}

We postpone the proof of Proposition~\ref{prop:induction} until the end
of this paragraph, and we here deduce the basechange compatibility of 
the characteristic ideal of $(\mathrm{Sel}_{\mathcal{A}_\eta^\mathrm{CM}}^{\Sigma,\mathrm{str}})^\vee$ under intermediate specialisation procedures from
Proposition~\ref{prop:induction}.   Firstly we note that, when $j$ equals zero, all the conditions $(\Gamma0)_0$, $(\Gamma1)_0$ and $(\Gamma2)_0$ are fulfilled; indeed the condition $(\Gamma0)_0$ is empty by convention,
and the condition $(\Gamma1)_0$ follows from Lemma~\ref{lem:tors} and the cotorsionness of the strict Selmer group $\mathrm{Sel}_{\mathcal{A}_\eta^\mathrm{CM}}^{\Sigma,\mathrm{str}}$ (refer to the argument in Section~\ref{sssc:first_step}). The triviality of the dual Selmer group $\mathrm{Sel}_{\mathcal{L}^*_{\mathrm{str}}}(F, \mathcal{T}_\eta^*)$, namely the condition $(\Gamma2)_0$,  
 also follows from Lemma~\ref{lem:tors} and the cotorsionness of the strict Selmer group $\mathrm{Sel}_{\mathcal{A}_\eta^\mathrm{CM}}^{\Sigma,\mathrm{str}}$, combined with
 Proposition~\ref{prop:surj}; observe that an arbitrary place
 $\mathfrak{P}^c$ contained in $\Sigma_p^c$ satisfies the extra condition (c) 
proposed in Proposition~\ref{prop:surj}.

Now let $j$ be a natural number with $1\leq j\leq d+\delta_{F,p}-1$, and 
assume that we have already chosen elements $ \gamma_1, \dotsc,\gamma_{j-1}$ 
 of $\mathrm{Gal}(\widetilde{F}/F^\mathrm{cyc}_\infty)$ so that all the conditions 
 $(\Gamma0)_{j-1}, (\Gamma1)_{j-1}$ and $(\Gamma2)_{j-1}$ are fulfilled for them.
As an induction hypothesis, we further assume that
{\em $\mathrm{Sel}^{\Sigma,\mathrm{str}}_{\mathcal{A}^{\mathrm{CM}}_\eta[\mathfrak{A}_{j-1}]}$
is cotorsion} as a $\Lambda_\mathcal{O}^{\mathrm{CM}}/\mathfrak{A}_{j-1}$-module. 
Theorem~\ref{thm:adiv} then enables us to conclude that the strict Selmer group 
$\mathrm{Sel}^{\Sigma, \mathrm{str}}_{\mathcal{A}^{\mathrm{CM}}_\eta[\mathfrak{A}_{j-1}]}$
is almost $\Lambda_0^{(j-1)}$-divisible.
 Let
 $(\mathrm{Sel}^{\Sigma, \mathrm{str}}_{\mathcal{A}^{\mathrm{CM}}_\eta[\mathfrak{A}_{j-1}]})^\vee$
 denote the Pontrjagin dual of the strict Selmer group 
$\mathrm{Sel}^{\Sigma, \mathrm{str}}_{\mathcal{A}^{\mathrm{CM}}_\eta[\mathfrak{A}_{j-1}]}$. We
apply Proposition~\ref{prop:induction} and 
find an element $\gamma_j$ of $\mathrm{Gal}(\widetilde{F}/F^\mathrm{cyc}_\infty)$ 
so that $x_j=\gamma_j-1$ does not divide the characteristic ideal
of
$(\mathrm{Sel}^{\Sigma, \mathrm{str}}_{\mathcal{A}^{\mathrm{CM}}_\eta[\mathfrak{A}_{j-1}]})^\vee$
and all the conditions $(\Gamma0)_j, (\Gamma1)_j$ and $(\Gamma2)_j$ are fulfilled
for the tuples $\gamma_1, \dotsc ,\gamma_{j-1}$ and $\gamma_j$. 
Then, from the Exact Control Theorem~\ref{thm:control} and the Specialisation Lemma~\ref{lem:specialisation}, we readily deduce the
cotorsionness of the strict Selmer group $\mathrm{Sel}^{\Sigma,\mathrm{str}}_{\mathcal{A}^{\mathrm{CM}}_\eta[\mathfrak{A}_{j}]}$ as a $\Lambda_\mathcal{O}^{\mathrm{CM}}/\mathfrak{A}_{j}$-module,
and obtain the equality

 \begin{multline*}
\left(
\mathrm{Char}_{\Lambda_\mathcal{O}^{\mathrm{CM}}/\mathfrak{A}_{
  j-1}}
 \left(\mathrm{Sel}^{\Sigma,
  \mathrm{str}}_{\mathcal{A}^{\mathrm{CM}}_\eta[\mathfrak{A}_{
  j-1}]}\right)^\vee
  \right)
  \otimes_{\Lambda_\mathcal{O}^{\mathrm{CM}}/\mathfrak{A}_{
  j-1}} 
  \Lambda_\mathcal{O}^{\mathrm{CM}}/\mathfrak{A}_{j} = \mathrm{Char}_{\Lambda_\mathcal{O}^{\mathrm{CM}}/\mathfrak{A}_{j}}\left(\mathrm{Sel}^{\Sigma,
  \mathrm{str}}_{\mathcal{A}^{\mathrm{CM}}_\eta[\mathfrak{A}_{j}]}\right)^\vee.
\end{multline*}

The induction proceeds until $j$ achieves $d+\delta_{F,p}-1$, and
consequently we see that the strict Selmer group
$\mathrm{Sel}^{\Sigma,\mathrm{str}}_{\mathcal{A}^{\mathrm{CM}}_\eta[\mathfrak{A}_{
d+\delta_{F,p}-1}]}$
is cotorsion as a $\Lambda_\mathcal{O}^{\mathrm{CM}}/\mathfrak{A}_{d+\delta_{F,p}-1}$-module and the following equality among ideals of $\Lambda_\mathcal{O}^\mathrm{CM}/\mathfrak{A}_{d+\delta_{F,p}-1}$ holds:

\begin{multline}  \label{eq:spec_ind}
 \left( 
 \mathrm{Char}_{\Lambda_\mathcal{O}^{\mathrm{CM}}}
 \left(\mathrm{Sel}^{\Sigma,
  \mathrm{str}}_{\mathcal{A}^{\mathrm{CM}}_\eta}\right)^\vee
  \right)
  \otimes_{\Lambda_\mathcal{O}^{\mathrm{CM}}}
  \Lambda_\mathcal{O}^{\mathrm{CM}}/\mathfrak{A}_{
  d+\delta_{F,p}-1} \\
= \mathrm{Char}_{\Lambda_\mathcal{O}^{\mathrm{CM}}/\mathfrak{A}_{
  d+\delta_{F,p}-1}}\left(\mathrm{Sel}^{\Sigma,
  \mathrm{str}}_{\mathcal{A}^{\mathrm{CM}}_\eta[\mathfrak{A}_{d+\delta_{F,p}-1}]}\right)^\vee.
\end{multline}

\begin{rem}
 We warn that we here exclude the case {\em where $j$ equals $d+\delta_{F,p}$} due to the constraint imposed on $j$ in Proposition~\ref{prop:induction};  
the case where $j$ equals $d+\delta_{F,p}$ shall be dealt with later in
Section~\ref{sssc:final_step} as the final specialisation procedure. 
\end{rem}

Let us return to the proof of Proposition~\ref{prop:induction}. 
We first observe that  
the condition $(\Gamma2)_j$ ---the triviality of the dual Selmer group---
is {\em automatically} fulfilled for an appropriate choice of $\gamma_j$.

 \begin{pro} \label{prop:dSelvan}
 Let $j$ be as above and let $\gamma_1, \dotsc, \gamma_{j-1}$ be 
  elements of
  $\mathrm{Gal}(\widetilde{F}/F^\mathrm{cyc}_\infty)$ for which all the conditions
  $(\Gamma0)_{j-1}$, $(\Gamma1)_{j-1}$ and $(\Gamma2)_{j-1}$ stated before
  Proposition~$\ref{prop:induction}$ are fulfilled. Assume further that the 
  condition {\upshape (ntr)$_{\mathfrak{P}^c}$}  is fulfilled for every place
  $\mathfrak{P}^c$ of $F$ contained in $\Sigma_p^c$. 
  Let $\gamma_j$ be an element of
  $\mathrm{Gal}(\widetilde{F}/F^\mathrm{cyc}_\infty)$ such that the tuples $\gamma_1, \dotsc,
  \gamma_{j-1}, \gamma_j$ satisfy the condition $(\Gamma0)_j$
  and that $x_j=\gamma_j-1$ does not divide the characteristic ideal of 
  the torsion part of
  $(\mathrm{Sel}^{\Sigma,\mathrm{str}}_{\mathcal{A}^{\mathrm{CM}}_\eta[\mathfrak{A}_{j-1}]})^\vee$
  in
  $\Lambda_\mathcal{O}^{\mathrm{CM}}/\mathfrak{A}_{
  j-1}$.
  Then the dual strict Selmer group 
  $\mathrm{Sel}_{\mathcal{L}^*_\mathrm{str}}(F,
  (\mathcal{T}^{\mathrm{CM}}_\eta/\mathfrak{A}_{j}\mathcal{T}^{\mathrm{CM}}_\eta)^*)$ of $(\mathcal{T}^{\mathrm{CM}}_\eta/\mathfrak{A}_{j}\mathcal{T}^{\mathrm{CM}}_\eta)^*$ is trivial.
 \end{pro}
 
 Let us verify Proposition~\ref{prop:induction} admitting Proposition~\ref{prop:dSelvan}.

 \begin{proof}[Proof of Proposition~$\ref{prop:induction}$] 
Let $\gamma_1, \dotsc, \gamma_{j-1}$ be elements of
  $\mathrm{Gal}(\widetilde{F}/F^\mathrm{cyc}_\infty)$ as in the statement of Proposition~\ref{prop:induction} and $H$ the closed subgroup of $\mathrm{Gal}(\widetilde{F}/F^\mathrm{cyc}_\infty)$ topologically
  generated by $\gamma_1, \dotsc, \gamma_{j-1}$. 
Define $\mathscr{C}_H$ as the set of all the elements of 
  $\overline{\Gamma}_{j-1}:=\mathrm{Gal}(\widetilde{F}/F^\mathrm{cyc}_\infty)/H$
  with nontrivial images in $\overline{\Gamma}_{j-1}/(\overline{\Gamma}_{j-1})^p$. 
  Let us denote by $Y_{j-1}^{(1)}$ the finite set of height-one prime ideals in
  $\Lambda_0^{(j-1)}$ dividing the characteristic ideal of the torsion part of 
  $(\mathrm{Sel}^{\Sigma,
  \mathrm{str}}_{\mathcal{A}^{\mathrm{CM}}_\eta[\mathfrak{A}_{j-1}]})^\vee$,
  which is obviously a finite set. Meanwhile, 
as we have remarked in Section~\ref{sssc:hyp}, the validity of the hypothesis
  (LEO$_{\mathcal{A}^{\mathrm{CM}}_\eta[\mathfrak{A}_{
  j-1}]}$)  
implies the existence of a finite set $Y^{(2)}_{j-1}$ of
  exceptional prime ideals of height one in
  $\Lambda_0^{(j-1)}$ in the sense that the hypothesis
  (LEO$_{\mathcal{A}^{\mathrm{CM}}_\eta[\mathfrak{A}_{j-1}][\bar{\varPi}]}$)
  is true for every height-one prime ideal
  $(\bar{\varPi})$ of $\Lambda_0^{(j-1)}$ which is not contained in $Y^{(2)}_{j-1}$.
  Since the set of principal ideals in $\Lambda_0^{(j-1)}$
  defined as $\{\,  (\bar{\gamma} -1)\Lambda_0^{(j-1)} \mid  \bar{\gamma} \in \mathscr{C}_H \, \}$  
  is infinite, we can choose an element
  $\bar{\gamma}_j$ of $\mathscr{C}_H$ such that the prime ideal
  $(\bar{\gamma}_j-1)$ in $\Lambda_0^{(j-1)}$
  generated by $\bar{\gamma}_j-1$ is contained in neither
  $Y^{(1)}_{j-1}$ nor $Y^{(2)}_{j-1}$. Let us take an arbitrary lift $\gamma_j$  of $\bar{\gamma}_j$ to
  $\mathrm{Gal}(\widetilde{F}/F^\mathrm{cyc}_\infty)$.
  By construction, both the conditions $(\Gamma 0)_j$ and $(\Gamma 1)_j$ are fulfilled
  on the tuples $\gamma_1, \dotsc, \gamma_{j-1}, \gamma_j$ of
  $\mathrm{Gal}(\widetilde{F}/F^\mathrm{cyc}_\infty)$.
  We now complete the proof of Proposition~\ref{prop:induction} since the condition $(\Gamma2)_j$ is also fulfilled on them by virtue of Proposition~\ref{prop:dSelvan}.
 \end{proof}
 
 The remaining issue is the verification of Proposition~$\ref{prop:dSelvan}$.
 Let us abbreviate in the proof of Proposition~\ref{prop:dSelvan} the continuous $\mathrm{Gal}(F_S/F)$-representation
  $(\mathcal{T}^{\mathrm{CM}}_\eta/\mathfrak{A}_{k}\mathcal{T}^{\mathrm{CM}}_\eta)^*$
  and the discrete $\mathrm{Gal}(F_S/F)$-representation $\mathcal{A}_\eta^{\mathrm{CM}}[\mathfrak{A}_k]$ 
  as $\mathcal{T}^*_{\eta, (k)}$ and $\mathcal{A}_{\eta,(k)}$ respectively where $k$ equals $j-1$ or $j$. We also use the abbreviation $\mathcyr{Sh}^1_{\mathcal{A}_{\eta,(k)}}$ for the
  $S$-fine Selmer group $\mathcyr{Sh}^1(F,S,\mathcal{A}_{\eta,(k)})$ of $\mathcal{A}_{\eta,(k)}$.
  Note that 
$\mathcal{T}^*_{\eta,(k)}$ is a free 
$\Lambda^{\mathrm{CM}}_\mathcal{O}/\mathfrak{A}_{k}$-module of rank one 
on which every element $g$ of $\mathrm{Gal}(F_S/F)$ acts by the
  multiplication of 
  $\chi_{p,\mathrm{cyc}}^{-1}\eta^\mathrm{gal}(g^{-1})g\vert_{\widetilde{F}}$.
  First recall that the cokernel of the global-to-local morphism
  \begin{align*}
   \phi_{\mathcal{L}_\mathrm{str}^{*,(k)}} \colon H^1(F_S/F, \mathcal{T}_{\eta,(k)}^*)
   \rightarrow Q_{\mathcal{L}_\mathrm{str}^{*,(k)}} (F,\mathcal{T}_{\eta,(k)}^*)
  \end{align*}
  is isomorphic to the kernel of the natural surjection
  $(\mathrm{Sel}^{\Sigma,
  \mathrm{str}}_{\mathcal{A}_{\eta,(k)}})^\vee
  \rightarrow
  (\mathcyr{Sh}^1_{\mathcal{A}_{\eta,(k)}})^\vee$ for every $k$   (see Section~\ref{sssc:dSel} for details); in particular, there exists a short exact sequence
  \begin{align} \label{ses:cokernel_k}
   \xymatrix{
   0 \ar[r] & \mathrm{Coker}(\phi_{\mathcal{L}^{*,(k)}_\mathrm{str}}) \ar[r] & (\mathrm{Sel}^{\Sigma,\mathrm{str}}_{\mathcal{A}_{\eta, (k)}})^\vee \ar[r] & (\mathcyr{Sh}^1_{\mathcal{A}_{\eta,(k)}})^\vee \ar[r] & 0
   }  
  \end{align}
 for $k=j-1$ or $j$. We consider the commutative diagram 
  \begin{align} \label{ses:cokernel}
   \xymatrix{
   0 \ar[r] & \mathrm{Coker}(\phi_{\mathcal{L}^{*,(j-1)}_\mathrm{str}}) \ar[r] \ar[d]_{\times x_{j}} & (\mathrm{Sel}^{\Sigma,\mathrm{str}}_{\mathcal{A}_{\eta,(j-1)}})^\vee \ar[r] \ar[d]_{\times x_{j}} &
   (\mathcyr{Sh}^1_{\mathcal{A}_{\eta,(j-1)}})^\vee \ar[r] \ar[d]_{\times x_{j}} & 0 \\
    0 \ar[r] & \mathrm{Coker}(\phi_{\mathcal{L}^{*,(j-1)}_\mathrm{str}}) \ar[r]  & (\mathrm{Sel}^{\Sigma,\mathrm{str}}_{\mathcal{A}_{\eta,(j-1)}})^\vee \ar[r] &
   (\mathcyr{Sh}^1_{\mathcal{A}_{\eta,(j-1)}})^\vee \ar[r]
   & 0,}
  \end{align}
whose rows are the exact sequences (\ref{ses:cokernel_k}) for $k=j-1$ and vertical maps are multiplication by $x_j$.
Note that the Pontrjagin dual of the strict Selmer group
$\mathrm{Sel}^{\Sigma, \mathrm{str}}_{\mathcal{A}_{\eta,(j-1)}}$
  does not contain nontrivial pseudonull
  $\Lambda_0^{(j-1)}$\nobreakdash-submodules
due to the assumptions $(\Gamma0)_{j-1}$, $(\Gamma1)_{j-1}$ and
  $(\Gamma2)_{j-1}$ combined with Theorem~\ref{thm:adiv}. 
Since $x_j$ does not divide the characteristic ideal of
  $(\mathrm{Sel}^{\Sigma,\mathrm{str}}_{\mathcal{A}_{\eta,(j-1)}})^\vee$
  by assumption, the triviality of the pseudonull submodules of
  $(\mathrm{Sel}^{\Sigma,\mathrm{str}}_{\mathcal{A}_{\eta,(j-1)}})^\vee$
  implies that the middle vertical arrow of (\ref{ses:cokernel}) is
  injective (and so is the left vertical arrow).
  Thus, applying the snake lemma to the diagram (\ref{ses:cokernel}), we obtain a four-term exact sequence
  \begin{align} \label{ses:snake_cokernel}
   \xymatrix{
   0 \ar[r] &  (\mathcyr{Sh}^1_{\mathcal{A}_{\eta,(j-1)}})^\vee [x_j] \ar[r]^(0.65){\tilde{\delta}_1^\vee} & \mathcal{C}^{(j-1)} \ar[r]  & (\mathrm{Sel}^{\Sigma,\mathrm{str}}_{\mathcal{A}_{\eta,(j-1)}}[x_j])^\vee \ar[r]  &
   (\mathcyr{Sh}^1_{\mathcal{A}_{\eta,(j-1)}}[x_j])^\vee \ar[r] & 0}
  \end{align}
  where we denote by $\tilde{\delta}^\vee_1$ the connecting homomorphism and 
  define $\mathcal{C}^{(j-1)}$ as follows: 
  \begin{equation}\label{equation:definition_C}
  \mathcal{C}^{(j-1)}:= \mathrm{Coker} \left[
  \mathrm{Coker}(\phi_{\mathcal{L}^{*,(j-1)}_\mathrm{str}}) 
  \xrightarrow{\times x_j }
  \mathrm{Coker}(\phi_{\mathcal{L}^{*,(j-1)}_\mathrm{str}})\right].
  \end{equation} 
  We also remark that there exists a natural commutative diagram
  \begin{align} \label{dg:square}
   \xymatrix{
   \mathrm{Sel}^{\Sigma,\mathrm{str}}_{\mathcal{A}_{\eta,(j-1)}}[x_j] & \mathcyr{Sh}^1_{\mathcal{A}_{\eta,(j-1)}}[x_j] \ar@{_(->}[l] \\
   \mathrm{Sel}^{\Sigma,\mathrm{str}}_{\mathcal{A}_{\eta,(j)}} \ar[u]^{\iota^{\mathrm{Sel}}_j}_{\rotatebox{90}{$\sim$}} & \mathcyr{Sh}^1_{\mathcal{A}_{\eta,(j)}} \ar[u]_{\iota^{\mathcyr{Sh}^1}_j} \ar@{_(->}[l] 
   }
  \end{align}
 both of whose vertical morphisms $\iota^{\mathrm{Sel}}_j$ and $\iota^{\mathcyr{Sh}^1}_j$ are induced from the natural inclusion $\iota^\mathcal{A}_j\colon \mathcal{A}_{\eta,(j)}=\mathcal{A}_{\eta,(j-1)}[x_j]\hookrightarrow \mathcal{A}_{\eta,(j-1)}$ 
  of $\mathrm{Gal}(F_S/F)$-modules. Note that the left vertical map $\iota^{\mathrm{Sel}}_j$ is an isomorphism due to the Exact Control Theorem~\ref{thm:control}. 
Combining the Pontrjagin dual of the commutative diagram (\ref{dg:square}) with the exact sequence (\ref{ses:snake_cokernel}), we obtain the following diagram with exact rows:
  \begin{align} \label{ses:comp_coker}
   \text{\footnotesize $\xymatrix{
   0 \ar[r] & (\mathcyr{Sh}^1_{\mathcal{A}_{\eta,(j-1)}})^\vee [x_j] \ar[r]^(0.6){\tilde{\delta}_1^\vee}  & \mathcal{C}^{(j-1)} \ar[r] \ar@{.>}[d]_{r^{\mathcal{C}}_j} & (\mathrm{Sel}^{\Sigma,\mathrm{str}}_{\mathcal{A}_{\eta,(j-1)}}[x_j])^\vee \ar[r] \ar[d]^{\rotatebox{90}{$\sim$}}_{(\iota^{\mathrm{Sel}}_j)^\vee} \ar@{}[rd]|\circlearrowright 
   &
   (\mathcyr{Sh}^1_{\mathcal{A}_{\eta,(j-1)}}[x_j])^\vee \ar[r] \ar[d]^{(\iota^{\mathcyr{Sh}^1}_j)^\vee} & 0 \\
    &  0 \ar[r] & \mathrm{Coker}(\phi_{\mathcal{L}^{*,(j)}_\mathrm{str}}) \ar[r]  & (\mathrm{Sel}^{\Sigma,\mathrm{str}}_{\mathcal{A}_{\eta,(j)}})^\vee \ar[r] &
   (\mathcyr{Sh}^1_{\mathcal{A}_{\eta,(j)}})^\vee \ar[r] & 0.
   }$}   
  \end{align}
Thanks to the commutativity of (\ref{dg:square}), a homomorphism $r^{\mathcal{C}}_j\colon \mathcal{C}^{(j-1)}\rightarrow \mathrm{Coker}(\phi_{\mathcal{L}^{*,(j)}_{\mathrm{str}}})$ (the dotted vertical arrow in the diagram (\ref{ses:comp_coker})) is induced.
An easy diagram chase on (\ref{ses:comp_coker}) enables us to verify that the kernel of the induced homomorphism $r^{\mathcal{C}}_j$ coincides with the image of $\tilde{\delta}^\vee_1$, and therefore we obtain an exact sequence
  \begin{align} \label{ses:Sha^1}
   \xymatrix{ 0\ar[r] & (\mathcyr{Sh}^1_{\mathcal{A}_{\eta,(j-1)}})^\vee[x_j] \ar[r]^(0.65){\tilde{\delta}_1^\vee} & \mathcal{C}^{(j-1)} \ar[r] & \mathrm{Coker}(\phi_{\mathcal{L}_\mathrm{str}^{*,(j)}})}.
  \end{align}

  Next recall that the dual Selmer group of $\mathcal{T}_{\eta,(k)}^*$, which we denote by $\mathrm{Sel}^{*,\mathrm{str}}_{\mathcal{T}^*_{\eta,(k)}}$ for brevity, is defined in terms of the exact sequence
  \begin{align} \label{ses:dualSel_def}
\text{\small    $\xymatrix{
0\ar[r] & \mathrm{Sel}^{*,\mathrm{str}}_{\mathcal{T}^*_{\eta,(k)}} \ar[r] &
H^1_\mathrm{cts}(F_S/F,
 \mathcal{T}^*_{\eta, (k)})  \ar[r]^{\phi_{\mathcal{L}^{*,(k)}_\mathrm{str}}} &
 Q_{\mathcal{L}_\mathrm{str}^{*,(k)}}(F, \mathcal{T}^*_{\eta, (k)})
 \ar[r]  &
 \mathrm{Coker}(\phi_{\mathcal{L}^{*,(k)}_\mathrm{str}}) \ar[r] 
 & 0
   }$}
  \end{align}
  for $k=j-1$ or $j$. When $k$ equals $j-1$, the triviality assumption $(\Gamma2)_{j-1}$ on $\mathrm{Sel}^{*,\mathrm{str}}_{\mathcal{T}^*_{\eta,(j-1)}}$ suggests that the global-to-local morphism $\phi_{\mathcal{L}^{*,(j-1)}_{\mathrm{str}}}$ in (\ref{ses:dualSel_def}) is injective. We thus consider the commutative diagram
\begin{align} \label{dg:dSel_phi}
\xymatrix{
0 \ar[r] & H^1_\mathrm{cts}(F_S/F,
 \mathcal{T}^*_{\eta, (j-1)})
 \ar[r]^{\phi_{\mathcal{L}^{*,(j-1)}_\mathrm{str}} } \ar[d]_{\times x_j}
 &
 Q_{\mathcal{L}_\mathrm{str}^{*,(j-1)}}(F, \mathcal{T}^*_{\eta, (j-1)})
 \ar[r] \ar[d]_{\times x_j} &
 \mathrm{Coker}(\phi_{\mathcal{L}^{*,(j-1)}_\mathrm{str}}) \ar[r] \ar@{^(->}[d]^{\times x_j}
 & 0 \\
0 \ar[r] & H^1_\mathrm{cts}(F_S/F, \mathcal{T}^*_{\eta, (j-1)})
 \ar[r]^{\phi_{\mathcal{L}^{*,(j-1)}_\mathrm{str}} }  &
 Q_{\mathcal{L}_\mathrm{str}^{*,(j-1)}}(F,\mathcal{T}^*_{\eta, (j-1)})
 \ar[r]  & \mathrm{Coker}(\phi_{\mathcal{L}^{*,(j-1)}_\mathrm{str}}) \ar[r]  & 0
.}
\end{align}
whose horizontal rows are short exact sequences obtained by ($\ref{ses:dualSel_def}$) for $k=j-1$ 
and all of whose vertical maps are 
  multiplication by $x_j$. Note that the right vertical arrow of (\ref{dg:dSel_phi}) is injective since it is the same as the left vertical arrow
  of (\ref{ses:cokernel}). 
  The snake lemma applied to the diagram \eqref{dg:dSel_phi} then suggests that the cokernels of its
  vertical morphisms  form a  short exact sequence
\begin{align} \label{ses:cokernel_xj}
 \xymatrix{
0 \ar[r] & H^{1,*}_{(j-1)}/x_jH^{1,*}_{(j-1)} \ar[r] & Q^*_{(j-1)}/x_jQ^*_{(j-1)} \ar[r] & \mathcal{C}^{(j-1)} \ar[r] & 0
 }
\end{align}
with abbreviated notation 
\begin{equation}\label{equation:definition_HQ}
 \begin{aligned}
H^{i,*}_{(j)}&:=H^i_\mathrm{cts}(F_S/F, \mathcal{T}^*_{\eta, (j)}), &
  Q^*_{(j-1)}&:= Q_{\mathcal{L}_\mathrm{str}^{*,(j-1)}}(F,\mathcal{T}^*_{\eta, (j-1)}).
 \end{aligned}
\end{equation}
Next we shall relate the short exact sequence  (\ref{ses:cokernel_xj}) with the sequence (\ref{ses:dualSel_def}) for $k=j$, which concerns the dual Selmer group $\mathrm{Sel}^{*,\mathrm{str}}_{\mathcal{T}^*_{\eta,(j)}}$. To this end, we first observe that the short exact sequence
\begin{align} \label{equation:T}
\xymatrix{0 \ar[r] &
 \mathcal{T}^*_{\eta, (j-1)} \ar[r]^{\times
 x_j} & 
 \mathcal{T}^*_{\eta, (j-1)} \ar[r]^{\pi_j^\mathcal{T}} & 
 \mathcal{T}^*_{\eta, (j)} \ar[r] & 0}
\end{align} 
of $\mathrm{Gal}(F_S/F)$-modules induces the following exact sequence: 
\begin{equation}\label{equation:definition_pi_j_h}
0 \longrightarrow  H^{1,*}_{(j-1)}/x_jH^{1,*}_{(j-1)}  \xrightarrow{\pi_j^H } H^{1,*}_{(j)}  \xrightarrow{\text{\sf \textgreek{d}}^1_{(j)}} H^{2,*}_{(j-1)} [x_j] 
\longrightarrow 0. 
\end{equation}
We shall construct a homomorphism
connecting $Q^*_{(j-1)}/x_jQ^*_{(j-1)}$ with $Q^*_{(j)}$ in a way similar to
the construction of $\pi_j^H$. Let $v$ be a place of $F$ contained in $S$ and let us denote by $\pi^H_{j,v}$ 
the injection on the local cohomology groups
\begin{align*}
\pi_{j,v}^H \colon H^1_\mathrm{cts}(F_v, \mathcal{T}^*_{\eta,
 (j-1)})/x_jH^1_\mathrm{cts}(F_v, \mathcal{T}^*_{\eta, (j-1)})
 \rightarrow H^1_\mathrm{cts}(F_v, \mathcal{T}^*_{\eta,(j)})
\end{align*}
induced by the cohomological long exact sequence associated to the sequence $(\ref{equation:T}$). 

\begin{lem} \label{lem:local_ses}
For each $v$ in $S$, the map $\pi_{j,v}^H$ induces an injection 
\begin{align*}
\pi^Q_{j,v} \colon Q_{\mathcal{L}^{*,(j-1)}_\mathrm{str}}(F_v,
 \mathcal{T}^*_{\eta,(j-1)})/x_jQ_{\mathcal{L}^{*,(j-1)}_\mathrm{str}}(F_v,
 \mathcal{T}^*_{\eta, (j-1)}) \longrightarrow
 Q_{\mathcal{L}^{*,(j)}_\mathrm{str}}(F_v,
 \mathcal{T}^*_{\eta, (j)})
\end{align*}
with a cokernel isomorphic to $H^2_\mathrm{cts}(F_v, \mathcal{T}^*_{\eta,
 (j-1)})[x_j]$. 
\end{lem}

\begin{proof}
In this proof, $k$ denotes either $j-1$ or $j$. 
There is nothing to prove 
for a place $\mathfrak{P}^c$ in $\Sigma_p^c$ since the local quotient
$Q_{\mathcal{L}^{*,(k)}_{\mathrm{str}}}(F_{\mathfrak{P}^c},
 \mathcal{T}^{*}_{\eta, (k)})$ itself is trivial. 
For a place $\mathfrak{P}$ in $\Sigma_p$, the local quotient 
$Q_{\mathcal{L}^{*,(k)}_\mathrm{str}}(F_\mathfrak{P},
  \mathcal{T}^{*}_{\eta, (k)})$ coincides with 
the whole local cohomology group $H^1_\mathrm{cts}(F_\mathfrak{P},
 \mathcal{T}^{*}_{\eta, (k)})$ by definition, and thus we define
 the desired injection $\pi^Q_{j, \mathfrak{P}}$ to be $\pi^H_{j,\mathfrak{P}}$ itself. 
Now let us consider a place $\lambda$ in $S$ which does not divide $p$. 
If the inertia subgroup $I_\lambda$ acts nontrivially on
 $\mathcal{T}^{*}_{\eta, (k)}$, or in other words, if
 $\eta^\mathrm{gal}$ is ramified at $\lambda$, the
 $I_\lambda$-invariant submodule of $\mathcal{T}^{*}_{\eta, (k)}$ is
 trivial; it is because every element $g$ of $I_\lambda$ acts as multiplication of the nontrivial element $\eta^\mathrm{gal}(g^{-1})$ on each component of the rank-one free $\Lambda_\mathcal{O}^\mathrm{CM}/\mathfrak{A}_k$-module $\mathcal{T}_{\eta,(k)}^*$, which is torsionfree as a $\Lambda_0^{(k)}$-module. This observation implies the triviality of the unramified cohomology 
 group $H^1_\mathrm{ur}(F_\lambda, \mathcal{T}^{*}_{\eta, (k)})$,
 and thus the desired map
  $\pi^Q_{j,\lambda}$ should be defined as $\pi^H_{j,\lambda }$. Finally
 assume that $I_\lambda$ acts trivially on $\mathcal{T}^*_{\eta,(k)}$, or in other words, assume that  
 $\eta^\mathrm{gal}$ is unramified at $\lambda$. Then the unramified
 cohomology group $H^1_\mathrm{ur}(F_\lambda, \mathcal{T}^{*}_{\eta,(k)})$ is just the first continuous cohomology group $H^1_\mathrm{cts}(D_\lambda/I_\lambda, \mathcal{T}^{*}_{\eta,(k)})$ of the procyclic group $D_\lambda/I_\lambda$, and thus the surjection $\pi_j^\mathcal{T}\colon \mathcal{T}_{\eta,(j-1)}^*\rightarrow \mathcal{T}_{\eta,(j)}^*$ induces an injection
 \begin{align*}
  \pi^\mathrm{ur}_{j,\lambda} \colon H^1_\mathrm{ur}(F_\lambda, \mathcal{T}^{*}_{\eta,(j-1)})/x_jH^1_\mathrm{ur}(F_\lambda, \mathcal{T}^{*}_{\eta,(j-1)}) \longrightarrow H^1_\mathrm{ur}(F_\lambda, \mathcal{T}^{*}_{\eta,(j)})
 \end{align*}
 with a cokernel isomorphic to $H^2_\mathrm{cts}(D_\lambda/I_\lambda, \mathcal{T}^{*}_{\eta,(j-1)})[x_j]$ due to the cohomological long exact sequence associated to (\ref{equation:T}) regarded as the short exact sequences of continuous $D_\lambda/I_\lambda$\nobreakdash-modules. However the procyclic group $D_\lambda/I_\lambda\cong \hat{\mathbb{Z}}$ has cohomological dimension one, and hence the the cokernel of $\pi^\mathrm{ur}_{j,\lambda}$ should vanish; in other words the injection $\pi^\mathrm{ur}_{j,v}$ is indeed an isomorphism. 
 We therefore obtain a commutative diagram with exact rows
\begin{align} \label{diag:beta}
\xymatrix{
 & H^{1,\mathrm{ur},*}_{\lambda,(j-1)}/x_jH^{1,\mathrm{ur},*}_{\lambda,(j-1)} \ar[r] \ar[d]^{\pi^\mathrm{ur}_{j,\lambda}}_{\rotatebox{90}{$\sim$}} & H^{1,*}_{\lambda, (j-1)}/x_jH^{1,*}_{\lambda,(j-1)} \ar[r] \ar[d]^{\pi^H_{j,\lambda}} & Q^*_{\lambda,(j-1)}/x_jQ^*_{\lambda,(j-1)} \ar[r] \ar@{.>}[d]^{\pi^Q_{j,\lambda}} & 0 \\
 0 \ar[r] & H^{1,\mathrm{ur},*}_{\lambda,(j)} \ar[r] & H^{1,*}_{\lambda,(j)} \ar[r] & Q^*_{\lambda,(j)}  \ar[r] & 0. 
}
\end{align}
 Here $H^1_\mathrm{cts}(F_\lambda, \mathcal{T}^*_{\eta,(k)})$, $H^1_\mathrm{cts}(D_\lambda/I_\lambda, \mathcal{T}^*_{\eta,(k)})$ and $Q_{\mathcal{L}^{*,(k)}_\mathrm{str}}(F_\lambda, \mathcal{T}^{*}_{\eta,(k)})$ are abbreviated as $H^{1,*}_{\lambda,(k)}$,  $H^{1,\mathrm{ur},*}_{\lambda,(k)}$ and $Q^*_{\lambda,(k)}$ respectively.
The commutativity of the left square in (\ref{diag:beta})is due to the functoriality of the inflation map $H^{1,\mathrm{ur},*}_{\lambda,(k)} \rightarrow H^{1,*}_{\lambda,(k)}$, by virtue of which $\pi^H_{j,\lambda}$ induces a map $\pi^Q_{j,\lambda}$ on 
 the quotient modules (the dotted arrow in (\ref{diag:beta})). Moreover
 the cokernel of $\pi^Q_{j,\lambda}$ is isomorphic to that of $\pi^H_{j,\lambda}$ due to the isomorphy of the left vertical map $\pi^\mathrm{ur}_{j,\lambda}$.
 This is the end of the proof since the cokernel of $\pi^H_{j,\lambda}$ is isomorphic to $H^2(F_\lambda, \mathcal{T}^{*}_{\eta,(j-1)})[x_j]$ due to the cohomological long exact sequence associated to (\ref{equation:T}).
 \end{proof}  
 
 Set $\pi^Q_j$ as $\pi^Q_j =(\pi^Q_{j,v})_{v\in S}$ and consider the diagram

\begin{align} \label{dg:big}
 \text{\footnotesize $\xymatrix{
  & & & & 0\ar[d] \\
 & & 0\ar[d] & 0\ar[d] &
  \mathrm{Im}\, (\tilde{\delta}_1^\vee)
 \ar[d] \\ 
  &0 \ar[r] & H^{1,*}_{(j-1)}/x_jH^{1,*}_{(j-1)} \ar[r] \ar[d]_{\pi^H_j} \ar@{}[dr]|\circlearrowright
 & Q^*_{(j-1)}/x_jQ^*_{(j-1)}
 \ar[r] \ar[d]^{\pi^Q_j} &
 \mathcal{C}^{(j-1)} \ar[r] \ar[d]^{r^\mathcal{C}_j}
 & 0 \\
0\ar[r] & \mathrm{Sel}^{*,\mathrm{str}}_{\mathcal{T}^*_{\eta,(j)}} \ar[r] & H^{1,*}_{(j)}
 \ar[r]_{\phi_{\mathcal{L}^{*,(j)}_\mathrm{str}}} \ar[r]
 \ar[d]_{\text{\sf \textgreek{d}}^1_{(j)}}
 &
 Q^*_{(j)}
 \ar[r] \ar[d] &
 \mathrm{Coker}(\phi_{\mathcal{L}^{*,(j)}_\mathrm{str}}) \ar[r] 
 & 0 \\
  & &
  H^{2,*}_{(j-1)}[x_j]
  \ar[d] &\displaystyle \prod_{v\in S}
 H^{2,*}_{v,(j-1)}[x_j] \ar[d] & & 
 \\
  & & 0 & 0 & &  }$}
\end{align}
where the left vertical sequence is the one obtained at ($\ref{ses:cokernel_xj}$) 
and we use the following abbreviations for each $v$ in $S$,
$i=1,2$ and $k=j-1,j$:
\begin{equation}
 \mathrm{Sel}^{*,\mathrm{str}}_{\mathcal{T}^*_{\eta,(j)}} :=
 \mathrm{Sel}_{\mathcal{L}^{*,(j)}_\mathrm{str}}(F,\mathcal{T}^*_{\eta,(j)}), \ \ \  
 H^{i,*}_{v,(k)}:=H^i_\mathrm{cts}(F_v,
 \mathcal{T}^*_{\eta,(k)}). 
\end{equation}
in addition to the abbreviations ($\ref{equation:definition_C}$) and ($\ref{equation:definition_HQ}$). 
The top and bottom rows of (\ref{dg:big}) are the sequences (\ref{ses:cokernel_xj}) and (\ref{ses:dualSel_def}) for $k=j$ respectively, which are hence exact.
The right column is the exact sequence (\ref{ses:Sha^1}). Here we replace $(\mathcyr{Sh}^1_{\mathcal{A}_{\eta,(j-1)}})^\vee[x_j]$ in (\ref{ses:Sha^1}) by $\mathrm{Im}(\tilde{\delta}^\vee_1)$ for later convenience in Appendix~\ref{app:duality}. The other columns are also exact due to the cohomological long exact sequence and Lemma~\ref{lem:local_ses}. Since both the maps $\pi_j^H$ and $\pi_j^Q$ are induced from the canonical surjection $\pi_j^\mathcal{T}\colon \mathcal{T}_{\eta,(j-1)}^*\rightarrow \mathcal{T}_{\eta,(j)}^*$ of the continuous Galois modules, the functoriality of the restriction maps guarantees the commutativity of the left square. We here verify the commutativity of the right square in (\ref{dg:big}).
 
 \begin{lem} \label{lem:commutes}
  The right square in the diagram $(\ref{dg:big})$ commutes.
 \end{lem}
 
 \begin{proof}
  Consider the diagram 
  \begin{align} \label{dg:big_commutes}
   \xymatrix{
   Q^*_{(j-1)}/x_jQ^*_{(j-1)} \ar[r] \ar[d]_{\pi^Q_j} & \mathcal{C}^{(j-1)} \ar@{^(->}[r] \ar[d]^{r^{\mathcal{C}}_j} \ar@{}[rd]|{\text{\large $\circlearrowright$}} & (\mathrm{Sel}^{\Sigma,\mathrm{str}}_{\mathcal{A}_{\eta,(j-1)}}[x_j])^\vee \ar[d]^{(\iota^{\mathrm{Sel}}_j)^\vee} \\
   Q^*_{(j)} \ar[r]  & \mathrm{Coker}(\phi_{\mathcal{L}^{*,(j)}_{\mathrm{str}}}) \ar@{^(->}[r] &  (\mathrm{Sel}^{\Sigma,\mathrm{str}}_{\mathcal{A}_{\eta,(j)}})^\vee
   }
  \end{align}
 obtained as the composition of the right square in (\ref{dg:big}) and the left commutative square in (\ref{ses:comp_coker}). It suffices to prove that the composite square (\ref{dg:big_commutes}) commutes since the morphism  $\mathrm{Coker}(\phi_{\mathcal{L}^{*,(j)}_{\mathrm{str}}}) \rightarrow (\mathrm{Sel}^{\Sigma,\mathrm{str}}_{\mathcal{A}_{\eta,(j)}})^\vee$ in the bottom row is injective. Now recall that we have canonical isomorphisms $Q^*_{(j-1)}/x_jQ^*_{(j-1)}\cong (L_{\mathrm{str},(j-1)}[x_j])^\vee$ and $Q^*_{(j)}\cong (L_{\mathrm{str},(j)})^\vee$ via the local Tate duality, where we set
  \begin{align*}
   L_{\mathrm{str},(k)} =\prod_{v\in S}L_\mathrm{str}(F_v, \mathcal{A}_{\eta,(k)})=\prod_{\substack{v\in S \\ v\notin \Sigma_p \cup \Sigma_p^c}} H^1_\mathrm{ur}(F_v, \mathcal{A}_{\eta,(k)})\times \prod_{v\in \Sigma_p} H^1(F_v, \mathcal{A}_{\eta,(k)}) 
  \end{align*}
  for $k=j-1$ and $j$. By construction both of the compositions
   \begin{align*}
    &\xymatrix{
    (L_{\mathrm{str}, (j-1)}[x_j])^\vee \cong Q^*_{(j-1)}/x_jQ^*_{(j-1)} \ar[r] & \mathcal{C}^{(j-1)} \ar[r] & (\mathrm{Sel}^{\Sigma,\mathrm{str}}_{\mathcal{A}_{\eta,(j-1)}}[x_j])^\vee}, \\
    &\xymatrix{
    (L_{\mathrm{str},(j)})^\vee \cong Q^*_{(j)} \ar[r] & \mathrm{Coker}(\phi_{\mathcal{L}^{*,(j)}_{\mathrm{str}}}) \ar[r] & (\mathrm{Sel}^{\Sigma,\mathrm{str}}_{\mathcal{A}_{\eta,(j)}})^\vee
    }
   \end{align*}
 of the rows in the diagram (\ref{dg:big_commutes}) with the local Tate duality isomorphisms are induced from the dual of the global-to-local restriction map $\phi_{\mathcal{L}_{\mathrm{str}}^{(k)}} \colon \mathrm{Sel}^{\Sigma,\mathrm{str}}_{\mathcal{A}_{\eta,(k)}} \rightarrow L_{\mathrm{str},(k)}$. Let us define
  $\iota^L_j \colon L_{\mathrm{str},(j-1)}\rightarrow L_{\mathrm{str},(j)}$ as
  the homomorphism induced from the canonical inclusion  $\iota^{\mathcal{A}}_j \colon \mathcal{A}_{\eta,(j)} \hookrightarrow \mathcal{A}_{\eta,(j-1)}$. Then by virtue of the functoriality of the restriction maps, we
  readily obtain a commutative square
  \begin{align*}
   \xymatrix{
    L_{\mathrm{str},(j-1)}[x_j]  && \ar[ll]_{\phi_{\mathcal{L}_\mathrm{str}^{(j-1)}}}  \mathrm{Sel}^{\Sigma,\mathrm{str}}_{\mathcal{A}_{\eta,(j-1)}}[x_j] \\
    L_{\mathrm{str},(j)} \ar[u]^{\iota^L_j} && \mathrm{Sel}^{\Sigma,\mathrm{str}}_{\mathcal{A}_{\eta,(j)}} \ar[u]_{\iota^{\mathrm{Sel}}_j} \ar[ll]^{\phi_{\mathcal{L}^{(j)}_{\mathrm{str}}}}
   }
  \end{align*}
  whose Pontrjagin dual coincides with (\ref{dg:big_commutes}). 
  Indeed it is straightforward, from the construction of the local Tate pairing as the usual cup product of the group cohomology, to verify that the morphism $\pi_j^Q$ corresponds to the dual 
  $(\iota^L_j)^\vee \colon (L_{\mathrm{str},(j-1)}[x_j])^\vee \rightarrow (L_{\mathrm{str},(j)})^\vee$ of $\iota^L_j$ with respect to the local Tate duality, and thus the proof is completed.  
  Note that $\iota^\mathcal{A}_j$ and $\pi^\mathcal{T}_j$ correspond to each other under the Pontrjagin duality.
 \end{proof}
 
 Now let us consider the coimage $\mathrm{Coim}(\phi_{\mathcal{L}^{*,(j)}_\mathrm{str}})=H^{1,*}_{(j)}/\mathrm{Sel}^{*,\mathrm{str}}_{\mathcal{T}^*_{\eta,(j)}}$ of the global-to-local morphism $\phi_{\mathcal{L}^{*,(j)}_\mathrm{str}}$ and split the bottom row of (\ref{dg:big}) (or in other words, the exact sequence (\ref{ses:dualSel_def}) for $k=j$) into two short exact sequences:
 \begin{align*}
  \xymatrix @R=-.5pt{
  0\ar[r] & \mathrm{Sel}^{*,\mathrm{str}}_{\mathcal{T}^*_{\eta,(j)}} \ar[r] & H^{1,*}_{(j)}
 \ar[rr]^{\phi_{\mathcal{L}^{*,(j)}_\mathrm{str}}} \ar[dr]&
 &  
 Q^*_{(j)}
 \ar[r]  &
 \mathrm{Coker}(\phi_{\mathcal{L}^{*,(j)}_\mathrm{str}}) \ar[r] 
  & 0 
\\
  &  &    & \text{\footnotesize $\mathrm{Coim}(\phi_{\mathcal{L}^{*,(j)}_\mathrm{str}})$}  \ar[ur] \ar[dr] &  & & 
\\
    &  &  0 \ar[ur]& & 0 && }.
 \end{align*}
 We define $\bar{\pi}^H_j \colon H^{1,*}_{(j-1)}/x_jH^{1,*}_{(j-1)} \rightarrow \mathrm{Coim}(\phi_{\mathcal{L}^{*,(j)}_\mathrm{str}})$ as the composition of $\pi^H_j$ with the natural surjection $H^{1,*}_{(j)} \twoheadrightarrow \mathrm{Coim}(\phi_{\mathcal{L}^{*,(j)}_\mathrm{str}})$, and define $\overline{\mathcal{H}}^{2,*}_j$ as the cokernel of $\bar{\pi}^H_{j}$. Then diagram (\ref{dg:big}) splits into the two diagrams
 \begin{align} \label{dg:left}
  \xymatrix{
  &  &  0 \ar[d] &  0 \ar[d]  &  \\
  & 0\ar[r] & H^{1,*}_{(j-1)}/x_jH^{1,*}_{(j-1)} \ar[d]_{\pi^H_j} \ar[r]^= & H^{1,*}_{(j-1)}/x_jH^{1,*}_{(j-1)} \ar[r] \ar[d]^{\bar{\pi}^H_j} & 0 \\
  0 \ar[r] & \mathrm{Sel}^{*,\mathrm{str}}_{\mathcal{T}^*_{\eta,(j)}} \ar[r] &  H^{1,*}_{(j)} \ar[r] \ar[d]_{\text{\sf \textgreek{d}}^1_{(j)}} & \mathrm{Coim}(\phi_{\mathcal{L}^{*,(j)}_\mathrm{str}}) \ar[r]  \ar[d] & 0 \\  
  &  &  H^{2,*}_{(j-1)}[x_j] \ar[r] \ar[d] & \overline{\mathcal{H}}^{2,*}_j \ar[d] \ar[r] & 0 \\
  &  & 0  & 0 & 
  }
 \end{align}
 and
 \begin{align} \label{dg:right}
  \xymatrix{
 &  &  &  0 \ar[d] &  \\ 
  & 0 \ar[d] & 0\ar[d]  &  \mathrm{Im}(\tilde{\delta}^\vee_1) \ar[d] & \\
  0 \ar[r ]& H^{1,*}_{(j-1)}/x_j H^{1,*}_{(j-1)} \ar[r] \ar[d]_{\bar{\pi}^H_j} & Q^*_{(j-1)}/x_j Q^*_{(j-1)} \ar[d]^{\pi^Q_j} \ar[r] & \mathcal{C}^{(j-1)}\ar[d]^{r^\mathcal{C}_j} \ar[r] & 0 \\
  0 \ar[r] & \mathrm{Coim}(\phi_{\mathcal{L}^{*,(j)}_\mathrm{str}}) \ar[d] \ar[r]^(0.56){\phi_{\mathcal{L}^{*,(j)}_\mathrm{str}}} & Q^*_{(j)} \ar[r] \ar[d] & \mathrm{Coker}(\phi_{\mathcal{L}^{*,(j)}_\mathrm{str}}) \ar[r] & 0 \\
  & \overline{\mathcal{H}}^{2,*}_j \ar[r] \ar[d] & \displaystyle \prod_{v\in S} H^{2,*}_{v,(j-1)} \ar[d] & & \\
  & 0 & 0 &  &
  }
 \end{align}
 with exact rows and columns. By applying the snake lemma to both the diagrams (\ref{dg:left}) and (\ref{dg:right}), we obtain two exact sequences
 \begin{align} \label{ses:dualSelvsH^2}
  \xymatrix{0\ar[r] &
  \mathrm{Sel}_{\mathcal{L}^{*,(j)}_\mathrm{str}}(F,\mathcal{T}^*_{\eta,(j)})
  \ar[r] & H^{2,*}_{(j-1)}[x_j] \ar[r]^(0.64){\mathrm{pr}_{H^2}} & \overline{\mathcal{H}}^{2,*}_j \ar[r]& 0}
 \end{align}
 and
 \begin{align} \label{ses:snake} 
  \xymatrix{0\ar[r] & \mathrm{Im}\, (\tilde{\delta}^\vee_1) \ar[r]^{\delta_2} & \overline{\mathcal{H}}^{2,*}_j\ar[r] & \displaystyle \prod_{v\in S} H^2_\mathrm{cts}(F_v, \mathcal{T}^*_{\eta,(j-1)})[x_j]  \\ 
 &
   (\mathcyr{Sh}^1_{\mathcal{A}_{\eta, (j-1)}})^\vee [x_j]
  \ar[u]^{\tilde{\delta}^\vee_1}_{\rotatebox{90}{$\sim$}}
  \ar[ur]_{\boldsymbol{\delta}_{\mathrm{snake}}} & & }. 
 \end{align}
 Here we use the symbol ``$\boldsymbol{\delta}_\mathrm{snake}$'' for
 the composition $\delta_2\circ \tilde{\delta}^\vee_1$ to indicate
 that it is constructed as the composition of the connecting
 homomorphisms derived from the snake lemma. 
 Meanwhile
 the Poitou-Tate nine-term exact sequence (see \cite[(8.6.10)]{NSW}
 for example) provides the exact sequence
  \begin{equation*} 
 \xymatrix{ \displaystyle \prod_{v\in S} H^0_\mathrm{cts}(F_v, \mathcal{T}^*_{\eta,(j-1)})^\vee \longrightarrow  
 (H^{2,*}_{(j-1)})^\vee \xrightarrow{\Phi_\mathrm{PT}^\vee}  \mathcyr{Sh}^1_{\mathcal{A}_{\eta,(j-1)}} \longrightarrow  0}.
  \end{equation*}
  
  By taking the Pontrjagin duals and the $x_j$-torsion parts,
  we obtain the exact sequence
  \begin{equation} \label{ses:PT}
   0 \longrightarrow  
   (\mathcyr{Sh}^1_{\mathcal{A}_{\eta,(j-1)}})^\vee[x_j]
   \xrightarrow{\Phi_\mathrm{PT}} 
   H^{2,*}_{(j-1)}[x_j]
   \xrightarrow{\mathrm{Res}^{2,*}_{(j-1)}}  \displaystyle \prod_{v\in S} H^{2,*}_{(j-1)}[x_j]
  \end{equation}
  where $\mathrm{Res}^{2,*}_{(j-1)}$ denotes the usual localisation morphism.
  
  Recall that the homomorphism $\Phi_\mathrm{PT}$, which is often
  called the {\em Poitou-Tate morphism}, is the one induced from the Poitou-Tate
  pairing. See also (\ref{eq:dual_hom}) in Appendix~\ref{app:duality}
  for the definition of $\Phi_{PT}$. 
  The following proposition, whose proof is slightly lengthy and is postponed to Appendix~\ref{app:duality}, 
  is a technical heart of the proof of Proposition~\ref{prop:dSelvan}. 
  
  \begin{pro} \label{prop:commutes}
   Consider the diagram 
   \begin{equation} \label{dg:PTvsSnake}
    \begin{CD} 
    0 @>>> (\mathcyr{Sh}^1_{\mathcal{A}_{\eta,(j-1)}})^\vee[x_j] @>{\Phi_{\mathrm{PT}}}>>  
    H^{2,*}_{(j-1)} [x_j] @>{\mathrm{Res}^{2,*}_{(j-1)}}>>  \displaystyle \prod_{v\in S} 
    H^2_\mathrm{cts}(F_v, \mathcal{T}^*_{\eta,(j-1)})[x_j] \\
    @. @| @VV{\mathrm{pr}_{H^2}}V @| \\ 
    0 @>>> (\mathcyr{Sh}^1_{\mathcal{A}_{\eta,(j-1)}})^\vee[x_j] @>>{\boldsymbol{\delta}_\mathrm{snake}}> \overline{\mathcal{H}}^{2,*}_j  @>>>  
    \displaystyle \prod_{v\in S} H^2_\mathrm{cts}(F_v, \mathcal{T}^*_{\eta,(j-1)})[x_j] 
   \end{CD}
   \end{equation}
   where the top and bottom rows are the exact sequences
   $(\ref{ses:PT})$ and $(\ref{ses:snake})$ respectively, and the middle vertical morphism
   $\mathrm{pr}_{H^2}$ is the quotient map defined in the short exact sequence $(\ref{ses:dualSelvsH^2})$. Then the diagram $(\ref{dg:PTvsSnake})$ commutes.
  \end{pro}
  
  The commutativity of the right square in the diagram (\ref{dg:PTvsSnake}) is obvious since both of the horizontal morphisms are induced from the
  global-to-local morphism
  \begin{align*}
 \mathrm{Res}^{2,*}_{(j-1)} \colon   H^{2,*}_{(j-1)} \rightarrow \prod_{v\in S} H^2_\mathrm{cts}(F_v, \mathcal{T}^*_{\eta,(j-1)}).
  \end{align*}
  Concerning the left square in (\ref{dg:PTvsSnake}), the top horizontal
  morphism $\Phi_{\mathrm{PT}}$ is the one induced
  from the Poitou-Tate pairing while the bottom horizontal morphism $\boldsymbol{\delta}_{\mathrm{snake}}$ is
  the composition of connecting morphisms derived from the snake lemma
  and the local Tate duality. Therefore  we must carefully study
  the relation between the global Poitou-Tate duality and the local Tate duality
  to verify the commutativity of the left square.
  As we shall show below, Proposition~$\ref{prop:dSelvan}$ follows immediately from Proposition~\ref{prop:commutes}.

\begin{proof}[Proof of Proposition~$\ref{prop:dSelvan}$ $($admitting Proposition~$\ref{prop:commutes})$]
  We observe  from the short exact sequence
 (\ref{ses:dualSelvsH^2}) that the kernel of the middle vertical morphism $\mathrm{pr}_{H^2}$ in
 (\ref{dg:PTvsSnake}) is isomorphic to the dual Selmer group
  $\mathrm{Sel}_{\mathcal{L}^{*,(j)}_{\mathrm{str}}}(F,\mathcal{T}^*_{\eta,(j)})$.
  We now readily verify
  its triviality by an easy diagram chase on (\ref{dg:PTvsSnake}).
 \end{proof}

\subsubsection{Final step$:$ specialisation from two variables to one
   variable} \label{sssc:final_step}

Suppose that we have already chosen elements $\gamma_1, \dotsc,
\gamma_{d+\delta_{F,p}-1}$ of $\mathrm{Gal}(F_S/F^\mathrm{cyc}_\infty)$ which satisfy 
all the conditions $(\Gamma0)_{d+\delta_{F,p}-1}$,
$(\Gamma1)_{d+\delta_{F,p}-1}$ and $(\Gamma2)_{d+\delta_{F,p}-1}$
proposed in Section~\ref{sssc:int_step} (the existence of such
tuples $\gamma_1, \dotsc, \gamma_{d+\delta_{F,p}-1}$ is justified in
Proposition~\ref{prop:induction}).
As the final step of the proof of Theorem~\ref{thm:specialisation}, 
we discuss the cyclotomic specialisation of the 
strict Selmer group
$\mathrm{Sel}^{\Sigma,\mathrm{str}}_{\mathcal{A}^{\mathrm{CM}}_\eta[\mathfrak{A}_{d+\delta_{F,p}-1}]}$. We
have already verified in Section~\ref{sssc:int_step} that $\mathrm{Sel}^{\Sigma,\mathrm{str}}_{\mathcal{A}^{\mathrm{CM}}_\eta[\mathfrak{A}_{d+\delta_{F,p}-1}]}$
is cotorsion and almost divisible as a
$\Lambda_0^{(d+\delta_{F,p}-1)}$-module
and the equality (\ref{eq:spec_ind}) holds.
 In this situation the kernel of the natural surjection
$\Lambda^{\mathrm{CM}}_{\mathcal{O}}/\mathfrak{A}_{d+\delta_{F,p}-1}
\twoheadrightarrow \Lambda^{\mathrm{cyc}}_{\mathcal{O}}$ is a principal
ideal of
$\Lambda^{\mathrm{CM}}_{\mathcal{O}}/\mathfrak{A}_{d+\delta_{F,p}-1}$
generated by $x_{d+\delta_{F,p}}=\gamma_{d+\delta_{F,p}}-1$, where
$\gamma_{d+\delta_{F,p}}$ is a lift of a topological generator of
$\mathrm{Gal}(\widetilde{F}/F^\mathrm{cyc}_\infty)/\langle \gamma_1,
\dotsc, \gamma_{d+\delta_{F,p}-1}\rangle \cong \mathbb{Z}_p$ to 
$\mathrm{Gal}(\widetilde{F}/F^\mathrm{cyc}_\infty)$.

We henceforth denote
$\Lambda^{\mathrm{CM}}_{\mathcal{O}}/\mathfrak{A}_{d+\delta_{F,p}-1}$ by 
$\Lambda_\mathcal{O}^{(d+\delta_{F,p}-1)}$ to simplify the notation. 
In order to apply the Specialisation Lemma~\ref{lem:specialisation} to 
the finitely generated torsion $\Lambda_\mathcal{O}^{(d+\delta_{F,p}-1)}$-module $(\mathrm{Sel}^{\Sigma,
\mathrm{str}}_{\mathcal{A}^{\mathrm{CM}}_\eta[\mathfrak{A}_{d+\delta_{F,p}-1}]})^\vee$ and the prime ideal $(x_{d+\delta_{F,p}})$
of $\Lambda_0^{(d+\delta_{F,p}-1)}$ of
height one, we must verify that $x_{d+\delta_{F,p}}$ does not divide 
the characteristic ideal 
of $(\mathrm{Sel}^{\Sigma, \mathrm{str}}_{\mathcal{A}^{\mathrm{CM}}_\eta[\mathfrak{A}_{d+\delta_{F,p}-1}]})^\vee$. 
Until the previous steps we could choose a specialising element 
$x_j=\gamma_j-1$ avoiding the prime divisors of the
characteristic power series of the Pontrjagin dual of 
the strict Selmer group. 
At this final step, however, the specialising element $x_{d+\delta_{F,p}}$ 
(or more precisely, the principal ideal
$x_{d+\delta_{F,p}}\Lambda_0^{(d+\delta_{F,p}-1)}$ which it generates) is uniquely determined, and we
are not allowed to choose it freely. 
We still require the following claim:

\begin{claim}
Let the notation be as before and let us assume the conditions (IMC$_{F,\psi}$) 
and (NV$_{\mathcal{L}^\mathrm{cyc}_p(\vartheta(\eta))}$) (see Theorem~\ref{thm:specialisation}
 for details on these conditions). Then the specialising element $x_{d+\delta_{F,p}}$ is relatively prime 
to the characteristic ideal of the Pontrjagin dual of the $\Lambda_{\mathcal{O}}^{(d+\delta_{F,p}-1)}$-module 
 $\mathrm{Sel}^{\Sigma,
 \mathrm{str}}_{\mathcal{A}^{\mathrm{CM}}_\eta[\mathfrak{A}_{d+\delta_{F,p}-1}]}$. 
\end{claim}

\begin{proof}[Proof of the Claim]
In order to prove the claim,
 we assume that $x_{d+\delta_{F,p}}$ {\em does divide} the characteristic
 ideal of $(\mathrm{Sel}^{\Sigma, \mathrm{str}}_{\mathcal{A}^{\mathrm{CM}}_\eta[\mathfrak{A}_{d+\delta_{F,p}-1}]})^\vee$
 and deduce contradiction. The
 cyclotomic specialisation of the characteristic ideal
{\small 
\begin{multline} \label{eq:contrapositive}
(\mathrm{Char}_{\Lambda_\mathcal{O}^{\mathrm{CM}}}(\mathrm{Sel}^{\Sigma,
 \mathrm{str}}_{\mathcal{A}^{\mathrm{CM}}_\eta})^\vee)
 \otimes_{\Lambda_\mathcal{O}^{\mathrm{CM}}}
 \Lambda_\mathcal{O}^{\mathrm{cyc}}  \\
=  (\mathrm{Char}_{\Lambda_{\mathcal{O}}^{(d+\delta_{F,p}-1)}}
  (\mathrm{Sel}^{\Sigma,
 \mathrm{str}}_{\mathcal{A}^{\mathrm{CM}}_\eta[\mathfrak{A}_{
 d+\delta_{F,p}-1}]})^\vee)
 \otimes_{\Lambda_{\mathcal{O}}^{(d+\delta_{F,p}-1)}}
 \Lambda_{\mathcal{O}}^{(d+\delta_{F,p}-1)}/x_{d+\delta_{F,p}}\Lambda_{\mathcal{O}}^{(d+\delta_{F,p}-1)}
\end{multline}}
is then trivial. 
 On the other hand, we obtain the equality of ideals of $\widehat{\mathcal{O}}^\mathrm{ur}[[\mathrm{Gal}(\widetilde{F}/F)]]$
\begin{align} \label{eq:IMC_eta}
\mathrm{Char}_{\Lambda^{\mathrm{CM}}_\mathcal{O}}(\mathrm{Sel}^{\Sigma,\mathrm{str}}_{\mathcal{A}^{\mathrm{CM}}_\eta})^\vee=\mathrm{Tw}_{\eta^\mathrm{gal}\psi^{-1}}(\mathrm{Char}_{\Lambda^{\mathrm{CM}}_\mathcal{O}}(\mathrm{Sel}^{\Sigma,\mathrm{str}}_{\mathcal{A}^{\mathrm{CM}}_\psi})^\vee)=(\mathrm{Tw}_{\eta^\mathrm{gal}\psi^{-1}}(\mathcal{L}^\Sigma_p(\psi)))
\end{align}
by the Iwasawa main conjecture (IMC$_{F,\psi}$) for $F$
 and $\psi$. We thus see from the equations (\ref{eq:contrapositive})
 and (\ref{eq:IMC_eta}) that 
 the image $\mathcal{L}^\mathrm{cyc}_{p,\mathrm{CM}}(\eta)$ of
 $\mathrm{Tw}_{\eta^\mathrm{gal}\psi^{-1}}(\mathcal{L}_p^\Sigma(\psi))$
 in the Iwasawa algebra  $\widehat{\mathcal{O}}^\mathrm{ur}[[\mathrm{Gal}(F^+(\mu_{p^\infty})/F^+)]]$ is trivial. On the other hand, we 
 have already observed in Corollary~\ref{cor:compare_p-adic_L} that the element $\mathcal{L}^\mathrm{cyc}_{p,\mathrm{CM}}(\eta)$ is a nonzero multiple of the cyclotomic 
 $p$\nobreakdash-adic $L$-function $\mathcal{L}_p^\mathrm{cyc}(\vartheta(\eta))$
 of $\vartheta(\eta)$ in each component of the semilocal Iwasawa algebra $\widehat{\mathcal{O}}^\mathrm{ur}[[\mathrm{Gal}(F^+(\mu_{p^\infty})/F^+)]]\otimes_{\mathbb{Z}_p}\mathbb{Q}_p$. The nontriviality assumption 
 (NV$_{\mathcal{L}^\mathrm{cyc}_p(\vartheta(\eta))}$) thus leads us to contradiction,
 which completes the proof of the claim.
\end{proof}

 Due to the claim, we can apply the Specialisation Lemma~\ref{lem:specialisation} to $(\mathrm{Sel}^{\Sigma, \mathrm{str}}_{\mathcal{A}^{\mathrm{CM}}_\eta[\mathfrak{A}_{d+\delta_{F,p}-1}]})^\vee$ 
 and $x_{d+\delta_{F,p}}\Lambda_0^{(d+\delta_{F,p}-1)}$. We thus observe that
 the strict Selmer group $\mathrm{Sel}^{\Sigma,\mathrm{str}}_{\mathcal{A}^{\mathrm{cyc}}_\eta}$ of the cyclotomic deformation $\mathcal{A}_\eta^\mathrm{cyc}$ of $\eta$ is cotorsion as a $\Lambda^{\mathrm{cyc}}_\mathcal{O}$-module, and
 and obtain the basechange compatibility of 
the characteristic ideal with respect to the cyclotomic specialisation: 
\begin{align*}
 (\mathrm{Char}_{\Lambda_\mathcal{O}^{\mathrm{CM}}}(\mathrm{Sel}^{\Sigma, \mathrm{str}}_{\mathcal{A}^{\mathrm{CM}}_\eta})^\vee)
  \otimes_{\Lambda_\mathcal{O}^{\mathrm{CM}}}
  \Lambda_\mathcal{O}^{\mathrm{cyc}} &=
  \mathrm{Char}_{\Lambda_\mathcal{O}^{\mathrm{cyc}}}
  (\mathrm{Sel}^{\Sigma,\mathrm{str}}_{\mathcal{A}^{\mathrm{cyc}}_\eta})^\vee  .
\end{align*}
This is the end of the proof of Theorem~\ref{thm:specialisation}.
 
%
\subsection{Application to the Iwasawa Main Conjecture} \label{sc:specialisation_IMC}
%

We shall prove the main theorem of this
article (Theorem~\ref{thm:IMC}). Firstly 
let us recall the notation and the settings.
As in Section~\ref{sc:Introduction}, let $p$ be an odd prime number 
and let $F^+$ be
a totally real number field of degree $d$ 
satisfying the unramifiedness condition (unr$_{F^+}$). 
Consider a nearly $p$-ordinary $p$-stabilised newform $f=\vartheta(\eta)$ with complex
multiplication defined on $F^+$, 
which is associated to a gr\"o{\ss}encharacter $\eta$ of
type ($A_0$) on a totally imaginary quadratic extension $F$ of $F^+$
satisfying the $p$-ordinarity condition (ord$_{F/F^+}$). We may assume that 
$\eta$ is ordinary with respect to an appropriate $p$-ordinary CM type
$\Sigma$ of $F$, which we henceforth fix. 
Finally we choose and fix a branch character $\psi$ associated to $\eta$ 
(see Definition~\ref{def:branch}).

 \begin{thm}[{Theorem~\ref{Thm:IMC}}] \label{thm:IMC}
Let the notation be as above and assume that all the following conditions
  are fulfilled$:$
\begin{itemize}
\item[-] the nontriviality condition {\upshape (ntr)$_{\mathfrak{P}}$}
      for every place $\mathfrak{P}$ of $F$ contained in $\Sigma_p^c;$
\item[-] the $(d+\delta_{F,p}+1)$-variable Iwasawa main conjecture
      {\upshape (IMC$_{F,\psi}$)} for the CM field $F$ and the branch
      character $\psi;$
\item[-] the nontriviality condition {\upshape
      (NV$_{\mathcal{L}^\mathrm{CM}_p(\vartheta(\eta))}$)} for the cyclotomic
      $p$-adic $L$-function associated to $\vartheta(\eta)$.
\end{itemize}
Then the cyclotomic Iwasawa main conjecture for the Hilbert cuspform
  $\vartheta(\eta)$ is true up to $\mu$-invariants$;$ that is, the equality
\begin{align} \label{eq:IMC_cyc_Hilb}
(\mathcal{L}^\mathrm{cyc}_p(\vartheta(\eta)))=\mathrm{Char}_{\Lambda^\mathrm{cyc}_{\mathcal{O}}}
 (\mathrm{Sel}^\vee_{\mathcal{A}^\mathrm{cyc}_{\vartheta(\eta)}})
\end{align}
holds as an equation of ideals of
  $\widehat{\mathcal{O}}^\mathrm{ur}[[\mathrm{Gal}(F^+(\mu_{p^\infty})/F^+)]]
  \otimes_{\mathbb{Z}_p} \mathbb{Q}_p$. Furthermore the equality $(\ref{eq:IMC_cyc_Hilb})$ holds as an equation
  of ideals of $\Lambda^\mathrm{cyc}_{\mathcal{O}}$ if
  Conjecture~$\ref{conj:periods}$ is true.
\end{thm}

\begin{proof}
The claim follows directly from Corollary~\ref{cor:compare_p-adic_L} and
 Theorem~\ref{thm:specialisation}.
\end{proof}

\begin{rem}[On the condition $($IMC$_{F,\psi})$] \label{rem:IMC}
 Concerning the multi-variable Iwasawa main conjecture for CM number fields
 (IMC$_{F,\psi}$), Ming-Lun Hsieh \cite{hsi-IMC} has recently obtained several results over the $(d+1)$-variable Iwasawa algebra associated to the Galois group of the compositum of the anticyclotomic $\mathbb{Z}_p^d$-extension ($d$-variable) and the cyclotomic $\mathbb{Z}_p$-extension ($1$-variable). If we assume the Leopoldt conjecture, which claim that the Leopoldt defect $\delta_{F,p}$ would equal zero, the Iwasawa algebra above coincides with $\mathcal{O}[[\mathrm{Gal}(\widetilde{F}/F)]]$ (which is isomorphic to each component of $\Lambda_\mathcal{O}^\mathrm{CM}$). Thus Hsieh's result \cite[Theorem~8.16]{hsi-IMC} combined with Leopoldt conjecture implies a one-sided divisibility relation
\begin{align*}
\mathcal{L}^\mathrm{CM}_p(\psi) \mid \mathrm{Char}_{\Lambda^{\mathrm{CM}}_\mathcal{O}}
 (\mathrm{Sel}^\vee_{\mathcal{A}_\eta^{\mathrm{CM}}}) 
\end{align*}
 in our cases under certain technical assumptions. Also \cite[Theorem~8.17, Theorem~8.18]{hsi-IMC} combined with Leopoldt conjecture implies the whole equality
 \begin{align*}
  \mathcal{L}^\mathrm{CM}_p(\psi) = \mathrm{Char}_{\Lambda^{\mathrm{CM}}_\mathcal{O}}
  (\mathrm{Sel}^\vee_{\mathcal{A}_\eta^{\mathrm{CM}}})
 \end{align*}
 holds in our cases under certain technical assumptions.
\end{rem} 

 \begin{rem}[On the nontriviality of the cyclotomic $p$-adic $L$-functions]
  We here discuss the validity of the nontriviality condition (NV$_f$) for general Hilbert cuspforms. 
 As in Theorem~\ref{thm:Hilb_L-func}, let $f$ be a normalised nearly $p$-ordinary eigencuspform in $S_\kappa (\mathfrak{N}, \underline{\varepsilon}; \overline{\mathbb{Q}})$ which is stabilised at $p$. If the region of the convergence of the (twisted) Dirichlet series
 \begin{align} \label{eq:Dirichlet_Series}
  L(f,\phi^{-1},s)=\sum_{\mathfrak{a}\subset \mathfrak{r}_{F^+}} \dfrac{C(\mathfrak{a};f)\phi^{-1}(\mathfrak{a})}{\mathcal{N}\mathfrak{a}^s}
 \end{align}
 contains at least one of  $\kappa_1^\mathrm{max}+1, \kappa_1^\mathrm{max}+2,\dotsc, \kappa_{2,\mathrm{min}}$, the cyclotomic $p$-adic $L$-function $\mathcal{L}_p^\mathrm{cyc}(f)$ associated to $f$ is obviously nontrivial. Indeed the  value $L(f,\phi,j)$ of the (complex) $L$-function at such a point never equals zero, and hence the nontriviality of $\mathcal{L}_p^\mathrm{cyc}(f)$ immediately follows from the interpolation formula \eqref{eq:interp_cusp}. The Ramanujan-Petersson conjecture for Hilbert modular forms, which was verified by Brylinski, Labesse \cite[Th\'eor\`eme~3.4.5]{BL} and Blasius \cite[Theorem~1]{Blasius}, suggests that the Dirichlet series \eqref{eq:Dirichlet_Series} absolutely converges in the region $\mathrm{Re}(s)>\dfrac{[\kappa]}{2}+1$, and thus the nontriviality of the cyclotomic $p$-adic $L$-function of $\mathcal{L}_p^\mathrm{cyc}(f)$ is automatically deduced when the inequality
 \begin{align*} 
\kappa_{2,\mathrm{min}} > \dfrac{[\kappa]}{2}+1
 \end{align*}
  holds. In contrast, it is very hard to verify that the critical values $L(f,\phi,j)$ of the (complex) $L$-function do not vanish when none of the critical points are contained in the region of convergence. For elliptic modular forms, Rohrlich has verified the nonvanishing of such critical values in general situations \cite{Rohrlich1, Rohrlich2}, and thus the condition (NV$_f$) is fulfilled for elliptic modular forms. For Hilbert modular forms, however, there have not been enough results yet to verify the condition (NV$_f$) for general $f$.
 \end{rem}

\appendix
\section{Complex multiplication of Hilbert modular cuspforms} \label{app:cm}
%

 Let $F^+$ be a totally real number field and consider 
 a Hilbert eigencuspform $f$ defined over $GL(2)_{F^+}$
 of cohomological weight $\kappa$, level $\mathfrak{N}$
 and nebentypus $\underline{\varepsilon}$, and suppose that $[\kappa]$ is strictly greater than zero. 
 We further assume that $f$ is a primitive form 
 in the sense of Miyake \cite{miyake}, and denote by $\mathbb{Q}_f$   
 the Hecke field associated to $f$. 
 Then due to results of 
 many people including Ohta \cite{ohta}, Carayol \cite{carayol},
 Wiles \cite{wiles}, Taylor \cite{taylor1} and Blasius
 and Rogawski \cite{BR}, 
 we can canonically attach to $f$ a strictly compatible system
 $(\rho_{f,\lambda})_\lambda$ of 2\nobreakdash-dimensional
 $\lambda$\nobreakdash-adic representations of the absolute Galois
 group $G_{F^+}$ of $F^+$; namely, for each finite place $\lambda$
 of $F^+$ with residue characteristic $\ell$,
 the $2$\nobreakdash-dimensional $\lambda$\nobreakdash-adic representation
\begin{align*}
 \rho_{f, \lambda} \colon G_{F^+}\rightarrow
 \mathrm{Aut}_{\mathbb{Q}_{f,\lambda}} V_{f,\lambda}
\end{align*}
is unramified outside $\ell \mathfrak{N}$ and characterised by
the formulae
\begin{align*}
\mathrm{Tr}\, \rho_{f,
 \lambda}(\mathrm{Frob}_\mathfrak{q})=C(\mathfrak{q}; f), \qquad \det
 \rho_{f,
 \lambda}(\mathrm{Frob}_\mathfrak{q})=\chi_{\ell, \mathrm{cyc}}^{-1}
 \varepsilon_+^\mathrm{gal}(\mathrm{Frob}_\mathfrak{q})
\end{align*}
for each prime ideal $\mathfrak{q}$ of $F^+$ relatively prime to $\ell
\mathfrak{N}$ where $\chi_{\ell, \mathrm{cyc}}$
denotes the $\ell$-adic cyclotomic character.
Further, each Galois representation $\rho_{f,\lambda}$ is known to be irreducible
 (see \cite[Theorem~3.1]{taylor2}).

The following statement is widely known for elliptic modular forms due to Ribet \cite{Ribet}.
 
 \begin{pro} \label{prop:cm}
Let $f$ be a primitive Hilbert cuspform as above. Then the following three statements on $f$ are equivalent{\upshape ;} 
 \begin{enumerate}[label=$(\arabic*)$]
 \item the primitive form $f$ has complex multiplication{\upshape ;}
 \item the absolute Galois group $G_{F^+}$ contains an open subgroup $H$ of index two such that  
 the image of $H$ under the associated $\lambda$-adic Galois representation  $\rho_{f,\lambda}$ is abelian for every finite place $\lambda$ of $\mathbb{Q}_f${\upshape ;}
 \item there exist a totally imaginary quadratic extension $F$ of
       $F^+$ and a primitive gr\"o{\ss}en\-character $\eta$
       $($in the sense that the modulus of $\eta$ coincides
       with its conductor$)$
       of type $(A_0)$ on $F$ such that $f$ coincides with
       the theta lift $\vartheta(\eta)$ of $\eta$.
       Furthermore there exists a CM type $\Sigma$ of $F$ such that
       the gr\"o{\ss}encharacter $\eta$ is
       admissible with respect to $\Sigma$ and its infinity type is
       described as 
       $\sum_{\sigma \in \Sigma} \kappa_{1, \sigma|_{F^+}}\sigma+
       \sum_{\bar{\sigma}\in \Sigma^c}
       \kappa_{2,\bar{\sigma}|_{F^+}}\bar{\sigma}$.
 \end{enumerate}
In the cases above, $\rho_{f, \lambda}$ is isomorphic to the induced
  representation $\mathrm{Ind}^{F^+}_F \eta^{\mathrm{gal}}$ of
  the $1$\nobreakdash-dimensional $\lambda$-adic representation
  $\eta^\mathrm{gal} \colon G_F\rightarrow
  \mathbb{Q}_{f,\lambda}^\times$ which corresponds to 
  the $\ell$-adic avatar $\hat{\eta}_\ell$  
  of the gr\"o{\ss}encharacter $\eta$ of type $(A_0)$ introduced in
  $(3)$. Here we consider the $\ell$-adic avatar $\hat{\eta}_\ell$
  with respect to a specific embedding $\overline{\mathbb{Q}}\hookrightarrow
  \overline{\mathbb{Q}}_\ell$ which induces $\lambda$ on $\mathbb{Q}_f$.
 \end{pro}
  
 The proof of the proposition proceeds analogously to 
 Ribet's arguments in \cite[Sections~3 and 4]{Ribet}.
 It is based upon a precise study of the Galois representation
 $\rho_{f,\lambda}$ associated to 
 the Hilbert modular cuspform $f$ with complex multiplication. 
 
\begin{rem}
Due to the lack of appropriate references, we decided to give
a proof of Proposition~\ref{prop:cm} in this appendix.
After the first redaction of the article, we learned that some of 
 the results in Proposition~\ref{prop:cm}, say the equivalence between the statement (1) and the statement (3), has been already proved
 in \cite[Proposition 6.5]{LL} with the language of automorphic representation. 
 We still leave the proof of Proposition~\ref{prop:cm} below
 believing that the proof with
 the language of Galois representation has its own value.
\end{rem}
 
 \begin{proof}[Proof of Proposition~{$\ref{prop:cm}$}] 
We first prove that the statement (3) implies the statement (1). 
Assume that $f$ is obtained as the theta lift $\vartheta(\eta)$
  of a primitive gr\"o{\ss}encharacter $\eta$ of type $(A_0)$ defined
  on a totally imaginary quadratic extension $F$ of $F^+$.
  Then one easily observes by the
  construction of the theta lift (refer to Proposition~\ref{prop:theta} for details) 
  that the Fourier coefficient $C(\mathfrak{q}, \vartheta(\eta))$ at 
a prime ideal $\mathfrak{q}$ of $F^+$ equals zero if and only if 
$\mathfrak{q}$ is inert or ramified in $F$, or equivalently, the evaluation of the quadratic character $\nu_{F/F^+}\colon \mathbb{A}_{F^+}^\times \rightarrow \mathbb{C}^\times$ associated to the quadratic extension $F/F^+$ at $\mathfrak{q}$ equals either $0$ or $-1$. 
  This is equivalent to the validity of the equation (\ref{eq:cm}) 
  when one replaces the character $\nu$ appearing in (\ref{eq:cm}) by the quadratic
  character $\nu_{F/F^+}$.
 
Next we prove that the statement (1) implies the statement (2).
Assume that $f$ has complex multiplication by a nontrivial
  gr\"o{\ss}encharacter $\nu$ on $\mathbb{A}_{F^+}^\times$: 
then $\nu$ is a quadratic character (see the arguments in
  Section~\ref{sssc:CMforms}). The equation (\ref{eq:cm})
  for each $\mathfrak{q}$ in a set of prime ideals of $F^+$ of density
  one implies that the traces $\mathrm{Tr}\,
  \rho_{f,\lambda}(\mathrm{Frob}_\mathfrak{q})$ and 
  $\mathrm{Tr} \, \rho_{f \otimes \nu
  ,\lambda}(\mathrm{Frob}_\mathfrak{q})$ coincide 
  for each $\mathfrak{q}$ in the same set of prime ideals. Since
  both $\rho_{f,\lambda}$ and $\rho_{f\otimes \nu, \lambda}$ are 
  irreducible by \cite[Theorem~3.1]{taylor2}, \v{C}ebotarev's density theorem suggests that they are isomorphic to each other
  as $\lambda$\nobreakdash-adic representations of $G_{F^+}$. In other words, 
  there exists a $(2\times 2)$\nobreakdash-matrix $M$ in
  $GL_2(\mathbb{Q}_{f,\lambda})$ such that the equality 
\begin{align} \label{eq:conj}
\rho_{f, \lambda}(g) =M\rho_{f\otimes\nu , \lambda}(g)M^{-1} =\nu(g)M 
\rho_{f,\lambda}(g)M^{-1}
\end{align}
holds for an arbitrary element $g$ of $G_{F^+}$ 
when we fix noncanonical identifications $V_{f,\lambda} \cong
  \mathbb{Q}_{f,\lambda}^{\oplus 2}$ and $V_{f\otimes \nu, \lambda} \cong
  \mathbb{Q}_{f,\lambda}^{\oplus 2}$. Let us denote by $H$ the kernel 
of the quadratic character $\nu$, which is a subgroup of index two
  of $G_{F^+}$. Take an element $g_0$ from the complement of $H$ in $G_{F^+}$
  and set $T=\rho_{f, \lambda}(g_0)$. We then obtain the equality
  $T=-MTM^{-1}$ by (\ref{eq:conj}), from which we readily observe
  that $M$ is semisimple but not a scalar matrix. The equation (\ref{eq:conj})
  also implies that the image of $H$ under $\rho_{f,\lambda}$ is
  contained in the commutant of the semisimple, nonscalar matrix $M$, and thus
  it is abelian.  The subgroup $H$ satisfies
  all the required condition in the statement (2), and hence
  the statement (2) follows from (1).
  We remark that since $H$ is defined as the kernel of
  the rational quadratic character $\nu$, it is determined independently of the
  choice of a finite place $\lambda$ of $\mathbb{Q}_f$.

We finally verify that the statement (2) implies the statement (3),
  which is a crucial part of the proof.
Suppose that $G_{F^+}$ contains an open subgroup $H$ of index two such that 
the image of $H$ under each $\lambda$-adic representation $\rho_{f, \lambda}$ is
  abelian. Let $F$ denote the subfield of $\overline{\mathbb{Q}}$
  corresponding to $H$, which is a quadratic extension of $F^+$.
  By assumption, the restriction $\rho_{f,\lambda}\vert_H$ of $\rho_{f,\lambda}$
  to $H$ is a semisimple, abelian $\lambda$-adic representation of $H$
  for every finite place $\lambda$ of $\mathbb{Q}_f$, and therefore 
  $\rho_{f, \lambda}\vert_H$ is locally algebraic due to the
  local algebraicity theorem proved by Serre \cite[Chapter~III.\ Section~3.]{serre}
  and Henniart \cite[Section~6.]{Henniart}. In other
  words, there exists a morphism of algebraic groups $r \colon
  {\mathbb{S}_\mathfrak{m}}_{/\mathbb{Q}_f} \rightarrow GL(2)_{/\mathbb{Q}_f}$ 
  giving rise to the strictly compatible system $(\rho_{f,\lambda}\vert_H)_\lambda$
  of the $\lambda$\nobreakdash-adic representations $\rho_{f,\lambda}\vert_H$ of $H$. 
  Here ${\mathbb{S}_\mathfrak{m}}_{/\mathbb{Q}}$ denotes
  Serre's algebraic group associated to the field $F$
  and an appropriate integral modulus $\mathfrak{m}$ of $F$ (called the
  modulus of definition) and ${\mathbb{S}_\mathfrak{m}}_{/\mathbb{Q}_f}$ denotes its basechange over $\mathbb{Q}_f$. Serre's algebraic group ${\mathbb{S}_\mathfrak{m}}_{/\mathbb{Q}}$
  is by construction an extension of the ray
  class group $\mathrm{Cl}(F)_\mathfrak{m}$ of $F$ modulo $\mathfrak{m}$ 
  (regarded as a constant group scheme) 
  by a certain algebraic torus $T_\mathfrak{m}$
  defined over $\mathbb{Q}$, and it is 
  in particular of multiplicative type. See \cite[Chapter~II]{serre} for details
  on properties of ${\mathbb{S}_\mathfrak{m}}_{/\mathbb{Q}}$.
  The algebraic representation $r$ is thus ($\mathbb{Q}_f$-rational and) 
  semisimple; namely there exists a $\mathbb{Q}_f$-rational pair 
  of algebraic characters $(\eta^\mathrm{alg}_1,\eta^\mathrm{alg}_2)$ 
  of ${\mathbb{S}_\mathfrak{m}}_{/\overline{\mathbb{Q}}}$ (in the sense
  that the summation $\eta^\mathrm{alg}_1+\eta^\mathrm{alg}_2$ is invariant 
  under the natural action 
  of $\mathrm{Gal}(\overline{\mathbb{Q}}/\mathbb{Q}_f)$) such
  that $r$ is isomorphic to the direct sum $\eta^\mathrm{alg}_1\oplus \eta^\mathrm{alg}_2$
  over $\overline{\mathbb{Q}}$.
  For each $i=1,2$ and each finite place
  $\lambda$ of $\mathbb{Q}_f$, let $\eta^\mathrm{gal}_{i, \lambda}$ denote
  the $1$\nobreakdash-dimensional
  $\lambda$\nobreakdash-adic representation associated to $\eta_i^\mathrm{alg}$,
  and let $\eta_i$ denote the gr\"o{\ss}encharacter of type $(A_0)$ on
  $F$ associated to $\eta_i^\mathrm{alg}$. Then by construction
  $\eta^\mathrm{gal}_{i,\lambda}$ is the Galois character of $H$ associated 
  to the $\ell$-adic avatar $\hat{\eta}_{i,\ell}$ of $\eta_i$,
  and $\rho_{f,\lambda}\vert_H$ is equivalent to the direct sum of
  $\eta^\mathrm{gal}_{1,\lambda}$ and $\eta^\mathrm{gal}_{2,\lambda}$.
  Now let $c$ denote the generator of the quotient group $G_{F^+}/H$, and let 
  us take its arbitrary lift $\tilde{c}$ to $G_{F^+}$. 
  We define the $c$-conjugation $\rho_{f,\lambda}\vert_H^c$
  of $\rho_{f,\lambda}\vert_H$ by
  $\rho_{f,\lambda}\vert_H^c(h)=\rho_{f,\lambda}\vert_H(\tilde{c}h\tilde{c}^{-1})$ for each element $h$ in $H$. 
  Note that $\rho_{f,\lambda}\vert_H^c$ is well defined independently
  of the choice of $\tilde{c}$.  
  Then, since  $\rho_{f,\lambda}$ is defined on $G_{F^+}$,
  the trace of $\rho_{f,\lambda}\vert_H^c$ obviously coincides
  with that of $\rho_{f,\lambda}\vert_H$. Furthermore both
  $\rho_{f,\lambda}\vert_H$ and $\rho_{f,\lambda}\vert_H^c$ are semisimple,
  and hence they are isomorphic to each other by \v{C}ebotarev's density theorem. Consequently either of the
  followings two cases occurs: ($\eta_{i,\lambda}^c=\eta_{i,\lambda}$ for $i=1,2$) 
or ($\eta_{1,\lambda}^c=\eta_{2,\lambda}$ and
  $\eta_{2,\lambda}^c=\eta_{1,\lambda}$). If the former case occurs, the equality
  $\rho_{f,\lambda}(\tilde{c}h)=\rho_{f,\lambda}(h\tilde{c})$ holds for every
  $h$ in $H$. This means that $\rho_{f,\lambda}$ is an abelian
  representation of $G_{F^+}$, which contradicts the irreducibility of $\rho_{f, \lambda}$. 
  Therefore the latter case does occur, and in particular
  the conjugation by $\rho_{f,\lambda}(\tilde{g})$ corresponds to the interchange of
  $\eta^\mathrm{gal}_{1,\lambda}$ and $\eta^\mathrm{gal}_{2,\lambda}$ if $\tilde{g}$ is
  an element of the complement of $H$ in $G_{F^+}$. From this fact
  we readily verify that the image $\rho_{f,\lambda}(\tilde{g})$ of
  such $\tilde{g}$ is conjugate to a matrix both of whose diagonal entries equal zero, and thus $\mathrm{Tr}\, \rho_{f,\lambda}(\mathrm{Frob}_\mathfrak{q})$
  is trivial for a prime ideal $\mathfrak{q}$ of $F^+$ which inerts in $F$.
  If a prime ideal $\mathfrak{q}$ of $F^+$ splits completely in $F$, on the contrary,
  the decomposition group at $\mathfrak{q}$ is naturally identified with $H$.
  Combining these observations, we obtain the following formula for
  each prime ideal $\mathfrak{q}$ which does not divide $\ell \mathfrak{N}$
  (note that $\eta^\mathrm{gal}_{2,\lambda}(\mathrm{Frob}_\mathfrak{Q})=\eta^\mathrm{gal}_{1,\lambda}(\tilde{c}\mathrm{Frob}_{\mathfrak{Q}}\tilde{c}^{-1})=\eta^\mathrm{gal}_{1,\lambda}(\mathrm{Frob}_{\mathfrak{Q}^c})$ holds):
\begin{align} \label{eq:trace_eta}
\mathrm{Tr}\, \rho_{f,\lambda}(\mathrm{Frob}_\mathfrak{q}) = \begin{cases} 
\eta^\mathrm{gal}_{1,\lambda}(\mathrm{Frob}_\mathfrak{Q})+\eta^\mathrm{gal}_{1,\lambda}(\mathrm{Frob}_{\mathfrak{Q}^c})
	     & \text{if $\mathfrak{q}$ splits in $F$ as
	     $\mathfrak{q}=\mathfrak{QQ}^c$}, \\
0 & \text{otherwise}.
\end{cases}
\end{align}
This calculation (combined with \v{C}ebotarev's density theorem and 
the irreducibility of $\rho_{f,\lambda}$) implies that $\rho_{f,\lambda}$ is isomorphic to the induced
  representation $\mathrm{Ind}^{G_{F^+}}_H \eta^\mathrm{gal}_{1,\lambda}$ of
  $\eta^\mathrm{gal}_{1,\lambda}$. We remark that this is canonically extended to an isomorphism
  between the strict compatible systems $(\rho_{f,\lambda})_\lambda$
  and $(\mathrm{Ind}^{G_{F^+}}_H \eta^\mathrm{gal}_{1,\lambda})_\lambda$; namely,
  $\rho_{f,\lambda}$ is isomorphic to $\mathrm{Ind}^{G_{F^+}}_H \eta^\mathrm{gal}_{1,\lambda}$ for
  {\em every} finite place $\lambda$ of $\mathbb{Q}_f$ which is not contained in either of
  the exceptional sets of the two strictly compatible systems.

 Now, we verify that the field $F$ corresponding to the subgroup $H$ of
  $G_{F^+}$ is a purely imaginary quadratic extension of $F^+$ (and thus
  $F$ is in particular a CM number field). Indeed if $F$ is not purely
  imaginary over $F^+$, every algebraic character of $\mathbb{T}_\mathfrak{m}$
  is described as an integral power of the norm character (see \cite[Chapter~II, Section~3.3]{serre}).
  In particular, each $\lambda$-adic character $\eta_{i,\lambda}^\mathrm{gal}$ is described as
\begin{align*}
\eta^\mathrm{gal}_{i,\lambda} =\chi_{\ell,\mathrm{cyc}}^{-n_i} \eta_{i,\lambda}^f
\end{align*}
  for a certain integer $n_i$ and a certain character $\eta_{i,\lambda}^f$
  of $H$ of finite order. Then the determinants of $\rho_{f,\lambda}$
  and  $\mathrm{Ind}_H^{G_{F^+}}\eta_{1,\lambda}^\mathrm{gal}$ are calculated as follows:
 \begin{align*}
  \det \rho_{f,\lambda}&=\chi_{\ell,\mathrm{cyc}}^{-1}\varepsilon_+^\mathrm{gal}=\chi_{\ell,\mathrm{cyc}}^{-[\kappa]}\varepsilon_+^f, \\
  \det \left(\mathrm{Ind}_H^{G_{F^+}} \eta_{1,\lambda}^\mathrm{gal}\right)&= \eta_{1,\lambda}^\mathrm{gal} \eta_{2,\lambda}^\mathrm{gal} =\chi_{\ell,\mathrm{cyc}}^{-n_1-n_2} (\eta_{1,\lambda}^f\eta_{2,\lambda}^f)
 \end{align*}
  where $\varepsilon_+^f$ denotes the finite part of $\varepsilon_+^\mathrm{gal}$.
  Since these two $1$\nobreakdash-dimensional $\lambda$\nobreakdash-adic representations must coincide,
  we obtain the equality $n_1+n_2=[\kappa]$. Furthermore the equation (\ref{eq:trace_eta}) implies that 
\begin{align*}
C(\mathfrak{q};f)=\mathrm{Tr}\, \rho_{f,\lambda}(\mathrm{Frob}_{\mathfrak{q}})=\eta_{1,\lambda}^f(\mathfrak{Q})
 \mathcal{N}
 \mathfrak{q}^{n_1}+\eta_{2,\lambda}^f(\mathfrak{Q})\mathcal{N}
 \mathfrak{q}^{n_2}
\end{align*}
  holds for each prime ideal $\mathfrak{q}$ of $F^+$ which splits completely in $F$ as $\mathfrak{q}=\mathfrak{QQ}^c$ (recall
  that $\chi_{\ell,\mathrm{cyc}}(\mathrm{Frob}_\mathfrak{q})$ equals $\mathcal{N}
  \mathfrak{q}^{-1}$). By virtue of the Ramanujan-Petersson conjecture 
\begin{align*}
\lvert C(\mathfrak{q};f)\rvert \leq 2\mathcal{N}\mathfrak{q}^{[\kappa]/2}
\end{align*}
  established by Brylinski, Labesse \cite[Th\'eor\`eme~3.4.5]{BR} and Blasius \cite[Theorem~1]{Blasius},
  both $n_1$ and $n_2$ must be less than or equal to $[\kappa]/2$.  
  We thus conclude that, combining this observation with the equality $n_1+n_2=[\kappa]$,
  both $n_1$ and $n_2$ equal $[\kappa]/2$. More specifically, the restriction $\rho_{f, \lambda}\vert_H$ is 
  equivalent to the $\lambda$\nobreakdash-adic representation of the form
\begin{align*}
\begin{pmatrix}\eta_{1,\lambda}^f & 0 \\ 0 & \eta_{2,\lambda}^f \end{pmatrix}
 \otimes \chi_{\ell, \mathrm{cyc}}^{-[\kappa]/2}, 
\end{align*} 
which has a finite image in $PGL_2(\mathbb{Q}_{f,\lambda})$. 
This contradicts the fact that $\rho_{f,\lambda}$ has an {\em infinite}
  image in $PGL_2(\mathbb{Q}_{f,\lambda})$ (one readily verifies this
  fact in the completely same manner as \cite[Theorem~(4.3)]{Ribet}), and thus 
  the field $F$ corresponding to $H$ is purely imaginary over $F^+$.

 Next, we verify that 
  the infinity type $\mu=\sum_{\sigma\in I_F} \mu_\sigma \sigma$
  of $\eta$ is described in terms of the weight of $f$ by setting $\eta=\eta_1$. 
  To this end, we take a finite place $\mathfrak{p}$ of $F^+$ which satisfies the following
  two properties:
  \begin{enumerate}[label=\roman*)$_\mathfrak{p}$]
   \item the prime ideal $\mathfrak{p}$ is 
	    contained in neither of the exceptional set
	    of the strictly compatible system
	    $(\rho_{f,\lambda})_\lambda$ nor
	    that of $(\mathrm{Ind}^{G_{F^+}}_H
	    \eta_\lambda^\mathrm{gal})_\lambda$;
   \item the unique prime ideal $p\mathbb{Z}$ of $\mathbb{Z}$ lying below
	    $\mathfrak{p}$ splits completely in the extension $F/\mathbb{Q}$.
  \end{enumerate}
  The existence of such a finite place $\mathfrak{p}$ is guaranteed
  by \v{C}ebotarev's density theorem. Now we fix an algebraic closure
  $\overline{\mathbb{Q}}_p$ of the
  completion $F^+_\mathfrak{p}$ of $F^+$ at
  $\mathfrak{p}$ and an embedding $\iota_p \colon
  \overline{\mathbb{Q}} \hookrightarrow \overline{\mathbb{Q}}_p$.
  For each embedding $\tau \colon F^+\hookrightarrow
  \overline{\mathbb{Q}}$ (that
  is, $\tau$ is an element of $I_{F^+}$ under the notation of
  Section~\ref{sssc:grossen}), let $\mathfrak{p}_\tau$ denote the
  finite place of $F^+$ lying above $p\mathbb{Z}$ induced by the
  composition $\iota_p\circ \tau \colon F^+ \hookrightarrow
  \overline{\mathbb{Q}}_p$. Note that, due to the condition
  ii)$_\mathfrak{p}$ on $\mathfrak{p}$, the correspondence $\tau
  \mapsto \mathfrak{p}_\tau$ induces a bijection between $I_{F^+}$ and
  the set of prime ideals of $F^+$ lying above $p\mathbb{Z}$. 
  Let $\lambda_0$ denote the finite place of $\mathbb{Q}_f$
  induced by the embedding $\mathbb{Q}_f \subset
  \overline{\mathbb{Q}}\xrightarrow{\iota_p} \overline{\mathbb{Q}}_p$
  and let us consider the isomorphism 
  \begin{align*}
   \rho_{f,\lambda_0} \cong \mathrm{Ind}^{G_{F^+}}_H
   \eta_{\lambda_0}^\mathrm{gal} \colon G_{F^+} \rightarrow GL_2(\mathbb{Q}_{f,\lambda_0})
  \end{align*}
  of the $\lambda_0$\nobreakdash-adic representations of $G_{F^+}$.
  It is obvious from their constructions that both of $\rho_{f,\lambda_0}
  \vert_{D_{\mathfrak{p}_\tau}}$ and $( \mathrm{Ind}_H^{G_{F^+}}
  \eta_{\lambda_0}^\mathrm{gal})\vert_{D_{\mathfrak{p}_\tau}}$ are
  Hodge-Tate representation of $D_{\mathfrak{p}_\tau}$ for each $\tau$ in $I_{F^+}$, and 
  we shall compare the Hodge-Tate weights of these representations.
  As we have remarked in the paragraphs preceding
  Remark~\ref{rem:neben}, the Hodge type
  at $\tau$ of the motive $M(f)_{/F^+}$ associated to $f$ is
  given by $\{\,(\kappa_{1,\tau},\kappa_{2,\tau}),
  (\kappa_{2,\tau},\kappa_{1,\tau})\,\}$, and we thus see that
  the Hodge-Tate weights of $\rho_{f,\lambda_0}\vert_{D_{\mathfrak{p}_\tau}}$
  are $\{ \, \kappa_{1,\tau}, \kappa_{2,\tau} \, \}$
  via the comparison isomorphism in $p$-adic Hodge theory.
  Now let us study the Hodge-Tate weight of
  $(\mathrm{Ind}^{G_{F^+}}_H
  \eta_{\lambda_0}^\mathrm{gal})\vert_{D_{\mathfrak{p}_\tau}}$ for
  each $\tau$ in $I_{F^+}$. Let $w_\tau$ denote the unique  (complex) place of
  $F$ lying above the real place of $F^+$ determined by $\tau$, which
  is identified with the pair $\{\, \sigma_{\tau,1}, \sigma_{\tau,2}
  \,\}$ of embeddings of $F$ into $\overline{\mathbb{Q}}$ whose
  restrictions to $F^+$ coincide with $\tau$. Then the
  $p$\nobreakdash-adic embeddings $\iota_p\circ \sigma_{\tau,1}$ and
  $\iota_p\circ \sigma_{\tau,2}$ of $F$ induce distinct prime ideals
  $\mathfrak{P}_{\tau,1}$ and $\mathfrak{P}_{\tau,2}$, which are
  interchanged by the complex conjugation. 
  Recall that the Hodge type at $w_\tau$ of the motive $M(\eta)_{/F}$
  associated to the gr\"o{\ss}encharacter $\eta$ of type $(A_0)$
  is given by $\{\,
  (\mu_{\sigma_{\tau,1}}, \mu_{\sigma_{\tau,2}}), (\mu_{\sigma_{\tau,2}},\mu_{\sigma_{\tau,1}})\,\}$ (see \cite[Proposition~(3.2.3)]{Blasius1} or \cite[Chapter~1, Section~4]{schappacher} for details). 
  Since $(\mathrm{Ind}^{G_{F^+}}_H
  \eta_{\lambda_0}^\mathrm{gal})\vert_{D_{\mathfrak{p}_\tau}}$ 
  is isomorphic to $\eta_{\lambda_0}^\mathrm{gal}\vert_{D_{\mathfrak{P}_{\tau,1}}}\oplus \eta_{\lambda_0}^\mathrm{gal}\vert_{D_{\mathfrak{P}_{\tau,2}}}$,
  we readily see that the Hodge-Tate weights of $(\mathrm{Ind}^{G_{F^+}}_H \eta_{\lambda_0}^\mathrm{gal})\vert_{D_{\mathfrak{p}_\tau}}$ 
  are $\{ \, \mu_{\sigma_{\tau,1}}, \mu_{\sigma_{\tau,2}}\, \}$.
  By comparing the Hodge-Tate weights, we can consider without loss of
  generality that
  the two integers $\kappa_{1,\tau}$ and
  $\kappa_{2,\tau}$ coincide with $\mu_{\sigma_{\tau,1}}$ and
  $\mu_{\sigma_{\tau,2}}$ respectively.
 Furthermore the inequality $\kappa_{1,\tau} < \kappa_{2,\tau}$ holds
  for each $\tau$ in $I_{F^+}$ since we have assumed that the weight
  $\kappa$ of $f$ was cohomological. Therefore
  if we set $\Sigma=\{ \, \sigma_{\tau,1} \colon F\hookrightarrow \overline{\mathbb{Q}} \mid \tau \in I_{F^+} \,\}$,
  it is straightforward to verify that $\Sigma$ is a $p$\nobreakdash-ordinary CM type of $F$ with respect to which
  the infinity type $\mu$ of the gr\"o{\ss}en\-character $\eta$ is admissible.
  Moreover we readily redescribe $\mu$ in terms of $\kappa$ and $\Sigma$ as in the statement (3). 
   
  Finally, we verify that the primitive form $f$ is described as the theta lift of $\eta$.
  Since the infinity type of $\eta$ is $\Sigma$-admissible, we have 
  the theta lift $\vartheta(\eta)$ of $\eta$ by Proposition-Definition \ref{prop:theta}. 
  The construction of $\eta$ implies that the local $L$\nobreakdash-factors of $f$
  and $\vartheta(\eta)$ coincide at every prime ideal $\mathfrak{q}$ of $F^+$ which
  splits completely in $F$, and we thus conclude that $f$ and
  $\vartheta(\eta)$ coincides up to a scalar multiple due to the 
  strong multiplicity one theorem for Hilbert modular forms.
  However both the Fourier coefficient at $\mathfrak{r}_{F^+}$ of $f$
  and that of $\vartheta(\eta)$ equal $1$, and hence
  the primitive form $f$ exactly coincides with the theta lift $\vartheta(\eta)$ of $\eta$.
\end{proof}

As an application of Proposition~\ref{prop:cm}, we can deduce conditions for the primitive cuspform $\vartheta(\eta)$ with complex multiplication to be nearly $p$-ordinary. In the following proposition we fix a $p$-adic embedding $\iota_p\colon \overline{\mathbb{Q}}\hookrightarrow \overline{\mathbb{Q}}_p$.

\begin{pro} \label{prop:cm_ord}
 Let $p$ be a  prime number and $f$ a primitive Hilbert modular cuspform as above. Then $f$ is nearly $p$-ordinary $($with respect to $\iota_p$$)$ if and only if there exist a totally imaginary quadratic extension $F/F^+$ satisfying the ordinarity condition $($ord$_{F/F^+})$ for the prime number $p$, a $p$-ordinary CM type $\Sigma$ and a $\Sigma$-admissible gr\"o{\ss}encharacter $\eta$ of type $(A_0)$ on $F$ ordinary with respect to $\Sigma$ such that $f$ is obtained as the theta lift $\vartheta(\eta)$ of $\eta$.
\end{pro}

\begin{proof}
 We can easily verify that the condition is sufficient. Indeed, let $F/F^+$, $\Sigma$ and $\eta$ be as in the statement. By the characterisation of the Fourier coefficients of the theta lifts, we readily see that, for every prime ideal $\mathfrak{p}=\mathfrak{P}\mathfrak{P}^c$ lying above $p$, the eigenvalue $C_0(\mathfrak{p};\vartheta(\eta))$ of the normalised Hecke operator $U_0(\mathfrak{p})$ at $\vartheta(\eta)$ equals the summation of $\{\mathfrak{p}^{\kappa_{\mu,1}}\}^{-1}\eta^*(\mathfrak{P})$ and $\{ \mathfrak{p}^{\kappa_{\mu,1}}\}^{-1}\eta^*(\mathfrak{P}^c)$. Note that the value $\{\mathfrak{p}^{\kappa_{\mu,1}}\}^{-1}\eta^*(\mathfrak{P})$ does not vanish since $\eta$ is unramified at the unique place $\mathfrak{P}$ in $\Sigma_p$ lying above $\mathfrak{p}$ due to the ordinarity of $\eta$ with respect to $\Sigma$. Moreover we readily observe by construction that $\{\mathfrak{p}^{\kappa_{\mu,1}}\}^{-1}\eta^*(\mathfrak{P})$ has the same $p$-adic valuation with the evaluation $\hat{\eta}_{\mathfrak{P}}(\varpi_\mathfrak{P})$ of the $\mathfrak{P}$-component of the $p$-adic avatar $\hat{\eta}$ of $\eta$ at a uniformiser $\varpi_\mathfrak{P}$ of $F_\mathfrak{P}$. Therefore $C_0(\mathfrak{p}; \vartheta(\eta))$ is a $p$-adic unit for each $\mathfrak{p}$ and consequently $\vartheta(\eta)$ is nearly $p$-ordinary.

\medskip 
 Conversely let $f$ be a nearly $p$-ordinary Hilbert modular cuspform with complex multiplication, and let us take a totally imaginary quadratic extension $F/F^+$, a CM type $\Sigma$ and a $\Sigma$-admissible gr\"o{\ss}encharacter $\eta$ of type $(A_0)$ on $F$ as in Proposition~\ref{prop:cm}; then $f$ is obtained as the theta lift of  $\eta$. We denote by $V_f$ the Galois representation associated to $f$, which is isomorphic to the induced representation $\mathrm{Ind}^{F^+}_F\eta^\mathrm{gal}$ by Proposition~\ref{prop:cm}. First assume that there exists a place $\mathfrak{p}$ of $F^+$ lying above $p$ which does not split in $F$, and let $\mathfrak{P}$ denote the unique place of $F$ above $\mathfrak{p}$. Since  we can regard the decomposition group $D_\mathfrak{P}$ of $G_F$ at $\mathfrak{P}$ as a subgroup of the decomposition group $D_\mathfrak{p}$ at $\mathfrak{p}$ of index $2$, we readily identify the restriction of $V_f$ to $D_\mathfrak{p}$ with the induced representation $\mathrm{Ind}^{D_\mathfrak{p}}_{D_\mathfrak{P}}\eta^\mathrm{gal}\lvert_{D_\mathfrak{P}}$, which is obviously irreducible. Therefore $V_f\vert_{D_\mathfrak{p}}$ admits no one-dimensional $D_\mathfrak{p}$-subrepresentations. This contradicts Proposition~\ref{prop:localGalois}, and thus all places of $F^+$ lying above $p$ split in $F$. In other words, the quadratic extension $F/F^+$ satisfies the condition (ord$_{F/F^+}$) for $p$.

 We next prove that $\Sigma$ is a $p$-ordinary CM type. Let $\mathfrak{p}$ be a place of $F^+$ lying above $p$, which splits completely in $F$ as $\mathfrak{p}=\mathfrak{PP}^c$ by the arguments above. The quadratic equation \eqref{eq:unitroot} in Proposition~\ref{prop:localGalois} has two roots $\varpi_\mathfrak{p}^{-\kappa_{1,\mathfrak{p}}}\eta^*(\mathfrak{P})$ and $\varpi_\mathfrak{p}^{-\kappa_{1, \mathfrak{p}}}\eta^*(\mathfrak{P}^c)$, one of which is a $p$-adic unit due to the near $p$-ordinarity of $f=\vartheta(\eta)$. We can assume without loss of generality that $\varpi_\mathfrak{p}^{-\kappa_{1,\mathfrak{p}}}\eta^*(\mathfrak{P})$ is a $p$-adic unit. Define $\Sigma_\mathfrak{P}$ and $\Sigma_{\mathfrak{P}}^c$ as follows:
 \begin{align*}
  \Sigma_\mathfrak{P}&=\{\,\sigma \in \Sigma \mid \iota_p\circ \sigma \text{ induces }\mathfrak{P}\,\}=\{\, \sigma_1,\dotsc, \sigma_s \,\}, \\
  \Sigma^c_\mathfrak{P}&=\{\,\bar{\sigma} \in \Sigma^c \mid \iota_p\circ \bar{\sigma} \text{ induces }\mathfrak{P}\,\}=\{\, \bar{\sigma}_{s+1},\dotsc, \bar{\sigma}_{s+t} \,\}.
 \end{align*}
 We shall verify that $\Sigma^c_{\mathfrak{P}}$ is empty, or in other words, that $t$ equals $0$. Since $\mathfrak{p}$ splits in $F$, the decomposition group $D_\mathfrak{p}$ at $\mathfrak{p}$ is contained in $G_F$, and thus the restriction of $V_f$ to $D_\mathfrak{p}$ is isomorphic to the direct sum of $\eta^\mathrm{gal}\vert_{D_\mathfrak{P}}$ and $\eta^\mathrm{gal,c}\vert _{D_\mathfrak{P}}$. Therefore the equation \eqref{eq:unitroot} in Proposition\ref{prop:localGalois} suggests that $\eta^\mathrm{gal}(\mathrm{Frob}_{\varpi_\mathfrak{p}})=\hat{\eta}(\varpi_\mathfrak{P})$ coincides with $\varpi_\mathfrak{p}^{-\kappa_{1,\mathfrak{p}}}\eta^*(\mathfrak{P})$. Here we identify $F^+_\mathfrak{p}$ with $F_\mathfrak{P}$ and define the uniformiser $\varpi_\mathfrak{P}$ of $F_\mathfrak{P}$ as $\varpi_\mathfrak{p}$ via this identification. By the definition of the $p$-adic avatar \eqref{eq:padicavatar}, we have
 \begin{align} \label{eq:exponent1}
  \hat{\eta}(\varpi_\mathfrak{P})=\varpi_\mathfrak{P}^{-\sum_{i=1}^s \mu_{\sigma_i}-\sum_{j=1}^t \mu_{\bar{\sigma}_{s+j}}} \eta^*(\mathfrak{P})
 \end{align}
 where we denote by $\mu=\sum_{\sigma \in \Sigma}(\mu_\sigma\sigma+\mu_{\bar{\sigma}}\bar{\sigma})$ the infinity type of $\eta$.
 On the other hand, since $\kappa_1=\kappa_{\mu,1}$ equals $\sum_{\sigma \in \Sigma}\mu_\sigma \sigma\vert_{F^+}$ by the characterisation of the weight of the theta lift (see \eqref{eq:weight_cm} for details), we have
 \begin{align} \label{eq:exponent2}
  \varpi_\mathfrak{p}^{-\kappa_{1,\mathfrak{p}}}\eta^*(\mathfrak{P}) =\varpi_\mathfrak{P}^{-\sum_{i=1}^{s+t}\mu_i} \eta^*(\mathfrak{P}).
 \end{align}
 Comparing the exponent of $\varpi_\mathfrak{P}$ in the equations \eqref{eq:exponent1} and \eqref{eq:exponent2}, we obtain the equality
 \begin{align*}
\sum_{j=1}^t (\mu_{\bar{\sigma}_{s+j}}-\mu_{\sigma_{s+j}})=0.
 \end{align*}
 The $\Sigma$-admissibility of $\eta$ implies the inequality $\mu_{\bar{\sigma}_{s+j}}>\mu_{\sigma_{s+j}}$ for each $j$, which forces $t$ to be $0$. Therefore the CM type $\Sigma$ is indeed a $p$-ordinary CM type.

 Finally since $\varpi_\mathfrak{p}^{-\kappa_{1,\mathfrak{p}}}\eta^*(\mathfrak{P})$ is a $p$-unit for every place $\mathfrak{p}=\mathfrak{P}\mathfrak{P}^c$ of $F^+$ lying above $p$, the value $\eta^*(\mathfrak{P})$ does not vanish for every $\mathfrak{P}$ in $\Sigma_p$. This implies that the gr\"o{\ss}encharacter $\eta$ is ordinary with respect to $\Sigma$. 
\end{proof}

%
\section{Comparison of the global and local duality pairings}
\label{app:duality}
%

In this appendix we provide the proof of
Proposition~\ref{prop:commutes}.
We take over the same setting and use the same notation
as in Section~\ref{sssc:int_step}.
Firstly we recall that, as is explained in the preceding paragraphs of
\cite[Proposition~3.3.1]{gr-surj}, the Pontrjagin dual of
the cokernel of the global-to-local morphism $\phi_{\mathcal{L}_\mathrm{str}^{*,(j-1)}}$ is identified
with the image of the global-to-local map
\begin{align*}
 \mathrm{Res}^1_{(j-1)}\rvert_{\mathrm{Sel}^{\Sigma,\mathrm{str}}_{\mathcal{A}_{\eta,(j-1)}}} \colon
 \mathrm{Sel}^{\Sigma,\mathrm{str}}_{\mathcal{A}_{\eta,(j-1)}}
 \rightarrow L_\mathrm{str}(F,\mathcal{A}_{\eta,(j-1)})
 \qquad \left( \subset \prod_{v\in S} H^1(F_v,
 \mathcal{A}_{\eta,(j-1)}) \right) 
\end{align*}
via the local Tate duality isomorphism
\begin{align*}
 \Phi_{\mathrm{local}}=(\Phi_{\mathrm{local},v})_{v\in S} \colon \prod_{v\in S} H^1(F_v,
 \mathcal{A}_{\eta,(j-1)})^\vee \xrightarrow{\sim} \prod_{v\in S}
 H^1_\mathrm{cts}(F_v, \mathcal{T}^*_{\eta,(j-1)}).
\end{align*}
In other words, the summation of the local Tate pairing
\begin{align*}
 \langle \, , \, \rangle_\mathrm{local}=\sum_{v\in S}\langle \, , \,
 \rangle_v \colon \prod_{v\in S} H^1(F_v,\mathcal{A}_{\eta,(j-1)})\times \prod_{v\in S}
 H^1_\mathrm{cts}(F_v,\mathcal{T}^*_{\eta,(j-1)})\rightarrow \mathbb{Q}_p/\mathbb{Z}_p
\end{align*}
induces a perfect pairing
\begin{align*}
\langle \, , \, \rangle_\mathrm{local}\colon 
\mathrm{Im}
\Bigl( \mathrm{Res}^1_{(j-1)}\vert_{\mathrm{Sel}^{\Sigma,\mathrm{str}}_{\mathcal{A}_{\eta,(j-1)}}}\Bigr)
 \times
 \mathrm{Coker}(\phi_{\mathcal{L}^{*,(j-1)}_{\mathrm{str}}})\rightarrow
 \mathbb{Q}_p/\mathbb{Z}_p,
\end{align*}
for which we use the same symbol $\langle \, , \,
\rangle_{\mathrm{local}}$ to simplify the notation. Since the module
$\mathcal{C}^{(j-1)}$ introduced in the diagram (\ref{ses:comp_coker})
is defined as the cokernel of the multiplication
of $x_j$ on $\mathrm{Coker}(\phi_{\mathcal{L}^{*,(j-1)}_\mathrm{str}})$,
the perfect pairing $\langle \, , \, \rangle_\mathrm{local}$ above
also induces a perfect pairing
\begin{align*}
 \langle \, , \, \rangle_\mathrm{local} \colon
 \mathcal{K}^{(j-1)}\times
 \mathcal{C}^{(j-1)}\rightarrow \mathbb{Q}_p/\mathbb{Z}_p
\end{align*}
where $\mathcal{K}^{(j-1)}$ is defined as 
\begin{equation}\label{equation:definition_K}
  \mathcal{K}^{(j-1)}:= \mathrm{Ker} \left[
  \mathrm{Im}
\Bigl( \mathrm{Res}^1_{(j-1)}\vert_{\mathrm{Sel}^{\Sigma,\mathrm{str}}_{\mathcal{A}_{\eta,(j-1)}}}\Bigr)
  \xrightarrow{\times x_j }
  \mathrm{Im}
\Bigl( \mathrm{Res}^1_{(j-1)}\vert_{\mathrm{Sel}^{\Sigma,\mathrm{str}}_{\mathcal{A}_{\eta,(j-1)}}}\Bigr)\right].
\end{equation}

Now consider the commutative diagram
\small 
\begin{equation} \label{dg:cokernel_dual}
 \begin{CD}
 @. @. @. \mathcal{K}^{(j-1)} @.  
 \\
@. @. @.  @VVV @. 
 \\ 
 0 @>>>  
 \mathcyr{Sh}^1_{\mathcal{A}_{\eta,(j-1)}}
 @>>> 
 \mathrm{Sel}^{\Sigma,\mathrm{str}}_{\mathcal{A}_{\eta,(j-1)}}
 @>{\mathrm{Res}^1_{(j-1)}}>>  
 \mathrm{Im}
\Bigl( \mathrm{Res}^1_{(j-1)}\vert_{\mathrm{Sel}^{\Sigma,\mathrm{str}}_{\mathcal{A}_{\eta,(j-1)}}}\Bigr) 
 @>>> 
 0  
 \\
@. @VV{\times x_j}V @VV{\times x_j}V @VV{\times x_j}V @. 
 \\ 
 0 @>>>  
 \mathcyr{Sh}^1_{\mathcal{A}_{\eta,(j-1)}}
 @>>> 
 \mathrm{Sel}^{\Sigma,\mathrm{str}}_{\mathcal{A}_{\eta,(j-1)}}
 @>{\mathrm{Res}^1_{(j-1)}}>>  
 \mathrm{Im}
\Bigl( \mathrm{Res}^1_{(j-1)}\vert_{\mathrm{Sel}^{\Sigma,\mathrm{str}}_{\mathcal{A}_{\eta,(j-1)}}}\Bigr) 
 @>>> 
 0  
 \\
@. @VVV @VVV @. @. 
 \\
@.  \mathcyr{Sh}^1_{\mathcal{A}_{\eta,(j-1)}}/x_j
 \mathcyr{Sh}^1_{\mathcal{A}_{\eta,(j-1)}} @>>> 
  0 @. @.  
 \end{CD}
\end{equation}
\normalsize 
whose rows are exact by the definition of the $S$-fine Selmer group $\mathcyr{Sh}^1_{\mathcal{A}_{\eta,(j-1)}}$.
We denote by
\begin{align*}
 \delta_1 \colon 
 \mathcal{K}^{(j-1)} \twoheadrightarrow \mathcyr{Sh}^1_{\mathcal{A}_{\eta,(j-1)}}/x_j\mathcyr{Sh}^1_{\mathcal{A}_{\eta,(j-1)}}
\end{align*}
the connecting homomorphism associated to the diagram
(\ref{dg:cokernel_dual}) through the snake lemma. 

The diagram (\ref{dg:cokernel_dual}) is obtained as the Pontrjagin dual
of the diagram of (\ref{ses:cokernel}), and we thus observe that the connecting
homomorphism $\tilde{\delta}_1^\vee$ in the diagram
(\ref{ses:comp_coker}) is obtained as the composition
\begin{align} \label{dg:delta_tilde}
\xymatrix{ (\mathcyr{Sh}^1_{\mathcal{A}_{\eta,(j-1)}})^\vee[x_j]
 \ar@{^(->}[r]^(0.6){\delta_1^\vee} 
& 
(\mathcal{K}^{(j-1)})^\vee
 \ar[rr]^(0.5){\Phi_{\mathrm{local}}}_(0.5){\sim}&& \mathcal{C}^{(j-1)} 
 }
\end{align}
where $\delta_1^\vee$ denotes the dual morphism of
$\delta_1$ and $\Phi_\mathrm{local}$ is the isomorphism induced by the
perfect pairing $\langle \, , \rangle_\mathrm{local}$. By construction
the local Tate duality map $\Phi_\mathrm{local}$ isomorphically
sends the image of $\delta_1^\vee$ onto that of
$\tilde{\delta}_1^\vee$, or in other words, the perfect pairing
$\langle \, , \, \rangle_\mathrm{local}$ induces a perfect pairing
\begin{align*}
\langle \, , \, \rangle_\mathrm{local}\colon \mathrm{Coim}(\delta_1)\times
 \mathrm{Im}(\tilde{\delta}^\vee_1)\rightarrow \mathbb{Q}_p/\mathbb{Z}_p
\end{align*}
which makes the following diagram commutative:
\begin{align} \label{dg:pairing1}
 \xymatrix{
 \mathcal{K}^{(j-1)}  \ar@{->>}[d] & \times
 &\mathcal{C}^{(j-1)} \ar[rr]^(0.55){\langle\, ,
 \,\rangle_\mathrm{local}} && \mathbb{Q}_p/\mathbb{Z}_p \ar@{=}[d] \\
  \mathrm{Coim}(\delta_1) & \times
 &\mathrm{Im}(\tilde{\delta}^\vee_1) \ar[rr]_(0.55){\langle\, ,
 \,\rangle_\mathrm{local}} \ar@{^(->}[u] && \mathbb{Q}_p/\mathbb{Z}_p.
 }
\end{align}

Recall that we have defined the (injective) morphism $\delta_2 \colon \mathrm{Im}\, (\tilde{\delta}_1^\vee)
 \rightarrow \overline{\mathcal{H}}^{2,*}_j$ as the connecting
 homomorphism associated to the diagram (\ref{dg:right}) via the snake
 lemma. See the paragraph preceding the short exact sequences
 (\ref{ses:dualSelvsH^2}) and (\ref{ses:snake}) for the definition of
 the module $\overline{\mathcal{H}}^{2,*}_j$. 
 Let $\mathcyr{Sh}^2_{\mathcal{T}^*_{\eta,(j-1)}}$ denote
 the kernel of the local-to-global morphism
 \begin{align*}
\mathrm{Res}^{2,*}_{(j-1)} \colon    H^{2,*}_{(j-1)} \rightarrow
   \prod_{v\in S}H^2_\mathrm{cts}(F_v, \mathcal{T}^*_{\eta,(j-1)}).
 \end{align*}
 Combining the short exact sequence defining $\mathcyr{Sh}^2_{\mathcal{T}^*_{\eta,(j-1)}}$ and the sequence (\ref{ses:snake}) with $\mathrm{Im}(\tilde{\delta}^\vee_1)$ replaced by $\mathrm{Im}(\delta_2)$, we obtain the following commutative diagram with exact rows:
 \begin{align*}
  \xymatrix{
  0\ar[r] & \mathcyr{Sh}^2_{\mathcal{T}^*_{\eta,(j-1)}}
  \ar[r] \ar@{.>>}[d]_{\mathrm{pr}_{\mathcyr{Sh}^2}} &
   H^{2,*}_{(j-1)}
  \ar@{->>}[d]^{\mathrm{pr}_{H^2}} \ar[rr]^(0.30){\mathrm{Res}^{2,*}_{(j-1)}} && \displaystyle\prod_{v\in S}
  H^2_\mathrm{cts}(F_v,\mathcal{T}^*_{\eta,(j-1)})
  \ar@{=}[d] \ar[r] & 0 \\
  0 \ar[r] & \mathrm{Im}(\delta_2) \ar@{^(->}[r] &
  \overline{\mathcal{H}}^{2,*}_j \ar[rr] && \displaystyle \prod_{v\in
  S}H^2_\mathrm{cts}(F_v,\mathcal{T}^*_{\eta,(j-1)}) \ar[r] &0. 
  }
 \end{align*}
 We readily observe that the middle vertical morphism $\mathrm{pr}_{H^2}$, which
 is defined in (\ref{ses:dualSelvsH^2}), induces a surjection
 $\mathrm{pr}_{\mathcyr{Sh}^2}\colon
 \mathcyr{Sh}^2_{\mathcal{T}^*_{\eta,(j-1)}}\twoheadrightarrow
 \mathrm{Im}(\delta_2)$. Note
 that the right commutative square of the diagram above is
 the same as that of the diagram (\ref{dg:PTvsSnake}). 
 The rest of this appendix is devoted to the verification of the following proposition.

  \begin{pro} \label{prop:coupling}
   Let
   \begin{align*}
    \langle \, , \, \rangle_{\mathrm{PT}}\colon
    \mathcyr{Sh}^1_{\mathcal{A}_{\eta,(j-1)}}/x_j\mathcyr{Sh}^1_{\mathcal{A}_{\eta,(j-1)}}\times
    \mathcyr{Sh}^2_{\mathcal{T}^*_{\eta,(j-1)}}[x_j]\rightarrow \mathbb{Q}_p/\mathbb{Z}_p
   \end{align*}
   be the perfect pairing induced from the Poitou-Tate 
   pairing $\mathcyr{Sh}^1_{\mathcal{A}_{\eta,(j-1)}}\times
   \mathcyr{Sh}^2_{\mathcal{T}^*_{\eta,(j-1)}}\rightarrow \mathbb{Q}_p/\mathbb{Z}_p$. Then the equality
 \begin{align} \label{eq:PTvsLocal}
\langle (s_v)_{v\in S},(t_v)_{v\in S}\rangle_\mathrm{local} =\langle
  \delta_1((s_v)_{v\in S}), \delta_2((t_v)_{v\in S})^{\sim}\rangle_{\mathrm{PT}}
 \end{align}
 holds for arbitrary elements $(s_v)_{v\in S}$ and $(t_v)_{v\in S}$ in
 $\mathrm{Coim}(\delta_1)$ and
 $\mathrm{Im}(\tilde{\delta}^\vee_1)$
 respectively, where $\delta_2((t_v)_{v\in S})^{\sim}$ denotes
 an arbitrary element of 
   $\mathcyr{Sh}^2(F,S,\mathcal{T}^*_{\eta,(j-1)})[x_j]$
   which is sent to $\delta_2((t_v)_{v\in S})$ by the map $\mathrm{pr}_{\mathcyr{Sh}^2}$ 
   introduced above.
In particular the perfect pairing $\langle \, , \,
   \rangle_{\mathrm{PT}}$ induces a pairing
    \begin{align*}
     \xymatrix{
    \mathcyr{Sh}^1_{\mathcal{A}_{\eta,(j-1)}}/x_j\mathcyr{Sh}^1_{\mathcal{A}_{\eta,(j-1)}}\times
    \mathrm{Im}(\delta_2) \ar[r]& \mathbb{Q}_p/\mathbb{Z}_p}
    \end{align*}
which makes the following diagram commutative$:$
 \begin{align*}
  \xymatrix{
 \mathrm{Coim}(\delta_1)
  \ar@{->}[d]_{\delta_1}^{\rotatebox{90}{$\sim$}}
  & \times & \mathrm{Im}\, (\tilde{\delta}_1^\vee)
  \ar[d]^{\delta_2}_{\rotatebox{90}{$\sim$}} \ar[rr]^(0.55){\langle \,
  , \, \rangle_\mathrm{local}}&
  & \mathbb{Q}_p/\mathbb{Z}_p \ar@{=}[d]\\
  \mathcyr{Sh}^1_{\mathcal{A}_{\eta,(j-1)}}/x_j\mathcyr{Sh}^1_{\mathcal{A}_{\eta,(j-1)}}
  & \times  & \mathrm{Im}(\delta_2) \ar@{.>}[rr] & &
  \mathbb{Q}_p/\mathbb{Z}_p \ar@{=}[d] \\
 \mathcyr{Sh}^1_{\mathcal{A}_{\eta,(j-1)}}/x_j\mathcyr{Sh}^1_{\mathcal{A}_{\eta,(j-1)}} \ar@{=}[u]
  & \times  &
  \mathcyr{Sh}^2_{\mathcal{T}^*_{\eta,(j-1)}}[x_j]
  \ar[rr]_(0.55){\langle \, , \, \rangle_\mathrm{PT}} \ar@{->>}[u]_{\mathrm{pr}_{\mathcyr{Sh}^2}} & &
  \mathbb{Q}_p/\mathbb{Z}_p.}
 \end{align*}
 \end{pro}
 
 We readily observe that Proposition~\ref{prop:commutes} is a direct consequence of
 Proposition~\ref{prop:coupling}. As we remarked after Proposition~\ref{prop:commutes}, 
 the right square of the diagram in Proposition~\ref{prop:commutes} is obvious and 
 we only need to check the commutativity of the left square of the diagram in Proposition~\ref{prop:commutes}. 
 Let $f\colon \mathcyr{Sh}^1_{\mathcal{A}_{\eta,(j-1)}}\rightarrow
 \mathbb{Q}_p/\mathbb{Z}_p$ be an arbitrary element of
 $(\mathcyr{Sh}^1_{\mathcal{A}_{\eta,(j-1)}})^\vee[x_j]$
 and $a$ an arbitrary element of
 $\mathcyr{Sh}^1_{\mathcal{A}_{\eta,(j-1)}}/x_j\mathcyr{Sh}^1_{\mathcal{A}_{\eta,(j-1)}}$. Then
 there exists a unique element $b$ of $\mathrm{Coim}(\delta_1)$
 satisfying $\delta_1(b)=a$. By definition the duality isomorphisms
 $\Phi_\mathrm{local}$ and $\Phi_\mathrm{PT}$ are characterised by
 the relations
 \begin{align} \label{eq:dual_hom}
  \delta^\vee_1(f)(b)&=\langle b, \Phi_\mathrm{local}(\delta^\vee_1(f))\rangle_\mathrm{local},
  & f(a) &=\langle a,\Phi_\mathrm{PT}(f)\rangle_\mathrm{PT}.
 \end{align}
 
Proposition~\ref{prop:coupling} thus provides the following equality:
 \begin{align*}
\langle a, \Phi_\mathrm{PT}(f)\rangle_\mathrm{PT}
  &=f(a)=f(\delta_1(b))=\delta^\vee_1(f)(b)    \\
  &=\langle b,\Phi_\mathrm{local}(\delta^\vee_1(f))\rangle_\mathrm{local}  \\
  &=\langle \delta_1(b),
  \delta_2(\Phi_\mathrm{local}(\delta^\vee_1(f)))^{\sim}\rangle_\mathrm{PT}
  \qquad (\text{here we apply Proposition~\ref{prop:coupling}})\\
  &=\langle a,\delta_2(\Phi_\mathrm{local}(\delta^\vee_1(f)))^{\sim}\rangle_\mathrm{PT}. 
 \end{align*}
The perfectness of the Poitou-Tate pairing thus implies the equality
\begin{align} \label{eq:perf}
 \Phi_\mathrm{PT}(f)=\delta_2(\Phi_\mathrm{local}(\delta^\vee_1(f)))^{\sim}.
\end{align}
 Since the connecting homomorphism $\boldsymbol{\delta}_\mathrm{snake}$ appearing in
 Proposition~\ref{prop:commutes} is decomposed
 as $\boldsymbol{\delta}_\mathrm{snake}=\delta_2\circ \Phi_\mathrm{local}\circ
 \delta^\vee_1$ due to (\ref{ses:snake}) and (\ref{dg:delta_tilde}),
 we finally deduce the desired equality
 \begin{align*}
  \mathrm{pr}_{H^2}\circ\Phi_\mathrm{PT}(f)=\delta_2(\Phi_\mathrm{local}(\delta^\vee_1(f)))=\boldsymbol{\delta}_\mathrm{snake}(f)
 \end{align*}
 by applying $\mathrm{pr}_{\mathcyr{Sh}^2}(=\mathrm{pr}_{H^2})$ to the both sides of (\ref{eq:perf}).  
 The final equality shows the commutativity of the left square of the diagram in Proposition~\ref{prop:commutes}. 
  
 \medskip
 The proof of Proposition~\ref{prop:coupling} requires a little lengthy computation. We first calculate the cohomology classes $\delta_1((s_v)_{v\in S})$ and $\delta_2((t_v)_{v\in S})$ explicitly, and then check that the value $\langle \delta_1((s_v)_{v\in S}), \delta_2((t_v)_{v\in S})^\sim \rangle_\mathrm{PT}$  coincides with $\langle (s_v)_{v\in S}, (t_v)_{v\in S}\rangle_\mathrm{local}$ by a direct computation. In order to estimate the value $\langle \delta_1((s_v)_{v\in S}), \delta_2((t_v)_{v\in S})^\sim \rangle_\mathrm{PT}$,
 we utilise the explicit formula of the Poitou-Tate pairing
 $ \langle \, , \, \rangle_{\mathrm{PT}}\colon \mathcyr{Sh}^1_{\mathcal{A}_{\eta,(j-1)}} \times \mathcyr{Sh}^2_{\mathcal{T}^*_{\eta,(j-1)}}\rightarrow \mathbb{Q}_p/\mathbb{Z}_p$, which we recall below (see \cite[Section~3]{Tate} for details\footnote{Indeed the explicit description of the Poitou-Tate pairing is proposed only for finite Galois modules in \cite{Tate} and other literatures, but it is straightforward to justify the same description under our settings by the standard limit argument based upon Tate's theorem on the inverse limits of Galois cohomology groups.}).  In the rest of the article, we denote by $C^\bullet(G,M)$ the standard cochain complex of a continuous $G$-module $M$ over a profinite group $G$, and by $d_M$ (or just $d$ if it is clear from the context) its coboundary homomorphism. Let $[f]$ be an element of $\mathcyr{Sh}^2_{\mathcal{T}^*_{\eta,(j-1)}}$ represented by a continuous 2-cocycle $f$ in $C^2(\mathrm{Gal}(F_S/F), \mathcal{T}^*_{\eta,(j-1)})$, and $[f']$ an element of $\mathcyr{Sh}^1_{\mathcal{A}_{\eta,(j-1)}}$ represented by a continuous 1\nobreakdash-cocycle $f'$ in $C^1(\mathrm{Gal}(F_S/F), \mathcal{A}_{\eta,(j-1)})$. For each $v$ in $S$, we denote by $f_v$ the restriction of $f$ to the decomposition group $D_v$ at $v$. 
Since $[f]$ is locally trivial at $v$, 
there exists a 1\nobreakdash-cochain $g_v$ in $C^1(D_v, \mathcal{T}^*_{\eta,(j-1)})$ satisfying $dg_v=f_v$ . 
Recall that $F$ has no real places, which implies the triviality of the global cohomology group $H^3(F_S/F, \mu_{p^\infty})$ (see \cite[(8.6.10) (ii)]{NSW} for example). Thus there exists a 2\nobreakdash-cochain $h$ in $C^2(\mathrm{Gal}(F_S/F), \mu^{p^\infty})$ satisfying $dh=f'\cup f$. Then the value $\langle [f'], [f]\rangle_{\mathrm{PT}}$ is explicitly calculated as
 \begin{align} \label{eq:explicit_PT}
  \langle [f'], [f] \rangle_\mathrm{PT} = \sum_{v\in S} \{- f' _v \cup g_v -h_v \}_v
 \end{align}
 where, for each $v$ in $S$, $h_v$ denotes the restriction of $h$ to the decomposition subgroup $D_v$, and $\{ \, \cdot \,\}_v \colon H^2(F_v, \mu_{p^\infty}) \xrightarrow{\sim} \mathbb{Q}_p/\mathbb{Z}_p$ is the invariant isomorphism at $v$.

 Now let us calculate $\delta_1((s_v)_{v\in S})$ and $\delta_2((t_v)_{v\in S})$. 

 \paragraph{\bf \em Calculation of $\delta_1((s_v)_{v\in S})$.} Let
  $(s_v)_{v\in S}$ be an arbitrary element of
  $\mathrm{Coim}(\delta_1)$. There exists an element $[z]$ of
  $\mathrm{Sel}^{\Sigma,\mathrm{str}}_{\mathcal{A}_{\eta,(j-1)}}$, represented by a 1-cocycle $z$ in $C^1(\mathrm{Gal}(F_S/F), \mathcal{A}_{\eta,(j-1)})$, which is sent to $(s_v)_{v\in S}$ by
  $\mathrm{Res}^1_{(j-1)}$.
  Then, following the definition of the connecting homomorphism $\delta_1$ given at the diagram (\ref{dg:cokernel_dual}), 
  the evaluation of $\delta_2$ at $(s_v)_{v\in S}$ is calculated as
  \begin{align*}
   \delta_1((s_v)_{v\in S})=[x_jz] \mod x_j\mathcyr{Sh}^1_{\mathcal{A}_{\eta,(j-1)}}.
  \end{align*}
 For each $v$ in $S$, we denote by $z_v$ the restriction of the 1-cocycle $z$ to the decomposition group $D_v$ at $v$.

\medskip

 \paragraph{\bf \em Calculation of $\delta_2((t_v)_{v\in S})$.}

  Let $(t_v)_{v\in S}$ be an arbitrary element of $\mathrm{Im}\,
  (\tilde{\delta}^\vee_1)$. We calculate
  $\delta_2((t_v)_{v\in S})$ by the diagram chasing on (\ref{dg:big}) (or on \eqref{dg:right}). For each $v$ in $S$, take a
  cocycle $w_v$ of $C^1(F_v, \mathcal{T}^*_{\eta,(j-1)})$ so
  that the element represented by the cohomology class $([w_v])_{v\in S}$ in $Q^*_{(j-1)}$ is sent to $(t_v)_{v\in S}$ under the natural surjection $Q^*_{(j-1)}/x_jQ^*_{(j-1)}  \twoheadrightarrow \mathcal{C}^{(j-1)}$. Here we use the same symbol $(t_v)_{v\in S}$ for its image in $\mathcal{C}^{(j-1)}$. 
  Denoting by $\bar{w}_v$ the image
  of  $w_v$ under the surjection $C^1(F_v,
  \mathcal{T}^*_{\eta,(j-1)})\twoheadrightarrow
  C^1(F_v,\mathcal{T}^*_{\eta,(j)})$, we have
  $\pi^\mathcal{Q}_j(([w_v])_{v\in S})=([\bar{w}_v])_{v\in S}$.
  Then the diagram \eqref{dg:right} implies that there exists a
  unique element $\varrho$ of $\mathrm{Coim}(\phi_{\mathcal{L}_\mathrm{str}^{*,(j)}})$ satisfying $\phi_{\mathcal{L}^{*,(j)}_\mathrm{str}}(\varrho)=([\bar{w}_v])_{v\in S}$.
  The image of $\varrho$ in $\overline{\mathcal{H}}^{2,*}_j$ is none other than $\delta_2((t_v)_{v\in S})$ by definition.

  We do a little more precise computation. Let us take an arbitrary lift $\tilde{\varrho}$ of $\varrho$ with respect to the canonical surjection $H^{1,*}_{(j)}\twoheadrightarrow \mathrm{Coim}(\phi_{\mathcal{L}^{*,(j)}_\mathrm{str}})$, which is represented by a 1\nobreakdash-cocycle $\tilde{w}$ in $C^1(\mathrm{Gal}(F_S/F), \mathcal{T}^*_{\eta,(j)})$. Then the diagram \eqref{dg:left} implies that the image of $\text{\sf \textgreek{d}}^1_{(j)}([\tilde{w}])$ in $\overline{\mathcal{H}}^{2,*}_j$ coincides with $\delta_2((t_v)_{v\in S})$. In other words, $\text{\sf \textgreek{d}}^1_{(j)}([\tilde{w}])$ is regarded as a lift $\delta_2((t_v)_{v\in S})^\sim$ of $\delta_2((t_v)_{v_\in S})$. The element $\text{\sf \textgreek{d}}^1_{(j)}([\tilde{w}])$ is, however, calculated in the usual manner; namely, if we take a 1-cochain (not necessarily a 1-cocycle) $\hat{w}$ of $C^1(\mathrm{Gal}(F_S/F), \mathcal{T}^*_{\eta,(j-1)})$ which is sent to $\tilde{w}$ under the natural surjection $C^1(\mathrm{Gal}(F_S/F),\mathcal{T}^*_{\eta,(j-1)})\twoheadrightarrow C^1(\mathrm{Gal}(F_S/F),\mathcal{T}^*_{\eta,(j)})$, the image $\text{\sf \textgreek{d}}^1_{(j)}([\tilde{w}])$ of $[\tilde{w}]$ coincides with $[x_j^{-1}d\hat{w}]$. Note that the lift $\delta_2((t_v)_{v\in S})^\sim$ of $\delta_2((t_v)_{v\in S})$ is determined uniquely modulo $\mathrm{Sel}^{*,\mathrm{str}}_{\mathcal{T}^*_{\eta,(j)}}$, but this ambiguity corresponds to that of the choices of lifts $\tilde{\varrho}=[\tilde{w}]$ of $\varrho$. Therefore each $\delta_2((t_v)_{v\in S})^\sim$ is obtained in this procedure, or more specifically, we obtain an explicit description of an arbitrary lift $\delta_2((t_v)_{v\in S})^\sim$ of $\delta_2((t_v)_{v\in S})$ as $[x_j^{-1}d\hat{w}]$, for an appropriate $\hat{w}$ as above.

  \medskip
   \paragraph{\bf \em Replacement of local cochains.}
  As we have obtained explicit descriptions of $\delta_1((s_v)_{v\in S})$ and $\delta_2((t_v)_{v\in S})^\sim$, we evaluate these elements 
  under the Poitou-Tate pairing:
  \begin{align} \label{eq:pairing_cocycle}
   \langle \delta_1((s_v)_{v\in S}), \delta_2((t_v)_{v\in S})^\sim \rangle_\mathrm{PT} =\langle [x_j z], [x_j^{-1}d\hat{w}]\rangle_\mathrm{PT}.
  \end{align}
  In order to apply the explicit formula \eqref{eq:explicit_PT} to \eqref{eq:pairing_cocycle}, we need to find a 1-cochain $\check{w}_v$ of $C^1(F_v, \mathcal{T}^*_{\eta,(j-1)})$ satisfying  the equality $d\check{w}_v=x_j^{-1}d\hat{w}_v$ {\em of cocycles} for each $v$ in $S$, where $\hat{w}_v$ denotes the restriction of $\hat{w}$ to the decomposition group $D_v$. In order to find such a nice 1-cochain $\check{w}_v$, we first study the relation between $\hat{w}_v$ and $w_v$, the cocycle which we first took in the computation of $\delta_2((t_v)_{v\in S})$ above.

  \begin{lem} \label{lem:relation_w_v}
   For each $v$ in $S$, there exist 1-cochains $\ell_v$ and $c_v$ in $C^1(F_v, \mathcal{T}^*_{\eta,(j-1)})$ satisfying the following two properties.
   \begin{enumerate}[label=$(\arabic*)$]
    \item the image of $\ell_v$ under the canonical surjection $C^1(F_v,\mathcal{T}^*_{\eta,(j-1)})\twoheadrightarrow C^1(F_v,\mathcal{T}^*_{\eta,(j)})$ is a $1$-cocycle representing a cohomology class contained in $L_\mathrm{str}(F_v,\mathcal{T}^*_{\eta,(j)});$
    \item the cochains $\hat{w}_v$ and $w_v+\ell_v+x_jc_v$ coincide modulo coboundaries.
   \end{enumerate}
  \end{lem}
  
\begin{proof}
 Let $\tilde{w}_v$ denote the restriction of the cocycle $\tilde{w}$ to the decomposition group $D_v$. Then by construction, $\tilde{w}_v$ satisfies the equation
 \begin{align*}
  ([\tilde{w}_v])_{v\in S} (=\phi_{\mathcal{L}^{*,(j)}_\mathrm{str}}([\tilde{w}]))=([\bar{w}_v])_{v\in S}
 \end{align*}
 {\em in $Q^*_{(j)}$}, and hence there exists a 1-cocycle $\bar{\ell}_v$ of $C^1(F_v,\mathcal{T}^*_{\eta,(j)})$, representing a cohomology class contained in $L_\mathrm{str}^*(F_v, \mathcal{T}^*_{\eta,(j)})$, such that $\tilde{w}_v$ coincides with $\bar{w}_v+\bar{\ell}_v$ modulo coboundaries.
 Let $\ell_v$ be an arbitrary 1-cochain in $C^1(F_v,\mathcal{T}^*_{\eta,(j-1)})$ which is sent to $\bar{\ell}_v$ under the natural surjection $C^1(F_v,\mathcal{T}^*_{\eta,(j-1)})\twoheadrightarrow C^1(F_v,\mathcal{T}^*_{\eta,(j)})$. Again by construction, $\hat{w}_v$ and $w_v+\ell_v$ has the same image $\bar{w}_v+\bar{\ell}_v$ in $C^1(F_v,\mathcal{T}^*_{\eta,(j)})$. We now readily verify the existence of a cochain $c_v$ satisfying the claim due to the natural exact sequence
 \begin{align*}
  C^1/B^1(F_v, \mathcal{T}^*_{\eta,(j-1)}) \xrightarrow{\times x_j}   C^1/B^1(F_v, \mathcal{T}^*_{\eta,(j-1)}) \twoheadrightarrow   C^1/B^1(F_v, \mathcal{T}^*_{\eta,(j)}) \rightarrow 0.
 \end{align*}
 Here $C^1/B^1(F_v, \mathcal{T}^*_{\eta,(k)})$ denotes the quotient of $C^1(F_v, \mathcal{T}^*_{\eta,(k)})$ with respect to its submodule consisting of all coboundaries.
\end{proof}

Lemma~\ref{lem:relation_w_v} enables us to replace $\hat{w}_v$ by $w_v+\ell_v+x_jc_v$ in the computation of the Poitou-Tate pairing \eqref{eq:pairing_cocycle}.

  \medskip
   \paragraph{\bf \em Computation of the Poitou-Tate pairing}

  We are ready to calculate the evaluation of the Poitou-Tate pairing \eqref{eq:pairing_cocycle}. On the one hand, we readily observe that the local cocycle $x_j^{-1}d\hat{w}_v$ is redescribed as
  \begin{align} \label{eq:dl_v}
   x_j^{-1} d\hat{w}_v = x_j^{-1}d(w_v+\ell_v+x_jc_v)=x_j^{-1}d\ell_v+dc_v
  \end{align}
  up to coboundaries by virtue of Lemma~\ref{lem:relation_w_v}. On the other hand, the cohomology class $([x_j^{-1}d\hat{w}_v])_{v\in S}$ is the restriction of the element $\delta_2((t_v)_{v\in S})^\sim=[x_j^{-1}d\hat{w}]$, which is an element of $\mathcyr{Sh}^2_{\mathcal{T}^*_{\eta,(j-1)}}$ by construction. This observation implies that each $x_j^{-1}d\hat{w}_v=x_j^{-1}d \ell_v+dc_v$ itself is a coboundary. Let us take for every $v$ in $S$ a 1-cochain $\lambda_v$ in $C^1(F_v,\mathcal{T}^*_{\eta,(j)})$  satisfying $x_j^{-1}d\ell_v=d\lambda_v$ so that $x_j^{-1}d\hat{w}_v=d(\lambda_v+c_v)$ holds up to coboundaries. Moreover the equation
  \begin{align*}
  x_jz \cup x_j^{-1}d\hat{w}=z\cup d\hat{w}=-d(z\cup \hat{w})
  \end{align*}
  holds since $z$ is a 1-cocycle. Applying the explicit formula \eqref{eq:explicit_PT} for $f=x_j^{-1}d\hat{w}$, $f'=x_jz$, $g_v=\lambda_v+c_v$ and $h_v=-(z\cup \hat{w})_v=-z_v\cup \hat{w}_v$, we calculate the value of the Poitou-Tate pairing $\langle \delta_1((s_v)_{v\in S}), \delta_2((t_v)_{v\in S})^\sim\rangle_\mathrm{PT}$ as follows:
     \begin{align} \label{eq:calculation_PT}
    \langle \delta_1((s_v)_{v\in S}), \delta_2((t_v)_{v\in
    S})^\sim \rangle_\mathrm{PT} &= \sum_{v\in S} \left\{ -[x_jz_v]\cup [\lambda_v+c_v]-(-z_v\cup \hat{w}_v) \right\}_v \\ \nonumber
    &= \sum_{v\in S} \left\{ -[z_v]\cup x_j[\lambda_v+c_v] +[z_v]\cup [w_v+\ell_v+x_jc_v] \right\}_v  \\ \nonumber
    &= \sum_{v\in S} \{ [z_v]\cup [w_v]\}_v +\sum_{v\in S} \{ [z_v]\cup [\ell_v-x_j\lambda_v] \}_v \\ \nonumber
    &=\sum_{v\in S} \langle s_v, t_v \rangle_v +\sum_{v\in S} \langle s_v,  [\ell_v'] \rangle_v. \nonumber
     \end{align}
     Here we set $\ell_v'=\ell_v-x_j\lambda_v$, which is indeed a {\em cocycle} as one readily checks:
     \begin{align*}
      d\ell_v'=d\ell_v-x_jd\lambda_v=d\ell_v-x_j(x_j^{-1}d\ell_v)=0.
     \end{align*}
    We also note that the last equality of \eqref{eq:calculation_PT} is just the definition of the local Tate duality: $\langle [a_v],[b_v]\rangle_v=\{[a_v]\cup [b_v]\}_v$. 

    \medskip
    
 \paragraph{\bf \em Completion of the proof.}  

Let us complete the verification of \eqref{eq:PTvsLocal}. Due to the previous computation \eqref{eq:calculation_PT}, it suffices to show that $\langle s_v, [\ell_v']\rangle_v$ equals $0$ for each place $v$ in $S$. Since $s_v$ is an element of $L_\mathrm{str}(F_v, \mathcal{A}_{\eta,(j-1)})[x_j]$ and $L^*_\mathrm{str}(F_v, \mathcal{T}^*_{\eta,(j-1)})$ is defined as the orthogonal compliment of $L_\mathrm{str}(F_v, \mathcal{A}_{\eta,(j-1)})$ with respect to the local Tate pairing $\langle \, , \, \rangle_v$, the triviality of $\langle s_v, [\ell_v']\rangle_v$ is reduced to the following claim:
 \begin{claim}
  For each place $v$ in $S$, the $1$\nobreakdash-cocycle $\ell_v'$ represents a cohomology class contained in $L^*_\mathrm{str}(F_v, \mathcal{T}^*_{\eta,(j-1)})+x_jH^1_\mathrm{cts}(F_v,\mathcal{T}^*_{\eta,(j-1)})$. 
 \end{claim}

 The following lemma is a technical key to the verification of the claim.
Let
  \begin{align*}
   \text{\sf \textgreek{d}}^1_{v,(j)} \colon H^1_\mathrm{cts}(F_v, \mathcal{T}^*_{\eta,(j)}) \rightarrow H^2_\mathrm{cts}(F_v, \mathcal{T}^*_{\eta,(j-1)})
  \end{align*}
  denote the connecting homomorphism of the cohomological long exact sequence associated to (\ref{equation:T}).
  
 \begin{lem} \label{lem:surj_local}
  Let $v$ be a finite place in $S$ and let $\bar{\xi}_v$ be an arbitrary cocycle 
 of $C^1(F_v, \mathcal{T}^*_{\eta,(j)})$ representing a cohomology class contained in $L^*_\mathrm{str}(F_v, \mathcal{T}^*_{\eta,(j)})$. Further assume that the image of the cohomology class $[\bar{\xi}_v]$ under $\text{\sf \textgreek{d}}^1_{v,(j)}$ is trivial when $v$ is a place belonging to $\Sigma^c_p$. 
 Then there exists a lift
 $\tilde{\xi}_v$ of $\xi_v$ to $C^1(F_v, \mathcal{T}^*_{\eta,(j-1)})$ such that $\tilde{\xi}_v$ is also a cocycle and it represents a cohomology class in $L_\mathrm{str}^*(F_v, \mathcal{T}^*_{\eta,(j-1)})$.
\end{lem}
    
\begin{proof}[Proof of Lemma~$\ref{lem:surj_local}$]
 The statement is nontrivial only when $v$ is a place in $\Sigma^c_p$ or a place not lying above $p$ such that the inertia group $I_v$ acts trivially on $\mathcal{T}^*_{\eta,(k)}$;
 otherwise the local condition $L^*_\mathrm{str}(F_v,\mathcal{T}^*_{\eta,(k)})$ is trivial and there is nothing to prove.
 In the former case, the local condition $L_\mathrm{str}^*(F_v, \mathcal{T}^*_{\eta,(k)})$ coincides with the whole cohomology group $H^1_\mathrm{cts}(F_v, \mathcal{T}^*_{\eta,(k)})$. The assumption $\text{\sf \textgreek{d}}^1_{v,(j)}([\bar{\xi}_v])=0$ then guarantees the existence of a cohomology class $[\tilde{\xi}'_v]$ of $H^1_\mathrm{cts}(F_v, \mathcal{T}^*_{\eta,(j-1)})$ whose image in $H^1_\mathrm{cts}(F_v, \mathcal{T}^*_{\eta,(j)})$ coincides with $[\bar{\xi}_v]$. In the latter case, the local condition $L_\mathrm{str}^*(F_v, \mathcal{T}^*_{\eta,(k)})$ is isomorphic to $H^1_\mathrm{cts}(D_v/I_v, \mathcal{T}^*_{\eta,(k)})$. Consider the cohomological long exact sequence associated to (\ref{equation:T}):
 \begin{align*}
  \xymatrix{
  H^1_\mathrm{cts}(D_v/I_v, \mathcal{T}^*_{\eta,(j-1)}) \ar[r]^{(\star)} & H^1_\mathrm{cts}(D_v/I_v, \mathcal{T}^*_{\eta,(j)}) \ar[r]& H^2_\mathrm{cts} (D_v/I_v, \mathcal{T}^*_{\eta,(j-1)}).
  }
 \end{align*}
 Since the procyclic quotient $D_v/I_v\cong \hat{\mathbb{Z}}$ has cohomological dimension one, the second cohomology group  $H^2_\mathrm{cts}(D_v/I_v, \mathcal{T}^*_{\eta,(j-1)})$ should be trivial and hence the natural map $(\star)$ is surjective.

 We thus find in both cases a cocycle $\tilde{\xi}'_v$ of $C^1(F_v, \mathcal{T}^*_{\eta,(j-1)})$ representing a cohomology class contained in $L^*_\mathrm{str}(F_v, \mathcal{T}^*_{\eta,(j-1)})$ such that its image in $C^1(F_v, \mathcal{T}^*_{\eta,(j)})$ coincides with $\bar{\xi}_v$ up to coboundaries. However,
 the natural surjection $\mathcal{T}^*_{\eta,(j-1)}\twoheadrightarrow \mathcal{T}^*_{\eta,(j)}$ induces a surjection $C^\bullet(F_v,\mathcal{T}^*_{\eta,(j-1)})\twoheadrightarrow C^\bullet(F_v,\mathcal{T}^*_{\eta,(j)})$ on the cochain complexes, and hence we readily take a lift of an arbitrary coboundary in $C^1(F_v, \mathcal{T}^*_{\eta,(j)})$ from the submodule of coboundaries in $C^1(F_v,\mathcal{T}^*_{\eta,(j-1)})$. This enables us to construct a desired cocycle $\tilde{\xi}_v$ by modifying $\tilde{\xi}'_v$ in the same cohomology class.
\end{proof}

\begin{proof}[Proof of Proposition~$\ref{prop:coupling}$] 

 We first recall that the cocycle $\bar{\ell}_v$ satisfies $\text{\sf \textgreek{d}}^1_{v,(j)}([\bar{\ell}_v])=0$ for every place $v$ in $S$. Indeed, $\text{\sf \textgreek{d}}^1_{v,(j)}([\bar{\ell}_v])$ is calculated as $[x_j^{-1}d\ell_v]$ due to the definition of the connecting homomorphism $\text{\sf \textgreek{d}}^1_{v,(j)}$. Then, the cohomology class $[x_j^{-1}d\ell_v]$ is trivial by the equation \eqref{eq:dl_v} and the fact that $[x_j^{-1}d\hat{w}]=\delta_2((t_v)_{v\in S})$ is an element of $\mathcyr{Sh}^2_{\mathcal{T}^*_{\eta,(j-1)}}$.  Therefore, by applying Lemma~\ref{lem:surj_local}, we find a 1-cocycle $\tilde{\ell}_v$ in $C^1(F_v, \mathcal{T}^*_{\eta,(j-1)})$  which is a lift of $\bar{\ell}_v$ and represents a cohomology class contained in $L^*_\mathrm{str}(F_v, \mathcal{T}^*_{\eta,(j-1)})$. It is straightforward to see that $\ell_v'$ is also a lift of $\bar{\ell}_v$ by definition, and thus $\tilde{\ell}_v-\ell_v'$ is contained in $x_jC^(F_v,\mathcal{T}^*_{\eta,(j-1)})$ due to the exact sequence
 \begin{align*}
  0 \rightarrow C^1(F_v,\mathcal{T}^*_{\eta,(j-1)}) \xrightarrow{\times x_j} C^1(F_v,\mathcal{T}^*_{\eta,(j-1)}) \twoheadrightarrow C^1(F_v,\mathcal{T}^*_{\eta,(j)}) \rightarrow 0.
 \end{align*}
This calculation verifies the claim above, and hence the proof of Proposition~\ref{prop:coupling} is completed.
\end{proof}


\begin{thebibliography}{BlRo94}
 \bibitem[AS86]{AS} A.~Ash and G.~Stevens, {\em Modular forms in
	       characteristic $\ell$ and special values of their
	       $L$-functions}, Duke Math.\ J., \textbf{53}, no.~3,
	       849--868 (1986).

 \bibitem[Bl86]{Blasius1} D.~Blasius, {\em On the Critical Values of Hecke $L$-Series}, Ann.\ of Math.\ (2), \text{124}, no.~1, 23--63 (1986).
	       
 \bibitem[Bl06]{Blasius} D.~Blasius, {\em Hilbert modular forms and the
	       Ramanujan conjecture}, in: {\em Noncommutative geometry
	       and number theory}, Aspects Math., \textbf{E37}, Vieweg,
	       Wiesbaden, 35--56 (2006).
 \bibitem[BlRo94]{BR} D.~Blasius and J.~D.~Rogawski, {\em Motives for
	     Hilbert modular forms}, Invent.\ Math., \textbf{114}, 
	     Issue~1, 55--87 (1994).
 \bibitem[BrLa84]{BL} J.-L.~Brylinski and J.-P.~Labesse, {\em Cohomologie
	       d'intersection et Fonctions $L$ de certaines vari\'et\'es
	       Shimura}, Ann.\ Sci.\ Ec.\ Norm.\ Super.~(4) \textbf{17}, 
	       361--412 (1984).
 \bibitem[Ca86]{carayol} H.~Carayol, {\em Sur les repr\'esentations $p$-adiques associ\'ees aux formes modulaires de Hilbert}, 
 Ann.\ Sci.\ Ec.\ Norm.\ Super.\ (4) \textbf{19}, 409--468 (1986).
\bibitem[CS05]{CS} J.~Coates and R.~Sujatha, {\em Fine Selmer groups of
      elliptic curves over $p$-adic Lie extensions}, Math.\ Ann.,
      \textbf{331}, 809--839 (2005).
 \bibitem[Da94]{dabrowski} A.~Dabrowski, {\em $p$-adic $L$-functions of Hilbert modular forms}, Ann.\ Inst.\ Fourier, \textbf{44}, 1025--1041 (1994).
 \bibitem[De79]{Deligne} P.~Deligne, {\em Valeurs de
		fonctions~$L$ et p\'eriodes d'int\'egrales}, in: {\em
		Automorphic forms, representations, and $L$-functions}
		(Part 2), Proc.\ Sympos.\ Pure Math., XXXIII,
		Providence, R.I., Amer.\ Math.\ Soc., 313--346 (1979). 
 \bibitem[Di13]{dimitrov} M.~Dimitrov, {\em Automorphic symbols, $p$-adic $L$-functions and ordinary cohomology of Hilbert modular varieties},  Amer.\ J.\ Math., \textbf{135}, no.~4 (2013).
 \bibitem[Ge75]{gelbert} S.~Gelbart, {\em Automorphic forms on ad\`{e}le groups}, Annals of Mathematics Studies, No. 83. Princeton University Press, 1975. 
\bibitem[Gr76]{gr-iw} R.~Greenberg, {\em  On the Iwasawa invariants of
             totally real number fields}, Amer.\ J.\ Math., \textbf{98},
	      no.~1, 263--284 (1976).
\bibitem[Gr78]{Gr_Galois} R.~Greenberg, {\em  On the structure of
	      certain Galois Groups}, Invent.\ Math., \textbf{47},
	      85--99 (1978).
\bibitem[Gr94]{gr-mot} R.~Greenberg, {\em Iwasawa theory
         and $p$-adic deformations of motives}, in: {\em Motives} 
         (Seattle, WA, 1991) Proc.\ Sympos.\ Pure Math.,
         \textbf{55}, 193--223, Part 2, Amer.\ Math.\ Soc., 
         Providence, RI (1994).
\bibitem[Gr06]{gr-coh} R.~Greenberg, {\em On the structure 
	     of certain Galois cohomology groups},  
	     Doc.\ Math.\ Extra Vol., (4) 
	     {\em John H.\ Coates Sixtieth' Birthday}, 335--391 (2006).
 \bibitem[Gr10]{gr-surj} R.~Greenberg, {\em Surjectivity of 
	      the global-to-local map defining a Selmer group}, 
	      Kyoto J.\ of Math., Vol.~\textbf{50}, No.~4, 853--888 (2010).
 \bibitem[Gr]{gr-sel} R.~Greenberg, {\em On the structure of Selmer groups}, preprint, 
available at his webpage: {\tt
	     http://www.math.washington.edu/\~{}greenber/research.html}.
 \bibitem[Ha85]{Hasse} H.~Hasse, {\em \"Uber die Klassenzahl abselscher
		Zahlk\"orper}, Springer-Verlag (1985).
 \bibitem[He82]{Henniart} G.~Henniart, {\em Repr\'esentations
	      $\ell$-adiques ab\'eliennes}, in: {\em Seminaire de
	      theorie des nombres}, Seminaire
	      Delange-Pisot-Poitou, Paris 1980--81, Birkh\"auser,
	      107--126 (1982).  
 \bibitem[Hi88]{hecke} H.~Hida, {\em On $p$-adic Hecke algebras for
	       $GL(2)$ over totally real fields}, Ann.\ of Math.,
	       \textbf{128}, 295--384 (1988).
	       
 \bibitem[Hi89]{nearlyhecke} H.~Hida, {\em Nearly ordinary Hecke algebras and Galois representations of several variables}, in: {\em Algebraic analysis, geometry, and number theory} (Baltimore, MD, 1988), 115--134, Johns Hopkins Univ. Press, Baltimore, MD (1989). 
 \bibitem[Hi06]{HIT} H.~Hida, {\em Hilbert modular forms and Iwasawa
		theory}, Oxford Mathematical Monographs, Oxford
		University Press (2006).
\bibitem[HT93]{HT-katz} 
H.~Hida and J.~Tilouine, {\em Anticyclotomic Katz $p$-adic $L$-functions 
and congruence modules}, Ann.\ Sci.\ \'Ec.\ Norm.\ Sup\'er.\ (4)
\textbf{26}, no.~2, 189--259 (1993).
 \bibitem[HT94]{HT-aIMC} H.~Hida and J.~Tilouine, {\em On the anticyclotomic main conjecture for CM fields}, Invent.\ Math., \textbf{117}, 
 no.~1, 89--147 (1994).
 \bibitem[Hsi14]{hsi-IMC} M.-L.~Hsieh, {\em Eisenstein congruence on unitary groups and 
Iwasawa main conjecture for CM fields}, J.\ Amer.\ Math.\ Soc., \textbf{27}, no.~3, 753--862 (2014). 
 \bibitem[Kato04]{Kato} K.~Kato, {\em $p$-adic Hodge theory and
		values of zeta functions of modular forms}, in: {\em
		Cohomologies $p$-adiques et applications arithmetiques}
		III, Ast\'erisque \textbf{295}, 117--290 (2004).
 \bibitem[Katz78]{katz} 
N.\ M.~Katz, 
\textit{$p$-adic $L$-functions for CM fields}, 
Invent.\ Math., \textbf{49}, 199--297 (1978). 
 \bibitem[Ki94]{kitagawa} K.~Kitagawa, {\em On standard $p$-adic
	       $L$-functions of families of elliptic cusp forms}, in:
	       {\em $p$-adic monodromy and the Birch and Swinnerton-Dyer
	       conjecture} (Boston, MA, 1991), vol.\ 165 of Contemp.\
	       Math., Amer.\ Math.\ Soc., Providence, RI, 81--110
	       (1994).	       
 \bibitem[LL79]{LL} J.-P.~Labesse and R.~P.~Langlands, {\em $L$-indistinguishability for $SL(2)$}, Canad.\ J.\ Math., \textbf{31}, no.~4, 726--785 (1979).
 \bibitem[Man76]{Manin} J.~I.~Manin, {\em Non-Archimedean integration and $p$-adic Jacquet-Langlands $L$-functions}, Uspehi Mat.\ Nauk., \textbf{31}, 5--54 (1976).
 \bibitem[MTT86]{MTT} B.~Mazur, J.~Tate and J.~Teitelbaum, {\em On $p$-adic analogues of the conjectures of Birch and Swinnerton-Dyer}, Invent.\ Math., \textbf{84}, 1--48 (1986).
		
 \bibitem[MW84]{MazurWiles} B.~Mazur and A.~Wiles, {\em Class fields of abelian extensions of $\mathbb{Q}$}, Invent.\ Math., \textbf{76} no.~2, 179--330 (1984). 
  \bibitem[Mi71]{miyake} T.~Miyake, {\em On automorphic forms on $GL_2$ and Hecke operators}, Ann.\ of Math., Second Series, Vol.~\textbf{94}, no.~ 1, 174--189 (1971). 
  \bibitem[Mok09]{mok} C.-P.~Mok, 
	       {\em The exceptional zero conjecture for
	       Hilbert modular forms}, Compos.\ Math., \textbf{145},
	       no.~1, 1--55 (2009).
  \bibitem[NSW00]{NSW} J.~Neukirch, A.~Schmidt and K.~Wingberg, {\em
		 Cohomology of number fields}, Grundlehren der Math.,
		 Wissenschaften, \textbf{323}, Springer (2000).
 \bibitem[Och05]{och_euler} T.~Ochiai, {\em Euler system for Galois
		deformations}, Ann.\ Inst.\ Fourier (Grenoble)
		\textbf{55}, no.~1, 113--146 (2005).
 \bibitem[Och12]{ochiai} T.~Ochiai, {\em Several variable $p$-adic $L$-functions for Hida families of 
 Hilbert modular forms}, Doc.\ Math., \textbf{17}, 807--849 (2012).
 \bibitem[OP]{OP}
T.~Ochiai and K.~Prasanna, {\em Two-variable Iwasawa theory for Hida families with 
complex multiplication}, in preparation.
 \bibitem[Oh83]{ohta} M.~Ohta, {\em On the zeta function of an abelian
	      scheme over the Shimura curve}, Jpn.\ J.\ Math.,
	      \textbf{9},  1--26 (1983).
 \bibitem[Pa94]{panchishkin} A.~A.~Panchishkin, {\em Motives over totally real fields and $p$-adic $L$-functions}, Ann.\ Inst.\ Fourier, \textbf{44}, 989--1023 (1994).
 \bibitem[PR81]{PR_CM} B.~Perrin-Riou, {\em Groupe de Selmer d'une courbe elliptique \`a multiplication complexe}, Compositio Math.,
 \textbf{43}, no.~3, 387--417 (1981).
 \bibitem[Ri77]{Ribet} K.~A.~Ribet, {\em Galois representations attached
              to eigenforms with nebentypus}, in: {\em Modular functions
              of one variable}: V (Proc.\ Second Internat.\ Conf.,
              Univ.\ Bonn, Bonn, 1976), pp.~17--51, Lecture Notes in
              Math., Vol.~\textbf{601}, Springer, Berlin (1977).
\bibitem[Ro84]{Rohrlich1} D.~Rohrlich, {\em On $L$-functions of elliptic
	      curves and cyclotomic towers}, Invent.\ Math.,
	      \textbf{75}, 409--423 (1984).
\bibitem[Ro88]{Rohrlich2} D.~Rohrlich, {\em $L$-functions and division
	      towers}, Math.\ Ann.,
	      \textbf{281}, 611--632 (1988).
\bibitem[Ru91]{Rubin_CM} K.~Rubin, {\em The {\upshape ``}main
	      conjectures{\upshape''} of Iwasawa theory for imaginary
	      quadratic fields}, Invent.\ Math., \textbf{103}, no.~1,
	      25--68 (1991).
 \bibitem[Ru00]{Rubin} K.~Rubin, {\em Euler systems}, Annals of
              Mathematics Studies, vol.~147, Princeton: Princeton
              University Press (2000). 
 \bibitem[Sc86]{schappacher} N.~Schappacher, {\em Periods of Hecke characters}, Lecture Notes in Mathematics, 1301. Springer-Verlag, Berlin, 1988.
 \bibitem[Se98]{serre} J.-P.~Serre, {\em Abelian $l$-adic representations
	      and elliptic curves}, Research Notes in Mathematics,
	      \textbf{7}, A K Peters, Ltd., Wellesley, MA (1998). 
 \bibitem[Sh78]{shimura} G.~Shimura, {\em The special values of the zeta functions associated with Hilbert modular forms}, 
	       Duke Math.\ J., \textbf{45}, no.~3, 637--679  (1978).

 \bibitem[Sh94]{shimura_intro} G.~Shimura, {\em Introduction to the arithmetic theory of automorphic functions}, Reprint of the 1971 original, Publications of the Mathematical Society of Japan, \textbf{11}, Princeton University Press, Princeton, NJ (1994).
	       
 \bibitem[St85]{Stevens} G.~Stevens, {\em The cuspidal group and special
	       values of $L$-functions}, Trans.\ Amer.\ Math.\ Soc.,
	       \text{291}, no.~2,519--550 (1985).
\bibitem[SU14]{SU} C.~Skinner and E.~Urban, {\em
	      The Main Conjecture for $GL(2)$}, Invent.\ Math.,
	      \textbf{195}, Issue~1, 1--277 (2014).
 \bibitem[Tat63]{Tate} J.~Tate, {\em Duality theorems in Galois
		cohomology over number fields}, Proc.\ Internat.\ Congr.\ Mathematicians (Stockholm, 1962), Inst.\ Mittag-Leffler, Djursholm, 288--295 (1963).
 \bibitem[Tay89]{taylor1} R.~Taylor, {\em On Galois representations
	      associated to Hilbert modular forms}, Invent.\ Math.,
	      \textbf{98}, 265--280 (1989).
\bibitem[Tay95]{taylor2} R.~Taylor, {\em On Galois representations
	      associated to Hilbert modular forms}: II, in: {\em Elliptic
	      curves, modular forms and Fermat's last theorem}, Ser.\
	      Number Theory, I, Int.\ Press, Cambridge, MA, 185--191 (1995).
 \bibitem[Wa97]{Was} L.~C.~Washington, {\em Introduction to Cyclotomic
	      Fields}, 2nd edition, Graduate Texts in Mathematics,
	       \textbf{83}, Springer-Verlag (1997).
 \bibitem[Wi86]{Wiles_prep} A.~Wiles, {\em On $p$-adic representations for totally real fields}, Ann.\ Math.\ (2) \textbf{123}, 407--456 (1986).
 \bibitem[Wi88]{wiles} A.~Wiles, {\em On ordinary $\lambda$-adic representations associated to modular forms},
	       Invent.\ Math., \textbf{94}, 529--573 (1988).	       
 \bibitem[Win80]{Wintenberger} J.-P.~Wintenberger, {\em Structure galoisienne de limites projectives d'unites locales}, Comp.\ Math., \textbf{42}, no.~1, 89--103 (1980/1981).
\end{thebibliography}
\end{document}